\documentclass[11pt]{article}

\usepackage{pratik}

\begin{document}

\newcommand{\titletext}{Precise Asymptotics of Bagging Regularized M-estimators}

\title{\titletext}

\author{
    \setcounter{footnote}{-1}
    \blfootnote{$^\ast$Corresponding authors.}
    \setcounter{footnote}{1}
    Takuya Koriyama$^\ast$\footremember{chicagobooth}{Booth School of Business, University of Chicago, Chicago, IL 60637, USA.}
    \\ {\small \url{{tkoriyam@uchicago.edu}}}
    \and
    Pratik Patil$^\ast$\footremember{utstats}{Department of Statistics and Data Science, University of Texas, Austin, TX 78712, USA.}
    \\ {\small \url{{pratikpatil@utexas.edu}}}
    \and
    Jin-Hong Du\footremember{hkuids}{Institute of Data Science, University of Hong Kong, Pokfulam, Hong Kong, China.}
    \\ {\small \url{{jinhongd@hku.hk}}}
    \and
    Kai Tan\footremember{stanfordstats}{Department of Statistics, Stanford University, Stanford, CA 94305, USA.}
    \\ {\small \url{{kaitan9@stanford.edu}}}
    \and 
    Pierre {C.} Bellec\footremember{rutgersstats}{Department of Statistics, Rutgers University, New Brunswick, NJ 08854, USA.}
    \\ {\small \url{{pierre.bellec@rutgers.edu}}}
}

\date{\vspace{-15pt}}

\maketitle

\begin{abstract}
We characterize the squared prediction risk of ensemble estimators obtained through subagging (subsample bootstrap aggregating) regularized M-estimators and construct a consistent estimator for the risk. Specifically, we consider a heterogeneous collection of $M \ge 1$ regularized M-estimators, each trained with (possibly different) subsample sizes, convex differentiable losses, and convex regularizers. We operate under the proportional asymptotics regime, where the sample size $n$, feature size $p$, and subsample sizes $k_m$ for $m \in [M]$ all diverge with fixed limiting ratios $n/p$ and $k_m/n$. Key to our analysis is a new result on the joint asymptotic behavior of correlations between the estimator and residual errors on overlapping subsamples, governed through a (provably) contractive nonlinear system of equations. Of independent interest, we also establish convergence of trace functionals related to degrees of freedom in the non-ensemble setting (with $M = 1$) along the way, extending previously known cases for squared loss with ridge and lasso regularizers. 

When specialized to homogeneous ensembles trained with a common loss, regularizer, and subsample size, the risk characterization sheds some light on the implicit regularization effect due to the ensemble and subsample sizes $(M,k)$. For any ensemble size $M$, optimally tuning subsample size yields sample-wise monotonic risk. For the full-ensemble estimator (when $M \to \infty$), the optimal subsample size $k^\star$ tends to be in the overparameterized regime $(k^\star \le \min\{n,p\})$, when explicit regularization is vanishing. Finally, joint optimization of subsample size, ensemble size, and regularization can significantly outperform regularizer optimization alone on the full data (without any subagging).
\end{abstract}

\section{Introduction}

Ensemble methods combine predictions of multiple models to improve predictive accuracy \cite{hastie2009elements}.
Among these methods, bagging (bootstrap aggregating) trains individual models on bootstrapped samples of the dataset and averages their predictions to reduce variance and mitigate overfitting \cite{breiman_1996}. 
A popular variant of bagging, known as subagging (subsample bootstrap aggregating), trains models on random subsamples rather than full bootstrapped samples \cite{buhlmann2002analyzing}.
Apart from the computational advantages, subagging can substantially improve predictive performance, especially in the overparameterized regimes and near model interpolation thresholds \cite{patil2022mitigating}.
In this paper, we analyze the squared prediction risk of subagging of regularized M-estimators trained with convex loss and regularizer.

There is growing interest in understanding the prediction risk asymptotics of ensemble methods, particularly subagging, across different types of predictors and under various data assumptions.
For example, \cite{lejeune2020implicit} study subagging in the context of ordinary least squares regression without any explicit regularization in the underparameterized regime (where the number of subsamples is higher than the number of features).
This is extended to ridge and ridgeless regression by \cite{patil2022bagging} for both the underparameterized and overparameterized regimes (where the number of subsamples is lower than the number of features).
\cite{patil2023generalized} further generalizes these results and identifies explicit equivalence paths between subsampled estimators and ridge regularized estimators.
Beyond linear and ridge regression, \cite{bellec2024asymptotics} examines the behavior of subagging in logistic and Huber regression models without regularization.
In addition to subagging, there has also been considerable work on feature sketching and other ensemble methods. 
For instance, \cite{loureiro2022fluctuations} and \cite{patil2023asymptotically} study feature sketching and ensembling to optimize predictive performance in high-dimensional settings.
For other recent developments in the analysis of ensemble methods and related work details, see \Cref{sec:related_work}.

\begin{table*}[!t]
    \centering
    \caption{
    \textbf{Literature on subagging risk analysis in high dimensions.}
    We summarize the various settings (estimator structure and data structure) of some of the recent works that characterize the prediction risk of subagging regularized M-estimators trained with loss function $\loss$ and regularization function $\reg$ (defined in \Cref{sec:setup}).
    Convex$^{\mysolidcircle}$ denotes that in addition to being convex, we require the loss function to be differentiable and have Lipschitz continuous derivatives.
    By RMT features, we refer to features of the form $\bx = \bSigma^{1/2} \bv$ where $\bv$ contains independent entries of bounded moments (of order $4^+$) that are common in the random matrix theory literature.
    Arbitrary$^{\blacktriangle}$ response $y$ refers to no additional modeling assumption on the response other than bounded moments (of order $4^+$).
    Signal and noise refer to $\btheta$ and $\eps$ in the linear data model: $y = \bx^\top \btheta + \eps$.
    Arbitrary$^{\mysolidsquare}$ signal and noise distributions $F_\theta$ and $F_z$ refer to general marginal distribution on the coordinates of signal $\btheta$ and noise $\eps$, respectively, subject to mild regularity conditions (\Cref{asm:regularity-conditions}).
    }
    \footnotesize
    \begin{tabularx}{\textwidth}
    {C{1cm}C{1cm}C{1.2cm}C{1.6cm}C{1.4cm}C{1.75cm}C{1.6cm}C{1.7cm}C{1.7cm}}
    \toprule
    \textbf{Loss} & \textbf{Penalty} & \textbf{Features} & \textbf{Covariance} & \textbf{Response} & \textbf{Signal} & \textbf{Signal~Dist.} & \textbf{Noise~Dist.} & \multirow{2}{*}{\textbf{Reference}} \\
    ($\loss$) & ($\reg$) & ($\bx$) & ($\bSigma$) & ($\by$) & ($\btheta$) & ($F_\theta$) & ($F_\eps$) & \\
    \midrule
    Square & None & Gaussian & Isotropic & Linear & Random & Gaussian & Gaussian & \cite{lejeune2020implicit} \\
    Square & Ridge & RMT & Anisotropic & Linear & Deterministic & \cellcolor{lightgray!25} & Bnd.~Mom.\ & \cite{patil2022bagging} \\
    Square & Ridge & RMT & Anisotropic & Arbitrary$^{\blacktriangle}$ & \cellcolor{lightgray!25} & \cellcolor{lightgray!25} & \cellcolor{lightgray!25} & \cite{patil2023generalized} \\
    Huber & None & Gaussian & Isotropic & Linear & Deterministic & \cellcolor{lightgray!25} & Arbitrary$^{\mysolidsquare}$ & \cite{bellec2024asymptotics} \\
    Logistic & None & Gaussian & Isotropic & Logistic & Deterministic & \cellcolor{lightgray!25} & \cellcolor{lightgray!25} & \cite{bellec2024asymptotics} \\
    Convex & Ridge & Gaussian & Isotropic & GLM & Random & Gaussian & Arbitrary$^{\mysolidsquare}$ & \cite{clarte2024analysis} \\
    \arrayrulecolor{black!25}\midrule
    \colorhighlight Convex$^{\mysolidcircle}$ & \colorhighlight Convex & \colorhighlight Gaussian & \colorhighlight Isotropic & \colorhighlight Linear & \colorhighlight Random & \colorhighlight Arbitrary$^{\mysolidsquare}$ & \colorhighlight Arbitrary$^{\mysolidsquare}$ & \colorhighlight \Cref{th:nonlinear} \\
    \arrayrulecolor{black}\bottomrule
    \end{tabularx}
    \label{tab:risk-landscape}
\end{table*}

We generalize these previous works by characterizing the prediction risk of subagging a collection of $M$ regularized M-estimators and constructing a consistent estimator for this risk.
We allow the collection to be heterogeneous, consisting of estimators trained with convex differentiable loss function $\loss$ with Lipschitz continuous derivative and convex regularization function $\reg$, which can be non-differentiable and non-strongly convex (that includes $\ell_1$-regularized Huber regression, for instance).
We operate under the proportional asymptotic regime, where the sample size $n$, feature size $p$, and subsample sizes $k_m$ for $m \in [M]$ diverge while maintaining fixed limiting ratios $n/p \to \delta \in (0, \infty)$ and $k_m/n \to c_m \in (0, 1]$.
{Throughout the paper, we refer to $n/p$ as inverse data aspect ratio and $p/n$ as data aspect ratio of the design $\bX \in \RR^{n \times p}$, viewing $n$ as the height and $p$ as the width of the rectangular design matrix $\bX$.}
Our results assume a linear response model $y = \bx^\top \btheta + z$ with isotropic Gaussian features $\bx$, a random signal $\btheta$ with independent coordinates drawn from an arbitrary marginal distribution $F_\theta$, and noise variable $z$ with an arbitrary distribution $F_z$, both subject to a mild regularity condition. 
We summarize our main results below and situate them within the context of recent related work in \Cref{tab:risk-landscape}.

\subsection{Summary of results and paper outline}

A summary of our main results along with an outline for the paper is as follows.

\begin{enumerate}[(a),leftmargin=7mm]
    \item
    \textbf{Risk characterization and estimation.} 
    In \Cref{sec:general-risk-characterization}, we obtain a precise characterization of the squared risk of subagging of regularized M-estimators under proportional asymptotics (\Cref{thm:corr-sigerror-reserror} and \Cref{th:nonlinear}).
    The asymptotic risk is governed by two nonlinear systems of equations (\Cref{sys:general_ensemble-M=1,sys:general_ensemble-M=infty}) that depend on the loss and regularization functions $\loss$, $\reg$, and the limiting subsample ratios $c_m$ for $m \in [M]$ for the component estimators and the limiting inverse data aspect ratio $\delta$. 
    \Cref{sys:general_ensemble-M=1} is a known system that characterizes the asymptotic risk of regularized M-estimators in non-ensemble settings \cite{thrampoulidis2018precise} (see \Cref{sec:related_work} for more details), while \Cref{sys:general_ensemble-M=infty} is a new contribution of this paper.
    Each scalar unknown in the 2-dimensional \Cref{sys:general_ensemble-M=infty} is shown to be the fixed-point equation of a contraction (\Cref{th:existence-uniqueness-sys:general_ensemble-M=infty}-(2)). 
    This property plays a crucial role in proving the existence and uniqueness of the solution to \Cref{sys:general_ensemble-M=infty}.
    This contraction property is also crucial in establishing
    the asymptotic behavior of the inner products between estimator errors and their residuals for estimators trained on overlapping subsamples (\Cref{thm:corr-sigerror-reserror}), leading to the subagged risk asymptotics (\Cref{th:nonlinear}).
    The contractivity, along with a ridge smoothing technique, also allows us to maintain weak assumptions on the regularizer, specifically allowing it to be non-strongly convex.
    Moreover, we also construct a consistent estimator of the ensemble risk (\Cref{thm:risk-estimator-general-subagging} and \Cref{cor:general-ensemble-risk-estimator}), which can be employed for data-adaptive tuning of hyperparameters such as loss functions, regularizers, and subsample sizes.

    \item
    \textbf{Homogeneous ensembles.}
    In \Cref{sec:specific-examples}, we consider homogeneous ensembles (where components are trained on the same $\loss$ and $\reg$ functions and a common subsample size $k$).
    In \Cref{sec:homogeneous-risk-properties}, we analyze (oracle) optimal ensemble optimal risk with optimal ensemble size $\Mstar$ and subsample size $\kstar$.
    We first establish the monotonicity of the risk with respect to the ensemble size $M$ (\Cref{prop:monotonicity-ensemble-size}), illustrating the benefits of ensembling, which leads to $\Mstar = \infty$.
    We then prove that the risk at the optimal subsample size $k^*$ decreases as the limiting data aspect ratio $p/n$ decreases (\Cref{prop:monotonicity-risk}). 
    In particular, this implies that the optimally subsampled risk avoids the typical ``double (or multiple) descents'' observed in regularized M-estimators without subagging.
    In \Cref{sec:examples-square-loss-and-square-penalty}, we specialize our main result to convex regularized least squares (including $\ell_q$-regularized least squares for $q \ge 1$, such as ridge and lasso estimators) and to general M-estimators (including regularized  Huber regression).
    These recover and generalize various known results in the literature (see \Cref{sec:examples-square-loss-and-square-penalty} for more details).
    
    \item
    \textbf{Subagging and overparameterization.}    
    In \Cref{sec:ensembles-interpolators}, we investigate subagging of estimators with vanishing regularization ($\lambda \to 0^{+}$) and also contrast with estimators with optimal explicit regularization (over $\lambda \ge 0$).
    Our first insight is that when subagging estimators without any explicit regularization, the optimal subsample size $k^*$ is in the overparameterized regime, regardless of whether the original data aspect ratio $p/n$ is overparameterized.
    In other words, the optimal subsample size $\kstar$ satisfies $k^* \le \min\{n, p\}$ in such cases.
    We verify this for the lassoless (minimum $\ell_1$-norm interpolator) ensemble (\Cref{fig:optimum_phis_overparameterized-2-optsubsample}).
    This highlights the advantages of overparameterization in subagging in that full-ensemble subsampled lassoless can outperform the optimal lasso on the full data (without any subagging).
    Our second insight is that the joint optimization of the subsample size and explicit regularization parameter can outperform optimizing explicit regularization alone on the full data. 
    We verify this property for ensembles of the lasso, unregularized Huber, and $\ell_1$-regularizer Huber (see \Cref{fig:advantage_subsampling_underparameterized}, \ref{fig:fig_lasso_risk_opt}, and \Cref{sec:optimal-subagging-versus-optimal-regularization}). 
    This highlights the benefits of subagging on top of optimal explicit regularization.
\end{enumerate}

\textbf{Independent results in the non-ensemble setting.}
In the process of characterizing the risk of the subagging ensemble, we also establish the convergence of certain trace functionals (in particular, see \eqref{eq:nu-kappa-convergence} below or the last three rows of \Cref{tab:interpretations_squared_loss}).
This implies that the observable adjustments developed in
\cite{bellec2022observable} for inference using a single estimator 
converge to their deterministic counterparts defined as solutions
to the \Cref{sys:general_ensemble-M=1}, unifying the mean-field
asymptotics featuring \Cref{sys:general_ensemble-M=1} of \cite{thrampoulidis2018precise} and the inference results of \cite{bellec2022observable}. 
This was known only for the lasso \cite[Theorem 8]{celentano2020lasso}
and unregularized M-estimators \cite{bellec2024asymptotics}.
These convergence results are new for regularized estimators (beyond ridge and lasso) and robust loss functions (beyond squared loss) and are of independent interest even for a single estimator (that is, in the non-ensemble setting with $M = 1$).

{
\textbf{Extension to anisotropic covariance and deterministic signal.}
In \Cref{sec:discussion}, we study the setting of an anisotropic design $(\bx_i\sim \cN(\bm 0_p, p^{-1}\bm\Sigma)$ for $\bm\Sigma\ne \bI_p$) with a deterministic signal $\btheta\in\R^p$, allowing for the possibility of non-separable regularizers.
Using a heuristic argument, we derive a new nonlinear system (\Cref{sys:general_ensemble-M=infty_sigma}) that generalizes \Cref{sys:general_ensemble-M=infty} to this more general regime. 
This new system is validated through numerical simulations (\Cref{fig:bagging_solution_cov}). 
Importantly, we prove that this generalized system retains the same
crucial contraction property (\Cref{th:contraction_Sigma}) that we established for the original system under isotropic covariance and
random signal case (\Cref{th:existence-uniqueness-sys:general_ensemble-M=infty}).
This property is essential, as it not only guarantees the uniqueness of the solution but also underpins the heuristic argument linking this new system to the asymptotic overlapping correlations between estimation errors (and residual errors) under anisotropic design with deterministic signals (\Cref{conjecture} and \Cref{proof:conjecture}). 
Finally, based on the new system, we examine through numerical simulations how structured covariance and deterministic signal affect the optimal subsample size and optimal subagged risk (\Cref{fig:opt_risk_aniso,fig:optimum_phis_overparameterized-2-optsubsample-ar1_aniso,fig:advantage_subsampling_underparameterized_aniso}).
These conclusions mirror those in \Cref{sec:ensembles-interpolators} mentioned above.
}

\subsection{Other related work}
\label{sec:related_work}

Resampling methods, such as bagging and subsampling, are widely used in statistics and machine learning.
Given their broad applicability, there is a vast literature on these methods.
In this section, we provide an overview of the literature related to the risk analysis of ensemble methods, particularly in high-dimensional regimes that have received considerable recent interest.

Classical work on bagging and subagging includes the work by \cite{breiman_1996,breiman2001using,buhlmann2002analyzing}, among others.
Beyond bagging, analysis of ensemble methods of different predictors includes smooth weak predictors \cite{buja2006observations,friedman_hall_2007}, nonparametric estimators \cite{buhlmann2002analyzing,loureiro2022fluctuations}, and classifiers \cite{hall_samworth_2005,samworth2012optimal}.
Historically, there are also early works by \cite{sollich1995learning,krogh1997statistical} on risk asymptotics for ridge ensembles under Gaussian features. 
We also mention here some other early work on ensembles, including: \cite{hansen1990neural,perrone1993putting,sollich1995learning,krogh1997statistical}.
For a comprehensive overview of early work on bagging and ensemble methods in general, we refer readers to \cite{patil2022bagging}.

Substantial progress has been made in the last decade in understanding the asymptotic behavior of regularized M-estimators in high-dimensional settings, particularly under the proportional asymptotic regime where the number of features scales with the number of observations.
Frameworks of Approximate Message Passing (AMP) (developed in a series of papers \cite{donoho2009message,donoho2011noise,bayati2011dynamics}), the Convex Gaussian Min-Max Theorem (CGMT) (developed in a series of papers \cite{oymak2013squared,oymak2016sharp,thrampoulidis2015regularized,thrampoulidis2018precise}), and leave-one-out (LOO) and martingale-based analysis common in Random Matrix Theory (RMT) (used in \cite{karoui_2013,karoui2018impact}, for example) have been instrumental in deriving the limiting test risk, often as solutions to (nonlinear) systems of self-consistent equations.
More specifically, these include analyses of unregularized estimators \cite{el2013robust,karoui_2013,karoui2018impact,donoho2016high,bean2013optimal}, ridge estimator \cite{dobriban_wager_2018}, lasso \cite{bayati2011lasso,miolane2021distribution}, bridge estimators \cite{weng2018overcoming}, logistic regression \cite{sur2019modern,mai2019large,salehi2019impact}, convex regularized M-estimators \cite{thrampoulidis2015regularized,thrampoulidis2018precise}, among others.
Recently, triggered by the empirical success of neural networks that interpolate, these risk analyses have been extended to interpolating estimators with vanishing regularization (in the overparameterized regimes that allow for interpolation), such as ridgeless \cite{hastie2022surprises}, lassoless \cite{li2021minimum}, max-margin interpolators \cite{montanari2019generalization,deng2022model, liang2022precise}; see the survey papers \cite{bartlett2021deep, belkin_2021} for other related references.

Beyond individual regularized M-estimators, there has now been considerable interest over the last few years in the analysis of ensembles of estimators in high-dimensional settings, especially in the overparameterized regime just mentioned.
In particular, \cite{lejeune2020implicit} consider least squares ensembles obtained by subsampling such that the final subsampled dataset has more observations than the number of features.
The work of \cite{patil2022bagging} provides the characterization of the asymptotic risk of ensembles of ridge regression using RMT results.
Furthermore, recent extensions by \cite{du2023subsample,patil2023generalized} expand the scope of these results by establishing risk equivalences for both optimal and suboptimal risks, considering arbitrary feature covariance and signal structures.
Other follow-up works for subagging of ridge and ridgeless regression include \cite{chen2023sketched,ando2023high}. 
This paper develops tools to study ensembles of regularized estimators with general loss and regularizers, beyond ridge regularization.
For instance, our theory accommodates $\ell_1$-regularized Huber regression. 

Another line of research focuses on ensemble methods involving random features and feature sketching rather than subsampling. 
In random features models, the effect of ensembling on various components of the risk has been studied in \cite{adlam2020understanding,d2020double,loureiro2022fluctuations}.
Recently, \cite{patil2023asymptotically} analyze ensembles of ridge regression with sketched features with asymptotically free sketching.
There are also analyses of alternative resampling and averaging techniques. 
For example, in the context of distributed learning, \cite{dobriban_sheng_2020,dobriban_sheng_2021,mucke_reiss_rungenhagen_klein_2022} consider the divide-and-conquer approach, or splagging (split aggregating), and investigate their properties for ridge and ridgeless predictors.

Very recently, \cite{clarte2024analysis} analyzed the limiting equations of several resampling schemes, including bootstrap and resampling without replacement, and characterized self-consistent equations for the limiting bias and variance functionals of estimators obtained by minimization of the negative log-likelihood plus an additive ridge penalty. 
This is related to our risk characterization as \cite{clarte2024analysis} also covers sampling without replacement, but our nonlinear systems (\Cref{sys:general_ensemble-M=1,sys:general_ensemble-M=infty}) characterizing the subagging risk do not appear explicitly in their work, which instead focuses on self-consistent equations for bias and variance functionals of the specific resampling scheme.
The results of \cite{clarte2024analysis} rely on the general AMP
analysis and state evolution laid out in \cite[Lemmas B.3 and B.5]{loureiro2022fluctuations}, generalizing \cite{bayati2011lasso}. 
This analysis requires the existence of unique solutions to the corresponding limiting system of equations, which is granted under strong convexity (e.g., with a ridge penalty), but was not established until the present paper
for the case of ensembling of subsampled regularized estimators.

Finally, complementary to risk characterization, there has also been considerable interest in the cross-validation and model selection of ensemble methods.
In particular, \cite{du2023subsample} study cross-validation for bagging of ridge regression.
\cite{bellec2023corrected} examine the consistency of generalized cross-validation (GCV) for estimating the prediction risk of arbitrary ensembles of regularized least squares estimators for strongly convex penalties.
They show that GCV is inconsistent for any finite ensemble of size greater than one and identify a correction to GCV that is consistent for any finite ensemble size, termed corrected GCV (CGCV).
In this paper, we generalize one of the data-dependent estimators proposed in \cite{bellec2023corrected} for the general setting of this paper, allowing for general convex losses and heterogeneous component estimators in the ensemble.
While we do not attempt to interpret the estimator as a corrected GCV for homogeneous ensembles in the general setting, it may be possible to perform such an analysis further, which we leave for future work.

The proof strategy in this paper extends the approach of \cite{bellec2024asymptotics}, which studies the bagging of unregularized M-estimators. 
While their analysis is based on a relatively simple 1-dimensional nonlinear system, the new \Cref{sys:general_ensemble-M=infty} below is 2-dimensional, introducing additional complexity to the analysis. 
The rise in complexity is similar to that from unregularized regression \cite{karoui_2013,donoho2016high,karoui2018impact} and its 2-dimensional system to the 4-dimensional system of regularized M-estimators \cite{thrampoulidis2018precise} given in \Cref{sys:general_ensemble-M=1}. 
One challenge arises from the stochastic control of the trace terms in \eqref{eq:nu-kappa-convergence}.
In the unregularized case, these trace terms can be approximated by a straightforward product of the norms of the error vector and residuals (see \cite[Lemma 5.7]{bellec2024asymptotics}), and the stochastic behavior of these norms are well-understood in the existing literature (see, e.g., \cite{thrampoulidis2018precise, loureiro2021learning}). 
However, in the regularized case, the trace functional cannot be approximated by such a simple expression, which prevents the direct application of these existing results based on the CGMT. 
We overcome this by showing that, with high probability, the trace term is a stationary point of a certain (random) strongly convex function. 
This allows us to control the perturbation of the trace term through the behavior of the convex function (see \Cref{subsec:proof_convergence_trace}).

\subsection{Notation}
\label{sec:notation} 
For a natural number $n$, the shorthand notation $[n]$ denotes the set $\{ 1, \dots, n \}$.
For two real numbers $x$ and $y$, we use $x \wedge y$ to denote $\min\{x,y\}$. 
For a univariate function $f$ and a vector $\ba \in \RR^{n}$, with a slight overload of notation, we use $f(\ba) \in \RR^{n}$ to denote the component-wise application of $f$ to $\ba$. 
For any proper, closed, convex function $f \colon \RR \to \R\cup\{+\infty\}$, the proximal operator and Moreau envelope of $f$ with a parameter $\tau > 0$ at a point $x \in \RR$ are, respectively, denoted as:
$$
    \textstyle
    \prox_{f}(x; \tau) \coloneq \argmin_{y \in \RR} f(y) + \frac{1}{2\tau} (x - y)^2 \ \text{ and } \ 
    \env_{f}(x; \tau) \coloneq \min_{y \in \RR} f(y) + \frac{1}{2\tau} (x - y)^2.
$$
For a proper, closed, convex $f$, the $\argmin$ exists and is unique, and consequently $x \mapsto \prox_f(x; \tau)$ is a well-defined function.
Let $\partial f$ denote the subdifferential of $f$, which is the set of all subgradients of $f$.
We note two key relationships between the proximal operator, subdifferential, and Moreau envelopes of $f$ below for the reader's convenience  (see, e.g., \cite[Section 3]{parikh2014proximal}).
{By Danskin's theorem \cite[Proposition B.22]{bertsekas2016nonlinear} and the KKT conditions, it holds that:}
\begin{equation}
    \textstyle
    \label{eq:prox_subdiff_relation}
\frac{\partial}{\partial x} \env_{f}(x;\tau) =  \frac{1}{\tau}\bigl(x- \prox_f(x; \tau) \bigr) \in  \partial f(\prox_f(x; \tau)). 
\end{equation}
For simplicity, we often use $\env_f'(x;\tau)$ to denote the partial derivative $\frac{\partial}{\partial x} \env_{f}(x;\tau)$.
Finally, we use $\Op$ and $\op$ to denote probabilistic big-O and little-o notation, respectively, while the convergences in probability are denoted by $\pto$. For the reader's convenience, we also give a quick overview of the specific notation used in this paper in \Cref{tab:notation}.

\section{Setup and assumptions}
\label{sec:setup}

We consider the standard supervised regression setting, in which we observe $n$ data points $(\bx_i, y_i)$ for $i \in [n]$.
The feature matrix $\bX \in \RR^{n\times p}$ contains $\bx_{i}^\top$ in its $i$-th row and the response vector $\by \in \RR^n$ contains $y_{i}$ in its $i$-th entry. 
We assume the following distribution on the dataset $(\bX, \by)$:
\begin{assumption}
    [Data distribution]
    \label{asm:linear_model}
    The distribution of ($\bX$, $\by$) is specified by:
    \begin{enumerate}[leftmargin=7mm]
        \item 
        The design matrix $\bX \in \RR^{n \times p}$ has i.i.d.\ entries drawn from $\cN(0,1/p)$.
        \item
        The response vector $\by = \bX \btheta + \bm{\eps}$, where $\btheta \in \RR^{p}$ is a random signal vector and $\bm{\eps} \in \RR^{n}$ is a random noise vector, both independent of each other and of $\bX$, with:\newline
        ~a) The signal vector $\btheta \in \RR^{p}$ has i.i.d.\ entries drawn from distribution $F_{\theta}$. 
        \newline
        ~b) The noise vector $\bm{\eps} \in \RR^{n}$ has i.i.d.\ entries drawn from distribution $F_\eps$.
    \end{enumerate}
\end{assumption}

The above assumption implies that the covariates $\bx_i = \bX^\top\be_i\in\R^p$ are sampled from the standard Gaussian distribution $\cN(\bm{0}_p, p^{-1}\bI_p)$, and the ground truth signal $\btheta$ has i.i.d. entries. 
While our main theorems are formally established under these conditions, we provide strong evidence of their broader applicability.
Numerical simulations in \Cref{fig:effect-of-M-cov} suggest that our results extend beyond Gaussian design and exhibit universality. 
Furthermore, in \Cref{sec:discussion}, we develop an extension of our framework to the important case of anisotropic design ($\bx_i \sim \cN(\bm{0}_p, p^{-1}\bm{\Sigma})$ with $\bm{\Sigma} \neq \bI_p$) with deterministic signal $\btheta$, which is numerically validated in \Cref{fig:bagging_solution_cov}.

We subsample the dataset $(\bX, \by)$ to create $M$ subsampled datasets. 
Towards that end, define $M$ subsample index subsets $I_m \subset[n]$ of cardinality $k_m = |I_m|$ for $m \in [M]$.
The feature matrix and response vector associated with the subsampled dataset $(\bx_i, y_i)$ for $i \in I_m$ are denoted as $(\bX_{\setm}, \by_{\setm})$.
We assume the following sampling strategy for the subsample index sets $\{ I_m \}_{m \in [M]}$:
\begin{assumption}
    [Subsampling strategy]
    \label{assu:sampling}
    Given deterministic integers $\{ k_m \ge 1\}_{m \in [M]}$, the $M$ subsample index sets $\{I_m\}_{m\in[M]}$ are independent of $(\bX, \by)$ and are independently sampled from the uniform distribution over subsets of $[n]$ with cardinality $k_m$ for each $m\in [M]$. 
\end{assumption}
It is worth noting that if $I_m$ and $I_\ell$ (for $m \neq \ell$) are any two independent subsample sets of cardinality $k_m$ and $k_\ell$ per \Cref{assu:sampling}, then the cardinality of intersection $| I_m \cap I_\ell |$ follows a hypergeometric distribution with mean $k_m k_\ell/ n$.
Using the properties of the hypergeometric distribution, it follows that $|I_m \cap I_\ell | / n \pto c_m c_\ell$ as both $n, k_m, k_\ell, \to \infty$ with the subsample ratios $k_m/n \to c_m$ and $k_\ell/n \to c_\ell$ for some $c_m, c_\ell \in (0, 1]$ (this follows from Chebyshev's inequality and the variance formula of the hypergeometric distribution, see Section S.8.1 of \cite{patil2022bagging} for more details).
Intuitively, each sample lands in a subsample $I_m$ with probability $c_m$ (the limiting ratio $|I_m|/n$) and in the overlap of two subsamples with probability $c_m c_\ell$ (the limiting ratio $| I_m \cap I_\ell | / n$), as the subsamples are drawn independently.
The overlap between any two subsample sets $I_m$ and $I_\ell$ is thus of order $n$ with high probability.
The randomness in subsampling in \Cref{assu:sampling} is not important.
Our results can accommodate deterministic sampling where the subsample sets $\{ I_m \}_{m \in [M]}$ are selected deterministically, provided that the ratios $|I_m|/n$, $I_\ell/n$, and $|I_m \cap I_\ell|/n$ converge to non-zero constants.

For each subsampled dataset $(\bX_{\setm}, \by_{\setm})$ for $m \in [M]$, we define the regularized M-estimator 
\begin{equation}
    \label{eq:def-hbeta}
    \hat \bbeta_m(I_m) 
    \in
    \argmin_{\signal \in \R^p} \sum_{i\in I_m} \loss_m (y_i - \bx_i^\top \signal) + \sum_{j\in[p]} \reg_m(\signal_j).
\end{equation}
When defining \eqref{eq:def-hbeta}, we allow the $\argmin$ operator to return any one of the minimizers (as emphasized by the element notation in \eqref{eq:def-hbeta}).
Here $\loss_m$ and $\reg_m$ are the loss and regularization functions that satisfy the following assumption for all $m \in [M]$:
\begin{assumption}
    \label{assu:loss_penalty}
    The loss function $\loss \colon \RR \to \R$ is proper, closed, convex, and differentiable with derivative $\loss'$ Lipschitz continuous, $\min_x\loss(x)$ $ = \loss(0)$. 
    The function $\reg \colon \RR \to \R\cup\{+\infty\}$ is proper, closed, convex and $\min_x\reg(x) = \reg(0)$. 
\end{assumption}

The final ensemble estimator, constructed using the component estimators \eqref{eq:def-hbeta}, is defined as:
\begin{equation}
    \textstyle
    \label{eq:def-M-ensemble}
    \tbeta_{M} \big(\{I_m\}_{m\in[M]}\big)
    \coloneq \frac{1}{M} \sum_{m\in[M]} \hbeta_m(I_m).
\end{equation}
For brevity, we omit the dependency of the component and ensemble estimators on $I_m$ and $\{ I_{m} \}_{m\in[M]}$ and simply write $\hat\bbeta_m$ and $\tbeta_M$, respectively, when it is clear from the context.
We evaluate the performance of the ensemble estimator $\tbeta_{M}$ via:
\begin{equation*}
    \textstyle
    R_{M}
    \coloneq
    \frac{1}{p} \| \tbeta_{M} - \btheta \|_2^2.
\end{equation*}
Note that $R_{M}$ is the (excess) out-of-sample squared error in our setup because of isotropic features:
For an independently sampled test feature $\bx_0 \in \RR^p$ with i.i.d.\ entries drawn from $\cN(0, 1/p)$, we have $R_M = \EE[(\bx_0^\top \tilde\btheta_M - \bx_0^\top \btheta)^2 \mid \by, \bX, \{ I_m \}_{m \in [M]}]$. 
We will refer to $R_M$ as the risk of the ensemble.
Observe that $R_{M}$ is a scalar random variable that depends on both the dataset $(\bX, \by)$ and the random samples $I_{m}$ for $m\in [M]$.
The goal of the paper is to characterize the asymptotic behavior of this random variable $R_M$ under the proportional asymptotic regime.
In this regime, the sample size $n$, feature size $p$, and subsample size $k_m$ all diverge while keeping the appropriate ratios fixed:
we will assume the inverse data aspect ratio $n/p \to \delta \in (0, \infty)$ and for each $m\in[M]$, the subsample ratio $k_m/n \to c_m \in (0,1]$ as $n, p, k_m \to \infty$.

\section{Risk characterization and estimation}
\label{sec:general-risk-characterization}

In this section, we will first describe a general technical result on the correlations of the error and residual vectors for regularized M-estimator in \Cref{sec:main-result}.
We then state our general result on the risk characterization of the ensemble estimator in \Cref{sec:ensemble-risk-asymptotics} and construct a consistent risk estimator for this risk in \Cref{sec:risk-estimation}.

\subsection{Asymptotics of correlations of estimator and residual errors}
\label{sec:main-result}

To state the risk characterization of the ensemble estimator, we first introduce two important nonlinear systems of equations: \Cref{sys:general_ensemble-M=1,sys:general_ensemble-M=infty}.
Intuitively, these systems correspond to the corner cases where the ensemble size $M = 1$ and $M = \infty$, respectively.
As we shall see in \Cref{sec:ensemble-risk-asymptotics}, these systems completely determine the risk asymptotics of the ensemble estimator \eqref{eq:def-M-ensemble}.

\addtocounter{system}{-1}
\renewcommand{\thesystem}{1a}
\begin{system}
    [Error norms of individual regularized M-estimators]
    \label{sys:general_ensemble-M=1}
    Given a triple $(\loss,\reg,c \delta)$ where $c \delta \in (0, \infty)$ and  $\loss$, $\reg$ are convex functions, define the following 4-scalar system of equations in variables $(\alpha, \beta, \kappa, \nu) \in \R_{>0}^4$:
    \begin{subequations}
    \label{eq:CGMT-1}
    \begin{alignat}{1}
        \alpha^2 &= 
        \E \big[
          \big(
            \tfrac{1}{\nu} \env_{\reg}' ( \Theta + \tfrac{\beta}{\nu} H; \tfrac{1}{\nu} ) - \tfrac{\beta}{\nu} H
          \big)^2
        \big] 
        \label{eq:CGMT-1a}\\
        \beta^2 &= 
        \E\big[\env_{\loss}'(Z + \alpha G; \kappa)^2\big] \cdot c \delta
        \label{eq:CGMT-1b} \\
        \kappa\beta &= 
         \E\big[
         \big(
          \tfrac{1}{\nu} \env_{\reg}' (\Theta + \tfrac{\beta}{\nu} H ; \tfrac{1}{\nu} ) 
          -
          \tfrac{\beta}{\nu} H
          \big)
          \cdot (-H)
        \big] 
        \label{eq:CGMT-1c}\\
        \nu\alpha &= 
        \E\big[\env'_{\loss}(Z + \alpha G; \kappa)\cdot G\big] \cdot c \delta
        \label{eq:CGMT-1d}
    \end{alignat}
    \end{subequations}
    where $H\sim \cN(0,1)$, $G\sim \cN(0,1)$, $\Theta \sim F_\theta$, $Z \sim F_\eps$, all mutually independent.
\end{system}

\Cref{sys:general_ensemble-M=1} can be found in the literature, specifically in \cite[Equation 15]{thrampoulidis2018precise}.
To be precise, we are applying the result of \cite{thrampoulidis2018precise} on the subsample estimator \eqref{eq:def-hbeta} using $k=|I_m|$ observations with $k/n\to c$, so that the limiting inverse aspect ratio $k / p = k / n \cdot n / p \to c \delta$.
\Cref{sys:general_ensemble-M=1} is known to characterize the limit in probability of the risk of \eqref{eq:def-hbeta} when $|I_m|/p \to c\delta$: if $(\alpha,\beta,\kappa,\nu)$ is a solution to \Cref{sys:general_ensemble-M=1}, then $p^{-1} \|\hbbeta_m - \bbeta^*\|^2 \pto \alpha^2$.
The existence and uniqueness of the fixed-point parameters in this system are central to applying results from the Convex Gaussian Min-Max Theorem (CGMT) to derive precise risk characterizations for regularized M-estimators (under proportional asymptotics).
This is guaranteed under conditions where both $\loss$ and $\reg$ are Lipschitz and the problem parameters are such that the perfect signal recovery is not possible, leading to non-zero asymptotic risk.
For a detailed discussion on these conditions, see \cite{bellec2023existence}.
Next, we describe our second system for risk characterization of the ensemble estimators.

\renewcommand{\thesystem}{1b}
\begin{system}
    [Error correlations of overlapped regularized M-estimators]
    \label{sys:general_ensemble-M=infty}
    Given $c, \tilde c \in (0, 1]$ and convex pairs of functions $(\loss, \reg)$, $(\tilde\loss, \tilde\reg)$, let $(\alpha, \beta, \kappa, \nu)$ and $(\tilde\alpha, \tilde\beta, \tilde\kappa, \tilde\nu)$ be parameters that satisfy \Cref{sys:general_ensemble-M=1} with $(\loss, \reg, c\delta)$ and $(\tilde\loss, \tilde\reg, \tilde c\delta)$, respectively. 
    Define the following 2-scalar system of equations in variable $(\etaG, \etaH) \in [-1,1]^2$:
    \begin{equation}
    \label{eq:contraction-map}
    {
        \etaG = F_{\reg}(\etaH) ,
        \qquad
        \etaH = F_\loss(\etaG)
    }
    \end{equation}
    where $F_\loss,F_\reg \colon [-1,1] \to \RR$ are functions defined as:
    \small
   \begin{subequations} \label{eq:CGMT-2}
       \begin{empheq}{align}
           F_{\loss}(\etaG) & \coloneq
          {\tfrac{c\tilde c \delta}{\beta\tilde\beta}} {\E \big[
          \env_{\loss}'(Z + \alpha G; \kappa) \cdot \env_{\tilde \loss}'(Z + \tilde \alpha \tilde G; \tilde \kappa)
        \big]} \label{eq:eta_def} \\ 
        \textstyle
           F_{\reg}(\etaH) & \coloneq
           {\tfrac{1}{\alpha\tilde\alpha}}
       {\E \bigl[
          \bigl(
            \tfrac{1}{\nu} \cdot \env_{\reg}' \bigl(\Theta + \tfrac{\beta}{\nu} H; \tfrac{1}{\nu}\bigr) - \tfrac{\beta}{\nu} H
          \bigr)
          \cdot 
        \bigl(
            \tfrac{1}{\tilde \nu} \cdot \env_{\tilde \reg}' \bigl(\Theta + \tfrac{\tilde \beta}{\tilde \nu} \tilde H ; \tfrac{1}{\tilde \nu}\bigr) - \tfrac{\tilde \beta}{\tilde \nu}\tilde  H
          \bigr)
        \bigr]} \label{eq:etaH_def}
       \end{empheq}
   \end{subequations}
   \normalsize
    {
    where the four random vectors $(G,\tilde G)$, $(H,\tilde H)$, $\Theta$ and
    $Z$ are all mutually independent with
    }
    \begin{equation}
    \begin{pmatrix}
    G\\
    \tilde{G}
    \end{pmatrix} \sim \cN \biggl(\begin{bmatrix}
    0 \\
    0
    \end{bmatrix},
    \begin{bmatrix}
    1 & \etaG\\
    \etaG & 1
    \end{bmatrix}\biggr), 
    \begin{pmatrix}
    H\\
    \tilde{H}
    \end{pmatrix} \sim \cN \biggl(\begin{bmatrix}
    0 \\
    0
    \end{bmatrix}, \begin{bmatrix}
    1 & \etaH\\
    \etaH & 1
    \end{bmatrix}\biggr),
    \quad\Theta \sim F_\theta,
    Z \sim F_\eps.
    \label{distrib}
    \end{equation}
\end{system}

\Cref{sys:general_ensemble-M=infty} is new and one of the main contributions of this paper. 
Note that the parameters $(\etaG,\etaH)$ in the system are correlation parameters (up to scaling factors) of the two random variables visible inside expectations in \eqref{eq:eta_def} and of the two random variables in \eqref{eq:etaH_def}, respectively.
By the Cauchy--Schwarz inequality, the function $F_\loss$ and $F_\reg$ are uniformly bounded as $|F_\loss(\eta)| \le \sqrt{c\tilde c}$ and $|F_\reg(\eta)| \le 1$ so that any solution $(\etaG,\etaH)$ to the system \eqref{eq:contraction-map} lies in the set $[-1,1] \times [-\sqrt{c\tilde c}, \sqrt{c\tilde{c}}]$.
As stated, it is not immediately clear if \Cref{sys:general_ensemble-M=infty} admits any solution and whether it is unique.
Our first result establishes that this is indeed the case:

\begin{theorem}
    [Existence and uniqueness of solutions to \Cref{sys:general_ensemble-M=infty}]
    \label{th:existence-uniqueness-sys:general_ensemble-M=infty}
    Let $c,\tilde c\in (0,1]$.
    The functions $F_{\loss}$ and $F_{\reg}$ defined in \eqref{eq:CGMT-2} satisfy the following properties:
\begin{enumerate}[leftmargin=7mm]
    \item
    $|F_{\loss}(\etaG)| \le \sqrt{c\tilde c}$ and $|F_{\reg}(\etaH)| \le 1$ for all $\etaG\in [-1,1]$ and $\etaH\in [-1,1]$. 
    \item
    $F_\loss$ and $F_\reg$ are non-decreasing, differentiable, and the compositions $F_\loss \circ F_\reg$ and $F_\reg\circ F_\loss$ are $\min\{c, \tilde c\}$-Lipschitz.
    \item
    \Cref{sys:general_ensemble-M=infty} admits a unique solution $(\etaGstar,\etaHstar)
    \in [-1, 1] \times [-\sqrt{c\tilde c},\sqrt{c\tilde c}]$.
\end{enumerate}
\end{theorem}

Some remarks on \Cref{th:existence-uniqueness-sys:general_ensemble-M=infty} are in order.
Among the properties listed in \Cref{th:existence-uniqueness-sys:general_ensemble-M=infty}, the most interesting is the second property: the two maps ($F_\loss\circ F_\reg$, $F_\reg\circ F_\loss$) are strict contractions if {$\min\{c, \tilde c\} < 1$}.
Given this property, the third property (the uniqueness and existence of the solution) easily follows.  
We briefly explain this next.
Indeed, if $(\etaHstar, \etaGstar)$ is a solution to \Cref{sys:general_ensemble-M=infty}, then $\etaHstar$ automatically satisfies the following 1-dimensional fixed-point equation:
$$
    \etaHstar = F_\loss(\etaGstar) = F_\loss\circ F_\reg(\etaHstar).
$$
The other direction is also true in the following sense: if $\etaHstar$ is a solution to the fixed-point equation $\etaH = F_\loss\circ F_\reg(\etaH)$, then letting $\etaGstar = F_\reg(\etaHstar)$, we observe that the pair $(\etaHstar, \etaGstar)$ satisfies \Cref{sys:general_ensemble-M=infty}.
Since $F_\loss\circ F_\reg$ is continuous and maps $[-1,1]$ to itself, Brouwer's fixed-point theorem implies that such $\etaHstar$ exists.
Uniqueness of the solution is straightforward if $\min(c,\tilde c)<1$ as the
Lipschitz constant of $F_\loss\circ F_\reg$ is strictly less than 1;
uniqueness in the case $c=\tilde c = 1$ still holds but requires a more subtle
argument.
See \Cref{subsec:proof-exisntence-uniqueness} for the full proof details.
This contraction property also certifies that the fixed-point iteration algorithm $\etaH^{(k+1)} = F_\loss\circ F_\reg(\etaH^{(k)})$, which we use in our experiments to solve \Cref{sys:general_ensemble-M=infty}, numerically converges to the correct solution $\etaHstar$ exponentially fast {if $\min(c,\tilde c)<1$}. {\Cref{sec:correlation_signs} discusses the signs of $\etaGstar,\etaHstar$.}

We next show that the correlation parameters $(\etaG,\etaH)$ from \Cref{sys:general_ensemble-M=infty} are the limiting correlations between the estimator and residual errors of estimators trained on overlapped samples.
To do so, besides \Crefrange{asm:linear_model}{assu:loss_penalty}, we will need mild regularity conditions that the $\loss$ and $\reg$ in \Cref{assu:loss_penalty} need to satisfy in relation to the distribution $F_\theta$ and $F_\eps$ of the signal and noise coordinates in \Cref{asm:linear_model}.

\begin{assumption}
    [Regularity conditions]
    \label{asm:regularity-conditions}
    Let $\Theta \sim F_\theta$ and $Z \sim F_\eps$ be the signal and noise random variables as in \Cref{asm:linear_model}, and $G \sim \cN(0, 1)$, $H \sim \cN(0,1)$, all mutually independent. In addition to \Cref{assu:loss_penalty}, the functions $\loss : {\RR} \to \R$ and $\reg : {\RR} \to \R\cup\{+\infty\}$ satisfy the following:
    \begin{enumerate}[leftmargin=7mm]
        \item 
        For all $c\in \R$, we have
        $
            \E[\loss_{+}'(cG+Z)^2]<+\infty$ and $\E[\reg_{+}'(cH+\Theta)^2] <+\infty
        $
        where we define $f_{+}'(x) := \sup_{s\in \partial f(x)}|s|$ for any convex function $f$.
        \item $\PP(\Theta\ne 0) > 0$.
        \item 
        \Cref{sys:general_ensemble-M=1} admits a unique positive solution $(\alpha, \beta, \kappa, \nu)\in \R_{>0}^4$. 
        \item
        There exists an interval $\mathcal{I}\subset \R$ where $\loss'$ is strictly increasing. 
        {The support of $F_z$ is $\R$.}
    \end{enumerate}
\end{assumption}

The conditions in \Cref{asm:regularity-conditions} are similar to those assumed in \cite[Theorem 4.1]{thrampoulidis2018precise} when characterizing the asymptotics for the non-ensemble case.
The main difference is that we do not require the second moment of $\Theta$ to be finite. 
These conditions ensure that the individual estimator and residual error norms converge, i.e., $\|\hat\btheta - \btheta\|_2^2 / p \pto \alpha^2$ and $\|\loss'(\by - \bX \hat\btheta)\|_2^2/p \pto \beta^2$ hold, where $\alpha$ and $\beta$ are solutions to \Cref{sys:general_ensemble-M=1}. 
(See \Cref{subsec:cgmt_assumption,subsec:convergence_psi} for proofs of convergences of error vector and loss gradient norm squared under the relaxation of the conditions.) 

While \Cref{asm:regularity-conditions}-(3) posits the existence and uniqueness of a solution to \Cref{sys:general_ensemble-M=1}, this has been formally established under slightly stronger assumptions in \cite{bellec2023existence}, specifically assuming Lipschitz loss and penalty.
We conjecture that the rest of the assumptions, namely \Cref{asm:regularity-conditions}-(1,2,4), are sufficient for the existence and uniqueness of a solution to \Cref{sys:general_ensemble-M=1} without requiring a Lipschitz condition on the loss and penalty.
Extending the analysis in \cite{bellec2023existence} to allow a pseudo-Lipschitz condition is a promising starting point to confirm this conjecture.

For the upcoming statement, recall that when used on a vector, the $\loss$ and $\reg$ functions are assumed to be operated component-wise.
In addition, we denote the feature matrix and response vector associated with the ``overlapped'' dataset $(\bx_i,y_i)$ for $i \in I \cap \tilde I$ using $(\bX_{I \cap \tilde I}, \by_{I \cap \tilde I})$.

\begin{theorem}
    [Estimator and residual errors correlations characterization]
    \label{thm:corr-sigerror-reserror}
    Let $\hat{\btheta}_{I}$ and $\hat{\btheta}_{\tilde I}$ be component estimators \eqref{eq:def-hbeta} trained on subsamples $(\bX_I,\by_I)$ and $(\bX_{\tilde I}, \by_{\tilde I})$ corresponding to index sets $I$ and $\tilde I$ with parameters $(\loss,\reg)$ and $(\tilde \loss, \tilde \reg)$.
    Under \Crefrange{asm:linear_model}{asm:regularity-conditions}, as $n, p, k, \tilde{k} \to \infty$ with $n/p \to \delta \in (0, \infty)$, $k/n \to c \in {(0, 1]}$ and $\tilde{k}/n \to \tilde{c} \in {(0,1]}$,
    \begin{equation}
    \begin{split}
       p^{-1} (\hat{\btheta}_{I}-\btheta)^\top (\hat{\btheta}_{\tilde I}-\btheta) &\pto \etaG \alpha \tilde\alpha, \\
        p^{-1} \loss'(\by_{I \cap \tilde I} - \bX_{I \cap \tilde I} \, \hat{\btheta}_{I})^\top \, \tilde \loss'(\by_{I \cap \tilde I} - \bX_{I \cap \tilde I} \, \hat{\btheta}_{\tilde I}) &\pto \etaH \beta\tilde\beta, 
    \end{split}
       \label{eq:correlation-limits}
    \end{equation}
    where $(\etaG, \etaH)$ is the solution to \Cref{sys:general_ensemble-M=infty}. 
    Furthermore, for any $i\in {[n]}$ and $j\in [p]$, there exists a jointly normal $(G_i, \tilde G_i)$ with mean $0$, variance $1$ and correlation $\etaG$ and $(H_j, \tilde{H}_j)$ with mean $0$, variance $1$ and correlation $\etaH$ (as in \Cref{sys:general_ensemble-M=infty}) such that the residuals and estimators are jointly approximated as follows:
    {\begin{subequations}
    \begin{alignat}{1}
        \max_{j\in [p]} 
        \E 
        \biggl[
            1 ~\wedge~ 
            \Bigl\|
            \begin{pmatrix}
                \be_j^\top \hat\btheta_I \\
                \be_j^\top  \hat\btheta_{\tilde I}
            \end{pmatrix}-
            \begin{pmatrix}
                \prox_{\reg}(\theta_j+ \tfrac{\beta}{\nu} H_j; \tfrac{1}{\nu})\\
                \prox_{\tilde \reg}(\theta_j+ \tfrac{\tilde\beta}{\tilde \nu} \tilde{H}_j; \tfrac{1}{\tilde{\nu}})
            \end{pmatrix}
            \Bigr\|_2^2 
        \biggr] 
        &= o(1), \label{eq:joint-distribution-estimators} \\
        \max_{i\in [n]} 
        \E
        \biggl[
            \ind_{i\in I\cap \tilde I}
             ~\wedge~ \Bigl\|
            \begin{pmatrix}
                y_i - \bx_i^\top\hat{\btheta}_{I}\\
                y_i - \bx_i^\top\hat{\btheta}_{\tilde I}
            \end{pmatrix}
            - 
            \begin{pmatrix}
                \prox_{\loss}(z_i +\alpha G_i; \kappa)\\
                \prox_{\tilde \loss}(z_i + \tilde\alpha \tilde G_i; \tilde\kappa)
            \end{pmatrix}
            \Bigr\|_2^2
        \biggr] &= o(1) \label{eq:joint-distribution-residuals}.
    \end{alignat}
    \end{subequations}}
\end{theorem}
The proof is given in \Cref{sec:proof-thm:corr-sigerror-reserror}.
Put in words, $\etaG$ and $\etaH$ from \Cref{sys:general_ensemble-M=infty} encode the cosines of the angles between the estimator errors and loss gradient residual errors of estimators $\hat{\btheta}_{I}$ and $\hat{\btheta}_{\tilde I}$ respectively.
{
Furthermore, \eqref{eq:joint-distribution-estimators} and \eqref{eq:joint-distribution-residuals} imply that the pair of distinct estimators (or their corresponding residuals) can be jointly approximated by the proximal operators applied to the signal (or noise) plus correlated Gaussian vectors, with the correlation structure determined by $\etaH$ (or $\etaG$).
}

The main difficulty in showing \Cref{thm:corr-sigerror-reserror} is the non-trivial dependence between the two estimators $\hat{\bbeta}_{I}$ and $\hat{\bbeta}_{\tilde I}$ as they share the samples $\bX_{I\cap \tilde I}$.
In the case of squared loss and ridge regularizer, the estimators $\hat{\bbeta}_{I}$ and $\hat{\bbeta}_{\tilde I}$ have closed-form expressions, and prior work in \cite{patil2022bagging} explicitly analyze the overlapped resolvents by developing conditional calculus of resolvents.
However, for general loss and regularizers, the overlapped terms are more challenging to analyze due to the lack of closed-form expressions.
Our strategy in this paper is to exploit the recently developed technique in \cite{bellec2024asymptotics} to analyze the overlapped terms.

To prove \Cref{thm:corr-sigerror-reserror}, we show that the quantities on the left-hand side
of \eqref{eq:correlation-limits} concentrate around scalars
independent of $\bm X$, and that these two scalars are approximate
fixed-points 
of $F_\loss \circ F_\reg$ and $F_{\reg} \circ F_\loss$, respectively.
This strategy of first proving the concentration of certain quantities
and then obtaining approximate fixed-point equations is reminiscent
of the leave-one-out analysis of \cite{el2013robust,karoui2018impact}
and was previously used in \cite{bellec2024asymptotics} to characterize
the ensemble risk in unregularized M-estimators (with no explicit penalty).
The setting studied here is significantly more complicated than these works
due to the presence of robust loss functions, penalty functions, and shared
samples between $\hat{\bbeta}_I$ and $\hat{\bbeta}_{\tilde I}$.

We believe that once the contractions of $F_\loss\circ F_\reg$ and
$F_\reg \circ F_\loss$ have been found and the existence and uniqueness
of the solution to \Cref{sys:general_ensemble-M=infty} have been established,
different techniques than those used here and discussed in the previous
paragraph could also be used to derive asymptotic results similar to
\Cref{thm:corr-sigerror-reserror}.
For instance, after existence and uniqueness
of the solution to \Cref{sys:general_ensemble-M=infty} is established,
there is hope to carry out an AMP analysis for matrix-valued parameters
(see for instance \cite[Lemmas B.3 and B.5]{loureiro2022fluctuations} or \cite{javanmard2013state,gerbelot2023graph}),
or by using the conditional CGMT technique of
\cite[Appendix F]{celentano2021cad}.
However, we emphasize that these alternate techniques would also first require us to establish the structure of \Cref{sys:general_ensemble-M=infty}
(as done in \Cref{th:existence-uniqueness-sys:general_ensemble-M=infty}) in order to guarantee the existence and uniqueness of the solution to \Cref{sys:general_ensemble-M=infty}. This is a prerequisite for both applying the CGMT results and ensuring 
the convergence of AMP to the regularized M-estimator.

{
The proximal representations in \eqref{eq:joint-distribution-estimators} and \eqref{eq:joint-distribution-residuals} allow one to provide limiting behavior of other functionals of the estimator and residual errors by assuming further moments on the signal and error distributions $F_\theta$ and $F_z$; for example, we can characterize the correlation between the raw residuals (rather than after applying the loss derivative) assuming finite second moment of $F_z$ or consider pseudo-Lipschitz functionals other than squared error.
Towards that end, our next result provides the convergence of the empirical measure of $(\be_j^\top \hat\btheta_I, \be_j^\top \hat\btheta_{\tilde I},\theta_j)_{j\in[p]}$ and residuals $(y_i-\bx_i^\top\hat\btheta_I, y_i-\bx_i^\top\hat\btheta_{\tilde I}, z_i)_{i\in I\cap \tilde I}$ for pseudo-Lipschitz test functions. 

\begin{theorem}
\label{theorem:average_prox_rep}
Let $\Phi:\R^3\to\R$ be a pseudo-Lipschitz function of order $2$. 
Suppose the assumptions in \Cref{thm:corr-sigerror-reserror} hold.
Further assume that the marginal distribution $F_\theta$ of the signal has a bounded second moment. 
Then it holds that: 
\begin{align*} 
\frac{1}{p} \sum_{j\in[p]}  \Phi    \begin{pmatrix}
                \be_j^\top \hat\btheta_I \\
                \be_j^\top  \hat\btheta_{\tilde I}\\
                \theta_j
            \end{pmatrix}
           \pto {\mathbb{E}} \Phi            \begin{pmatrix}
                \prox_{\reg}(\Theta+ \tfrac{\beta}{\nu} H; \tfrac{1}{\nu})\\
                \prox_{\tilde \reg}(\Theta+ \tfrac{\tilde\beta}{\tilde \nu} \tilde{H}; \tfrac{1}{\tilde{\nu}})\\
                \Theta 
            \end{pmatrix}.
\end{align*}
Similarly, if the marginal distribution $F_z$ of noise has a finite second moment, we have:
$$
 \frac{1}{|I\cap \tilde I|} \sum_{i\in I\cap\tilde I}  \Phi    \begin{pmatrix}
y_i - \bx_i^\top\hat\btheta_I\\
y_i - \bx_i^\top \hat\btheta_{\tilde I}\\
y_i - \bx_i^\top\btheta
\end{pmatrix}
       \pto {\mathbb{E}} \Phi            \begin{pmatrix}
               \prox_{\loss}(Z + \alpha G;\kappa)\\
               \prox_{\tilde\loss}(Z + \alpha\tilde{G};\tilde \kappa)\\
                Z
            \end{pmatrix}.
$$
\end{theorem}
Note that the finite second moments of $F_\theta$ and $F_z$ in \Cref{theorem:average_prox_rep} are required to ensure the probabilistic limits above are finite for any pseudo-Lipschitz function of order $2$. 
}

\begin{figure}[!t]
    \centering
    \includegraphics[width=0.99\textwidth]{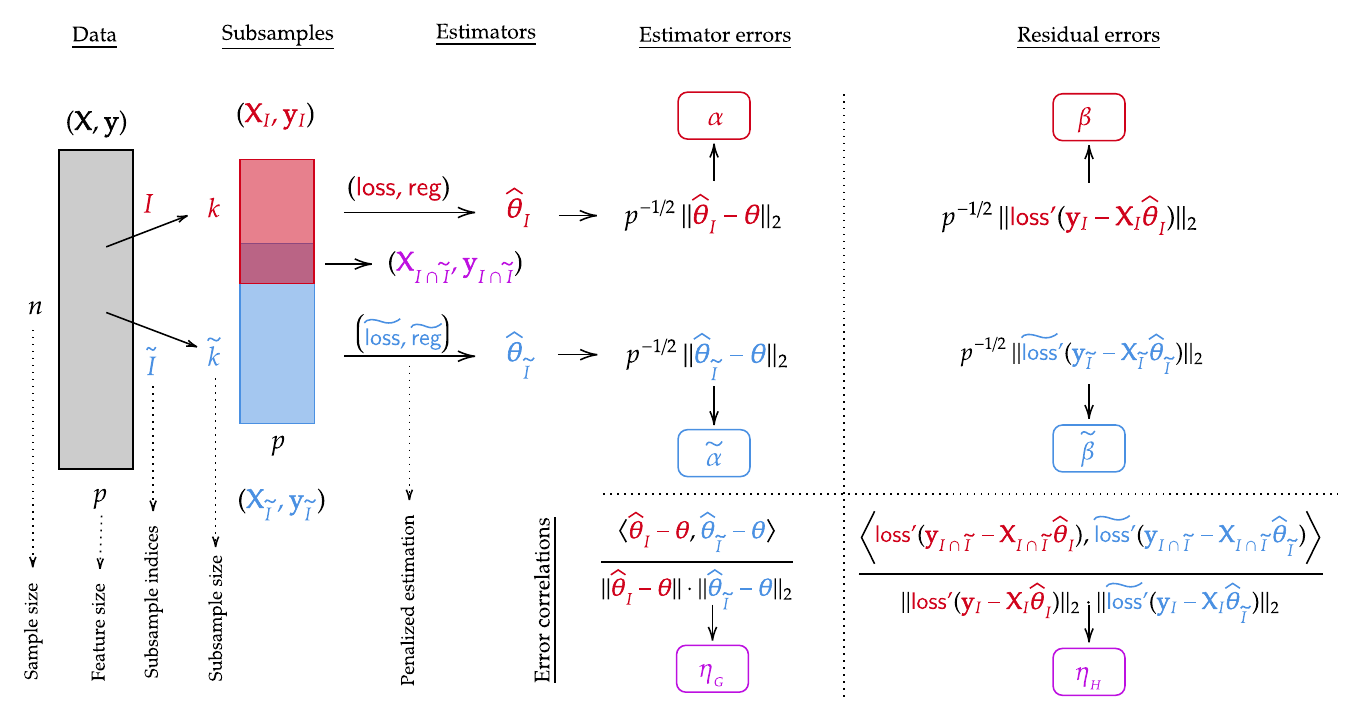}
    \caption{
        Illustration of the asymptotics of the estimator and residual error norms and correlations of overlapped regularized M-estimator.
    }
    \label{fig:risk-asymptotics-illustration}
\end{figure}

\subsection{Interpretation of the parameters in \Cref{sys:general_ensemble-M=1,sys:general_ensemble-M=infty}}
\label{sec:interpretation-general-ensemble}

As mentioned earlier, the six parameters $(\alpha, \beta, \kappa, \nu)$ and $(\etaG, \etaH)$ in \Cref{sys:general_ensemble-M=1,sys:general_ensemble-M=infty} essentially characterize the asymptotic risk of the ensemble estimator.
These deterministic parameters are limits of various stochastic (observable) quantities that we now describe (see also \Cref{fig:risk-asymptotics-illustration} for a visual illustration).

Here $\hat{\btheta}_{I}$ and $\hat{\btheta}_{\tilde I}$ are the component estimators \eqref{eq:def-hbeta} trained on subsamples $(\bX_I,\by_I)$ and $(\bX_{\tilde I},\by_{\tilde I})$ corresponding to index sets $I$ and $\tilde I$ and parameters $(\loss,\reg)$ and $(\tilde \loss, \tilde \reg)$, respectively.
We further define the scalar $\df_I$ and the matrix $\bV_I$ by
\begin{equation}
\df_{I} \coloneq \tr[(\partial/\partial \by_{I}) \bX_{I} \hat{\bbeta}_{I}],
\qquad
\bV_I \coloneqq 
(\partial/\partial \by_{I}) \loss'(\by_I - \bX_I\hat{\btheta}_I)
\in \R^{{|I|\times |I|}}
\label{eq:def-df-V_I}
\end{equation}
and similarly for $\tilde I$.
Two scalars of interest that relate the behavior of $\hat{\btheta}_I$ to the scalars $(\kappa,\nu)$ in \Cref{sys:general_ensemble-M=1}, are $\df_I$ and $\tr[\bV_I]$.

Assuming squared loss, $\loss'(\by_I - \bX_I\hat{\btheta}_I)=\by_I - \bX_I\hat{\btheta}_I$ is simply the residual vector, and the matrix
$\bV_{I}$ simplifies to
$
    \bV_{I} = \bI - (\partial/\partial \by_{I}) \bX_{I} \hat{\bbeta}_{I},
$
so that $\tr[\bV_I] = n - \df_{I}$.
That is, for the squared loss case these quantities can all be related to the usual notion of effective degrees of freedom \cite{stein1981estimation}.
The matrix $(\partial/\partial \by_{I}) \bX_{I} \hat{\bbeta}_{I}$ is usually referred to as the ``hat'' or ``smoothing'' matrix (for linear smoothers), whose trace is the effective degrees of freedom.

If $\loss$ is not the squared loss, but
$\loss'$ is 1-Lipschitz (as in the Huber loss or several robust
regression losses), the quantity
$\loss'(\by_I - \bX_I\hat{\btheta_I})$ 
is still related to a notion of residual vector, and
$\trace[\bV_I]$ is still related to a notion of degrees of freedom.
By \cite[Lemma 9.1]{bellec2020out}, the estimator
$\hat{\btheta}_I$ is the first part of a solution
$(\hat{\btheta}_I,\hat\bu)$ to the convex optimization problem
$$
\min_{\bb\in\R^p,
\bu\in\R^{|I|}}
\|\by_I - \bX_I\bb - \bu\|_2^2
+ \sum_{j=1}^p \reg(b_j)
+ \sum_{i\in I} h(u_i),
$$
where $h:\R\to\R$ is a deterministic convex function related to $\loss$. 
The interpretation of this new optimization problem is that
in the presence of heavy-tailed errors or outliers in some components
of $\by_I$,
we add additional variables $(u_i)_{i\in I}$ to fit those outliers.
As an example, for the Huber loss, $h(\cdot)$ is proportional to the absolute value.
The solution satisfies
$\loss'(\by_I - \bX_I\hat{\btheta_I})
=\by_I - \bX_I\hat{\btheta}_I - \hat\bu
= \by_{{I}} - [\bX\mid \bI_I](\hat{\btheta}_I^\top \mid \hat\bu^\top)^\top$.
That is,
$\loss'(\by_I - \bX_I\hat{\btheta_I})$ is the residual vector
of the optimization problem with enlarged design matrix $[\bX_I \mid \bI_I]\in\R^{|I|\times(|I|+p)]}$.
Consequently, $\tr[\bV_I]$ equals $|I|$ minus the effective degrees-of-freedom of the estimate $(\hat{\btheta}_I^\top,\hat\bu^\top)$ 
fitted using this enlarged design matrix.
With this in mind, we refer in \Cref{tab:interpretations_squared_loss}
to $\tr[\bV_I]$ as residual degrees of freedom in general, and robust residual degrees of freedom for the special case of the Huber loss.
    
Another interpretation of the matrix $\bV_I$ in \eqref{eq:def-df-V_I}
is the Hessian,  with respect to $\by_I$, of the objective value
\eqref{eq:def-hbeta} at $\hat\btheta_I$.
More precisely, with
$$F(\by_I) = 
\sum_{i\in I} \loss (y_i - \bx_i^\top \hat\btheta_I) + \sum_{j\in[p]} \reg((\hat\btheta_I)_j)$$ being the objective value at the minimizer,
{Danskin's theorem \cite[Proposition B.22]{bertsekas2016nonlinear} gives}
$(\partial/\partial y_i) F(\by_{I}) = \loss'(y_i - \bx_i^\top \hat\btheta_I)$.
Differentiating once more reveals that $\bV_I$ in \eqref{eq:def-df-V_I}
is the Hessian of $F(\cdot)$ at $\by_I$
and $\tr[\bV_I]$ is the Laplacian.
Since partial minimization preserves convexity
and the objective function in \eqref{eq:def-hbeta}
is jointly convex in $(\signal,\by)$,
the function $F(\cdot)$ is convex. This interpretation
explains why $\bV_I$ is a positive semi-definite matrix
in cases where closed-form expressions for $\bV_I$
are available (see \Cref{tab:df-V-examples}).

\begin{table}[!t]
    \caption{
        \textbf{Interpretations of various limiting quantities appearing in \Cref{sys:general_ensemble-M=1,sys:general_ensemble-M=infty}.}
        See \Cref{sec:interpretation-general-ensemble}
        for definitions and notations.
    }
    \centering
    \small
    \begin{tabular}{p{4.7cm} p{9.3cm} p{1cm}}
        \toprule
        \textbf{Interpretation} & \textbf{Stochastic quantity} & \textbf{Limit} \\
        \midrule
        Error vector norm squared
        & $\|\hat{\btheta}_{I} - \btheta\|_2^2/p$ & $\alpha^2$ \\
        \addlinespace[1ex]
        Loss gradient norm squared
        & $\|\loss'(\by_{I} - \bX_{I} \hat{\btheta}_{I})\|_2^2/p$
        & $\beta^2$ \\
        \addlinespace[1ex]
        \arrayrulecolor{black!25} \midrule \arrayrulecolor{black} 
        Inner product of error vectors
        & $(\hat{\btheta}_{I} - \btheta)^\top (\hat{\btheta}_{\tilde I} - \btheta) / p$ & $\etaG \alpha \tilde \alpha$ \\
        \addlinespace[1ex]
        Inner product of loss gradients
        & 
        $\loss'(\by_{I \cap \tilde I} - \bX_{I \cap \tilde I} \hat{\btheta}_{I})^\top \tilde\loss'(\by_{I \cap \tilde I} - \bX_{I \cap \tilde I} \hat{\btheta}_{\tilde I}) / p$ & $\etaH \beta \tilde\beta$ \\
        \addlinespace[1ex]
        \arrayrulecolor{black!25} \midrule \arrayrulecolor{black}
        Degrees of freedom
        & 
        $\df_I/p$ where
        $\df_{I} \coloneq \tr[(\partial/\partial \by_{I}) \bX_{I} \hat{\bbeta}_{I}]$
        & $\nu\kappa$ \\
        \addlinespace[1ex]
        Residual degrees of freedom
        & $\tr[\bV_{I}] / p$ where $\bV_{I} \coloneq (\partial/\partial \by_{I}) \loss'(\by_{I}-\bX_{I} \hat \btheta)$ 
        & $\nu$ \\
        \addlinespace[1ex]
        Generalized resolvent trace
        & 
        $\trace[(\bX_{I}^\top   \diag[\loss'' (\by_{I} - \bX_{I} \hat{\bbeta}_{I})] \bX_{I} + \diag[\reg''(\hat{\bbeta}_{I})])^{-1}]$
        if $\reg$ is twice differentiable and
        $\df_{I} / \tr[\bV_{I}]$ if $\reg$ is non-smooth
        & $\kappa$ \\
        \arrayrulecolor{black}
        \bottomrule
    \end{tabular}
    \label{tab:interpretations_squared_loss}
\end{table}

The convergence of the estimator and residual error norms
(first two rows of \Cref{tab:interpretations_squared_loss})
is proved in \cite{thrampoulidis2018precise, celentano2020lasso, loureiro2021learning} using the CGMT. 
Convergence of the corresponding inner products
(third, fourth row of \Cref{tab:interpretations_squared_loss})
is novel and established in \Cref{thm:corr-sigerror-reserror}.
The convergence
\begin{equation}
    \label{eq:nu-kappa-convergence}
    \tr[\bV_{I}]/p \pto \nu,
    \qquad
    \df_I/p \pto \nu \kappa,
    \qquad
    \df_I / \tr[\bV_I]
    \pto \kappa,
\end{equation}
was so
far only known for the lasso 
\cite[Theorem 8]{celentano2020lasso} or the squared loss
\cite[Corollary 3.2]{bellec2020out}.
To our knowledge, the present paper
is the first to establish the above convergence in probability
for regularized estimators and
robust loss functions beyond the squared loss.
The proof is given in \Cref{subsec:proof_convergence_trace}.

Assuming twice differentiable $\loss$ and $\reg$ functions, the parameter $\kappa$ is also the limiting trace of a resolvent-like matrix $\bA_{I} \coloneq (\bX_{I}^\top   \diag[\loss'' (\by_{I} - \bX_{I} \hat{\bbeta}_{I})] \bX_{I} + \diag[\reg''(\hat{\bbeta}_{I})])^{-1}$, which in the further special case of squared loss and squared regularizer (with regularization level $\lambda$) simplifies to the standard ridge resolvent: $\bA_{I} = (\bX_{I}^\top \bX_{I} + \lambda \bI)^{-1}$.
We refer to $\bA_I$ as the generalized resolvent for convenience in \Cref{tab:interpretations_squared_loss}.

\subsection{Asymptotics of ensemble risk}
\label{sec:ensemble-risk-asymptotics}

Using the parameters in \Cref{sys:general_ensemble-M=1,sys:general_ensemble-M=infty}, we are now ready for our main result on the squared risk asymptotics of the ensemble estimator.
The squared risk of the ensemble estimator $\tilde\btheta_M = \frac{1}{M}\sum_{m\in [M]} \hat\btheta_m$ can be decomposed as
\begin{align*}
    \frac{1}{p} \big\| \tilde\btheta_M - \btheta \big\|_2^2 
    = 
    \frac{1}{M^2} \sum_{m\in[M]} \frac{1}{p} \|\hat\btheta_m-\btheta\|_2^2
    +  \frac{1}{M^2} \sum_{{m,\ell \in [M]:m \neq \ell}} \frac{1}{p} (\hat{\bbeta}_m- \bbeta)^\top (\hat{\bbeta}_\ell- \bbeta). 
\end{align*}
Noting $\|\hat\btheta_m-\btheta\|_2^2/p \pto \alpha_m^2$ and applying \Cref{thm:corr-sigerror-reserror} to the cross term $(\hat{\bbeta}_m- \bbeta)^\top (\hat{\bbeta}_\ell- \bbeta) / p$ for each $m\ne\ell$, we arrive at the following result:
\begin{corollary}
    [General ensemble risk characterization]
    \label{th:nonlinear} 
    Suppose the assumptions of \Cref{thm:corr-sigerror-reserror} hold.
    For $m \in [M]$, let $\alpha_m$ be the parameter satisfying \Cref{sys:general_ensemble-M=1}.
    For $m, \ell \in [M]$, let $\etaG(m,\ell)$ be the parameter satisfying \Cref{sys:general_ensemble-M=infty}. 
   Then, as $n, p, k \to \infty$ with $n/p \to \delta \in (0, \infty)$ and $k_m / n \to c_m \in (0, 1]$, we have 
   \begin{equation}
       \label{eq:ensemble-risk-limit-heterogeneous}
       \frac{1}{p} \big\| \tilde\btheta_M - \btheta \big\|_2^2
       \pto 
       \cR_M
       \coloneq
       \frac{1}{M^2} \sum_{m\in[M]} \alpha_m^2 + 
       \frac{1}{M^2} \sum_{{m,\ell \in [M]: m \neq \ell}} \etaG(m,\ell) \cdot \alpha_{m} \alpha_{\ell}.
   \end{equation}
\end{corollary}

Since the parameters $\alpha_m$ and $\etaG(m, \ell)$ implicitly depend on $\delta$ and $c_m$, the asymptotic risk $\cR_M$ also implicitly depends on these parameters.
For brevity, we will simply write $\sR_M$ unless we wish to explicitly point out this dependence.
The factor $\etaG$ captures the benefit of ensembling.
It is worth noting that a negative $\etaG$ (which intuitively corresponds to a component that does better in a different direction) will improve the ensemble risk if the components themselves also have small risks.
This aligns with the higher level intuition in ensembling that one should ensemble predictors that each perform well, preferably on different parts of the input space.

\subsection{Risk estimation}
\label{sec:risk-estimation}

The risk characterization in \Cref{th:nonlinear} depends on the population-level characteristics (such as the signal and noise distributions $F_\theta$ and $F_\eps$) and provides useful theoretical insights into the risk behavior of the ensemble estimator in terms of these quantities.
In practical applications, however, the statistician needs to estimate the risk accurately to tune ensemble hyperparameters effectively using the observed data $(\bX,\by)$.
These hyperparameters include the choice of component estimators (through $\loss$ and $\reg$), their level of regularization (regularization level for $\reg$), the subsample sizes ($k$), and the ensemble size ($M$).
For this purpose, we next construct a data-dependent proxy for the squared risk, which one can then tune with respect to various hyperparameters.

\begin{definition}
    [Risk estimator component]
    \label{def:overlapped-risk-estimator-term}
    Let $\hat{\btheta}_{I}$ and $\hat{\btheta}_{\tilde I}$ be the component estimators \eqref{eq:def-hbeta} trained on subsamples $(\bX_I,\by_I)$ and $(\bX_{\tilde I},\by_{\tilde I})$ corresponding to index sets $I$ and $\tilde{I}$ with parameters $(\loss, \reg)$ and $(\tilde\loss, \tilde\reg)$.
    Corresponding to estimators $\hat{\btheta}_{I}$ and $\hat{\btheta}_{\tilde I}$, define:
    \begin{enumerate}[leftmargin=5mm]
        \item
        Degrees of freedom: $\df_I = \tr[(\partial/\partial \by_{I}) \bX_{I} \hat{\btheta}_{I}]$ and $\df_{\tilde I}= \tr[(\partial/\partial \by_{\tilde I}) \bX_{\tilde I} \hat{\btheta}_{\tilde I}]$.
        \item
        Residual errors: $\br = \by_{I} - \bX_{I} \hat{\btheta}_{I}$ and $\tilde{\br} = \by_{\tilde I} - \bX_{\tilde I} \hat{\btheta}_{\tilde I}$.
        \item
        Residual degrees of freedom: traces of
        \smash{$\bV_{I} = (\partial/\partial \by_I) \loss'(\by_I - \bX_I \hat{\btheta}_{I})$} \newline and \smash{${\bV}_{\tilde I} = (\partial/\partial \by_{\tilde I}) \tilde \loss'(\by_{\tilde I} - \bX_{\tilde I} \hat{\btheta}_{\tilde I})$}.
    \end{enumerate}
    Using these quantities, define an observable quantity $\EST_{I, \tilde I}$ as follows:
    \begin{equation}
        \label{eq:risk-estimator-general-subagging}
        \EST_{I, \tilde I} :=  \frac{1}{n} \sum_{i\in [n]} \Bigl(r_i + \ind_{\{i\in I\}} \frac{\df_I}{\tr[\bV_I]} \loss'(r_i) \Bigr) \Bigl(\tilde{r}_i + \ind_{\{i\in \tilde I\}} \frac{\df_{\tilde I}}{\tr[\bV_{\tilde I}]} \tilde\loss'(\tilde r_i)\Bigr)
    \end{equation}
    where $\ind_\Omega$ denotes the indicator function associated with an event $\Omega$. 
\end{definition}

The quantities $\tr[\bV_{{I}}]$ and $\df_{{I}}$
have explicit closed-form expressions for special choices of $\loss$ and $\reg$.
Some of these are summarized in \Cref{tab:df-V-examples}.
We show next that $\EST_{I, \tilde I}$ approximates well the component of prediction risk corresponding to the inner product of estimator errors of $\hat \btheta_I$ and $\hat{\btheta}_{\tilde I}$.
We then naturally construct a risk estimator for the prediction risk of the ensemble estimator.

\begin{theorem}
    [Consistency of risk estimator component]
    \label{thm:risk-estimator-general-subagging}
    In addition to \Crefrange{asm:linear_model}{asm:regularity-conditions}, assume that $\reg$ and $\tilde\reg$ are strongly convex. Then we have
    \begin{equation*}
        \textstyle
        \frac{1}{p} (\hat{\btheta}_{I}-\btheta)^\top (\hat{\btheta}_{\tilde I}-\btheta) + \frac{\|\bz\|_2^2}{n} =  \EST_{I, \tilde I} + \Op(n^{-1/2}) \bigl(1+\frac{\|\bz\|_2}{\sqrt{n}} \bigr).
    \end{equation*}
\end{theorem}

The risk estimator $\EST_{I, \tilde I}$ is a generalization of the criterion originally proposed by \cite{bellec2022derivatives} for non-ensemble regularized M-estimator. 
Although in this paper we focus on the isotropic Gaussian design $\bSigma = \bI_p$, the same argument in the proof of \Cref{thm:risk-estimator-general-subagging} works in the anisotropic design $\bSigma \ne \bI_p$. 
As a result, we can show that the $\EST_{I, \tilde I}$ (without any modification) approximates $p^{-1} (\hat{\btheta}_{I}-\btheta)^\top \bSigma (\hat{\btheta}_{\tilde I}-\btheta) + n^{-1} \|\bz\|_2^2$ under the event that $(\tr[\bV_I], \tr[\bV_{\tilde I}])$ are bounded from below by a positive constant as in \cite[Theorem 5.3]{bellec2022derivatives}. 

We believe that the strong convexity assumption on $\reg$ in \Cref{thm:risk-estimator-general-subagging} is an artifact of our proof (see \Cref{fig:effect-of-M-huber} for an illustration where $\reg$ is not strongly convex). 
Note that this type of assumption of strong convexity has already been assumed in \cite{bellec2022derivatives}. 
This assumption guarantees for free that the coefficient $\df/\tr[\bV]$ does not blow up. 
However, we emphasize that \Cref{thm:corr-sigerror-reserror} and \Cref{th:nonlinear} for risk characterization do not require the strong convexity assumption.
The approximation argument (see \Cref{sec:ridge_smoothing}) used to
prove \Cref{thm:corr-sigerror-reserror} and \Cref{th:nonlinear} in the
non-strongly convex case is again applicable in the context 
of \Cref{thm:risk-estimator-general-subagging}, although it is not currently
sufficient to conclude a version \Cref{thm:risk-estimator-general-subagging}
for non-strongly convex regularizers due to the difficulty of establishing
the continuity of $\df_I$ and $\tr[\bV_I]$ with respect to the perturbation
parameter $\mu$ as $\mu \to 0$ in \eqref{eq:def_h_smoothed}.

Equipped with the component risk estimator \eqref{eq:risk-estimator-general-subagging}, we can now construct a consistent risk estimator for the ensemble estimator \eqref{eq:def-M-ensemble}:

\begin{corollary}
    [General ensemble risk estimation]
    \label{cor:general-ensemble-risk-estimator}
    Fix $M \ge 1$ and consider the ensemble estimator $\tilde\btheta_M = \tfrac{1}{M} \sum_{m\in [M]} \hat \btheta_m$ where the component estimator $\hat \btheta_m$ is trained with $(\loss_m,\reg_m)$ on subsample $I_m$ for $m \in [M]$ as in \eqref{eq:def-hbeta}. Define an estimator $\EST$ for the squared prediction risk:
    \begin{equation*}
        \textstyle
        \EST := \frac{1}{M^2} \sum_{m,\ell \in [M]}\EST_{m,\ell}
    \end{equation*}
    where $\EST_{m,\ell}$ is $\EST_{I, \tilde I}$ as defined in \eqref{eq:risk-estimator-general-subagging} with $(\loss, \reg, I) = (\loss_m, \reg_m, I_m)$ and $(\tilde\loss, \tilde\reg, \tilde I) = (\loss_\ell, \reg_\ell, I_\ell)$. Then we have
    \begin{equation}
        \textstyle
        \label{eq:risk-estimator-general-subagging-gaurantee}
        \frac{1}{p} \bigl\|\tilde\btheta_M - \btheta\bigr\|_2^2 
        + \frac{\|\bz\|_2^2}{n}
        =\EST + \Op(n^{-1/2}) \bigl(1 + \frac{\|\bz\|_2}{\sqrt{n}} \bigr).
    \end{equation}
\end{corollary}

If the noise distribution has enough moments, the guarantee \eqref{eq:risk-estimator-general-subagging-gaurantee} implies that $\EST$ approximates the (full) prediction risk (that includes the irreducible error) of the ensemble estimator $\tilde\btheta_M$ under \Cref{asm:linear_model}:
$$
    \E\bigl[\bigl(y_0 - \bx_0^\top \tilde\btheta_M \bigr)^2 \mid \by, \bX, \{ I_m \}_{m\in [M]} \bigr] 
    = \tfrac{1}{p} \|\tilde\btheta_M-\btheta\|_2^2 + \E[Z^2],
$$
where $(y_0, \bx_0)\in\R\times \R^{p}$ is an independent test point sampled from the same distribution as the training data $(\by,\bX)$ .
More precisely, if the noise has a finite second moment, then $\EST$ is consistent for the prediction risk. 
Furthermore, if the noise distribution has a finite fourth moment, by the central limit theorem (on the terms involving noise averages), $\EST$ is $\sqrt{n}$-consistent for the prediction risk.
This rate is not improvable because, for the single ordinary least squares (OLS) estimator (with $\loss(x) = x^2/2$, $\reg=0$, $|I|=|\tilde I| = n$) and $F_z = \cN(0,1)$, the risk estimator gives 
$
\EST = \frac{1}{n} \frac{\|\by-\bX\hat\btheta_{\text{ols}}\|_2^2}{(1-p/n)^2} \overset{\text{d}}{=} \frac{\chi_{n-p}^2}{n(1-p/n)^2}
$
and the standard deviation of the $\chi^2_{n-p}$ incurs an unavoidable term of order $n^{-1/2}$.

If the noise distribution $F_\eps$ does not have a finite second moment but has a finite $(1+\epsilon)$-moment for $\epsilon \in [0, 1)$, then even if the estimator $\EST$ may not track the prediction risk (because the prediction risk does not necessarily converge), minimizing $\EST$ is approximately equivalent to minimizing the excess squared risk $p^{-1} \|\tilde\btheta_M-\btheta\|_2^2$ (which does converge). 
This is because the moment assumption implies $\sum_{i\in [n]} z_i^2 = \op(n^{2/(1+\epsilon)})$ 
so that \eqref{eq:risk-estimator-general-subagging-gaurantee} yields 
$$
    \textstyle
   \frac{1}{p} \|\tilde\btheta_M-\btheta\|_2^2
    = \EST 
    - \frac{\|\bz\|_2^2}{n}
    + \op(n^{-\frac{\epsilon}{1+\epsilon}}),
$$
where the subtraction term $\frac{\|\bz\|_2^2}{n}$ is independent of hyperparameters $(\loss_m, \reg_m, I_m)_{m\in [M]}$. 
We illustrate this in \Cref{fig:effect-of-M-huber} with noise following Student's $t_2$ distribution (that does not have a finite second moment).

{
A natural question arising from the above discussion is whether hyperparameters (such as the subsample ratio $c$ and the regularization level $\lambda$ in $\reg(x)=\lambda|x|^q$) selected by tuning $\EST$ are close to the oracle parameters that minimize the true excess risk $p^{-1}\|\tilde\btheta_M -\btheta\|_2^2$ (or its deterministic limit $\mathcal{R}_M$ in \eqref{eq:ensemble-risk-limit-heterogeneous}). 
Establishing such a result theoretically requires proving a suitable form of smoothness (such as \Holder~or Lipschitz continuity) of the excess risk (or its deterministic limit) as a function of the hyperparameters. 
This has been achieved in specific settings, including for the lasso in \cite{celentano2020lasso}, for the unregularized robust M-estimators in \cite{bellec2023error}, and for the ridge estimator in \cite{patil2021uniform,han2023distribution}, among others. 
However, these analyses are highly tailored to the specific estimators and leverage their unique structures, typically involving substantial technical machinery.
A general investigation of this type for the broad class of estimators studied in this paper is an important direction for future work.
}
\begin{figure*}[!t]
    \centering
    \includegraphics[width=0.8\textwidth]{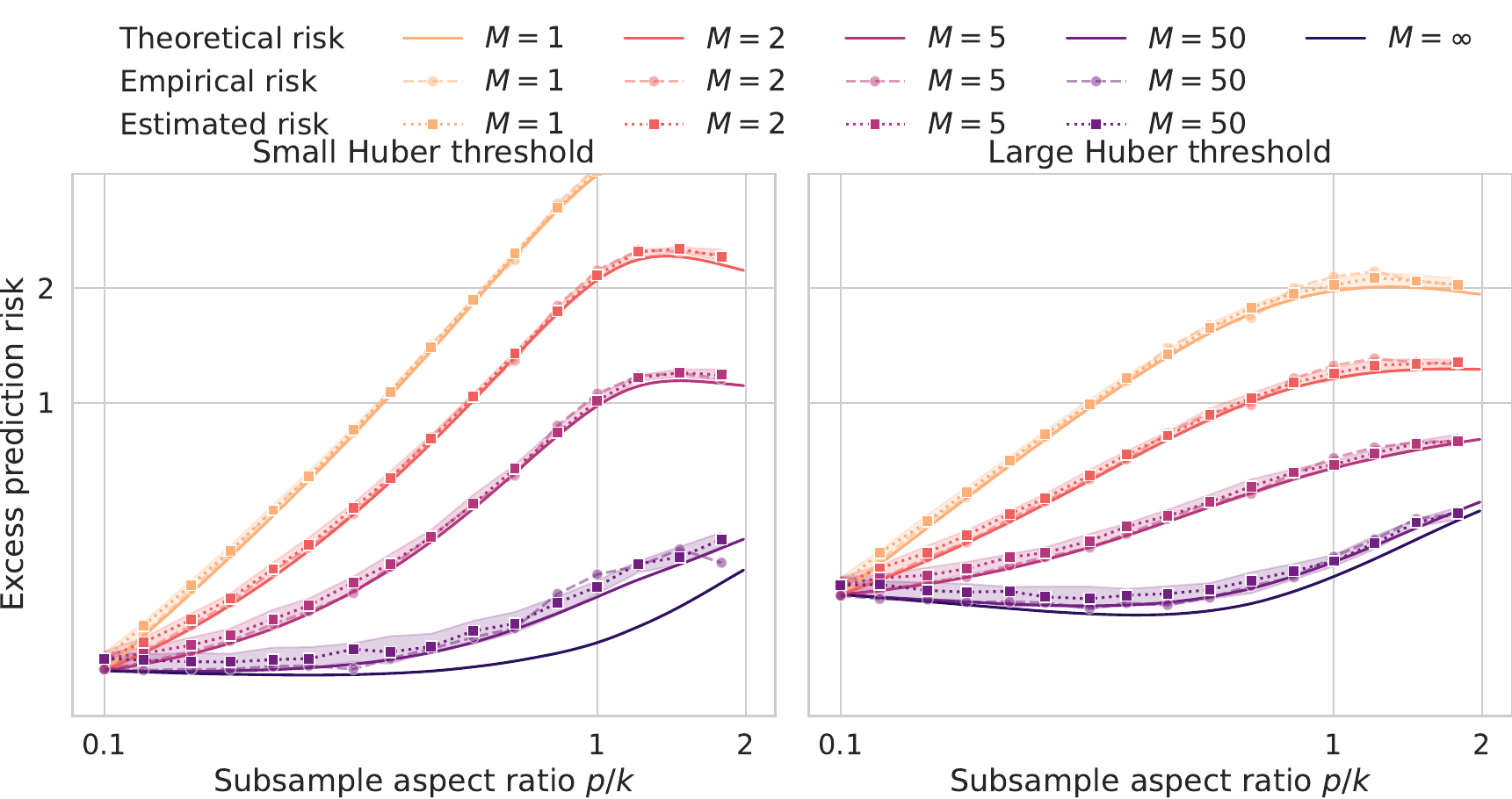}
  \caption{
      Risk of $\ell_1$-regularized Huber ensemble at different subsample aspect ratios $p/k$ with $\ell_1$-regularization parameter $\lambda=0.2$ and varying ensemble size $M$ in the underparameterized regime when $p/n=0.1$ and $n = 5000$.
    The solid lines represent the theoretical risks, the dashed lines represent the empirical risks averaged over $50$ simulations, and the shaded regions represent the empirical standard errors.
    The data model is given by \Cref{subsec:Huber-ex} where the noise follows Student's $t$ distribution $t_{2}$.
    \emph{Left}: Huber threshold parameter $1$.
    \emph{Right}: Huber threshold parameter $5$. 
  }
  \label{fig:effect-of-M-huber}
\end{figure*}

\section{Homogeneous ensembles}
\label{sec:specific-examples}

While \Cref{th:nonlinear} applies for a generic heterogeneous ensemble
(where $\reg_m,\loss_m,|I_m|$ are allowed to differ for distinct $m$),
concrete theoretical insights can be obtained for the homogeneous case
(where $\reg_m,\loss_m,|I_m|$ are the same for every $m$).
In the general heterogeneous case, the effect of increasing ensemble size $M$ is not straightforward, as in general we only have:
\begin{equation*}
    \min_{m \in [M]}\{ \alpha_m^2 \}
    \overset{?}{\lesseqgtr}
   \frac{1}{M^2} \sum_{m\in[M]} \alpha_m^2 + 
   \frac{1}{M^2} \sum_{\substack{m,\ell \in [M] \\ m \neq \ell}} \etaG(m,\ell) \cdot \alpha_{m} \alpha_{\ell}
   \le
   \max_{m \in [M]}\{ \alpha_m^2 \}.
\end{equation*}
In other words, we will do no worse than the worst component but may not do better than the best component.
More ensembles may or may not improve performance depending on the risks of individual estimators.
In contrast, for homogeneous ensembles, we show in the next subsection that increasing ensemble size does indeed help reduce the risk. 
This uniformity allows for more concrete analytical results and practical insights.

\subsection{Risk properties}
\label{sec:homogeneous-risk-properties}

For the homogeneous ensemble of component estimators trained with the same $\loss$, $\reg$, and subsample size $k$, as $n, p, k \to \infty$ with $n/p \to \delta \in (0, \infty)$ and $k / n \to c \in (0, 1]$, the limiting risk \eqref{eq:ensemble-risk-limit-heterogeneous} is given by:
\begin{equation}
   \label{eq:ensemble-risk-limit-generic}
   \mathcal{R}_M = \frac{1}{M}  \mathcal{R}_1 + \Bigl(1-\frac{1}{M}\Bigr)  \mathcal{R}_\infty \quad \text{where} \quad \begin{dcases}
        \mathcal{R}_1 \coloneq \alpha^2 \quad \text{(non-ensemble risk)}, \\
        \mathcal{R}_\infty \coloneq \etaG\alpha^2 \quad \text{(full-ensemble risk)}.
   \end{dcases}
\end{equation}
Observe that the limit $\cR_M$ is simply a convex combination of the asymptotic risk of the single estimator $\cR_1$ and of the full-ensemble estimator $\cR_\infty$.
The reason we call this the ``full'' ensemble is that the ensemble estimator $\hat{\btheta}_M$ when $M \to \infty$ is almost surely equal (coordinate-wise and conditioned on the data) to an ensemble estimator fitted on all possible (and distinct) $\binom{n}{k}$ subsamples of size $k$; see Lemma A.1 of \cite{du2023subsample} for a precise statement and proof.
(Note that when $M \to \infty$, $\cR_M$ does indeed converge to $\cR_\infty$, justifying the notation for the limit for the case when $m \neq \ell$.)
As a sanity check, note that in the special case when $c=1$, the above setting corresponds to the non-ensemble case discussed in \cite[Section 4]{thrampoulidis2018precise}.

The following result shows the advantage of ensembling.
In the classical bagging and subagging literature, it is well known that the risk of the ensemble estimator decreases as the ensemble size increases, due to the reduction in variance that comes with having more predictors in the ensemble.
In the proportional asymptotic regime that we study in this paper, since subagging also introduces bias, this is not immediate.
A general result along these lines that verifies the monotonicity of the risk of the ensemble itself (not the asymptotic limit) follows from Proposition 3.1 of \cite{patil2022bagging}; see Equation (10) of \cite{patil2022bagging}.
Below we verify that the asymptotic risk is strictly monotonic in $M$ by showing that {$\etaG \in [0,1)$} in general.
This implies that the asymptotic risk is strictly decreasing in $M$.

\begin{proposition}
    [Improvement due to ensembling]
    \label{prop:monotonicity-ensemble-size}
    Fix the subsample ratio $c = k/n \in (0, 1)$ and let $\cR_M$ be the limiting risk as defined in \eqref{eq:ensemble-risk-limit-generic}. 
    Then $\cR_M$ is strictly decreasing in the number of ensembles $M$, i.e., $\cR_{M+1} < \cR_{M}$ for all $M\in \mathbb{N}$. 
\end{proposition} 

This monotonicity in the ensemble size $M$ is illustrated by \Cref{fig:effect-of-M-huber} for the ensemble of $\ell_1$-regularized Huber regression.

Because the risk decreases in $M$, the optimal ensemble size is $M = \infty$.
(In practice, setting $M = \binom{n}{k}$ suffices by only averaging over estimators trained on distinct subsamples (see Appendix A.1 of \cite{du2023subsample} for more details), but this still can be quite large.)
However, it may not be feasible to use an ensemble size of $M = \infty$.
In practice, it suffices to use a large enough $M$ that gives a suboptimal risk close to the full-ensemble risk.
For this purpose, a natural idea is to estimate the risk of non-ensemble estimator $M = 1$ and the full estimator $M = \infty$, and obtain an estimate for the risk of $M$-ensemble using the relationship in \eqref{eq:ensemble-risk-limit-generic}.
This is very similar to the extrapolated cross-validation estimator (ECV) of \cite{du2023extrapolated}, which estimates the risk of $M = 1$ and $M = 2$. 

We also show that for any ensemble size $M$, when the subsample ratio $c$ is optimized, the resulting risk decreases in the inverse data aspect ratio $\delta=\lim n/p$. 
To prepare for the forthcoming statement, let us write the limiting risk $\cR_M$ in \eqref{eq:ensemble-risk-limit-generic} by $\cR_M(\delta, c)$ to make the dependence on the limit $\delta = \lim n/p$ and subsample ratio $c=\lim k/n$ clear. 
With this notation, we can say the following about the optimally subsampled risk.
\begin{proposition}
    [Monotonicity of optimally subsampled risk]
    \label{prop:monotonicity-risk} 
    The map $\delta \mapsto \inf_{c\in (0,1)} \cR_M(\delta, c)$ is non-increasing over $\delta \in (0, \infty)$ for all $M \in \NN {~\cup \{+\infty\}}$.
\end{proposition}

{
The proof of this proposition is based on an explicit formula for the partial derivative of $\mathcal{R}_M$ with respect to the data aspect ratio $\delta^{-1} = \lim p/n$ while holding the subsample-to-feature-size ratio $c \delta = \lim k/p$ fixed; see \Cref{proof:monotonicity-risk}. 
}

A consequence of this proposition is that the risk of the optimal ensemble estimator is decreasing in the sample size $n$ for a fixed (large enough) feature size $p$.
Moreover, combined with \Cref{prop:monotonicity-ensemble-size}, we also have that this monotonic risk profile lies below the function $\cR_1(\delta, 1)$, the risk profile of the original predictor trained once on the full dataset $(\bX, \by)$ with no ensembling (the risk of which, as we discussed above, can be non-monotonic). 
Such monotonicity in the inverse data aspect ratio is important because it ensures that increasing the amount of data relative to features consistently improves the estimator's performance. 
In a sense, a monotonic decrease in risk with the optimal subsample ratio certifies that the model effectively utilizes all the additional data, leading to better performance as the data size grows.
This result is illustrated in \Cref{fig:optimum_phis_overparameterized-2-optrisk-huber-lasso} for the $\ell_1$-regularized Huber regression.

\begin{figure}[!t]
    \centering
    \includegraphics[width=0.99\textwidth]{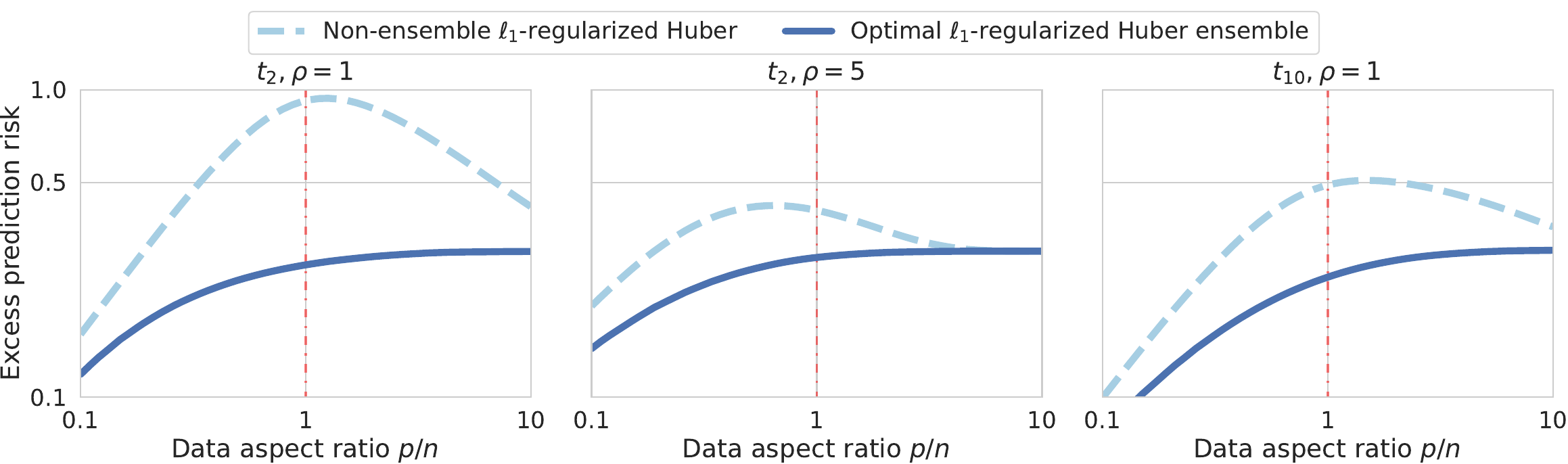} 
    \caption{
        \textbf{Optimal subsample risk of the Huber lasso ensemble is monotonic in the data aspect ratio.}
        {Comparison of the the limiting risks of non-ensemble ($M = 1$) $\ell_1$-regularized Huber regression fitted on full data $(c=1)$ and optimally subsampled (over the subsample ratio $c=\lim (k/n)\in (0,1]$) full ensemble ($M=\infty$) $\ell_1$-regularized Huber regression}, for fixed $\ell_1$-regularization parameter $\lambda=0.5$ and varying Huber parameter $\rho$, at different data aspect ratios $p/n$ ranging from $0.1$ to $10$.
        The data model is given in \Cref{subsec:Huber-ex}.
        \emph{Left}: noise follows Student's $t$ distribution $t_{2}$ and Huber parameter $\rho=1$.
        \emph{Middle}: noise follows Student's $t$ distribution $t_{2}$ and Huber parameter $\rho=5$.
        \emph{Right}: noise follows Student's $t$ distribution $t_{10}$ and Huber parameter $\rho=1$.
    }
    \label{fig:optimum_phis_overparameterized-2-optrisk-huber-lasso}
\end{figure}

\subsection{Examples and connections to literature}
\label{sec:examples-square-loss-and-square-penalty}

In this section, we specialize the general risk characterization in \Cref{th:nonlinear} in several examples of interest.
Throughout this section, we will consider ensembles of component predictors trained on the same $\loss$, $\reg$ (with a tuning parameter $\lambda$), and subsample size $k$.
We begin by considering convex regularized least squares, with further specialization to bridge, ridge, and lasso.
We then consider general ridge regularized estimators allowing for non-squared loss in \Cref{sec:general-train-loss} with further specialization to the Huber loss.
For the reader's convenience, the proximal operators, their derivatives, Moreau envelopes, and their derivatives for the ridge and lasso regularization and Huber loss functions are recalled in \Cref{tab:prox_and_derivatives_ridge_lasso}.

\subsubsection[Ensemble of regularized least squares]{Ensemble of regularized least squares}
\label{sec:ensembles-bridge-estimators}

In this section, we consider subagging regularized least squares. 
Given a common regularizer $\reg$ and a regularization parameter $\lambda > 0$, the component estimators are given by:
\begin{equation*}
    \hat \bbeta_m
    \coloneq \argmin_{\signal \in \R^p} \frac{1}{2} \| \by_{\setm} - \bX_{\setm} \signal \|_2^2 + \lambda \sum_{j\in [p]} \reg(b_j).
\end{equation*}
Now \Cref{asm:regularity-conditions}-(1) translates to having a bounded second moment of the noise distribution $F_\eps$, which we denote by $\sigma^2$. 
Using the explicit formula $\prox_{\loss}(x;\tau) = x/(1+\tau)$ for $\loss(x)=x^2/2$ and performing some change of variables, the risk convergence in \eqref{eq:ensemble-risk-limit-generic} holds with $(\cR_1, \cR_\infty)$ given by: 
$$
    \cR_1 = \tau^2 - \sigma^2
    \quad
    \text{and}
    \quad
    \cR_\infty = \xi^2- \sigma^2,
$$
where $\tau$ and $\xi$ are the solutions to the following systems:

\renewcommand{\thesystem}{\arabic{system}}
\begin{system}
    [Ensembles of regularized least squares]
    \label{sys:ensembles-penalized-least-squares}
    Given $\lambda \in (0, \infty)$, $\delta \in (0, \infty)$, $c \in (0, 1]$, $\sigma^2 \in [0, \infty)$, define the following 2-scalar system of equations in variables $(\tau, a)$:
    \begin{subequations}
    \label{eq:amp-bridge-M=1}
    \begin{empheq}{align}
    \tau^2 
    &=  \mathbb{E}\big[\big(\prox_{\reg}\big(\Theta + \tfrac{\tau}{\sqrt{c\delta}} H; \tfrac{a\tau}{\sqrt{c\delta}}\big) - \Theta\big)^2\big] + \sigma^2 \label{eq:amp-bridge-tau}, \\
    \tfrac{\lambda}{\sqrt{c\delta}} &= a\tau \big( 1 -  \tfrac{1}{c\delta}\E\big[
          \prox_{\reg}' \big(\Theta + \tfrac{\tau}{\sqrt{c\delta}} H ; \tfrac{a\tau}{\sqrt{c\delta}}\big) \big] \big), \label{eq:amp-bridge-a} 
    \end{empheq}
    \end{subequations}
    where $H \sim \cN(0,1)$ and $\Theta \sim F_\theta$ are independent.
    Given $(\tau,a) \in \RR^{2}_{> 0}$ that satisfy the above systems, define the following 1-scalar system of equations in variable $\xi$:
    \small
    \begin{subequations}
    \label{eq:amp-bridge-M=infty}
    \begin{empheq}{align}
    \hspace{-1em}
    \xi^2 
    = \mathbb{E}\big[\big(\prox_{\reg}\big(\Theta + \tfrac{\tau}{\sqrt{c\delta}} H; \tfrac{a\tau}{\sqrt{c\delta}}\big) - \Theta\big) \cdot \big(\prox_{\reg}\big(\Theta + \tfrac{\tau}{\sqrt{c\delta}} \tilde{H}; \tfrac{a\tau}{\sqrt{c\delta}}\big) - \Theta\big)\big] + \sigma^2
    \label{eq:amp-bridge-M=infty-a}
    \end{empheq}
    \end{subequations}
    \small
    where $(H,\tilde H)$ {are as in \eqref{distrib}}
    with
    $\etaH = c \frac{\xi^2}{\tau^2}$,
    and $\Theta \sim F_\theta$ independent.
\end{system}

Compared to \Cref{sys:general_ensemble-M=1,sys:general_ensemble-M=infty}, the parameterization in \Cref{sys:ensembles-penalized-least-squares} is slightly different.
This is done to match the result with some of the existing results for regularized least squares (for $M = 1$).
The invertible transformations are given by: $a = \tfrac{\lambda}{\beta}$, $\tau = \sqrt{c \delta} \tfrac{\beta}{\nu}$, and $\xi^2 = \etaG \alpha^2 + \sigma^2$.
These parameters also have interpretations as in \Cref{sec:interpretation-general-ensemble}, summarized below.

\begin{remark}
    [Interpretation of parameters in \Cref{sys:ensembles-penalized-least-squares}]
    \label{rem:interpretation-squaredloss-ensemble}
    The parameter $\tau^2$ is simply the (full) prediction risk of the non-ensemble estimator in the limit, which is $\alpha^2 + \sigma^2$.
    Note that $\tau^2 \ge \sigma^2$ and is sometimes referred to as effective ``inflated'' noise variance due to the high dimensionality \cite{donoho2009message,bayati2011lasso}.
    Moreover, the fact that $\tau^2 = \tfrac{(c \delta) \beta^2}{\nu^2} = \tfrac{\beta^2/(c \delta)}{\nu^2 / (c \delta)^2}$ is also at the core of consistency of generalized cross-validation (discussed in \Cref{sec:risk-estimation}) for the non-ensemble estimator. 
    Here, the numerator is the asymptotic training error and the denominator is the asymptotic degrees of freedom correction.
    (Observe that the factors of $c \delta$ arise because both the asymptotic training error $\beta^2$ and the asymptotic degrees of freedom correction $\nu$ are defined with normalization of $p$ in \Cref{tab:interpretations_squared_loss}.)
    The parameter $a \tau$ is the effective threshold parameter at which one applies the proximal operator to the noise-inflated effective observation and appears in approximate message passing (AMP) formulations \cite{donoho2009message}.
    The parameter $a$ serves as a proportionality constant between the effective threshold $a \tau$ and standard deviation $\tau$ of the inflated effective noise.
    Finally, the parameter $\xi^2$, which is the main contribution of this paper, is the full-ensemble predictor risk (when $M \to \infty$), which is also $\etaG \alpha^2 + \sigma^2$.
\end{remark}

In the following, we isolate some special cases of regularized M-estimators to compare with existing work.

\begin{remark}
    [Bridge ensembles]
    \label{rem:ensembles-bridge}
    Bridge estimators are also known as $\ell_q$-regularized least squares and are a popular class of regularized M-estimator \cite{frank1993statistical,fu1998penalized}.
    For $\ell_q$ regularizer $\reg_q(x)=|x|^q$, the risk of the bridge for $M = 1$ is derived in \cite{weng2018overcoming,wang2020bridge} for general $q \in [1,2]$. 
    For instance, we recover \cite[Theorem 2.1]{weng2018overcoming} by the change of variables $(\lambda',\Theta')=(\lambda/\sqrt{c\delta},\sqrt{c\delta}\Theta)$ such that the limiting prediction risks match $\tau' = \tau$.
     \Cref{eq:amp-bridge-M=infty-a} generalizes it for any $M \ge 1$.
    Further special boundary cases of $q = 2$ (ridge) and $q = 1$ (lasso) are further isolated in the next two remarks.
\end{remark}

\begin{remark}
    [Ridge ensembles]
    \label{rem:ensembles-ridge}
    For ridge regression when $q=2$, \Cref{sys:ensembles-penalized-least-squares} recovers Theorem 4.1 of \cite{patil2022bagging} under isotropic features, using a slight change of variables $(\lambda', \Theta')= (\lambda/\sqrt{c\delta}, \sqrt{c\delta}\Theta)$.
    The solution $v$ to the (limiting) Stieltjes transform of the spectrum of the sample Gram matrix therein satisfies $v=(a\tau)^{-1}$.
    For ridge ensembles, there is a deeper connection between $\cR_{1}$ and $\cR_{\infty}$.
    It turns out that $\cR_{\infty}(\lambda,  c)$ is exactly equal to $\cR_1(\mu, 1)$ for a new level of regularization $\mu = v^{-1} \ge \lambda$ that depends on $c$ (and properties of the data distribution).
    The lower the subsampling proportion $c$, the higher the value of this implicit regularization $\mu$.
    There are entire paths of equivalences in the $(\lambda, c)$ plane where not only are the asymptotic squared risks the same, but also the estimators themselves are equivalent.
    In a sense, one can think of the combined effect of subsampling and ensembling as an additional (implicit) ridge regularization.
    The prediction risk equivalences are first proved in \cite{du2023subsample} and later generalized to other risks and estimator equivalences in \cite{patil2023generalized}. 
\end{remark}

\begin{remark}
    [Lasso ensembles]
    \label{rem:ensembles-lasso}
    For the lasso predictor when $q=1$, the first set of equations \eqref{eq:amp-bridge-M=1} for $M=1$ in \Cref{sys:ensembles-penalized-least-squares} recovers \cite[Theorem 1.5]{bayati2011lasso} with a change of variables $\lambda'=\lambda/\sqrt{c\delta}$ and $\Theta'=\sqrt{c\delta}\Theta$.
    On the other hand, the second set of equations in \Cref{sys:ensembles-penalized-least-squares} for $M=\infty$ is new to the literature.
    We show empirically in the next section that the optimal full-ensemble subsampled lassoless is not the same as the optimal non-ensemble lasso (on full data).
    Thus, the subsampling and ensembling of the lasso are qualitatively different from the subsampling and ensembling of the ridge regression, as pointed out in \Cref{rem:ensembles-ridge}.
    In particular, the effect of subsampling for lasso is not merely additional lasso regularization.
\end{remark}

\subsubsection{Ensembles of general ridge regularized estimators}
\label{sec:general-train-loss}

In this section, we specialize the results of \Cref{th:nonlinear} for ridge regularized ensembles allowing for general loss functions.
Specifically, given $\lambda > 0$, consider component estimators of the form:
\begin{equation*}
    \hat \bbeta_m(I_m) 
    \coloneq \argmin_{\signal \in \RR^p} \sum_{i\in I_m} \loss(y_i - \bx_i^\top \signal) + \lambda \| \signal \|_2^2.
\end{equation*}
Recall that the risk convergence in \eqref{eq:ensemble-risk-limit-generic} holds with 
$\cR_1 = \alpha^2$ and $\cR_\infty = \etaG\alpha^2$ where $\alpha$ and $\etaG$ are solutions to \Cref{sys:general_ensemble-M=1} and \Cref{sys:general_ensemble-M=infty} respectively. 
Here, using the explicit formula $\prox_{\reg}(x;\tau) = x/(1+\lambda \tau)$ for $\reg(x) = \lambda |x|^2 / 2$, these systems can be simplified as follows:
\begin{system}
  [Nonlinear system for general ridge regularized ensembles]
  \label{sys:general_loss_ridge_ensembles}
  Given $\lambda \in (0, \infty)$, $\delta \in (0,\infty)$, and $c \in (0, 1]$, define the following 2-scalar system of equations in $(\alpha, \kappa)$:
  \begin{equation}
    \alpha^2 = 
    c \delta \cdot \kappa^2 \E[{\env}'_{\loss}(Z + \alpha G; \kappa)^2] + \lambda^2 \kappa^2 \EE[\Theta^2] 
    \text{ and }
    \alpha = 
    c \delta \cdot \tfrac{\kappa}{1-\lambda\kappa} \E[{\env}'_{\loss}(Z + \alpha G; \kappa) \cdot G], \label{eq:general-ridge-alpha-kappa}
    \end{equation}
  where $G \sim \cN(0,1)$, $\Theta \sim F_\theta$, $Z \sim F_\eps$, all mutually independent.
  Let $(\beta, \nu)$ be parameters expressed in terms of $(\alpha, \kappa)$ as:
  \begin{equation}
    \beta^2 = \tfrac{1}{\kappa^2} \alpha^2 - \lambda^2 \EE[\Theta^2] \quad \text{and} \quad 
    \nu = \tfrac{1}{\kappa} - \lambda. 
    \label{eq:general-ridge-beta-nu}
  \end{equation}
  Given parameters $(\alpha,\beta,\kappa,\nu)$ that satisfy \eqref{eq:general-ridge-alpha-kappa} and \eqref{eq:general-ridge-beta-nu}, define the following 2-scalar system of equations in variables $(\etaH, \etaG)$:
  \begin{equation}
    \etaH 
    = \tfrac{c^2\delta}{\beta^2} \E [
      \env_{\loss}'(Z + \alpha G; \kappa) \cdot \env_{\loss}'(Z + \alpha \tilde G; \kappa)]
    \quad
    \text{and}
    \quad
    \etaG = \tfrac{\etaH \beta^2 + \lambda^2 \EE[\Theta^2] }{\beta^2 + \lambda^2 \EE[\Theta^2]}, \label{eq:etaG-etaH-general-ridge-regularized}
    \end{equation}
  where $(G, \tilde{G})$ are jointly normal with $\E[G\tilde{G}] = \etaG$.
\end{system}

When $M = 1$, \Cref{sys:general_loss_ridge_ensembles} recovers \cite[Theorem 2.1]{karoui_2013}.
When $c\delta>1$ (underparameterized regime) and $\lambda=0$ (unregularized case),  we have $\nu = \kappa^{-1}$, $\beta=\kappa^{-1} \alpha$, $\etaH=\etaG$. {Substituting these into \eqref{eq:etaG-etaH-general-ridge-regularized}}, we recover Theorem 2.3 in the recent work of \cite{bellec2024asymptotics}.

\section{Subagging and overparameterization}
\label{sec:ensembles-interpolators}

In \Cref{sec:specific-examples}, we discussed subagging of regularized estimators with an explicit regularization level $\lambda > 0$. 
Triggered by the success of overparameterized neural networks that can (nearly) interpolate, there has been a surge of recent work analyzing the risk behavior of estimators with vanishing regularization, such as the minimum $\ell_2$-norm interpolator (ridgeless), minimum $\ell_1$-norm interpolator (lassoless), and max-margin interpolators, among others. 
In this section, we discuss subagging of minimum $\ell_q$-norm interpolators for $q \in \{1, 2\}$. 
We will demonstrate some interesting risk properties in \Cref{sec:subagging-interpolators}, showcase the benefits of subagging in overparameterized regimes in \Cref{sec:optimal-subsample-size}, and contrast with optimal explicit regularization in \Cref{sec:optimal-regularization-penalty}.

{
A central theme in the recent literature on interpolators is understanding the \emph{implicit bias} of optimization algorithms, especially in overparameterized regimes (e.g., \cite{gunasekar2018characterizing,gunasekar2018implicit,du2018gradient,chizat2018global,lee2019wide,allen2019convergence}).
For instance, gradient descent in an overparameterized linear model converges to the minimum-norm solution (see, e.g., \cite{belkin2018understand,hastie2022surprises}). 
Our analysis of subagging takes a complementary perspective. 
Subagging is not intended to model the trajectory of a single optimization run; rather, it is treated as an explicit ensembling procedure in its own right, designed to improve stability and predictive accuracy by averaging models trained on different subsamples of the data. 

Since minimum-$\ell_q$ interpolators serve as canonical models for modern predictors, it is natural and important to ask how their statistical properties change under bagging.
Our results in \Cref{sec:subagging-interpolators}-\ref{sec:optimal-regularization-penalty} show that the combination of interpolation with subsampling and averaging creates a distinctive form of \emph{algorithmic regularization}, often yielding better generalization compared to either a single interpolator or classical regularized estimators such as ridge or lasso.
This situates our contribution as an analysis of explicit ensembling method applied to modern base learners, complementing prior work on bagging least squares and ridge predictors in the overparameterized regimes (e.g., \cite{lejeune2020implicit,patil2022bagging,patil2023generalized}).
}

\subsection[Subagging of minimum ellq-norm interpolators]{Subagging of minimum $\ell_q$-norm interpolators}
\label{sec:subagging-interpolators}

We will focus in this section on subagging of bridgeless estimators, that is $\ell_q$-norm regularized least squares with $\reg(\bb) = \| \bb \|_q^q$.
{Our main cases of interest are the ``ridgeless'' estimator, which corresponds to $q = 2$ \cite{hastie2022surprises}, and the ``lassoless'' estimator, which corresponds to $q = 1$ \cite{li2021minimum}.}
The terminology ``less'' is motivated by the fact that these estimators can be defined in a limiting sense as $\lambda \to 0^{+}$ for bridge estimators with regularization level $\lambda$.
We consider the ensemble of predictors $(\hat\btheta_m)_{m\in [M]}$ of the form:
\[
    \textstyle
    \hat\btheta_{m} \coloneq \lim_{\lambda\to 0+} \hat\btheta_m(\lambda) \quad \text{where} \quad \hat\btheta_m(\lambda) \in \argmin_{\signal \in \R^p} \frac{1}{2} \| \by_{\setm} - \bX_{\setm} \signal \|_2^2 + \lambda \|\bb\|_q^q. 
\]
(We refer readers to \cite{tibshirani2013lasso} for details on how the sequence of estimators is defined when the estimators for $\lambda > 0$ are not unique, as in the case of the lasso.)

In the underparameterized regime ($p < n$), these are simply the least squares estimators: $\hat{\bbeta}_m = (\bX_{\setm}^\top \bX_{\setm})^{-1} \bX_{\setm}^\top \by_{\setm}$.
In the overparameterized regime ($p > n$), these correspond to the minimum $\ell_q$-norm interpolators: $\argmin_{{\bbeta} \in\R^p}\{ \| {\bbeta} \|_q \colon \by_{\setm} = \bX_{\setm} {\bbeta} \}$, when $\bX_{\setm}$ has independent rows to allow for interpolation.
For $q = 2$, when $\reg$ is the ridge penalty, this also has a closed-form expression given by: $\hat{\bbeta}_m = (\bX_{\setm}^\top \bX_{\setm})^{\dagger} \bX_{\setm}^\top \by_{\setm}$, where $\bA^{\dagger}$ denotes the Moore-Penrose pseudoinverse of a matrix $\bA$.
In other cases, we do not have a closed-form expression for the minimum $\ell_q$-norm interpolator.
The next system specializes \Cref{sys:ensembles-penalized-least-squares} to convex regularized least squares with vanishing regularization, by taking the limit as $\lambda \to 0^{+}$.

\begin{system}
    [Ensembles of minimum $\ell_q$-norm interpolators]
    \label{sys:interpolators}
    Given $\delta \in (0, \infty)$, $c \in (0, 1]$ such that $c\delta <1$, $\sigma^2 \in [0, \infty)$, define the following system of equations in variables $(\tau,a)$:
    \begin{subequations}
    \label{eq:ellq-interpolators}
    \begin{empheq}{align}
    \tau^2 
    &= \mathbb{E}\big[\big(\prox_{|\cdot|^q}\big(\Theta + \tfrac{\tau}{\sqrt{c\delta}} H; \tfrac{a\tau}{\sqrt{c\delta}}\big) - \Theta\big)^2\big] + \sigma^2 \label{eq:interpolators-tau}\\
        0 &= 1-\tfrac{1}{c\delta}\E\big[
          \prox_{|\cdot|^q}' \big(\Theta + \tfrac{\tau}{\sqrt{c\delta}} H ; \tfrac{a\tau}{\sqrt{c\delta}}\big) \big]
        \label{eq:interpolators-a}
    \end{empheq}
    \end{subequations}
    where $H \sim \cN(0,1)$ and $\Theta \sim F_\theta$ are independent.
    Given $(\tau,a)$ that satisfy \eqref{eq:interpolators-tau} and \eqref{eq:interpolators-a}, define the following 1-scalar system of equations in variable $\xi$:
    \small
    \begin{subequations}
    \begin{empheq}{align}
        \hspace{-1em}
        \xi^2 
        = \mathbb{E}\big[\big(\prox_{|\cdot|^q}\big(\Theta + \tfrac{\tau}{\sqrt{c\delta}} H; \tfrac{a\tau}{\sqrt{c\delta}}\big) - \Theta\big) \cdot \big(\prox_{|\cdot|^q}\big(\Theta + \tfrac{\tau}{\sqrt{c\delta}} \tilde{H}; \tfrac{a\tau}{\sqrt{c\delta}}\big) - \Theta\big)\big] + \sigma^2
          \label{eq:full-ensemble-risk-ellq-interpolators}
    \end{empheq}
    \end{subequations}
    where
    $(H,\tilde H)$ {are as in \eqref{distrib}}
    with
    $\etaH = c \frac{\xi^2}{\tau^2}$,
    and $\Theta \sim F_\theta$ is independent of $(H, \tilde H)$. 
\end{system}

To the best of our knowledge, the existence and uniqueness of the solution $(\tau, a)$ to \eqref{eq:ellq-interpolators} are not fully established in the literature, except for the special cases of $q = 2$ (ridgeless) and $q = 1$ (lassoless). 
Assuming this is the case, the existence and uniqueness of the solution $\xi$ to \eqref{eq:full-ensemble-risk-ellq-interpolators} follow from \Cref{th:existence-uniqueness-sys:general_ensemble-M=infty}.

Observe that the equations \eqref{eq:interpolators-tau} and \eqref{eq:full-ensemble-risk-ellq-interpolators} in \Cref{sys:interpolators} are special cases of \eqref{eq:amp-bridge-tau} and \eqref{eq:amp-bridge-M=infty-a} in \Cref{sys:ensembles-penalized-least-squares} for $\ell_q$ penalties.
Equation \eqref{eq:interpolators-a} is the limit of \eqref{eq:amp-bridge-a} as $\lambda \to 0^+$.
Indeed, for $q \in \{1,2\}$, the solution to \Cref{sys:ensembles-penalized-least-squares} with $\lambda > 0$ converges to the solution to \Cref{sys:interpolators} as $\lambda \to 0^+$ (see \Cref{subsec:derivation_system_interpolator} for a proof).
This means that
\begin{align*}
    \lim_{\lambda \to 0^{+}} {\underset{n \to +\infty}{\mathrm{p\!\!-\!\!lim}}} \Bigl\|\frac{1}{M} \sum_{m \in [M]} \hat\btheta_m(\lambda) - \btheta \Bigr\|_2^2 = \cR_M := M^{-1} \cR_1 + (1-M^{-1}) \cR_\infty
\end{align*}
where $\cR_1 = \tau^2 - \sigma^2$ and $\cR_\infty = \xi^2 - \sigma^2$, and by the same argument, the limiting risk $\cR_M$ satisfies \Cref{prop:monotonicity-ensemble-size} and \ref{prop:monotonicity-risk} (see also \Cref{fig:effect-of-M} and \ref{fig:optimum_phis_overparameterized-2-optrisk}).
However, it is challenging to show the above display with the order of $\lim_\lambda$ and $\mathrm{p\!\!-\!\!lim}_n$ swapped.
For the ridgeless estimator ($q = 2$) and any $M$, this is proved in \cite{patil2022bagging} using a uniform convergence argument.
For the lassoless estimator ($q = 1$ and $M = 1$), this is done in \cite{li2021minimum}, where the authors directly analyze the interpolator by constructing a suitable AMP algorithm.
We conjecture that this is, in general, true at least for bridgeless estimators for any $q \in [1,2]$ and $M \ge 1$. 
Since this is not the main focus of our paper and is only intended as an illustrative case, we will not work towards this goal in the current paper.
We will instead investigate properties and consequences of \Cref{sys:interpolators}.

Further special cases of $q = 2$ (ridgeless) and $q = 1$ (lassoless) are isolated in the next two remarks.
These will serve as our two main running examples in this section.
\begin{remark}[Ridgeless ensembles]
    \label{rem:ridgeless_ensembles}
    For squared loss and ridge regularizer ($q=2$), when $c\delta < 1$ and $\lambda \rightarrow 0^+$, we get
    \begin{align*}
    \cR_1(\delta, c) + \sigma^2 = \EE[\Theta^2](1 - c\delta) + \frac{{\sigma^2}}{1 - c\delta}, ~~
    \cR_{\infty}(\delta, c) + \sigma^2 = \EE[\Theta^2] \frac{(1 - c\delta)^2}{\delta (\delta - (c\delta)^2)} + \frac{{\sigma^2} \delta}{\delta - (c\delta)^2}.
    \end{align*}
    The result above aligns with the risk ensemble of ridgeless estimators presented in Corollary 6.1 of \cite{patil2022bagging}.
    This can be seen by substituting $\delta = 1 / \phi$ and $c = \phi / \psi$, or equivalently $c \delta = 1 / \psi$.
\end{remark}

\begin{remark}
    [Lassoless ensembles]
    \label{rem:lassoless_ensembles}
    For squared loss and lasso regularizer ($q=1$), when $c\delta < 1$ and $\lambda \rightarrow 0^+$, $\tau^2=\cR_1(\delta, c)+\sigma^2$ is the solution to the following equations:
    \begin{equation}
    \begin{aligned}
        \tau^2 = \sigma^2 + \mathbb{E}\big[\big(\soft\big(\tfrac{\tau}{\sqrt{c\delta}} H + \Theta ; \tfrac{a\tau}{\sqrt{c\delta}}\big) - \Theta\big)^2\big], 
        \quad 1 = \tfrac{1}{c\delta}\PP\big(|\tfrac{\tau }{\sqrt{c\delta}} H + \Theta| > \tfrac{a\tau}{\sqrt{c\delta}} \big),
    \end{aligned}
    \label{eq:lassoless_system_remark7}
    \end{equation}
    and $\xi^2 = \cR_\infty(\delta, c) + \sigma^2$ is the solution to the following equations:
    \begin{align*}
        \xi^2 &= \sigma^2 + \mathbb{E}\big[\big(\soft\big(\tfrac{\tau}{\sqrt{c\delta}} H + \Theta ; \tfrac{a\tau}{\sqrt{c\delta}} \big) - \Theta\big) \cdot \big(\soft\big(\tfrac{\tau}{\sqrt{c\delta}} \tilde H + \Theta ; \tfrac{a\tau}{\sqrt{c\delta}} \big) - \Theta\big)\big],
    \end{align*}
    with $\EE[H\tilde{H}] = c \frac{\xi^2}{\tau^2}$.
    Note that $\soft$ is the soft threshold function defined by $\soft(x; \tau) = (|x| - \tau)_+ \sign(x)$.
    In the ensemble setting, the case $M = 1$ corresponds to \cite[Theorem 2]{li2021minimum} with a slight change of variables $\Theta' = \sqrt{c\delta} \Theta$.
    The full-ensemble case when $M = \infty$ is new.
\end{remark}

\begin{figure}[!t]n
    \centering
    \includegraphics[width=0.99\textwidth]{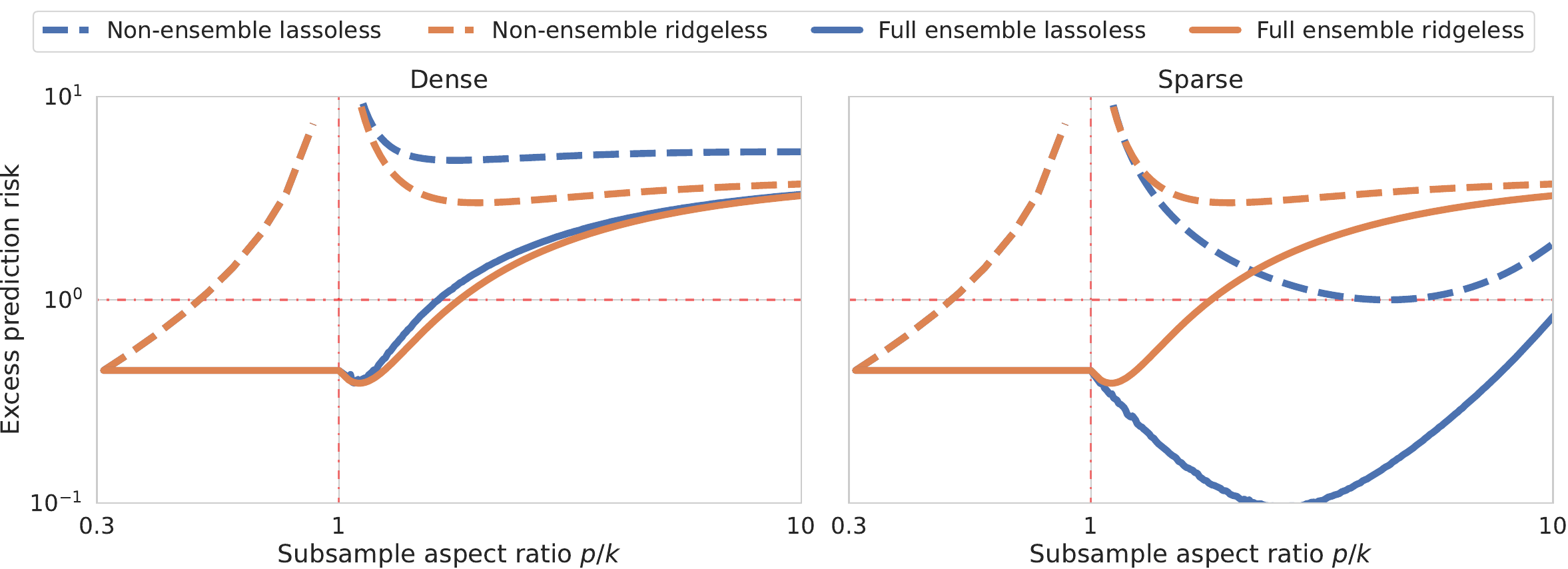}
    \caption{
        Prediction risks of the full-ensemble lassoless and ridgeless at different subsample aspect ratios $p/k={(c\delta)^{-1}}$, where $\delta^{-1} = p/n$ is fixed to $0.3$ while $c=k/n\in (0,1]$ varies. 
        The data model is given by \eqref{model:sparse} with signal strength $\rho=2$, noise level $\sigma=1$. 
        The support proportion $s$ varies.
        \emph{Left}: dense regime with $s=0.9$.
        \emph{Right}: sparse regime with $s=0.01$.
        We observe that the full-ensemble risk is continuous at the interpolation threshold $(p/k=1)$.
        Also, the full-ensemble risk has a negative derivative to the right of the interpolation threshold.
        This implies that the optimal subsample size is in the overparameterized regime.
    }
    \label{fig:optimum_phis_overparameterized-continuity-derivative}
\end{figure}

\begin{remark}
    [Avoiding risk divergence with ensembling]
    \label{rem:lassoless-fullensemble-continuity-psi=1}
    Note that in the underparameterized regime when $c \delta > 1$ is fixed, the estimator of interest is simply the ensemble of least squares.
    Thus, a standard Stieltjes transform argument or the explicit formula for the expectation of inverse Wishart matrices gives
    \[
        \sR_1(\delta,c) + \sigma^2 = \sigma^2 \frac{c \delta}{c \delta - 1}
        \quad
        \text{and}
        \quad
        \sR_\infty(\delta,c) + \sigma^2 = \sigma^2 \frac{\delta}{\delta - 1} \quad \text{for all $c > \delta^{-1}$}.
    \]
    It just so happens that for the full-ensemble least squares estimators, only the inverse aspect ratio $\delta$ of the original data matters!
    In particular, as $c \to (\delta^{-1})^{+}$, $\sR_1$  diverges, while the full-ensemble risk $\sR_\infty$ is still bounded. 
    Now let us consider the overparameterized regime $c\delta < 1$.
    By simple algebra, for any regularizer $\reg$, the solution $\tau$ to the sub-system \eqref{eq:interpolators-tau}-\eqref{eq:interpolators-a} in \Cref{sys:interpolators} is uniformly bounded from below as:
\begin{equation}\label{eq:tau_lower_bound}
        \tau^2 \ge (1-c\delta)^{-1} \sigma^2.
    \end{equation}
    (See \Cref{sec:proof_tau_lowerbound} for the proof.)
    Recalling $\cR_1 = \tau^2 - \sigma^2$, this means that the risk of the non-ensemble interpolators blows up as $c \to (\delta^{-1})^{-}$.
    For the minimum $\ell_2$- and $\ell_1$-norm interpolators, this is shown in \cite{hastie2022surprises,li2021minimum}.
    For the full-ensemble cases, we experimentally observe from \Cref{fig:optimum_phis_overparameterized-continuity-derivative} that the risk $\cR_\infty$ is continuous in $c$ for the full ridgeless and lassoless ensembles.
    In particular, it does not blow up around $c=\delta^{-1}$.
    For ridgeless, this claim is easy to verify (and holds more generally, as shown in \cite{patil2022bagging}).
    For lassoless, given the solution $(a,\tau)$ to \eqref{eq:lassoless_system_remark7}, in \Cref{sec:proof-continuity-Rinfty}, we identify that the condition $\lim_{c\to(\delta^{-1})^-} (a\tau) = 0$ is sufficient to obtain the conclusion $\lim_{c\to(\delta^{-1})^-}(\xi^2) =\frac{\delta}{\delta-1}\sigma^2$, and we observe experimentally that $a\tau\to0$ holds (\Cref{fig:atau-limit-to-zero}), however we are currently not able to provably establish that $\lim_{c\to(\delta^{-1})^-} (a\tau) = 0$.
\end{remark}

\subsection{Optimal subsample size}
\label{sec:optimal-subsample-size}

An intriguing observation from \Cref{fig:optimum_phis_overparameterized-continuity-derivative} concerns the optimal subsample size $k^*$.
When the sample size $n$ and the number of features $p$ are fixed with $n > p$, the optimal subsample size $k^*$ that minimizes the full risk $\cR_\infty$ falls below $p$.
This suggests that even when the original sample lies in the underparameterized regime, the optimal subsample size shifts into the overparameterized regime.
This phenomenon is proved for ridgeless ensembles ($\reg(x) = x^2$) by \cite{patil2022bagging}.
Expanding on this, and utilizing \Cref{sys:interpolators}, we empirically show that this behavior extends beyond ridgeless ensembles to lassoless ensembles ($\reg(x) = |x|$) as well (see \Cref{fig:optimum_phis_overparameterized-2-optsubsample}).

\begin{figure}[!t]
    \centering
    \includegraphics[width=0.99\textwidth]{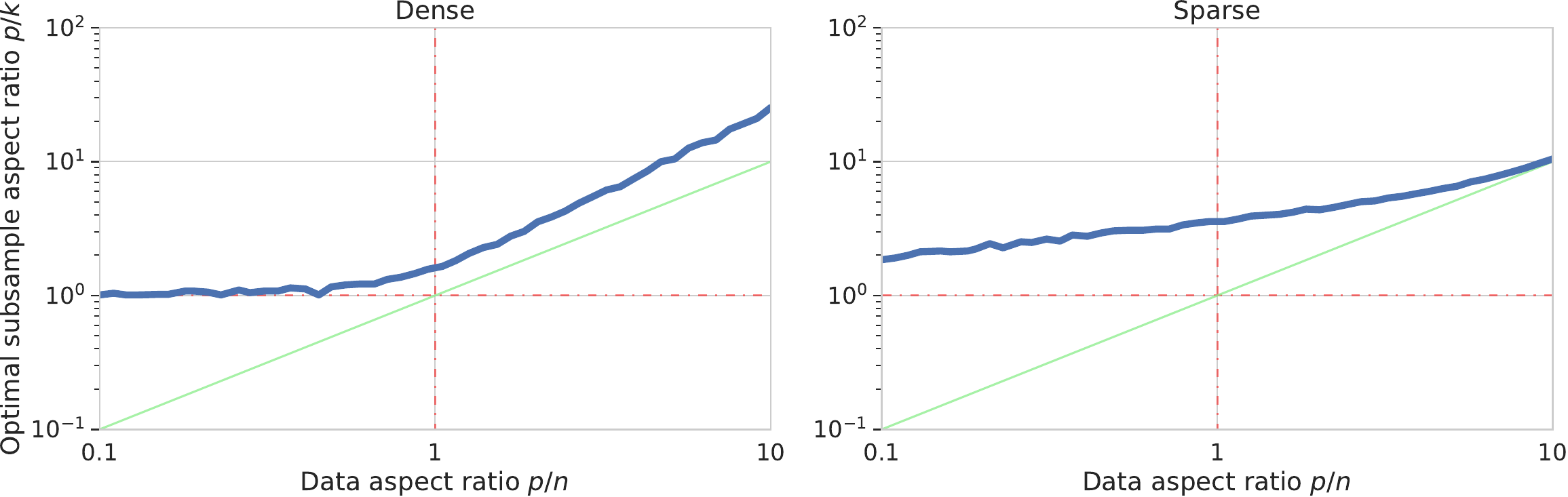}
    \caption{
        \textbf{Optimal subsample size for the lassoless ensemble is always in the overparameterized regime.}
        Optimal subsample aspect ratio $p/k$ that achieves the optimal risk for the lassoless ensemble at different data aspect ratios $p/n$ ranging from $0.1$ to $10$.
        The data model is as in \eqref{model:sparse} with signal strength $\rho=2$, noise level $\sigma=1$, data aspect ratio $p/n=0.1$, feature size $p=500$, and varying support proportion $s$.
        \emph{Left}: dense regime with $s=0.9$.
        \emph{Right}: sparse regime with $s=0.01$.
    }
    \label{fig:optimum_phis_overparameterized-2-optsubsample}
\end{figure}

{
The phenomenon of the optimal subsample size $k^*$  being in the overparameterized regime ($k^*<p$) arises from a combination of multiple effects.
First, subagging leverages correlation reduction.
As the subsample $k$ decreases, the overlap between any two subsamples shrinks, making the resulting estimators more diverse and less correlated.
This reduces the covariance term in the ensemble risk formula, incentivizing smaller values of $k$.
Second, the choice of $k$ also interacts with the ``double descent'' risk profile of the individual interpolators.
The risk of a single interpolator typically diverges at the threshold $k = p$.
By choosing $k < p$, we steer the base learners away from this high-risk region.
The optimal subsample size $k^*$ is therefore a ``sweet spot'' that balances these benefits against the eventual cost of training on smaller datasets, a balance that is surprisingly found in the overparameterized regime!
}

\subsection{Optimal subagging versus optimal (explicit) regularization}
\label{sec:optimal-regularization-penalty}

There are three parameters one can tune to optimize the asymptotic risk $\cR_M(\lambda, c)$ of the ensemble estimator, as in \Cref{rem:ensembles-ridge,rem:ensembles-lasso}: the regularization level $\lambda\in [0, \infty)$, the subsample ratio $c\in(0,1]$, and the ensemble size $M\in\mathbb{N}$. {Here, $\delta = \lim_{n,p\to\infty} (n/p) $ is not a tuning parameter, as it is determined by the dimensions of the full sample matrix $\bX\in\R^{n\times p}$, which has $n$ samples and $p$ features.
Furthermore, since $\mathcal{R}_M(\lambda, c)$ does not depend on $M$ when $c=1$, combined with \Cref{prop:monotonicity-ensemble-size} (which shows that the risk is monotonic in $M$ when $c<1$), we have $\inf_{M\in \mathbb{N}}\mathcal{R}_M(\lambda, c) = \mathcal{R}_\infty(\lambda, c)$ for all $(\lambda, c)$. Thus, it suffices to consider the following minimization: $\inf_{\lambda\in [0, \infty), \ c\in (0,1]}\mathcal{R}_\infty(\lambda, c)$. 
}

This hyperparameter optimization is simplified for ridge regression, as shown in Theorem 2.3 of \cite{du2023subsample}.
Minimization with respect to $(\lambda, c)$ is equal to the minimization over $\lambda$ when $M = 1$ and $c = 1$ (non-ensemble setting), which is the same as minimization $c$ when $\lambda = 0$ and $M=\infty$ (full ensemble of ridgeless predictors):
\begin{equation*}
    \underbrace{{\inf_{\lambda\in [0, \infty), \ c\in (0,1]}
    \cR_\infty(\lambda, c)}}_{{ \text{opt regularization and opt ensemble ($\color{pyred}\star$)}}}
    \qquad =  
    \underbrace{{\inf_{\lambda\in[0, \infty)} \cR_1(\lambda,c=1)}}_{{\text{opt regularization but no ensemble ($\color{pyblue}\mysolidcircle$)}}}
    =
    \underbrace{{\inf_{c\in (0,1]} \cR_\infty(\lambda=0, c)}}_{{\text{opt ensemble but no regularization ($\color{pygreen}\mysolidcircle$)}}}.
\end{equation*}
{(Here, the colored markers are as in \Cref{fig:advantage_subsampling_underparameterized,fig:advantage_subsampling_underparameterized_aniso,fig:advantage_subsampling_overparameterized,fig:huber,fig:huber-l1-rho,fig:huber-l1-lam}.)
In some situations, however, such an equivalence does not hold:} 
\begin{equation*}
    \underbrace{{\inf_{\lambda\in [0, \infty), \ c\in (0,1]}
    \cR_\infty(\lambda, c)}}_{{ \text{opt regularization and opt ensemble ($\color{pyred}\star$)}}}
    \qquad \ne  
    \underbrace{{\inf_{\lambda\in[0, \infty)} \cR_1(\lambda,c=1)}}_{{\text{opt regularization but no ensemble ($\color{pyblue}\mysolidcircle$)}}}
    \ne
    \underbrace{{\inf_{c\in (0,1]} \cR_\infty(\lambda=0, c)}}_{{\text{opt ensemble but no regularization ($\color{pygreen}\mysolidcircle$)}}}.
\end{equation*}
{
and the risk minimization over {$(\lambda, c)$} can be  strictly better}:
\begin{equation*}
    \underbrace{{\inf_{\lambda\in [0, \infty), \ c\in (0,1]}
    \cR_\infty(\lambda, c)}}_{{ \text{opt regularization and opt ensemble ($\color{pyred}\star$)}}}
    \qquad < 
    \underbrace{{\inf_{\lambda\in[0, \infty)} \cR_1(\lambda,c=1)}}_{{\text{opt regularization but no ensemble ($\color{pyblue}\mysolidcircle$)}}}
    \wedge
    \underbrace{{\inf_{c\in (0,1]} \cR_\infty(\lambda=0, c)}}_{{\text{opt ensemble but no regularization ($\color{pygreen}\mysolidcircle$)}}}.
\end{equation*}
\begin{figure}[!t]
    \centering
\includegraphics[width=0.99\textwidth]{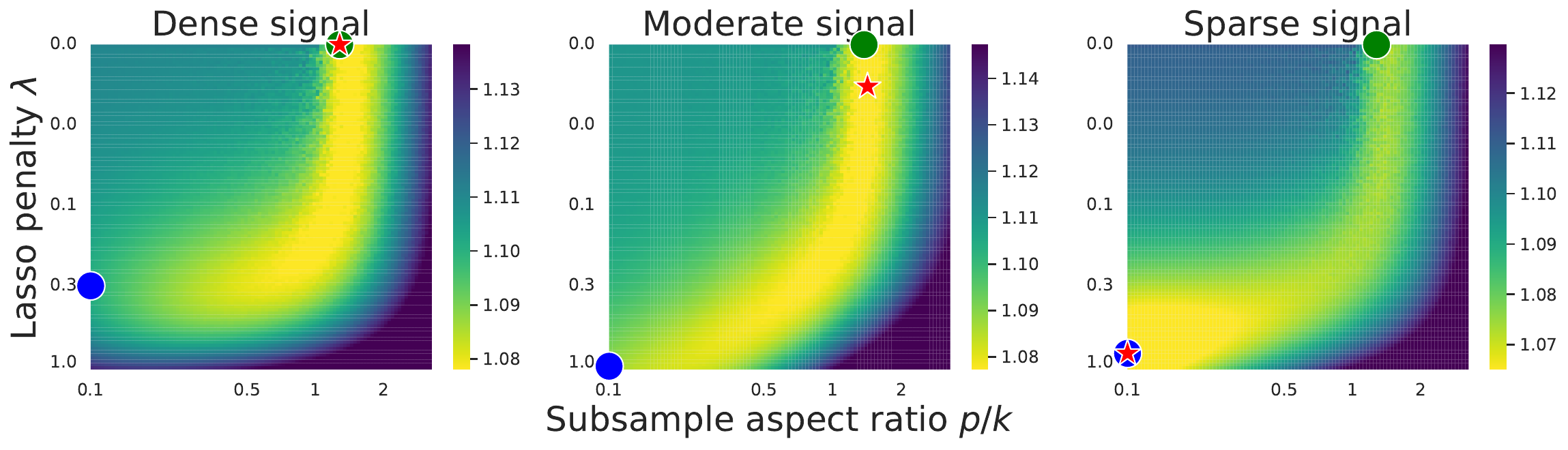}
    \caption{
        \textbf{Optimally subsampled lassoless regression can outperform optimal lasso regression.}
        Heatmaps of theoretical prediction risk in $\lambda$ and {$p/k=(p/n)\cdot c^{-1}$ with $c = k/n\in(0,1]$} of full lasso ensemble $(M=\infty)$ in the underparameterized regime ($p/n=0.1$).
        The data model is given by \eqref{model:sparse} with signal strength $\rho=0.5$ and noise level $\sigma=1$ at different sparsity levels $s$.
        {
        Blue dots ($\color{pyblue}\mysolidcircle$): the optimal lasso {($\lambda$ is optimized with $c$ fixed to $1$)}. 
        Green dots ($\color{pygreen}\mysolidcircle$): the optimally subsampled lassoless {($c$ is optimized with $\lambda$ fixed to $0$)}. 
        Red stars ($\color{pyred}\star$):  the optimal lasso ensemble {($c$ and $\lambda$ are jointly optimized)}.
        \emph{Left}: Dense regime with support proportion $s=0.9$. The optimally subsampled lassoless ($\color{pygreen}\mysolidcircle$) coincides with the optimal lasso ensemble ($\color{pyred}\star$) and outperforms the optimal lasso ($\color{pyblue}\mysolidcircle$).
        \emph{Middle}: Moderate sparse regime with support proportion $s=0.5$. The optimal lasso ensemble ($\color{pyred}\star$) is better than both the optimal lasso ($\color{pyblue}\mysolidcircle$) and the optimally subsampled lassoless ($\color{pygreen}\mysolidcircle$). 
        \emph{Right}: Sparse regime with support proportion $s=0.2$. The optimal lasso ($\color{pyblue}\mysolidcircle$) coincides with the optimal lasso ensemble ($\color{pyred}\star$) and outperforms the optimal subsample lassoless ($\color{pygreen}\mysolidcircle$).}
     }
    \label{fig:advantage_subsampling_underparameterized}
\end{figure}
We illustrate this through a numerical experiment with lasso.
We show that the optimal full-ensemble subsampled lassoless is not the same as the optimal non-ensemble lasso (on full data).
In \Cref{fig:advantage_subsampling_underparameterized}, we contrast the sparse and dense data settings.
In each case, we show the full-ensemble risk heatmap in $\lambda$ and {$p/k =(p/n) \cdot c^{-1}$ with $p/n$ held fixed while $c\in (0,1]$ varies.}
In the right panel (the sparse setting), we see that the optimal lasso is better than the optimal subsample lassoless. This is expected because the lasso is known to perform well in the sparse setting. In the left panel (the dense setting), we see that the optimally subsampled lassoless is better than the optimal lasso. This is interesting because it shows that the subsample and ensemble induce an implicit regularization effect.
In short, the optimally subsampled lassoless can be better or worse than the optimal lasso.
In other words, the full-ensemble lassoless when $c$ is optimized is not the same as the optimized lasso when $\lambda$ is optimized on the full data.

A similar conclusion holds for overparameterized settings, as shown in \Cref{fig:advantage_subsampling_overparameterized}.
In particular, depending on the SNR and $\delta$, the subsample optimized risk may be smaller or larger than the optimized lasso risk.
The conclusion is that, in general, it helps to optimize $\lambda$, but also to optimize the subsample size.
The subsample and ensemble induce an implicit regularization effect.
Once optimized, this implicit regularization can improve on the explicit regularization provided by the regularizer. {For further investigation on the interplay between subsampling and explicit regularization, see \Cref{sec:optimal-subagging-versus-optimal-regularization}.}

\begin{figure}[!t]
    \centering
    \includegraphics[width=0.99\textwidth]{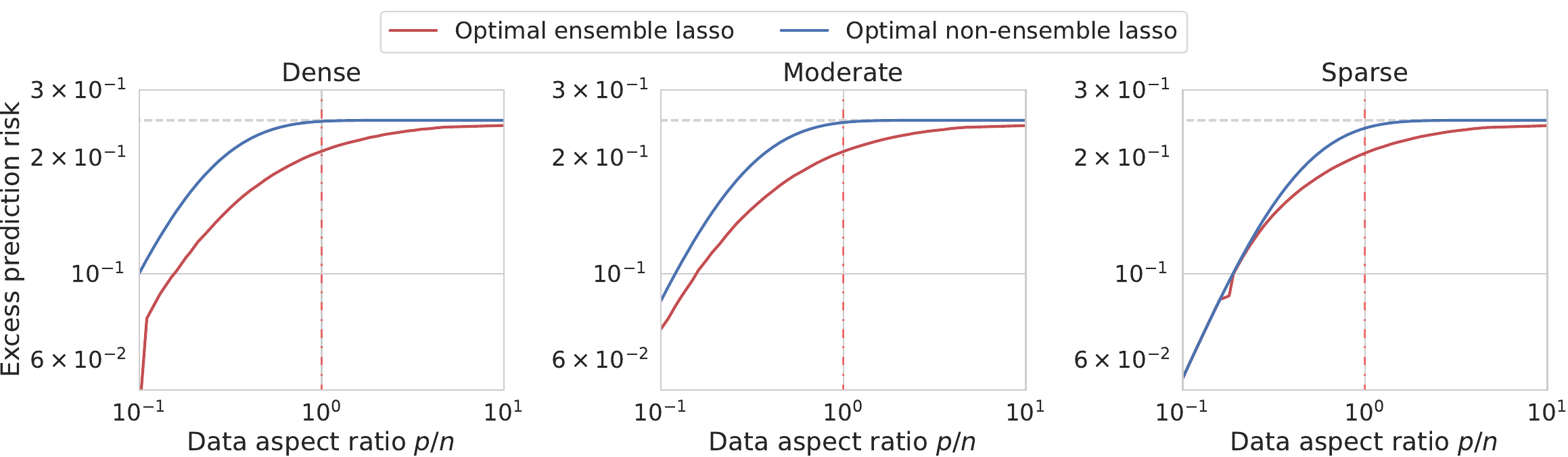}
    \caption{
        \textbf{Optimally subsampled ensemble lasso can uniformly beat optimally tuned non-ensemble lasso across different data aspect ratios.}
        The theoretical prediction risk of optimal ensemble and non-ensemble lasso at different data aspect ratios $p/n$ is shown.
        The data model is given by \eqref{model:sparse} with signal strength $\rho=0.5$ and noise level $\sigma=1$ at different sparsity levels $s$.
        \emph{Left}: Dense regime with support proportion $s=0.9$.
        \emph{Middle}: Moderate sparse regime with support proportion $s=0.5$.
        \emph{Right}: Sparse regime with support proportion $s=0.2$.
    }
    \label{fig:fig_lasso_risk_opt}
\end{figure}

Finally, \Cref{fig:fig_lasso_risk_opt} shows that the joint optimization of subsampling along with lasso penalty consistently outperforms the optimization of lasso penalty on the full data (without any subsampling) uniformly across varying data aspect ratios.
Particularly, in the dense regime ($s=0.9$) and moderate sparsity ($s=0.5$) regimes, the ensemble approach leverages implicit regularization to achieve lower prediction risk. 
In highly sparse scenarios ($s=0.2$), this effect reduces as one would expect.
Overall, we see complementary benefits of optimizing both subsampling and explicit regularization parameters to improve the predictive performance, with joint optimization being the overall clear winner.

\section{Extensions and open directions}
\label{sec:discussion}
In this paper, we provide a general risk characterization for the ensemble of regularized M-estimators.
The characterization depends on the inverse data aspect ratio $\delta = \lim n/p$, subsample ratio $c$, and loss and regularizer pair $(\loss, \reg)$.
We also specialize the results for specific cases of interest, such as the lasso and ridge regression, and analyze various properties related to optimal subsampling and ensembling.
Our goal in performing such analysis is to shed light on how subsampling and ensembling influence the risk of the ensemble estimator and the optimal choices of ensemble size and subsample size.
The key takeaway is that subsampling and ensembling can be beneficial in terms of reducing the risk of the estimator, and when the subsample ratio $c$ is optimized, the resulting risk is monotonic in $\delta$ for any ensemble size $M$.

There are several ``axes'' along which our results on the asymptotics of subagging can be extended.
These are all apparent by inspecting the last row of \Cref{tab:risk-landscape} (our current results) and contrasting it against the rows above (the known results in specific cases).
We briefly mention some of these open directions next to make them explicit.

\begin{itemize}
    \item
    \emph{Assumptions on the $\loss$ and $\reg$ functions}:
    Our current analysis can handle non-differentiable and non-strongly convex regularizers, but we require a differentiable loss. 
    Extending the analysis to non-differentiable losses through techniques like Moreau smoothing \cite[Section B.7]{celentano2020lasso} is a promising future direction. 
    In addition to relaxing conditions assumed for risk characterization, we are also interested in relaxing the assumptions for the risk estimation (\Cref{thm:risk-estimator-general-subagging}), particularly the strong convexity assumption on $\reg$. 
    A promising approach in this direction is to apply the Gaussian smoothing technique recently proposed by \cite{bellec2023error}.
    {
    Additionally, this paper's main results hold for separable regularizers, but as we show below in \Cref{sec:anisotropic_deterministic}, the framework can be extended to handle the non-separable case.
    }
    \item
    \emph{Assumptions on the response}:
    Another potential extension is relaxing the assumption of linear response models.
    Partial progress in this direction includes the recent work by \cite{bellec2024asymptotics} (for logistic models), \cite{clarte2024analysis} (for generalized linear models), and \cite{patil2023generalized} (for arbitrary response with bounded moments).
    \item
    \emph{Assumptions on the features}:
    While we assume isotropic Gaussian features, it is also of interest to extend the results to general feature distributions and anisotropic covariances.
    Note that this has been studied in special cases of ridge regression (\cite{patil2022bagging}, any $M$), ridgeless regression (\cite{han2023distribution}, $M=1$), lasso regression (\cite{celentano2020lasso}, $M = 1$), among others; see also, e.g., \cite{pesce2023gaussian,han2023universality}, for some progress towards establishing Gaussian universality (for $M = 1$).
    {
    We provide some numerical evidence for universality beyond Gaussian distributions in \Cref{fig:effect-of-M-cov}.
    Furthermore, we provide a concrete conjecture for the anisotropic covariance case in \Cref{sec:anisotropic_deterministic} that also handles deterministic signals, along with a heuristic derivation and numerical evidence for its validity.
    }
    \item
    \emph{Assumptions on the sampling strategies}:
    Broadening the scope of resampling strategies beyond subsampling without replacement to include more general schemes like sampling with replacement or according to a specific distribution, as recently studied in \cite{clarte2024analysis,du2024implicit}, is another promising future direction.
\end{itemize}

{
To illustrate one of these future directions, we extend below our analysis to the setting of anisotropic designs, deterministic signals, and (potentially) non-separable regularizers.

\subsection{Extension to anisotropic design and deterministic signal}\label{sec:anisotropic_deterministic}
We assume that the data $(y_i, \bx_i)_{i\in [n]}$ are generated according to the following linear model:
$$
\forall i\in[n], \quad 
y_i = \bx_i^\top \btheta + z_i, \quad (\bx_i)_{i\in [n]} \overset{\text{i.i.d.}}{\sim} \cN(\bm{0}_p, p^{-1}\bm{\Sigma}), \quad (z_i)_{i\in[n]} \overset{\text{i.i.d.}}{\sim} F_z, \quad \bx_i\indep z_i,
$$
where $\btheta\in \R^p$ is a fixed deterministic signal vector and $\bSigma \in \R^{p \times p}$ is a symmetric positive definite covariance matrix.
We estimate the signal vector $\btheta$ using regularized M-estimators $\hat\btheta$.
Following the subsampling setup in \Cref{sec:setup} and generalizing \eqref{eq:def-hbeta}, we now also allow for a general, potentially non-separable, proper, lower semicontinuous, convex regularizers $\vreg_m: \RR^p \to \RR \cup \{+ \infty\}$ for $m\in[M]$. 
For each subsampled dataset $(\bX_{\setm}, \by_{\setm})$ with $m \in [M]$, we define the regularized M-estimator as:
\begin{equation}
\label{eq:def-hbeta-nonseparable}
\hat \bbeta_m(I_m) \in \argmin_{\signal \in \R^p} \sum_{i\in I_m} \loss_m (y_i - \bx_i^\top \signal) + \vreg_m(\signal).
\end{equation}
To handle this setting, we first introduce a generalized Moreau envelope for vectors.
For any $\bv\in \R^p$ and symmetric positive definite matrix $\bm{\Lambda}\succ \bm{0}_{p\times p}$, we define:
\begin{align}
\label{eq:def_moreau_aniso}
\env_{\vreg}(\bv; \bm\Lambda) := \min_{\bx\in \R^p} \tfrac{1}{2} (\bv-\bx)^\top\bm\Lambda^{-1} (\bv-\bx) + \vreg(\bx).
\end{align}
The unique minimizer is denoted by $\bm\prox_\vreg(\bv; \bm\Lambda)\in\R^p$. Note that when $\bm\Lambda$ is a diagonal matrix and $\vreg$ is separable such that $\vreg(\bx) = \sum_{j\in[p]} \reg^{(j)}(x_j)$, the minimization problem \eqref{eq:def_moreau_aniso} becomes separable and the vector-valued proximal operator $\bm\prox_{\vreg}(\bv; \bm\Lambda)$ reduces to the component-wise scalar operator $(\prox_{\reg^{(j)}}(v_j; \Lambda_{jj}))_{j\in[p]}$, where $\prox_{\reg^{(j)}}(v_j; \Lambda_{jj})$ is the scalar proximal operator that we have been using in the previous sections.
For any $\bv\in \R^p$ and $\bm{\Lambda}\succ \bm{0}_{p\times p}$, we denote the gradient of $\env_{\vreg}(\bv; \bm\Lambda)$ with respect to $\bv$ by $\nabla \env_{\vreg}(\bv; \bm\Lambda)\in \R^p$. By Danskin's theorem \cite[Proposition B.22]{bertsekas2016nonlinear}, the gradient $\nabla \env(\bv; \bm\Lambda)$ and $\bm\prox_\vreg(\bv; \bm\Lambda)$ satisfy the relation:
$$
\nabla \env_{\vreg}(\bv; \bm\Lambda) = \bm\Lambda^{-1} (\bv-\bm\prox_\vreg(\bv; \bm\Lambda)),
$$
which is the vector analogue of \eqref{eq:prox_subdiff_relation}.
Furthermore, $\bv\mapsto \nabla\env_{\vreg}(\bv; \bm\Lambda)$ is $\lambda_{\min}(\bm\Lambda)^{-1}$-Lipschitz as a map from $\R^p\to\R^p$ (see \Cref{lemma:generalized_prox}). 
Thus, combined with Rademacher's theorem, we see that the Hessian $\nabla^2\env_{\vreg} (\bv; \bm\Lambda)\in\R^{p\times p}$ exists almost everywhere for $\bv\in\R^p$. 

With these new notations, we first introduce a nonlinear system of equations that characterizes the asymptotics of the single estimator $\hat\btheta_I$, generalizing \Cref{sys:general_ensemble-M=1} to the setting of anisotropic designs, deterministic signals, and (potentially) non-separable regularizers:
\renewcommand{\thesystem}{5a}
\begin{system}[Error norms of individual M-estimators]
    \label{sys:general_ensemble-M=1_sigma}
    Given $(\loss,\vreg,c \delta, \btheta, \bm\Sigma)$, 
    define the following 4-scalar system of equations in variables $(\alpha, \beta, \kappa, \nu) \in \R_{>0}^4$:
    \begin{subequations}\label{eq:anisotropic}
    \begin{alignat}{1}
        \alpha^2 &= \tfrac{1}{p}
        \E \Bigl[
          \bigl\|
            \tfrac{\bm{\Sigma}^{-\frac{1}{2}}}{\nu} \nabla \env_{\vreg} (\btheta+ \tfrac{\beta}{\nu} \bm{\Sigma}^{-\frac{1}{2}} \bh; \tfrac{\bm{\Sigma}^{-1}}{\nu} ) - \tfrac{\beta}{\nu} \bh
          \bigr\|_2^2
        \Bigr] \label{eq:anisotropic_1}\\
        \beta^2 &= c\delta \cdot \E\bigl[\env_\loss'(Z + \alpha G; \kappa)^2\bigr] \label{eq:anisotropic_2}
        \\
        \kappa\beta &= \tfrac{1}{p}
         \E\Big[
         \big(
          \tfrac{\bm{\Sigma}^{-\frac{1}{2}}}{\nu} \nabla \env_{\vreg} (\btheta + \tfrac{\beta}{\nu} \bm{\Sigma}^{-\frac{1}{2}} \bh ; \tfrac{\bm{\Sigma}^{-1}}{\nu} ) 
          -
          \tfrac{\beta}{\nu} \bh
          \big)^\top (-\bh)
        \Big] \label{eq:anisotropic_3}
        \\
        \nu\alpha &= c\delta\cdot
        \E\big[\env'_{\loss}(Z + \alpha G; \kappa)\cdot G\big]\label{eq:anisotropic_4}
    \end{alignat}
\end{subequations}
where the expectation is taken with respect to $\bh\sim \cN(\bm{0}_p,\bm{I}_p)$, $G\sim \cN(0,1)$, and $ Z\sim F_z$, which are all mutually independent.
\end{system}

Now we assume that there exists a solution $(\alpha, \beta, \kappa, \nu)$ to \Cref{sys:general_ensemble-M=1_sigma} such that the key statistics of the M-estimator $\hat\btheta(I)=\hat\btheta_I$ in \eqref{eq:def-hbeta-nonseparable} fitted by $(\loss, \vreg)$ are characterized by the solution $(\alpha, \beta, \kappa, \nu)$, as in \Cref{tab:interpretations_squared_loss}:
\begin{subequations}\label{eq:solution_concentrate}
    \begin{alignat}{1}
               p^{-1/2} \|\bm{\Sigma}^{\frac{1}{2}}(\hat\btheta_I-\btheta)\|_2 &=\alpha + \op(1) \label{eq:solution_concentrate_1} \\
    p^{-1/2} \|\by_I-\bX_I \hat\btheta_I\|_2 &= \beta + \op(1) \label{eq:solution_concentrate_2} \\
    \df_I/\tr[\bV_I] &= \kappa + \op(1) \label{eq:solution_concentrate_3} \\
        p^{-1} \tr[\bV_I]&= \nu + \op(1) \label{eq:solution_concentrate_4}
    \end{alignat}
\end{subequations}
where recall $\df_I = \tr[(\partial/\partial \by_I) \bX_I\hat\btheta_I]$ and $\bV_I = (\partial/\partial \by_I) \loss'(\by_I-\bX_I\hat\btheta_I)\in \R^{|I|\times |I|}$. 

Note that \Cref{sys:general_ensemble-M=1_sigma} and the concentration of the norms (\eqref{eq:solution_concentrate_1} and \eqref{eq:solution_concentrate_2}) are already established in the literature for general $(\loss, \vreg)$ (see, e.g., \cite{loureiro2021learning}).
In contrast, the concentration of the trace functionals,  \eqref{eq:solution_concentrate_3} and \eqref{eq:solution_concentrate_4},  is more delicate and has so far been proven only for specific estimators in separate works. 
For instance, in the case of lasso with $\loss(x) = x^2/2$ and $\vreg(\bx) = \lambda \|\bx\|_1$, the convergences in \eqref{eq:solution_concentrate} are provided by \cite[Theorem 5, 8]{celentano2020lasso}, while the ridge case ($\loss(x)=x^2/2, \vreg(\bx)=\lambda \|\bx\|_2^2$) is studied by \cite[Theorems 6, 7]{patil2024revisiting}, \cite[Theorem 2.3, 2.4]{han2023distribution}, among others. 
However, to the best of our knowledge, \eqref{eq:solution_concentrate_3} and \eqref{eq:solution_concentrate_4} have not been rigorously established for general $(\loss,\vreg)$ that are convex but not strongly convex. 
Since \Cref{conjecture} below is stated for such general $(\loss,\vreg)$, and its proof strategy requires the full set of concentration results in \eqref{eq:solution_concentrate}, we provide empirical evidence by verifying \eqref{eq:solution_concentrate} for $\ell_1$-regularized Huber regression ($\loss=\text{Huber}$, $\vreg(\bx) = \|\bx\|_1$), through numerical simulations (see the left and middle panels of \Cref{fig:bagging_solution_cov}).

Next, we present our main extension: a new system that characterizes the correlations between two estimators $\hat\btheta(I)$ and $\hat\btheta(\tilde I)$ trained on overlapping subsamples, extending \Cref{sys:general_ensemble-M=infty} to anisotropic designs with deterministic signals:
\renewcommand{\thesystem}{5b}
\begin{system}[Error correlations of overlapped M-estimators]
    \label{sys:general_ensemble-M=infty_sigma}
   Let $(\alpha, \beta, \kappa, \nu)$ and $(\tilde\alpha, \tilde\beta, \tilde\kappa, \tilde\nu)$ be parameters that satisfy \Cref{sys:general_ensemble-M=1_sigma} with $(\loss, \vreg, c\delta, \btheta, \bm\Sigma)$ and $(\tilde\loss, \tilde \vreg, \tilde c\delta, \btheta, \bm\Sigma)$, respectively. 
    Define the following 2-scalar system of equations in variables $(\etaG, \etaH) \in [-1,1]^2$:
    \begin{equation*}
    {\etaG = F_{\vreg}(\etaH) ,\qquad \etaH = F_\loss(\etaG), 
    }
    \end{equation*}
    where $F_\loss,F_{\vreg}\colon [-1,1] \to \RR$ are functions defined as:
    \small
    \begin{align*}
            F_\loss(\etaG) &= \tfrac{\delta c \tilde c}{\beta\tilde\beta} \E\bigl[
    \env_\loss'(Z+\alpha G; \kappa) \cdot \env_{\tilde\loss}'(Z +\tilde\alpha\tilde G; \tilde\kappa) 
    \bigr]  \\
    F_{\vreg}(\etaH) &= \tfrac{1}{\alpha\tilde\alpha} \tfrac{1}{p}\E\Bigl[\bigl(\tfrac{\bm{\Sigma}^{-\frac{1}{2}}}{\nu} \nabla \env_{\vreg} (\btheta+ \tfrac{\beta}{\nu} \bm{\Sigma}^{-\frac{1}{2}} \bh; \tfrac{\bm{\Sigma}^{-1}}{\nu} ) - \tfrac{\beta}{\nu} \bh\bigr)^\top\bigl(
    \tfrac{\bm{\Sigma}^{-\frac{1}{2}}}{\tilde \nu} \nabla \env_{\tilde \vreg} (\btheta+ \tfrac{\tilde \beta}{\tilde \nu} \bm{\Sigma}^{-\frac{1}{2}} \tilde \bh; \tfrac{\bm{\Sigma}^{-1}}{\tilde \nu} ) - \tfrac{\tilde \beta}{\tilde \nu} \tilde \bh
    \bigr)\Bigr],
    \end{align*}
   \normalsize and the expectation $\E$ is taken with respect to $(\bh, \tilde\bh)\in\R^{p\times 2}$, $(G, \tilde G) \in\R^2$, $Z\in\R$ such that 
$$
\begin{pmatrix}
   h_j\\
 \tilde h_j
\end{pmatrix}_{j\in[p]} \overset{\text{i.i.d.}}{\sim} \cN\biggl(\begin{bmatrix}
    0\\
    0
\end{bmatrix}, \begin{bmatrix}
    1 & \etaH\\
    \etaH & 1
\end{bmatrix}\biggr), \quad \begin{pmatrix}
G\\
  \tilde G
\end{pmatrix}\sim \cN\biggl(\begin{bmatrix}
    0\\
    0
\end{bmatrix}, \begin{bmatrix}
    1 & \etaG\\
    \etaG &  1
\end{bmatrix}\biggr), \quad Z \sim F_z,
$$
which are all mutually independent. 
\end{system}

Crucially, this generalized system retains one essential structure of its isotropic counterpart.
We claim that this new nonlinear system admits the same contraction property as stated in \Cref{th:existence-uniqueness-sys:general_ensemble-M=infty}. 
\begin{theorem}\label{th:contraction_Sigma}
The results of \Cref{th:existence-uniqueness-sys:general_ensemble-M=infty} also hold for $F_\loss, F_\vreg$ in \Cref{sys:general_ensemble-M=infty_sigma}.
Specifically:
\begin{enumerate}[leftmargin=7mm]
    \item
    $|F_{\loss}(\etaG)| \le \sqrt{c\tilde c}$ and $|F_{\vreg}(\etaH)| \le 1$ for all $\etaG\in [-1,1]$ and $\etaH\in [-1,1]$. 
    \item
    $F_\loss$ and $F_{\vreg}$ are non-decreasing, differentiable, and the compositions $F_\loss \circ F_{\vreg}$ and $F_{\vreg}\circ F_\loss$ are $\min\{c, \tilde c\}$-Lipschitz.
    \item \Cref{sys:general_ensemble-M=infty_sigma} admits a unique solution $(\etaGstar,\etaHstar)
    \in [-1, 1] \times [-\sqrt{c\tilde c},\sqrt{c\tilde c}]$.
\end{enumerate}
\end{theorem}

We prove \Cref{th:contraction_Sigma} in \Cref{proof:contraction_Sigma}. 
This brings us to our central conjecture for this setting, which generalizes \Cref{thm:corr-sigerror-reserror} to anisotropic designs and deterministic signals:

\begin{conjecture}\label{conjecture}
Let $\hat{\btheta}_{I}$ and $\hat{\btheta}_{\tilde I}$ be the component M-estimators in \eqref{eq:def-hbeta-nonseparable} trained on subsamples $(\bX_I,\by_I)$ and $(\bX_{\tilde I}, \by_{\tilde I})$ with $(\loss,\vreg)$ and $(\tilde \loss, \tilde \vreg)$, respectively. 
Under suitable assumptions on $(\loss, \tilde{\loss}, \vreg, \tilde{\vreg}, \btheta, \bSigma, F_z)$, as $n, p, k, \tilde{k} \to \infty$ with $n/p \to \delta \in (0, \infty)$, $k/n \to c \in {(0, 1]}$ and $\tilde{k}/n \to \tilde{c} \in {(0,1]}$, it holds that: 
\begin{align}\label{eq:anisotropic_correlation}
\begin{split}
p^{-1} (\hat\btheta_I-\btheta)^\top \bm{\Sigma} (\hat\btheta_{\tilde I}-\btheta) &= \alpha\tilde\alpha  \etaG + \op(1), \\
p^{-1}  \loss'(\by_{I\cap\tilde I}-\bX_{I\cap\tilde I}\hat\btheta_I)^\top \tilde{\loss}'(\by_{I\cap\tilde I}-\bX_{I\cap\tilde I}\hat\btheta_{\tilde I}) &= \beta\tilde\beta\etaH+ \op(1). 
\end{split}
\end{align}
\normalsize
where $(\etaG, \etaH)$ is the solution to \Cref{sys:general_ensemble-M=infty_sigma}. 
Furthermore, there exist some Gaussian vectors $\bh, \tilde\bh\in\R^p$ and $\bg, \tilde\bg\in \R^n$ such that:
\small
\begin{align*}
\frac{1}{p}\biggl\|\begin{pmatrix}
   \hat\btheta_I \\
   \hat\btheta_{\tilde I}
  \end{pmatrix}  - \begin{pmatrix}
     \bm\prox_\vreg (\tfrac{\beta}{\nu} \bm{\Sigma}^{-\frac{1}{2}} \bh +  \btheta; \tfrac{\bm{\Sigma}^{-1}}{\nu}) \\
      \bm\prox_{\tilde \vreg} (\tfrac{\tilde\beta}{\tilde \nu} \bm{\Sigma}^{-\frac{1}{2}} \tilde\bh + \btheta; \tfrac{\bm{\Sigma}^{-1}}{\tilde\nu})
  \end{pmatrix}
  \biggr\|_2^2 &\pto 0,\\
      \frac{1}{|I\cap \tilde I|} \sum_{i\in I\cap \tilde I} \biggl\|
    \begin{pmatrix}
            y_i - \bx_i^\top \hat\btheta_I \\
            y_i - \bx_i^\top \hat\btheta_{\tilde I}
    \end{pmatrix} - \begin{pmatrix}
        \prox_\loss (z_i + \alpha g_i ; \kappa)\\
        \prox_{\tilde\loss}(z_i + \tilde \alpha \tilde g_i;\tilde\kappa)
    \end{pmatrix}
    \biggr\|_2^2 &\pto 0,
\end{align*}
\normalsize
where the conditional distributions of $(\bh, \tilde\bh)$ and $(\bg, \tilde\bg)$ given $(\bz, \btheta, \bm\Sigma)$ are characterized by
$$
\begin{pmatrix}
   h_j\\
 \tilde h_j
\end{pmatrix}_{j\in[p]} \overset{\text{i.i.d.}}{\sim} \cN\biggl(\begin{bmatrix}
    0\\
    0
\end{bmatrix}, \begin{bmatrix}
    1 & \etaH\\
    \etaH & 1
\end{bmatrix}\biggr), \quad \begin{pmatrix}
   g_i\\
 \tilde g_i
\end{pmatrix}_{i\in[n]} \overset{\text{i.i.d.}}{\sim} \cN\biggl(\begin{bmatrix}
    0\\
    0
\end{bmatrix}, \begin{bmatrix}
    1 & \etaG\\
    \etaG & 1
\end{bmatrix}\biggr). 
$$
\end{conjecture}

\begin{figure}[!t]
    \centering
    \includegraphics[width=\linewidth]{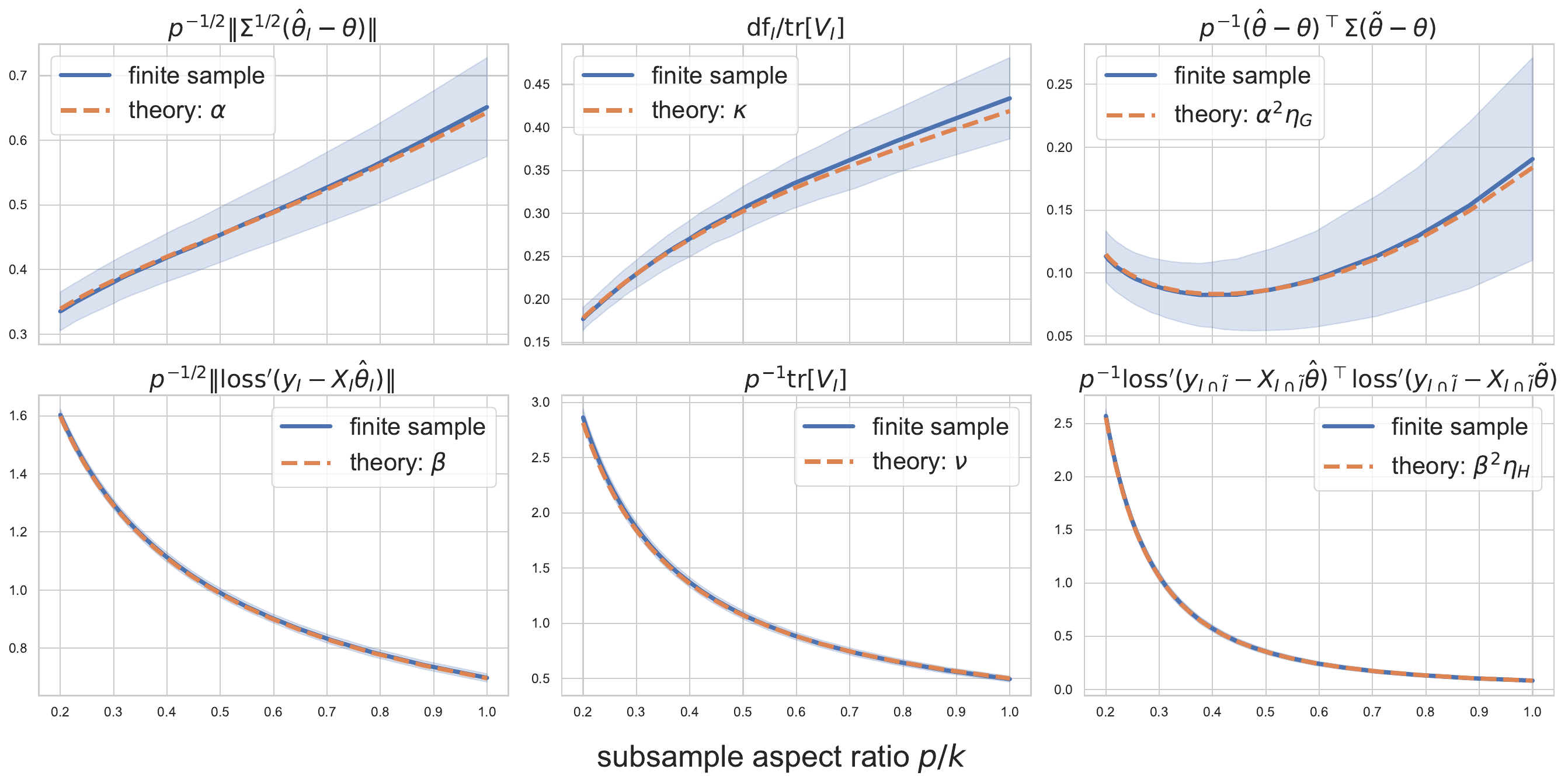}
    \caption{
    {
    \textbf{Verification of \eqref{eq:solution_concentrate} and \eqref{eq:anisotropic_correlation} for 
     $\ell_1$-regularized Huber regression under anisotropic design with deterministic signal.}
    We fix the $\ell_1$-regularization level and the Huber threshold parameter at $1$. The sample size is set to $n=1000$ and the feature dimension to $p=200$, while the subsample size $|I|=k$ is varied.  The noise distribution is $F_z = \text{t-dist}(\text{df}=4)$. The deterministic signal $\btheta \in \R^p$ is chosen such that $\theta_j = 0$ for $1 \le j \le 0.9p$ and $\theta_j = 2/\sqrt{0.1}$ for $0.9p \le j \le p$. The covariance matrix $\bm{\Sigma}$ follows the AR(1) structure, $\Sigma_{ij} = 0.5^{|i-j|}$. Error bars represent the standard deviation across $1000$ simulation runs.
    }
    }
    \label{fig:bagging_solution_cov}
\end{figure}

While \Cref{conjecture} has been proven for ridge estimators in \cite{patil2022bagging} and various generalizations of ridge estimators (that include non-uniform weighting of observations and features) in \cite{patil2023asymptotically,du2024implicit}, a proof for general M-estimators is challenging.
We outline a heuristic proof strategy in \Cref{proof:conjecture} for general convex $(\loss, \vreg)$.
As in the case of the isotropic design $\bm{\Sigma}=\bI_p$ and a random signal, the contraction property in \Cref{th:contraction_Sigma} plays a key role in this proof strategy. 
A full rigorous proof, however, requires a certain non-trivial extension of a normal approximation theorem (\Cref{lm:approx_multi_normal}) owing to the non-separability of the Moreau envelope \eqref{eq:def_moreau_aniso} that arises from the anisotropic covariance and non-separability of $\vreg$ (see \Cref{proof:conjecture} for details). 
We leave this as an important direction for future work.

We have verified \eqref{eq:anisotropic_correlation} of \Cref{conjecture} through a numerical simulation in the setting of $\ell_1$-regularized Huber regression, where $\btheta$ is a fixed sparse vector and $\bm{\Sigma}$ is taken as $\bm{\Sigma}_{ij} = 0.5^{|i-j|}$ (the AR(1) covariance model). 
The right panel of \Cref{fig:bagging_solution_cov} demonstrates that the solution of \Cref{sys:general_ensemble-M=infty_sigma} captures the asymptotic correlations as in \eqref{eq:anisotropic_correlation}.

Finally, analogous to \Cref{sec:ensembles-interpolators}, we derive the nonlinear system for the minimum $\ell_q$-norm interpolators under the anisotropic covariance and deterministic signal in the overparameterized regime $c\delta < 1$. 
Let us take $\vreg(\bx) = \tilde{\vreg}(\bx) = \lambda \|\bx\|_q^q$ for $q\in\{1, 2\}$ and $\loss(x) = \tilde{\loss}(x) = x^2/2$ in \Cref{sys:general_ensemble-M=1_sigma} and \Cref{sys:general_ensemble-M=infty_sigma}, and denote by $(\alpha(\lambda), \etaG(\lambda))$ the corresponding solutions.  
Assuming suitable stability (as $\lambda \to 0^+)$ of \Cref{sys:general_ensemble-M=1_sigma} under the overparameterized regime $c\delta < 1$, the same argument as in the isotropic case (see  \Cref{subsec:derivation_system_interpolator}) yields
$$
\alpha(\lambda)^2 \to \tau^2 - \E[Z^2] \quad \text{and} \quad \alpha^2(\lambda) \etaG(\lambda)\to\xi^2 -\E[Z^2] \quad \text{as $\lambda \to 0^+$}
$$
almost surely, where $(\tau, \xi)$ is a solution to \Cref{sys:interpolators_sigma} below. 
This generalizes \Cref{sys:interpolators} to the case of anisotropic design and deterministic signal.

\begin{figure}[!t]
    \centering
    \includegraphics[width=\linewidth]{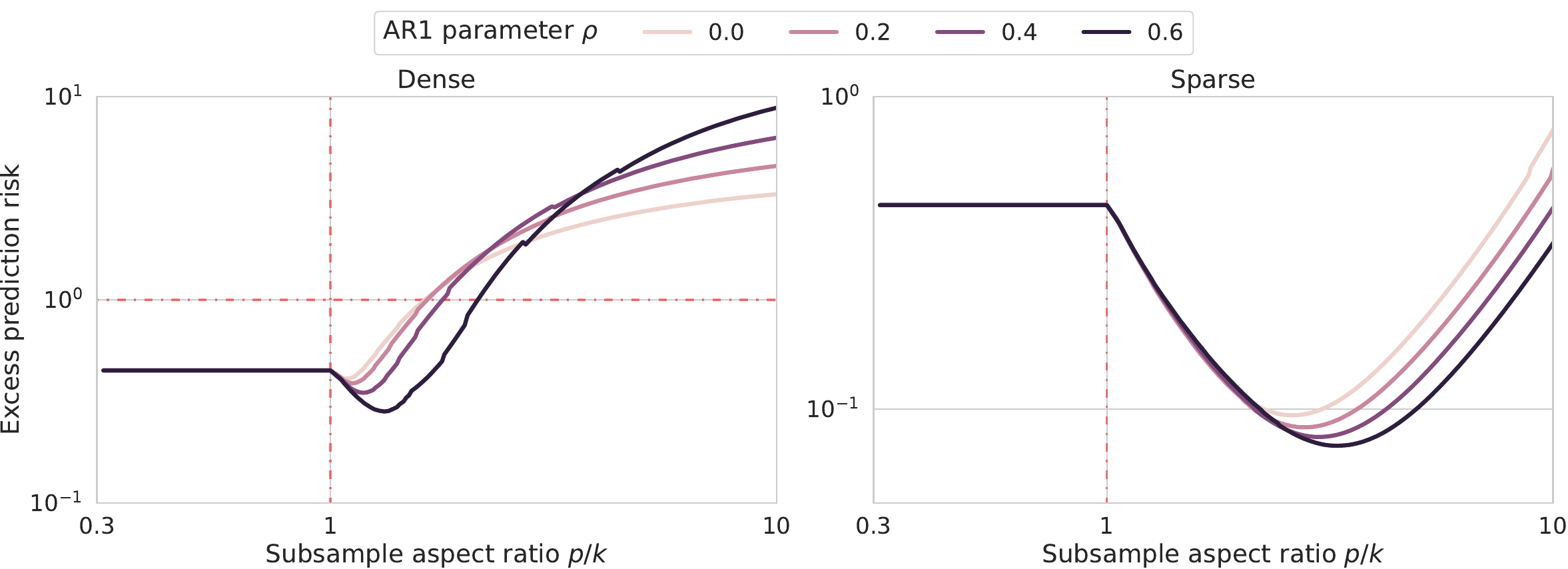}
    \caption{
    {
    Limiting excess prediction risks $\xi^2 - \sigma^2$ of the full-ensemble lassoless estimator under anisotropic covariance and deterministic signal across subsample aspect ratios $p/k$ ranging from $0.3$ to $10$ with the feature dimension $p$ fixed to $200$. The experimental setup follows \Cref{fig:optimum_phis_overparameterized-continuity-derivative}, except that the covariance matrix is replaced by the $\text{AR}(1)$ model, $\bm{\Sigma}_{ij} = \rho_{\mathrm{ar1}}^{|i-j|}$ with varying values of $\rho_{\mathrm{ar1}}$, and the random signal is replaced by a deterministic vector $\btheta \in \R^p$ defined by $\theta_j = 0$ for all $1\le j \le (1-s)p$ and $\theta_j = 2/\sqrt{s}$ for $(1-s)p \le j \le p$, for varying support proportions $s$.
    \emph{Left}: dense regime with $s=0.9$.
    \emph{Right}: sparse regime with $s=0.01$.
    }
    }
    \label{fig:opt_risk_aniso}
\end{figure}

\renewcommand{\thesystem}{6}
\begin{system}
    [Ensembles of minimum $\ell_q$-norm interpolators]
    \label{sys:interpolators_sigma}
    Given $(\btheta, \bSigma, \delta, c)$ such that $c\delta (=\lim k/p) <1$, assuming $\sigma^2 := \E[Z^2]$ is finite, define the following system of equations in variables $(\tau,a)\in \R_{>0}^2$:
    \begin{subequations}
    \label{eq:ellq-interpolators_sigma}
    \begin{empheq}{align}
     \tau^2 &=  \tfrac{1}{p} \E[\|
    \bm{\Sigma}^{\frac{1}{2}} (\bm\prox_{\|\cdot\|_q^q}(\btheta + \tfrac{\tau}{\sqrt{c\delta}}\bm{\Sigma}^{-\frac{1}{2}}\bh; \tfrac{a\tau}{\sqrt{c\delta}}\bm{\Sigma}^{-1})-\btheta)\|_2^2] + \sigma^2, \label{eq:interpolators-tau_sigma} \\
        c\delta &= \tfrac{\sqrt{c\delta}}{\tau}\tfrac{1}{p} \E\Bigl[\bh^\top 
     \bm{\Sigma}^{\frac{1}{2}} \bigl(\bm\prox_{\|\cdot\|_q^q}(\btheta + \tfrac{\tau}{\sqrt{c\delta}}\bm{\Sigma}^{-\frac{1}{2}}\bh; \tfrac{a\tau}{\sqrt{c\delta}}\bm{\Sigma}^{-1}) -\btheta
     \bigr)\Bigr],
        \label{eq:interpolators-a_sigma}
    \end{empheq}
    \end{subequations}
    \normalsize
    where $\bh \sim \cN(\bm{0}_p,\bI_p)$. 
    Given $(\tau,a)$ that satisfy \eqref{eq:interpolators-tau_sigma}-\eqref{eq:interpolators-a_sigma}, define the following 1-scalar system of equations in variable $\xi\in\R_{>0}$:
    \footnotesize
    \begin{align*}
        \xi^2 
    = \tfrac{1}{p}\mathbb{E}\Big[
    \bigl(\bm\prox_{\|\cdot\|_q^q}(\btheta + \tfrac{\tau}{\sqrt{c\delta}}\bm{\Sigma}^{-\frac{1}{2}}\bh; \tfrac{a\tau}{\sqrt{c\delta}}\bm{\Sigma}^{-1})-\btheta\bigr)^\top \bm\Sigma \bigl(\bm\prox_{\|\cdot\|_q^q}(\btheta + \tfrac{\tau}{\sqrt{c\delta}}\bm{\Sigma}^{-\frac{1}{2}}\tilde \bh; \tfrac{a\tau}{\sqrt{c\delta}}\bm{\Sigma}^{-1})-\btheta\bigr)
    \Big] + \sigma^2,
    \end{align*}
    \normalsize
    where
    $(\bh,\tilde \bh)$ follows a bivariate Gaussian distribution as in \Cref{sys:general_ensemble-M=infty_sigma}
    with
    $\etaH = c \frac{\xi^2}{\tau^2}$.
      \label{eq:full-ensemble-risk-ellq-interpolators_sigma}
\end{system}

\begin{figure}[!t]
    \centering
    \includegraphics[width=0.99\textwidth]{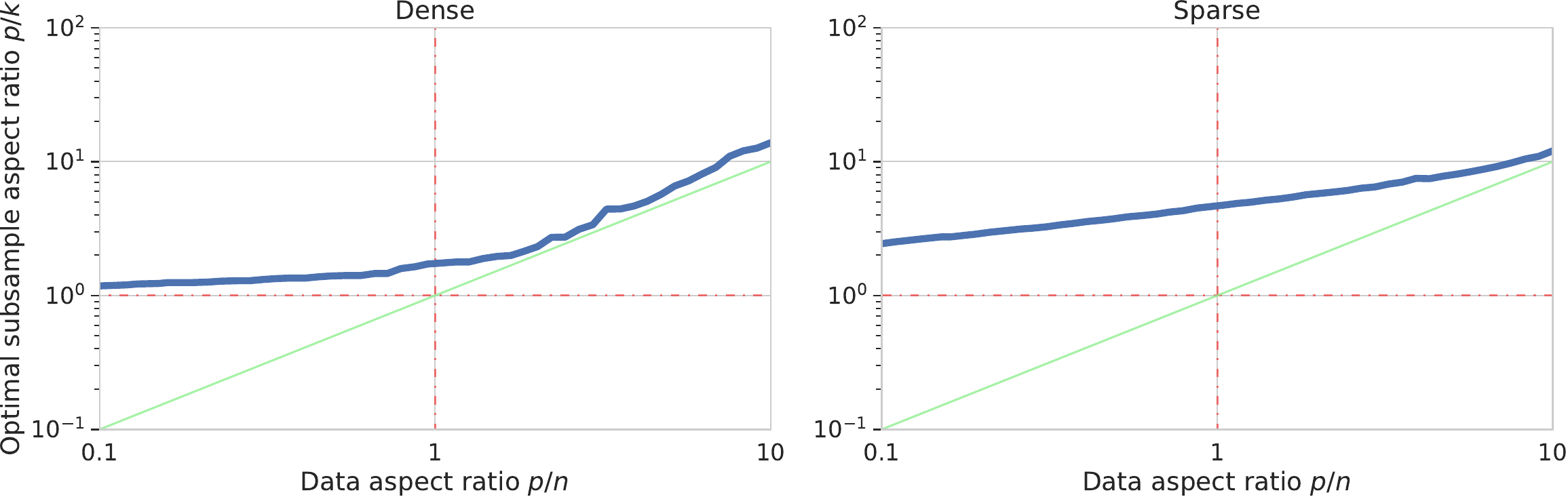}
    \caption{
        {
        \textbf{Optimal subsample size for the lassoless ensemble with anisotropic covariance is always in the overparameterized regime.}
        Optimal subsample aspect ratio $p/k$ that achieves the optimal risk for the lassoless ensemble at different data aspect ratios $p/n$ ranging from $0.1$ to $10$.
        The data model is as in \Cref{fig:opt_risk_aniso} with $\rho_{\textrm{ar1}}=0.6$, data aspect ratio $p/n=0.1$, feature size $p=200$, and varying support proportion $s$.
        \emph{Left}: dense regime with $s=0.9$.
        \emph{Right}: sparse regime with $s=0.01$.
        }
    }
    \label{fig:optimum_phis_overparameterized-2-optsubsample-ar1_aniso}
\end{figure}

\begin{figure}[!t]
    \centering
\includegraphics[width=0.99\textwidth]{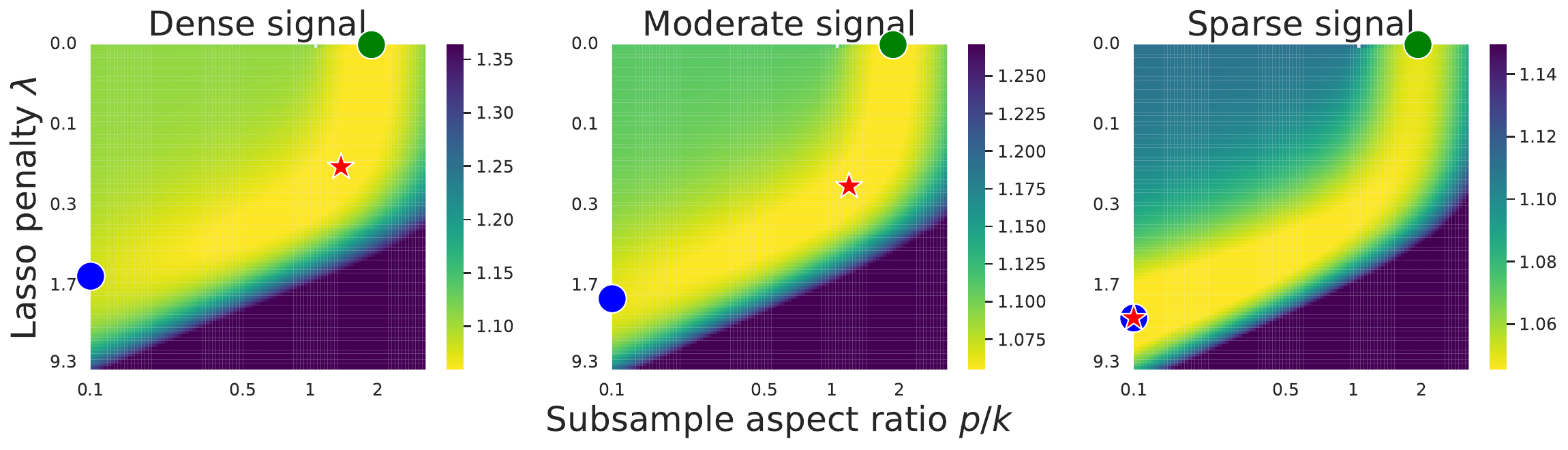}
    \caption{
        {
        Heatmaps of theoretical prediction risk $\alpha^2\etaG$ in $\lambda$ and $p/k$ with $k\in (0, n]$ of the full lasso ensemble with anisotropic covariance and deterministic signal in the underparameterized regime ($p/n=0.1$). 
        The data model is as in \Cref{fig:opt_risk_aniso} with $\rho_{\textrm{ar1}}=0.6$ at different sparsity levels $s$.
        Blue dots ($\color{pyblue}\mysolidcircle$): the optimal lasso ($\lambda$ is optimized with $k$ fixed to $n$). 
        Green dots ($\color{pygreen}\mysolidcircle$): the optimally subsampled lassoless ($k$ is optimized with $\lambda$ fixed to $0$). 
        Red stars ($\color{pyred}\star$):  the optimal lasso ensemble ($\lambda$ and $k$ are jointly optimized).
        \emph{Left}: Dense regime with support proportion $s=0.9$. \emph{Middle}: Moderate sparse regime with support proportion $s=0.5$. In these regimes, 
        the optimal lasso ensemble ($\color{pyred}\star$) outperforms both the optimal lasso ($\color{pyblue}\mysolidcircle$) and the optimally subsampled lassoless ($\color{pygreen}\mysolidcircle$). 
        \emph{Right}: Sparse regime with support proportion $s=0.2$. In this regime, the optimal lasso ($\color{pyblue}\mysolidcircle$) coincides with the optimal lasso ensemble ($\color{pyred}\star$) and outperforms the optimal subsample lassoless ($\color{pygreen}\mysolidcircle$). 
        }
    }
    \label{fig:advantage_subsampling_underparameterized_aniso}
\end{figure}

We note that the nonlinear system \eqref{eq:ellq-interpolators_sigma} with $q=2$ has appeared in prior works (see, e.g., \cite{han2023distribution}). 

Based on the solution $\xi$ from \Cref{sys:interpolators_sigma}, we conducted an experiment analogous to \Cref{fig:optimum_phis_overparameterized-continuity-derivative}. In \Cref{fig:opt_risk_aniso}, we plot the limiting excess risk $\xi^2-\sigma^2$ of the full-ensemble $(M=\infty)$ lassoless estimator under the AR(1) covariance model with a deterministic signal. The results show that the optimal subsample size $k^*$ once again lies in the overparameterized regime ($k^* < p$). Moreover, both the optimal subsample size $k^*$ and the corresponding prediction risk decrease as the AR(1) correlation parameter $\rho_{\mathrm{ar1}}$ increases. 

We also conducted experiments analogous to \Cref{fig:optimum_phis_overparameterized-2-optsubsample} and \Cref{fig:advantage_subsampling_underparameterized}. 
\Cref{fig:optimum_phis_overparameterized-2-optsubsample-ar1_aniso} shows that for the full lassoless ensemble, the optimal subsample size $k^*$ always lies in the overparameterized regime for varying data aspect ratios $p/n$. 
Moreover, \Cref{fig:advantage_subsampling_underparameterized_aniso} demonstrates that jointly optimizing the subsample size $k$ and the lasso penalty $\lambda$ can yield strictly smaller risk than both the non-ensemble lasso with optimally tuned penalty and optimally subsampled lassoless ensemble. These results suggest that the conclusions drawn in \Cref{sec:optimal-subsample-size}–\ref{sec:optimal-regularization-penalty} also extend to the case of anisotropic designs with deterministic signals.
}

\section*{Acknowledgments}

We thank Michael Celentano and Ryan Tibshirani for useful discussions.
We also thank Lucas Clarte and Bruno Loureiro for useful comments on an earlier draft of the paper.
PP and JHD acknowledge the computing support provided by the grant MTH230020 for experiments run on the Bridge-2 system at the Pittsburgh Supercomputing Center.
PCB acknowledges partial support from the NSF Grant DMS-1945428.

\bibliographystyle{alpha}
\bibliography{references}

\clearpage

\appendix

\newgeometry{left=0.25in,top=0.25in,right=0.25in,bottom=0.25in,head=.1in,foot=0.1in}

\begin{center}
\Large
{\bf \framebox{Supplement}}
\end{center}

\bigskip

This serves as a supplement to the paper ``\titletext.''
Below, we provide an outline of the supplement along with a summary of the notation used in the main paper and the supplement.

\section*{Organization}

\begin{table}[!ht]
\centering
\begin{tabularx}{\textwidth}{L{2.25cm}L{15cm}}
\toprule
\textbf{Section} & \textbf{Content} \\
\midrule
\Cref{sec:correlation_signs}
& 
{Sign of solutions $(\etaH, \etaG)$ to \Cref{sys:general_ensemble-M=infty}} \\
\arrayrulecolor{black!25}
\addlinespace[0.5ex]
\midrule
\Cref{sec:proofs-sec:general-risk-characterization}
& 
Proofs of \Cref{th:existence-uniqueness-sys:general_ensemble-M=infty,thm:corr-sigerror-reserror,thm:risk-estimator-general-subagging} from \Cref{sec:general-risk-characterization} \\
\arrayrulecolor{black!25}
\addlinespace[0.5ex]
\midrule
\Cref{sec:proofs-sec:specific_examples}
& Proof of \Cref{prop:monotonicity-ensemble-size}, \ref{prop:monotonicity-risk} and other miscellaneous details from \Cref{sec:specific-examples} \\
\addlinespace[0.5ex]
\midrule
\Cref{sec:proofs-sec:general-properties}
& Details of arguments in \Cref{rem:lassoless-fullensemble-continuity-psi=1} from \Cref{sec:ensembles-interpolators} \\
\addlinespace[0.5ex]
\midrule
\Cref{sec:appendix:anisotropic_deterministic}
& Proof of \Cref{th:contraction_Sigma} and \Cref{conjecture} from \Cref{sec:discussion} \\
\addlinespace[0.5ex]
\midrule
\Cref{sec:additional_numerical_illustrations}
& Additional numerical illustrations and experimental details \\
\arrayrulecolor{black}\bottomrule
\end{tabularx}
\caption{Outline of the supplement.}
\label{tab:outline-supplement}
\end{table}

\section*{Specific notation}

\begin{table}[!ht]
\centering
\begin{tabularx}{\textwidth}{L{3.3cm}L{16cm}}
\toprule
\textbf{Notation} & \textbf{Description} \\
\midrule
$(\bx_i,y_i)$, $i \in [n]$ & Observation vector with the feature vector in $\RR^{p}$ and the response variable in $\RR$ \\
$\bX,\by$ & Feature matrix in $\RR^{n\times p}$ and the response vector in $\RR^n$ \\
$\btheta$ & Signal vector in the linear model in $\RR^{p}$ with entries drawn i.i.d.\ from the distribution $F_\theta$ \\
$\bm{\eps}$ & Noise vector in the linear model in $\RR^{n}$ with entries drawn i.i.d.\ from the distribution $F_\eps$ \\
\arrayrulecolor{black!25}
\addlinespace[0.5ex]
\midrule
$k$, $M$ & Size of each subsample and the total number of subsamples (bags) used in the ensemble \\
$I_m$, $m \in [M]$ & Index set (subset of $[n]$) of the $m$-th subsample, with $|I_m| = k$ \\
$(\bX_{\setm},\by_{\setm})$, $m \in [M]$ & Feature matrix in $\RR^{k \times p}$ and response vector in $\RR^{k}$ of the subsampled dataset indexed by $I_m$ \\ 
$(\bX_{I_m \cap I_\ell}, \by_{I_m \cap I_\ell})$, $m \neq \ell \in [M]$ & Feature matrix in $\RR^{|I_m \cap I_\ell| \times p}$ and response vector in $\RR^{|I_m \cap I_\ell|}$ of the overlapped dataset corresponding to the overlap between the subsampled datasets indexed by $I_m$ and $I_\ell$ \\
\arrayrulecolor{black!25}
\addlinespace[0.5ex]
\midrule
$\loss$, $\reg$, $\lambda$ & Loss function, regularization function, and regularization level used for the regularized M-estimator \\
$\hat{\bbeta}_{m}$, $\tilde{\bbeta}_M$ & Component regularized M-estimator fitted on the $m$-th subsample and the ensemble estimator \\
$R_M$, $\cR_M$ & Squared prediction risk of the ensemble estimator and its asymptotic limit \\
$\cR_1$, $\cR_\infty$ & Asymptotic risk of the non-ensemble estimator ($M = 1$) and the full-ensemble estimator ($M = \infty$) \\
\arrayrulecolor{black!25}
\addlinespace[0.5ex]
\midrule
$\delta$, $c$ & Inverse data aspect ratio $n/p$ and subsample ratio $k/n$ \\
\arrayrulecolor{black!25}
\addlinespace[0.5ex]
\midrule
$(\alpha, \beta, \kappa, \nu)$ & Parameters characterizing the ensemble risk asymptotics with $M = 1$ (\Cref{sys:general_ensemble-M=1}) \\
$G$, $H$ & Standard normal random variables in the ensemble risk asymptotics with $M = 1$ (\Cref{sys:general_ensemble-M=1}) \\
$\Theta$, $Z$ & Random variables drawn according to distributions $F_\theta$ and $F_\eps$, respectively (\Cref{sys:general_ensemble-M=1}) \\
\arrayrulecolor{black!25}
\addlinespace[0.5ex]
\midrule
$(\etaG, \etaH)$ & Correlation parameters in the ensemble risk asymptotics with $M = \infty$ (\Cref{sys:general_ensemble-M=infty}) \\
$G, \tilde{G}$, $H, \tilde{H}$ & Random variables appearing in the ensemble risk asymptotics with $M = \infty$ (\Cref{sys:general_ensemble-M=infty}) \\
\arrayrulecolor{black!25}
\addlinespace[0.5ex]
\midrule
${(a, \tau)}$ & Parameters in alternate formulation of the ensemble risk asymptotics with $M = 1$ (\Cref{sys:ensembles-penalized-least-squares}) \\
$\xi$ & Parameter in alternate formulation of the ensemble risk asymptotics with $M = \infty$ (\Cref{sys:ensembles-penalized-least-squares}) \\
\arrayrulecolor{black}
\bottomrule
\end{tabularx}
\caption{Summary of some of the specific notation used in the paper.}
\label{tab:notation}
\end{table}

\restoregeometry

\newpage

\section{Correlation signs}
\label{sec:correlation_signs}

\begin{proposition}
\label{prop:sign_pattern}
    The signs of the solution $(\etaGstar,\etaHstar)$ to \Cref{sys:general_ensemble-M=infty} are characterized as
    $$
    \sign(\etaGstar)=\sign(F_{\reg} \circ F_\loss(0)), \qquad \sign(\etaHstar)
            =\sign(F_{\loss} \circ F_\reg(0))
    $$
    where $\sign(x) := \ind_{\{x>0\}} - \ind_{\{x<0\}}$. 
\end{proposition}
\begin{proof}
The claim immediately follows from the fact that $\etaGstar$ and $\etaHstar$ satisfies the fixed-point equations $\etaG - F_\reg\circ F_\loss(\etaG) =0$ and $\etaH - F_\loss \circ F_\reg(\etaH) =0$, respectively, and the maps $\etaG \mapsto \etaG-F_\reg\circ F_\loss(\etaG)$ and $\etaH\mapsto \etaH - F_\loss \circ F_\reg(\etaH)$ are nondecreasing since $F_\reg\circ F_\loss$ and $F_\loss \circ F_\reg$ are $(c\wedge\tilde c)$-Lipschitz with $c\wedge\tilde c\le 1$ (\Cref{th:existence-uniqueness-sys:general_ensemble-M=infty}). 
\end{proof}
Recall the property of $F_\loss$ and $F_\reg$ in \Cref{th:existence-uniqueness-sys:general_ensemble-M=infty}. 
Since $F_\loss$ and $F_\reg$ in \Cref{sys:general_ensemble-M=infty} are non-decreasing, combined with \Cref{prop:sign_pattern}, we get the following simple sufficient condition which determines the sign of $(\etaHstar, \etaGstar)$:
\begin{align*}
    \text{$F_\loss(0)$ and $F_\reg(0)$ are non-negative}  &\Rightarrow \text{$\etaHstar$ and $\etaGstar$ are non-negative}.
\end{align*}
Here $F_\loss(0)$ and $F_\reg(0)$ can be written as follows:
\begin{subequations}
\begin{align*}
    F_\loss(0) 
    &= \frac{c\tilde c \delta}{\beta\tilde\beta}\cdot \E\Bigl[\E \big[
              \env_{\loss}'(Z + \alpha G; \kappa) \mid Z\big] \cdot \E\big[\env_{\tilde \loss}'(Z + \tilde \alpha G; \tilde \kappa) \mid Z\big]  \Bigr],  \\
    F_\reg(0) 
    &= \frac{1}{\nu \tilde\nu \alpha \tilde\alpha} \cdot \E \Bigl[
              \E \bigl[\env_{\reg}' \bigl(\Theta + \tfrac{\beta}{\nu} H; \tfrac{1}{\nu}\bigr) \mid \Theta
             \bigr]
              \cdot 
            \E \bigl[
            \env_{\tilde \reg}' \bigl(\Theta + \tfrac{\tilde \beta}{\tilde \nu} H; \tfrac{1}{\tilde \nu}\bigr) \mid \Theta
              \bigr]
            \Bigr]. 
\end{align*}
\end{subequations}

In particular, if the same loss and regularizer are used, i.e., $\loss=\tilde\loss$ and $\reg=\tilde \reg$, and the subsample sizes are the same, i.e., $k=\tilde k$, then the solutions satisfying \Cref{sys:general_ensemble-M=infty} are same, i.e., $(\alpha, \beta, \kappa, \nu) = (\tilde\alpha, \tilde\beta, \tilde\kappa, \tilde\nu)$, so that it is easy to see from the above formula that $F_\loss(0)\ge 0$ and $F_\reg(0)\ge 0$. 
This means that the solutions $(\etaGstar, \etaHstar)$ are non-negative when the same loss, regularizer, and subsample size are used. 

Another case such that the solutions $(\etaGstar,\etaHstar)$ are positive is when $\loss$ and $\tilde\loss$ are least squares and $\reg$ and $\tilde{\reg}$ are ridge (but possibly different regularization parameters). 
This is because $\env'_{f};(x;\tau)$ is linear in $x$ for any squared loss (and regularizer) of the form $f(\cdot)=\lambda (\cdot)^2$ so that $F_\loss(0) = \C \E[Z^2] \ge 0$ and $F_\reg(0) = \C \E[\Theta^2] \ge 0$ for some positive constants $C_1, C_2$.

\begin{remark}
    [Negative estimator error correlation]
    Let us take $\reg, \tilde \reg$ as the indicator functions
    \begin{align*}
        \reg(x) = \Pi_{(-\infty,-t]} (x) \coloneq \begin{cases}
            +\infty  &x > - t \\
            0 & x\le -t
        \end{cases} 
        \quad \text{and} \quad
        \tilde \reg(x) = \Pi_{[\tilde t, +\infty)} (x) \coloneq \begin{cases}
            +\infty  &x < \tilde{t} \\
            0 & x \ge \tilde{t}
        \end{cases}
    \end{align*}
    where $t, \tilde t\ge 0$ are non-negative constants.
    Note that the two sets, $(-\infty, -t]$ and $[\tilde{t}, +\infty)$, are disjoint.
    Noting $\prox_\reg(x) =  \min\{x, -t\}$ and $\prox_{\tilde \reg}(x) = \max\{x, \tilde{t}\}$, we have 
    $$
        F_\reg(\etaH) = \frac{1}{\nu\tilde\nu \alpha\tilde\alpha}\E\big[\big(\min\big\{\Theta + \tfrac{\beta}{\nu}H, -t\big\} - \Theta\big) \cdot \big(\max\big\{\Theta + \tfrac{\tilde\beta}{\tilde \nu}\tilde H, \tilde{t}\big\} - \Theta\big)\big]. 
    $$
    Thus, if $\Theta$ is included in the closed set $[-t, \tilde t]$ with probability $1$, then we have $F_\reg(\etaH) \le 0$ for all $\etaH$ so that $\etaGstar = F_\reg(\etaH^\star)$ is non-positive. 
    This is intuitive as we will see in \Cref{thm:corr-sigerror-reserror} that $\etaG$ characterizes the limiting behavior of the correlations between two estimators trained on $\reg$ and $\tilde\reg$.
\end{remark}

\section{Proofs for results  in Section~{\ref{sec:general-risk-characterization}}}
\label{sec:proofs-sec:general-risk-characterization}

\subsection{Proof of \Cref{th:existence-uniqueness-sys:general_ensemble-M=infty}}\label{subsec:proof-exisntence-uniqueness}
\paragraph*{Part 1 }
We redefine $F_{\loss}$ and $F_{\reg}$ for convenience:
\begin{align*}
    F_\loss(\etaG) &= \tfrac{c\tilde c \delta}{\beta\tilde\beta}\cdot \E [
          \env_{\loss}'(\alpha G + Z; \kappa) \cdot \env_{\tilde \loss}'(\tilde \alpha \tilde G + Z; \tilde \kappa)], \\
    F_\reg(\etaH)  & = \tfrac{1}{\alpha \tilde\alpha} \cdot \E \Bigl[
          \Bigl(
            \tfrac{1}{\nu} \cdot \env_{
\reg}' \Bigl(\tfrac{\beta}{\nu} H + \Theta ; \tfrac{1}{\nu}\Bigr) - \tfrac{\beta}{\nu} H
          \Bigr)
          \cdot 
        \Bigl(
            \tfrac{1}{\tilde \nu} \cdot \env_{\tilde 
\reg}' \Bigl(\tfrac{\tilde \beta}{\tilde \nu} \tilde H + \Theta ; \tfrac{1}{\tilde \nu}\Bigr) - \tfrac{\tilde \beta}{\tilde \nu}\tilde  H
          \Bigr)
        \Bigr].
\end{align*}
By the Cauchy--Schwarz inequality, using \eqref{eq:CGMT-1b} in \Cref{sys:general_ensemble-M=1}, we have
\begin{align*}
    |F_\loss(\etaG)| &\le \tfrac{c\tilde c\delta}{\beta\tilde\beta} \cdot \E[\env_{\loss}'(\alpha G + Z; \kappa)^2]^{1/2} \cdot \E[\env_{\tilde \loss}'(\tilde \alpha \tilde G + Z; \tilde \kappa)^2]^{1/2}\\
    &= \tfrac{c\tilde c\delta}{\beta\tilde\beta} \cdot \sqrt{\tfrac{\beta^2}{c\delta}} \cdot \sqrt{\tfrac{\tilde\beta^2}{\tilde c\delta}}\\
    &=\sqrt{c\tilde c}.
\end{align*}
for all $\etaG\in [-1,1]$. By the same argument,  the Cauchy--Schwarz inequality and \eqref{eq:CGMT-1a} in \Cref{sys:general_ensemble-M=1} yield $|F_\loss(\etaH)| \le 1$ for all $\etaH\in [-1,1]$. 

\paragraph*{Part 2 }
Let us prove the differentiability and contraction of compositions.
We will use the following lemma to argue the differentiability of $F_\loss$ and $F_\reg$.

\begin{lemma}
    \label{lm:varphi_derivative}
    Let $G$ and $Z$ be independent $\cN(0,1)$ random variables. 
    Then, for any Lipschitz functions $(f, \tilde f)$ with bounded second moment $\E[f(G)^2], \E[\tilde f(G)^2]<+\infty$, the map 
    $$
    \varphi:[-1,1]\to\R, \quad  \eta \mapsto \E[f(G)\tilde f(\eta G + \sqrt{1-\eta^2}Z)]
    $$ 
    has the derivative
    $$
    \varphi'(\eta) = \E[f'(G)\tilde f'(\eta G + \sqrt{1-\eta^2}Z)].
    $$
\end{lemma}
\begin{proof}
    Since $\tilde f$ is Lipschitz and $\cN(0,1)$ has no point mass, $\tilde f$ is differentiable at $G\sim \cN(0,1)$ with probability $1$. 
    By the dominated convergence theorem, we have 
    \begin{align*}
    \varphi'(\eta) &= \E\bigl[
        f(G) \tilde f'\bigl(\eta G +\sqrt{1-\eta^2}Z\bigr) 
        \bigl(G-\frac{\eta}{\sqrt{1-\eta^2}} Z\bigr)
    \bigr].
    \end{align*}
    Let us define $A=\eta G + \sqrt{1-\eta^2}Z$ and $B=\sqrt{1-\eta^2}G-\eta Z$ so that $(A, B)$ are independent Gaussian $\cN(0,1)$ and $\varphi'(\eta) =  (1-\eta^2)^{-1/2} \E[f(\eta A + \sqrt{1-\eta^2}B) \tilde f'(A) B]$. 
    Using Stein's formula for $B$ conditionally on $A$, we are left with
    \begin{align*}
        \varphi'(\eta)
        = \E[\tilde f'(A) f'(\eta A + \sqrt{1-\eta^2}B)]= \E[\tilde f'(\eta G+ \sqrt{1-\eta^2} Z) f'(G)],
    \end{align*}
    where we have used $(A, \eta A +  \sqrt{1-\eta^2} B) \deq (\eta G + \sqrt{1-\eta^2} Z, G)$. 
    (Here and throughout $\deq$ refers to equality in distribution.)
\end{proof}

Notice that $\env_{\loss}'(\alpha G + Z;\kappa)$ and $\frac{1}{\nu} \cdot \env_{\reg}' \Bigl(\frac{\beta}{\nu} H + \Theta; \frac{1}{\nu}\Bigr) - \frac{\beta}{\nu} H$ have finite second moments, thanks to the existence of the solution to \Cref{sys:general_ensemble-M=1}. Thus, applying \Cref{lm:varphi_derivative} with $(f(x), \tilde f(x)) = (\env_{\loss}'(\alpha x + Z; \kappa), \env_{\loss}'(\tilde\alpha x+Z; \tilde\kappa))$ and $(f(x), \tilde f(x)) = (\frac{\beta}{\nu} x - \frac{1}{\nu} \env_{\reg}(\frac{\beta}{\nu} x+ \Theta; \frac{1}{\nu}), \frac{\tilde\beta}{\tilde\nu} x - \frac{1}{\tilde\nu} \env_{\reg}(\frac{\tilde\beta}{\tilde\nu} x+ \Theta; \frac{1}{\tilde\nu}))$, we find that $F_\loss$ and $F_{\reg}$ are differentiable and the derivatives are given by:
\begin{align*}
    F_\loss'(\etaG) &= \frac{c \tilde c \delta}{\beta\tilde\beta}\cdot \E[\alpha \env_{\loss}''(\alpha G + Z;\kappa ) \cdot \tilde\alpha \env_{\tilde\loss}''(\tilde\alpha \tilde G + Z;\tilde\kappa )] \\
    F_{\reg}'(\etaH) &= \frac{1}{\alpha\tilde\alpha} \E \Bigl[
        \frac{\beta}{\nu}
          \Bigl(1-
            \frac{1}{\nu} \cdot \env_{\reg}'' \Bigl(\frac{\beta}{\nu} H + \Theta ; \frac{1}{\nu}\Bigr) 
          \Bigr)
          \cdot 
          \frac{\tilde\beta}{\tilde\nu}
        \Bigl(
        1-
            \frac{1}{\tilde \nu} \cdot \env_{\tilde \reg}'' \Bigl(\frac{\tilde \beta}{\tilde \nu} \tilde H + \Theta ; \frac{1}{\tilde \nu}\Bigr) 
          \Bigr)
        \Bigr]
\end{align*}
for all $\etaG, \etaH \in [-1,1]$. 
Note that the non-expansiveness of the proximal operator implies that the map $x\mapsto \env_{f}'(x;\tau) = \tau^{-1}(x-\prox_f(x;\tau))$ is $\tau^{-1}$-Lipschitz and non-decreasing for any convex function $f$. Thus,  $F_\loss'(\etaG)$ and $F_{\reg}'(\etaH)$ are non-negative and uniformly bounded from above as follows:
    \begin{align*}
       0 \le  F_\loss'(\etaG)
       &\le \frac{c\tilde c \delta}{\beta\tilde\beta}\frac{\alpha}{\kappa} \cdot \E[\tilde\alpha \env_{\tilde\loss}''(\tilde\alpha \tilde G + Z; \tilde\kappa)] && (0\le \env_{\loss}''(\alpha G +Z;\kappa) \le \kappa^{-1}) \\
       &=  \frac{c \tilde c\delta}{\beta\tilde\beta}\frac{\alpha}{\kappa} \cdot \E[\tilde G \cdot \env_{\tilde\loss}'(\tilde\alpha \tilde G + Z; \tilde\kappa)] && (\text{by Stein's lemma}) \\
       &=  \frac{c\tilde c\delta}{\beta\tilde\beta}\frac{\alpha}{\kappa}  \cdot \frac{\tilde\nu \tilde\alpha}{\tilde c \delta} = c \cdot \frac{\alpha\tilde\alpha \tilde\nu}{\beta\tilde\beta\kappa} && (\text{using \eqref{eq:CGMT-1d} in \Cref{sys:general_ensemble-M=1}}); \\
               0 \le  F_\reg'(\etaH)
        &\le \frac{1}{\alpha\tilde\alpha}  \E \Bigl[
         \frac{\beta}{\nu}
          \Bigl(1-
            \frac{1}{\nu} \cdot \env_{\reg}'' \Bigl(\frac{\beta}{\nu} H + \Theta ; \frac{1}{\nu}\Bigr) 
          \Bigr)
        \Bigr] \cdot \frac{\tilde\beta}{\tilde\nu} && (0 \le \env_{\reg}'' (\frac{\tilde \beta}{\tilde \nu} H + \Theta; \frac{1}{\tilde \nu})\le \tilde \nu) \\
        &= \frac{1}{\alpha\tilde\alpha} \frac{\tilde\beta}{\tilde\nu} \E \Bigl[H
          \Bigl(\frac{\beta}{\nu} H -
            \frac{1}{\nu} \cdot \env_{\reg}' \Bigl(\frac{\beta}{\nu} H + \Theta ; \frac{1}{\nu}\Bigr) 
          \Bigr)
        \Bigr]  && (\text{by Stein's lemma}) \\
        &= \frac{1}{\alpha\tilde\alpha} \frac{\tilde\beta}{\tilde\nu}  \cdot \beta\kappa  = \frac{\beta\tilde\beta\kappa}{\alpha\tilde\alpha\tilde\nu} && (\text{using \eqref{eq:CGMT-1c} in \Cref{sys:general_ensemble-M=1}}).
    \end{align*}
Thus, by the chain rule, noting that $(\beta\tilde\beta\kappa)/(\alpha\tilde\alpha\tilde\nu)$ is cancelled out, we have
$$
0 \le (F_\loss\circ F_\reg)'(\etaH) \le c, \quad  0 \le (F_\reg\circ F_\loss)'(\etaG) \le c.
$$
By switching the role of $(\alpha, \beta, \kappa, \nu, c)$ and $(\tilde \alpha, \tilde \beta, \tilde \kappa, \tilde \nu, \tilde c)$, it also holds that 
$$
0 \le (F_\loss\circ F_\reg)'(\etaH) \le \tilde c, \quad 0 \le (F_\reg\circ F_\loss)'(\etaG) \le \tilde c.
$$
Thus, by taking the minimum of $(c, \tilde{c})$, we find that the compositions 
$$
\etaH \mapsto F_\loss \circ F_\reg(\etaH) , \quad \etaG \mapsto F_\reg\circ F_\loss(\etaG)
$$
are $(c \wedge \tilde c)$-Lipschitz. 

\paragraph*{Part 3}
Let us show the uniqueness and existence of the solution to the nonlinear system:
$$
\etaH = F_{\loss} (\etaG), \quad \etaG = F_{\reg}(\etaH).
$$
We have shown in the previous paragraph that the map $F_\reg\circ F_\loss(\etaH):[-1,1] \to [-1,1]$ are non-decreasing, differentiable, and $1$-Lipschitz. Thus, by Brouwer’s fixed-point theorem, there exists some $\etaG^\star \in [-1,1]$ such that 
$$
\etaG^\star  = F_\reg \circ F_\loss(\etaG^\star) 
$$
Letting $\etaH^\star = F_\loss(\etaH^\star)\in[-\sqrt{c\tilde c}, \sqrt{c\tilde c}]$, we find that the pair $(\etaG^\star, \etaH^\star)$ satisfies the nonlinear system. Next, we show the uniqueness. Suppose $(\etaG, \etaH)$ and $(\tetaG, \tetaH)$ satisfy the system. 
Then, $\etaG$ and $\tetaG$ are the solution to the fixed-point equation $\eta = F_\reg\circ F_\loss(\eta)$:
$$
\etaG = F_\reg(\etaH) = F_\reg\circ F_\loss(\etaG), \quad \tetaG = F_\reg(\tetaH) = F_\reg\circ F_\loss(\tetaG).
$$

Suppose $\etaG\ne\tetaG$. This implies that the map $\eta\mapsto \eta-F_\reg\circ F_\loss(\eta)$ is constant over a nonempty closed interval of $[-1,1]$, and hence we can find some $\etaG^\star\in (-1,1)$ such that $1-(F_\reg\circ F_\loss)'(\etaG^\star) = 0$. Recalling the argument in the previous paragraph, we must have $F_\loss'(\etaG^\star) = c \frac{\alpha\tilde\alpha\tilde\nu}{\beta\tilde\beta\kappa}$, and this equality case gives 
$$
\E\Bigl[\env_{\tilde\loss}''\bigl(\tilde\alpha \etaG^\star G + \sqrt{1-(\etaG^\star)^2} G_0 + Z;\tilde\kappa\bigr) \bigl(\kappa^{-1} -  \env_{\loss}''(\alpha G+Z; \kappa) \bigr) \Bigr] = 0. 
$$
for independent standard normals $(G, G_0)$. Since $\env_f''(\cdot; \kappa) \in [0,\kappa^{-1}]$ for any convex function $f$, the integrand is non-negative with probability $1$. Thus, we must have 
$$
\iint_{-\infty}^{\infty} \env_{\tilde\loss}''\bigl(\tilde\alpha \etaG^\star x + \tilde\alpha \sqrt{1-(\etaG^\star)^2} y + Z;\tilde\kappa\bigr) \bigl( \kappa^{-1} - \env_\loss''(\alpha x + Z; \kappa) \bigr) \phi(x) \phi(y) dxdy = 0
$$
with probability $1$ with respect to $Z$, where $\phi$ is the pdf of $\cN(0,1)$. 
By the change of variable $(x, y) \mapsto (u, v) = (ax+Z, \tilde\alpha \etaG^\star x +\tilde\alpha \sqrt{1-(\etaG^\star)^2} y + Z)$, since $(u, v)$ also spans $\R\times \R$ thanks to $|\etaG^\star| < 1$, the previous display reads to
$$
\iint_{-\infty}^{\infty} \env_{\tilde\loss}''(v;\tilde\kappa) \bigl(\kappa^{-1} - \env_\loss''(u; \kappa) \bigr) \phi\Bigl(\frac{u-Z}{\alpha}\Bigr) \phi\Bigl(\frac{v-Z -\etaG^\star (u-Z)}{\alpha \sqrt{1-|\etaG^\star|^2}}\Bigr) du dv = 0. 
$$
Let us decouple the integration $\iint (\cdot) dudv$ with a lower bound of the form $\int(\cdot) du\cdot \int(\cdot) dv$.  Using the inequality
$
\exp(-(x+y)^2) \ge \exp(-2x^2 - 2y^2) 
$
for the density $\phi$ of $\cN(0,1)$, we find that there exists some positive (deterministic) constants $(c, C)$ that depend on $(\alpha, \etaG^\star)$ only such that 
$$
 \phi\Bigl(\frac{u-Z}{\alpha}\Bigr) \phi\Bigl(\frac{v-Z -\etaG^\star (u-Z)}{\alpha \sqrt{1-|\etaG^\star|^2}}\Bigr) \ge C \cdot \phi(c\cdot Z) \phi(c \cdot u) \phi(c\cdot  v)
$$
for all $u, v\in\R$ and any realization of $Z$. Combining this lower estimate with the previous display, we obtain
$$
\phi(c Z) \cdot \int_{-\infty}^\infty \env_{\tilde\loss}''(v;\tilde\kappa) \phi(c v) dv  \cdot \int_{-\infty}^\infty \bigl(\kappa^{-1} - \env_\loss''(u; \kappa) \bigr) \phi(c u) du = 0. 
$$
with probability $1$ with respect to $Z$. Since $\phi(c Z)$ is always strictly positive, we must have 
$$
\int_{-\infty}^\infty \env_{\tilde\loss}''(v;\tilde\kappa) \phi(c v) dv = 0 \quad \text{or} \quad \int_{-\infty}^\infty \bigl(\kappa^{-1} - \env_\loss''(u; \kappa) \bigr)  \phi(c u) du = 0.
$$
Suppose $\int_{-\infty}^\infty \env_{\tilde\loss}''(v;\tilde\kappa) \phi(c v) dv = 0$ holds. This implies $\int_{K} \env_{\tilde\loss}''(v;\tilde\kappa)  dv = 0$ for any compact set $K\subset \R$ due to $\min_{v\in K} \phi(c v) > 0$ and the non-negativity of $\env''$. Since $x\mapsto \env_{\tilde \loss}'(x;\kappa)$ is Lipschitz, it is absolutely continuous, and hence it has the integral form $\env_{\tilde \loss}'(x;\tilde\kappa) = \env_{\tilde\loss}'(0;\tilde\kappa) + \int_{0}^x \env_{\tilde\loss}''(v;\tilde \kappa) dv$. Since we know $\int_{0}^x \env_{\tilde\loss}''(v;\tilde \kappa) dv=0$, we find that $x\mapsto \env_{\tilde\loss}'(x;\tilde\kappa)$ is constant. However, this cannot happen since \eqref{eq:CGMT-1d} in  \Cref{sys:general_ensemble-M=1} then gives $\tilde \nu\tilde \alpha=0$: a contradiction with the positiveness of $(\tilde \alpha,\tilde \beta, \tilde \kappa, \tilde \nu)$. Therefore, the other case $\int_{-\infty}^\infty \bigl(\kappa^{-1} - \env_\loss''(u; \kappa) \bigr) \phi(c u) du = 0$ must hold. However, by the same argument integrating the constant derivative over $[0,x]$,
this in turn gives 
$
\env_\loss'(x;\kappa) = \env_\loss'(0;\kappa) + \kappa^{-1} x.
$
Noting $\env_\loss'(x;\kappa) = (x-\prox_\loss(x;\kappa))/\kappa$, this means that the map
$
x\mapsto \prox_\loss(x;\kappa)
$ is constant. Denoting by $p_0$ the constant, we have that for all real $u$, 
$$
\kappa^{-1}(u-p_0)= \kappa^{-1} (u-\prox_\loss(u;\kappa)) \in \partial \loss(p_0), 
$$
which means $\partial \loss(p_0) =\R$.
For any $p\in\R\setminus\{p_0\}$, the definition of the subdifferential gives
$\loss(p)\ge \loss(p_0) + \sup_{u\in \partial \loss(p_0)} u(p-p_0) = +\infty$.
This contradicts the assumption that $\loss(p)<+\infty$ at two distinct $p\in\R$.

 
\subsection{Proof of \Cref{thm:corr-sigerror-reserror}}

In this section, for a vector $\bw$, the norm $\| \bw \|$ indicates the $\ell_2$ norm unless specified otherwise. By the assumption $0\in \argmin_x \loss(x) \cap \argmin_x\reg(x)$, taking $\loss_{\text{new}}(x) = \loss(x)-\loss(0)$ and $\reg_{\text{new}}(x) = \reg(x)-\reg(0)$ if necessary, we assume without loss of generality that $\loss$ and $\reg$ are non-negative and have $0$ as a minimizer. 

By the change of variable $\bb \mapsto \bh = (\bb-\btheta)/\sqrt{p}$, denoting $\bG=\sqrt{p}\bX$ so that $\bG$ has i.i.d.\ $\cN(0,1)$ entries, the regularized M-estimator $\hat\btheta$ of interest and the residual vector $\by-\bX\hat\btheta$ can be written as 
$$
\hat\btheta=\sqrt{p} \hat\bh +  \btheta, \quad \by-\bX\hat\btheta = \bm\eps-\bG\hat\bh
$$
where 
\begin{align}\label{eq:def_h_original}
\hat\bh \in \argmin_{\bh\in\R^p}
\obj(\bh) \quad \text{with} \quad \obj(\bh) := 
\sum_{i\in I}  \loss(\eps_i-\bg_i^\top\bh) + \sum_{j \in [p]} \reg(\sqrt{p} h_j + 
\theta_{j}). 
\end{align}
Throughout this section, we denote
$
\bpsi = \sum_{i\in I} \be_i \loss'(z_i-\bg_i^\top\bh) \in \R^n
$
where $\be_i$ are canonical basis of $\R^n$. 
Let 
\( \tilde \bpsi, \tilde \btheta, \tilde \bh \) be the corresponding notation for
another. 
Then, our goal is to show 
$$
\hat \bh^\top\tilde\bh \pto \alpha\tilde\alpha \etaG
\quad \text{and} \quad 
\bpsi^\top\tilde\bpsi/p \pto \beta\tilde\beta \etaH.
$$

\subsubsection{Smoothing by adding diminishing ridge penalty}\label{sec:ridge_smoothing}
For some positive and diminishing scalar $\mu=\mu_n\to 0$ to be specified later, we define the smoothed regularized M-estimator $\hat\bh_{\mu}$ as
\begin{align}\label{eq:def_h_smoothed}
    \hat\bh_{\mu} \in \argmin_{\bh\in \R^p}  \obj(\bh) + \frac{p\mu}{2} \|\bh\|^2    
\end{align}
where $\obj(\bh)$ is the objective function for the original regularized M-estimator $\hat\bh$ in \eqref{eq:def_h_original}. We denote by $\bpsi_\mu = \sum_{i\in I} \be_i \loss'(z_i -\bg_i^\top\bh_\mu)$ the smoothed version of $\bpsi$. 
This strategy of adding and additive ridge penalty for the mathematical analysis as a first step, and then arguing by approximation $(\mu\to 0)$ to study the original $\hat\bh_0 = \hat\bh$ is ubiquitous; see, for instance, \cite{karoui2018impact,celentano2022fundamental}, \cite[Appendix B.3]{loureiro2022fluctuations}, \cite{bellec2023error}.
Here, we use the following lemma to show that \eqref{eq:def_h_smoothed} approximates well the original estimator for any sequence $\mu=\mu_n$ indexed by $n$ and converging to 0.

\begin{lemma}[Ridge smoothing]\label{lm:ridge_smoothing}
    Let $\bh_\mu$ be the smoothed M-estimator and $\hat{\bh}$ be the original M-estimator. 
    \begin{enumerate}
        \item (Monotonicity) We have $\|\hat\bh-\hat\bh_\mu\|_2^2 \le \|\hat\bh\|_2^2-\|\hat\bh_\mu\|_2^2$ for any $\mu>0$. 
        \item (Convergence in $\|\cdot\|_2$) As $n,p\to+\infty$ with $n/p\to\delta$, $|I|/n\to c$, and $\mu = \mu_n$ for any $\mu_n\to 0$, we have
    \begin{align*}
    \|\hat \bh_{\mu}\|_2^2\pto \alpha^2, \quad  \frac{\|\bpsi_{\mu}\|_2^2}{p} \pto \beta^2, 
 \quad 
    \|\hat\bh_{\mu}-\hat\bh\|_2^2 = \op(1), \quad  \frac{\|\bpsi_{\mu}-\bpsi\|_2^2}{p} = \op(1),
    \end{align*}
    where $\alpha$ and $\beta$ are solutions to \Cref{sys:general_ensemble-M=1}.
    \end{enumerate}
\end{lemma}

\begin{proof}
    By the strong convexity of the ridge term $p\mu^2/2 \|\bh\|^2$, the smoothed one $\hat\bh_\mu$ also minimizes the convex function:
    $
    \bh\mapsto  \obj(\bh) + \frac{p\mu }{2} \|\bh\|^2 - \frac{p\mu}{2}\|\bh-\hat\bh_\mu\|^2. 
    $
    By the optimality of $\hat\bh_\mu$ and $\hat\bh$, we have  
    $$
    \obj(\hat\bh_\mu) + \frac{p\mu}{2} \|\hat\bh_\mu\|^2 \le \obj(\hat\bh) + \frac{p\mu}{2} \|\hat\bh\|^2 - \frac{p\mu}{2}\|\hat\bh-\hat\bh_\mu\|^2 \le \obj(\hat\bh_\mu) + \frac{p\mu}{2} \|\hat\bh\|^2 - \frac{p\mu}{2}\|\hat\bh-\hat\bh_\mu\|^2
    $$
    so that subtracting $\obj(\hat\bh_\mu)$ from the both sides and dividing by $p\mu/2>0$, we obtain the first claim. 

    Recall $\|\bh\|^2\pto \alpha^2$ and $\|\bpsi\|^2/p\pto \beta^2$. 
    Then, by \Cref{lm:ridge_smoothing}-(1) and the Lipschitz condition of $\loss'$, it suffices to show the convergence $\|\hat\bh_\mu\|^2\pto \alpha^2$ for the smoothed estimator $\hat\bh_\mu$. 
    Note that $\hat\bh_\mu$ minimizes the function below
    $$
    \bh\mapsto \sum_{i\in I} \loss(\eps_i - \bg_i^\top\bh) + \sum_{j \in [p]} \reg_{j}^{\mu_n} (\sqrt{p} h_j) \quad  \text{where} \quad \reg_{j}^\mu (x) \coloneq \reg(x+\theta_j) + \mu  \frac{x^2}{2}. 
    $$
    Now we suppose that for any standard normal $\bg = (g_j)_{j=1}^p \sim  \cN(\bm {0}_p, \bI_p)$, the convergence of the Moreau envelope
    $$
    \frac{1}{p} \sum_{j \in [p]} \env_{\reg_{j}^{\mu_n}}(cg _j ; \tau) - \reg_j^{\mu_n} (0)  \pto \E[\env_{\reg}(cH + \Theta; \tau) - \reg(\Theta)]
    $$
    holds for all $c\in\R$ and $\tau>0$.
   Then, by \cite[Theorem 3.1]{thrampoulidis2018precise}, we have  $\|\hat\bh_{\mu}\|^2\pto \alpha^2$ and complete the proof.  By the weak law of large numbers, the above display holds with $\mu=0$ (see \cite[Lemma 4.1]{thrampoulidis2018precise} for details):
   $$
   \frac{1}{p} \sum_{j \in [p]} \env_{\reg_{j}^{\mu=0} }(c g_j ; \tau) - \reg_j^{\mu=0} (0) \pto \E[\env_{\reg}(cH + \Theta; \tau) - \reg(\Theta)]. 
   $$
   Then, noting $\reg_{j}^{\mu=0}(0) = \reg(\theta_j) = \reg_j^{\mu}(0)$ for any $\mu \ge 0$, it suffices to show 
   \begin{align*}
    \frac{1}{p} \sum_{j \in [p]} \env_{\reg_{j}^{\mu_n} }(cg _j ; \tau) - \frac{1}{p} \sum_{j \in [p]} \env_{\reg_{j}^{\mu=0} }(cg _j ; \tau) = \op(1). 
   \end{align*}
   For each $j$, the monotonicity $\env_{\reg_{j}^{\mu=0} }(cg _j ; \tau) \le \env_{\reg_{j}^{\mu_n} }(cg _j ; \tau)$ holds since the objective function is monotone in the sense of $\reg_j^{\mu=0} (x) \le \reg_j^{\mu}(x)$ for all $x\in\R$ and all $\mu\ge 0$. 
   On the other hand, by the optimality of $\prox_{\reg_j^{\mu=0}}(c g_j;\tau)$ and $\prox_{\reg_j^{\mu_n}}(c g_j;\tau)$, we find 
   \begin{align*}
   \env_{\reg_{j}^{\mu_n}}(cg _j ; \tau) &\le \frac{1}{2\tau} (cg_j-\prox_{\reg_j^{\mu=0}}(cg_j; \tau))^2 + \reg_j^{\mu_n} (\prox_{\reg_j^{\mu=0}}(cg_j; \tau))\\
   &= \env_{\reg_j^{\mu=0}}(cg_j; \tau) + \frac{\mu_n}{2} (\prox_{\reg_j^{\mu=0}}(cg_j; \tau))^2
   \end{align*}
   for each $j\in [p]$. 
   Thus, it holds that 
    $$
    0 \le \frac{1}{p} \sum_{j \in [p]} \env_{\reg_{j}^{\mu_n} }(cg _j ; \tau) - \frac{1}{p} \sum_{j \in [p]} \env_{\reg_{j}^{\mu=0} }(cg _j ; \tau) \le \frac{\mu_n}{2} \cdot \frac{1}{p}\sum_{j \in [p]} (\prox_{\reg_j^{\mu=0}}(cg_j; \tau))^2,
    $$
    where $\prox_{\reg_j^{\mu=0}}(cg_j; \tau) = \prox_{\reg}(c g_j + \theta_j; \tau) - \theta_j$ by the definition of $\reg_j^{\mu}$. 
    Here, under \Cref{asm:regularity-conditions}-(1), the expectation $\E[(\prox_{\reg}(c g_j + \theta_j; \tau) - \theta_j)^2]$ under $g_j\sim \cN(0,1)$ and $\theta_j\sim \Theta$ is finite for all $c\in\R$ and $\tau>9$ (cf. \cite[equation (123)]{thrampoulidis2018precise}).
    Therefore, by the weak law of large numbers, we have $\frac{1}{p}\sum_{j \in [p]} (\prox_{\reg_j^{\mu=0}}(cg_j; \tau))^2 \pto \E[(\prox_{\reg}(c H + \Theta; \tau) - \Theta)^2] <+\infty$. 
    This means that for any $\mu=\mu_n\to 0$, the RHS of the previous display is $\op(1)$. 
\end{proof}

\subsubsection{Bounding norm of regularized M-estimators}\label{sec:bound_norm}
For some positive scalar $K$ to be specified later, we add another regularization term; we define $\hat\bh_{\mu, K}$ as
\begin{align*}
    \hat\bh_{\mu, K} &\in \argmin_{\bh\in \R^p}  \obj(\bh) + \frac{p\mu}{2} \|\bh\|^2 + \underbrace{\frac{\hat{\lambda}}{2} \mathsf{F} \Bigl(\frac{\|\bh\|^2 - K}{2}\Bigr)}_{\text{additional term}} \quad \text{with} \quad \quad \hat\lambda := \obj(\bm{0})
\end{align*}
where $\mathsf{F}:\R\to\R$ is convex, non-negative, and non-decreasing with $\lim_{x\to+\infty} \mathsf{F}(x) = +\infty$, as well as differentiable with
$\mathsf{F}'(u) = 0$ if $u\le 0$ and $\mathsf{F}'(u) = 1$ if $u\ge 1$. 
For instance, we may take $\mathsf{F}$ as an integral of the smoothed step function as follows:
$$
\mathsf{F} (x):= \int_{-\infty}^x f(u) \, \mathrm{d}u \quad \text{with} \quad f(u) := \begin{cases}
    1 & u\ge 1\\
    3u^2 - 2u &  u \in (0,1)\\
    0 & u \le 0
\end{cases}
$$
Note in passing that the regularization parameter $\hat{\lambda} = \obj(\bm 0) = \sum_{i\in I} \loss(\eps_i) + \sum_{j \in [p]} \reg (\theta_j)$ is independent of the design matrix $\bG$ and non-negative since $\loss$ and $\reg$ are non-negative. 

Now we claim that for sufficiently large $K>0$ this modified estimator $\hat{\bh}_{\mu, K}$ coincides with the smoothed estimator $\hat\bh_\mu$. 
Furthermore, thanks to the additional regularizer $\frac{\hat{\lambda}}{2} \mathsf{F} \Bigl(\frac{\|\bh\|^2 - K}{2}\Bigr)$, the norm of $\bh_{\mu, K}$ and $\bpsi_{\mu, K}$ are suitably bounded as follows:
\begin{lemma}
\label{lm:bound_h_psi}
The following convergences hold.
\begin{enumerate}
    \item If we set $K=2\alpha^2$ where $\alpha>0$ is the solution to \Cref{sys:general_ensemble-M=1}, we have 
    $$
    \PP(\hat{\bh}_{\mu, K}= \hat{\bh}_{\mu}) \pto 1.
    $$
    \item For any $K\ge 0$, there exists a positive constant $C_K$ that only depends on $K$ such that 
    \begin{align*}
    \|\hat\bh_{\mu, K} \|^2 &\le C_K, \quad \|\bpsi_{\mu,K}\|^2 \le C_K (1+\|\loss'\|_{\lip}^2) (\|\loss'(\bm\eps)\|^2 + \|\bG\|_{\oper}^2).
    \end{align*}
    Thus, the constant
    $C^*=
    1.1 C_K (1+\|\loss'\|_{\lip}^2) (\E[\loss'(Z)^2] + (1+\delta^{-1/2}))^2$
    satisfies
    $$
    \PP\Bigl(
    \|\hat\bh_{\mu, K}\|^{2} \vee (n^{-1}\|\bpsi_{\mu, K} \|^{2})
    \vee \sup_{m\ge 1}\E\Bigl[\|\hat\bh_{\mu, K}\|^{2m} \vee 
    (n^{-1}\|\bpsi_{\mu, K} \|^{2})^m
    \mid \bz, \btheta\Bigr]^{1/m} \le C^*\Bigr) \to 1.
    $$
\end{enumerate}
\end{lemma}
\begin{proof}
    Let us consider the event $\Omega \coloneq \{\|\bh\|^2 \le K\}$ with $K=2\alpha^2$,  
   which holds with high probability, since $\|\bh\|^2\pto \alpha^2>0$. 
    Combined with the monotonicity $\|\bh_\mu\|^2 \le \|\bh\|^2$ by \Cref{lm:ridge_smoothing}-(1), we have $\|\hat\bh_\mu\|^2 \le K$ under the event $\Omega$. 
    Since $\mathsf{F}'(u) = 0$ for all $u\le 0$, combined with the Karush-Kuhn-Tucker (KKT) condition $-p \mu\hat\bh_{\mu} \in \partial\obj(\hat\bh_{\mu})$ for the smoothed estimator $\hat\bh_{\mu}$, we observe that $\hat\bh_\mu$ satisfies the KKT conditions for $\hat\bh_{\mu, K}$ under the event $\Omega$:
    $$
    -p \mu\hat\bh_\mu  - \hat\lambda \mathsf{F}'\Bigl(\frac{\|{\hat\bh_\mu}\|^2 - K}{2}\Bigr) \hat\bh_\mu =  -p \mu\hat\bh_\mu  - 0\cdot \hat\bh_\mu  \in \partial \mathsf{obj}(\hat\bh_\mu).
    $$
    This implies $\PP(\hat\bh_{\mu}=\hat\bh_{\mu, K}) \ge \PP(\Omega) \to 1$. 
    
    By the non-negativity of $(\obj(\cdot), p\mu\|\cdot\|^2, \frac{\hat{\lambda}}{2} \mathsf{F}(\cdot))$ and the optimality of $\hat\bh_{\mu, K}$, it holds that 
    \begin{align*}
    0 \le \frac{\hat{\lambda}}{2} \mathsf{F} \Bigl(\frac{\|\hat\bh_{\mu, K}\|^2 - K}{2}\Bigr) &\le \obj(\hat{\bh}_{\mu, K}) + \frac{p\mu}{2} \|\bh_{\mu, K}\|^2 + \frac{\hat{\lambda}}{2} \mathsf{F} \Bigl(\frac{\|\bh_{\mu, K}\|^2 - K}{2}\Bigr)\\
    &\le \obj(\bm{0}) + \frac{p\mu}{2} \|\bm 0\|^2 + \frac{\hat{\lambda}}{2} \mathsf{F} \Bigl(\frac{\|\bm 0\|^2 - K}{2}\Bigr)\\
    &= \hat{\lambda} + 0 + 0.         
    \end{align*}
    When $\hat\lambda=0$ then all inequalities above holds with equality, which means that $\bm{0}$ minimizes the objective function $\obj(\bh) + \frac{p\mu}{2} \|\bh\|^2 +\frac{\hat{\lambda}}{2} \mathsf{F} \Bigl(\frac{\|\bh\|^2 - K}{2}\Bigr)$ for $\hat\bh_{\mu, K}$. Since this objective function is strongly convex due to the ridge term, the minimizer is unique. This means  $\hat{\bh}_{\mu, K}=\bm{0}$ when $\hat\lambda=0$. On the other hand, if $\hat\lambda>0$ then dividing the above display by $\hat\lambda>0$ we are left with 
    $\mathsf{F} \Bigl(\frac{\|\hat\bh_{\mu, K}\|^2 - K}{2}\Bigr)\le  2$. Since $\mathsf{F}$ is non-decreasing and coercive on the positive side, i.e., $\lim_{x\to +\infty} \mathsf{F}(x) \to + \infty$, this gives $\|\hat{\bh}_{\mu, K}\|^2 \le C(K)$ for a constant $C(K)>0$ depending on $K$ only. 
    Combined with the Lipschitz condition of $\loss'$, the norm of $\bpsi_{\mu, K}= \sum_{i\in I} \be_i \loss'(z_i-\bg_i^\top\bh_{\mu, K})$ is bounded as
    $$
    \|\bpsi_{\mu, K}\| \le \|\loss'(\bm\eps)\| + \|\loss'(\bm\eps - \bG\hat\bh_{\mu, K}) - \loss'(\bm\eps)\| \le \|\loss'(\bm\eps)\| + \|\loss'\|_{\lip} \|\bG\|_{\oper} \|\hat\bh_{\mu, K}\|.
    $$
    Since $\loss'(z_i)^2$ has a finite second moment by \Cref{asm:regularity-conditions}-(1), the weak law of large numbers gives 
    $n^{-1} \|\loss'(\bz)\|^2 \pto \E[\loss'(Z)^2] <+\infty$.
    Since $\bG\in\R^{n\times p}$ is a Gaussian matrix with i.i.d.\ $\cN(0,1)$ entries, we also have $\|\bG\|_{op}^2/n \pto (1+\delta^{-1/2})^2$ by standard results of the maximal singular value of a Gaussian matrix,
    e.g., \cite[Theorem II.13]{davidson2001local}.
    This completes the proof.
\end{proof}

\subsubsection{Derivative formulae for strongly convex regularizer}

\bigskip

\begin{lemma}[\cite{bellec2022derivatives}]\label{lm:derivative}
Let $\hat\bh\in\R^p$ be the regularized M-estimator of the form
$$
\hat \bh \in \argmin_{\bh\in\R^p} \sum_{i\in I} \loss(z_i-\bg_i^\top\bh) + \mathsf{R}(\bh),
$$
where $I\subset [n]$ is a subset independent of $(\bg_i)_{i\in[n]}$, 
$\loss:\R\to\R$ is a convex and differentiable function with Lipschitz derivative, 
$\mathsf{R}:\R^{p}\to\R\cup\{+\infty\}$ is a $\tau$-strongly convex regularizer for some $\tau>0$, and $\mathsf{R}$ is independent of $(\bg_i)$. Then, there exists a matrix $\bA\in\R^{p\times p}$ satisfying $\|\bA\|_{\oper} \le \tau^{-1}$ and $\tr[\bA] \ge 0$
such that $\hat\bh\in\R^p$ and $\bpsi=\sum_{i\in I} \be_i \loss'(z_i-\bg_i^\top \hat \bh)\in\R^n$ are both differentiable as functions of the design $(\bG) = (g_{ij})$ with the derivative given by 
\begin{equation}\label{eq:derivative_formula}
   \forall i\in [n], \forall j\in[p], \quad  \frac{\partial \hat\bh}{\partial g_{ij}} = \bA \Bigl(\be_j\be_i^\top \bpsi  - \bG^{\top}\bD\be_i \be_j^\top \hat\bh \Bigr), \quad \frac{\partial \bpsi}{\partial g_{ij}} = - \bD\bG \bA \be_j\be_i^\top \bpsi - \bV \be_i \be_j^\top \hat\bh.
\end{equation}
where $\bD$ and $\bV$ are $n\times n$ matrices defined by $\bD = \sum_{i\in I} \be_i\be_i^\top \loss''(z_i-\bg_i^\top \hat\bh)$ and $\bV = \bD - \bD\bG\bA\bG^\top \bD$.
Furthermore, the matrix $\bV$ defined above is positive semidefinite with its operator norm bounded as $\|\bV\|_{\oper}\le \|\loss'\|_{\lip}$. 
\end{lemma}

Recall that the modified estimator $\hat\bh_{\mu, K}$ in \Cref{sec:bound_norm} minimizes the objective function $\sum_{i\in I} \loss(z_i-\bg_i^\top\bh) + \mathsf{R}(\bh)$ where $\mathsf{R}$ is the regularizer of the form
$$
\mathsf{R}(\bh) := \sum_{j \in [p]} \reg (\sqrt{p} h_j + \theta_j ) + \frac{p\mu}{2} \|\bh\|^2 + \frac{\hat{\lambda}}{2} \mathsf{F} \Bigl(\frac{\|\bh\|^2 - K}{2}\Bigr),
$$
which is $(p\mu)$-strongly convex. Thus, we can apply \Cref{lm:derivative}; there exists a matrix $\bA_{\mu, K}\in\R^{p\times p}$ such that
\begin{equation}
\|\bA_{\mu, K}\|_{\oper} \le (p\mu)^{-1}, \quad \tr[\bA_{\mu, K}]\ge 0,
\label{bound-A}
\end{equation}
and $\hat \bh_{\mu, K}$ and $\bpsi_{\mu, K} = \sum_{i\in I} \be_i \loss'(z_i -\bg_i^\top\hat\bh_{\mu, K}) $ are differentiable with respect to the design matrix $\bG=(g_{ij})$ as in \eqref{eq:derivative_formula} with 
$$
\bD_{\mu, K} = \sum_{i\in I} \be_i\be_i^\top \loss''(z_i-\bg_i^\top\hat\bh_{\mu,K}) \qquad \bV_{\mu, K}=\bD_{\mu, K} - \bD_{\mu, K} \bG\bA_{\mu, K} \bG^\top \bD_{\mu, K}.
$$
Now we claim that the trace of $\bV_{\mu, K}$ and $\bA_{\mu, K}$ are empirical quantities that converge to the remaining solution $\nu$ and $\kappa$:
\begin{lemma}\label{lm:convergence_trace}
For any $\mu\to 0$ such that $\mu^{-1} = O(n^{1/8})$, we have 
    $$
    \tr[\bV_{\mu, K}]/p \pto \nu 
    \quad \text{and} \quad
    \tr[\bA_{\mu, K}]\pto \kappa. 
    $$
\end{lemma}
\begin{proof}
    See \Cref{subsec:proof_convergence_trace}.
\end{proof}

\subsubsection{Proof of \Cref{thm:corr-sigerror-reserror}}
\label{sec:proof-thm:corr-sigerror-reserror}

Below we will take the diminishing constant $\mu_n\to 0$  as in \Cref{lm:convergence_trace}. 
Dropping the dependence on $(\mu, K)$ for simplicity, we denote $\bh=\hat\bh_{\mu, K}$, $\bpsi=\bpsi_{\mu, K}$, and $\bV=\bV_{\mu, K}$, $\bA = \bA_{\mu, K}$. 
In the same way, we use the notation $(\tilde\bh, \tilde\bpsi, \tilde\bV, \tilde\bA)$ for the other estimator. 
Then, by the convergence in $\|\cdot\|_2$ from \Cref{lm:ridge_smoothing} and \Cref{lm:bound_h_psi}, it suffices to show the convergence of correlation for the modified ones:
$$
    \bh^\top\tilde\bh/(\alpha\tilde\alpha) \pto \etaG 
    \quad \text{and} \quad 
    \bpsi^\top\tilde\bpsi/(p\beta\tilde\beta)\pto \etaH.
$$
First, we will argue by the Gaussian Poincar\'e inequality that $\bh^\top\tilde\bh/(\alpha\tilde\alpha)$ and $\bpsi^\top\tilde\bpsi/(p\beta\tilde\beta)$ concentrate on random quantities that are independent of the scaled design matrix $\bG=\sqrt{p}\bX$. Throughout this section, we denote by $\bar\E[\cdot]=\E[\cdot|\bz, \btheta, I, \tilde I]$ the conditional expectation with respect to the design matrix $\bG$ given signal $\btheta$, noise $\bz$ and subsample index $I, \tilde I$.
Since $\bG$ is independent of $(\btheta,\bz,I,\tilde I)$, the conditional distribution of $\bG$ given $(\btheta,\bz,I,\tilde I)$ is still that of a matrix with i.i.d.\ $\cN(0,1)$ entries.
\begin{lemma}
    \label{lm:gaussian_Poincare_hpsi}
    For any $\mu_n>0$ such that $\mu_n^{-1} = o(n^{1/2})$, we have 
    \begin{align*}
        \bar\E\bigl[(\bh^\top\tilde\bh -\bar\E[\bh^\top\tilde\bh])^2\bigr] & = \Op(n^{-1}\mu^{-2}) = \op(1). \\
        \bar\E\bigl[(\bpsi^\top\tilde\bpsi -\bar\E[\bpsi^\top\tilde\bpsi])^2\bigr] &= \Op(n\mu^{-2}) = \op(n^2).
    \end{align*}
\end{lemma}
\begin{proof}
By the Gaussian Poincar\'e inequality, noting $\partial_{ij} (\bh^\top\tilde\bh) = \tilde\bh^\top \partial_{ij}\bh + \bh^\top\partial_{ij} \tilde \bh$ with $\partial_{ij} = \frac{\partial}{\partial g_{ij}}$, the conditional variance of $\bh^\top\tilde\bh$ is bounded from above as 
\begin{align*}
    \bar\E\Bigl[(\bh^\top\tilde\bh-\bar\E[\bh^\top\tilde\bh])^2\Bigr] \le \sum_{ij} \bar\E\Bigl[(\partial_{ij} (\bh^\top\tilde\bh))^2\Bigr] \le 2 \sum_{ij} \bar\E\Bigl[(\tilde\bh^\top \partial_{ij} \bh)^2\Bigr]  + 2\sum_{ij} \bar\E\Bigl[(\bh^\top \partial_{ij} \tilde\bh)^2\Bigr],
\end{align*} 
where we denote by $\sum_{ij}=\sum_{i=1}^n\sum_{j=1}^p$ for brevity.
By the derivative formula \eqref{eq:derivative_formula}, the first term on the RHS is bounded as
\begin{align*}
\sum_{ij} \Bigl(\tilde\bh^\top \partial_{ij} \bh\Bigr)^2  &= \sum_{i, j} \Bigl(\tilde\bh^\top \bA(\be_j \psi_i - \bG^\top\bD\be_i h_j)\Bigr)^2\\
    &\le 2 \|\bA\|_{\oper}^2 \|\tilde \bh\|^2 \|\bpsi\|^2 + 2 \|\bA\|_{\oper}^2 \|\bG\|_{\oper}^2 \|\bD\|_{\oper}^2 \|\tilde\bh\|^2 \|\bh\|^2. 
\end{align*}
Using the upper bound $\|\bA\|_{\oper}\le (p\mu_n)^{-1}$ from \eqref{bound-A} and the moment bound in \Cref{lm:bound_h_psi}-(2), the conditional expectation (with respect to $\bar\E$) of the RHS is $\Op(n^{-1} \mu_n^{-2})$.
By symmetry, we also get  $\bar\E[\sum_{ij} \bar\E[(\bh^\top \partial_{ij} \tilde\bh)^2]] = \Op(n^{-1}\mu_n^{-2})$. 

We use a similar argument for the variance of  $\bpsi^\top\tilde\bpsi$. 
The 
Gaussian Poincar\'e gives 
$\bar\E[(\bpsi^\top\tilde\bpsi-\bar\E[\bpsi^\top\tilde\bpsi])^2]\le 2 \sum_{ij} \bar\E[(\tilde\bpsi^\top \partial_{ij} \bpsi)^2 + (\bpsi^\top \partial_{ij}\tilde \bpsi)^2]$, 
where
\begin{align*}
\sum_{ij} \Bigl(\tilde\bpsi^\top \partial_{ij} \bpsi\Bigr)^2 &= \sum_{ij} \Bigl(\tilde\bpsi^\top (- \bD\bG\bA \be_j\psi_i - \bV \be_i\bw_j) \Bigr)^2\\
    &\lesssim \|\bA^\top\bG^\top\bD^\top\|_{\oper}^2 \|\tilde\bpsi\|^2\|\bpsi\|^2 + \|\bV\|_{\oper}^2 \|\tilde\bpsi\|^2 \|\bh\|^2.
\end{align*}
by the derivative formula. 
Using $\|\bA\|_{\oper} \le (p\mu_n)^{-1}$
and the moment bound in \Cref{lm:bound_h_psi}-(2), the conditional expectation $\bar\E[\cdot]$ of the RHS is $\Op(n \mu_n^{-2} + n) = \Op(n\mu_n^{-2})$. 
\end{proof}

Let us define $\hetaG$ and $\hetaH$ as the ``truncated'' values of the conditional expectation of $\bh^\top\tilde\bh$ and $\bpsi^\top\tilde\bpsi$, respectively:  
\begin{align*}
\hetaG := \Pi_{[-1,1]}\bigl({\bar\E[\bh^\top\tilde\bh]}/({\alpha\tilde\alpha})\bigr), \quad \hetaH := \Pi_{[-1,1]}\bigl({\bar\E[\bpsi^\top\tilde\bpsi]}/({p\beta\tilde\beta})\bigr) 
\end{align*}
where $\Pi_{[-1,1]}(x) := \max(\min(x, 1), -1)$ is the projection map onto $[-1,1]$. Here, we emphasize that $(\hetaH, \hetaG)$ is independent of the design matrix $\bG$ since $\bG$ is integrated out. Furthermore, the absolute values of $\hat\etaG$ and $\hat\etaH$ are less than $1$ due to the truncation, and hence the two  matrices 
$$\begin{pmatrix}
    1 & \hetaH\\
    \hetaH & 1
    \end{pmatrix} \text{ and } \begin{pmatrix}
    1 & \hetaG\\
    \hetaG & 1
\end{pmatrix}$$
are both positive semidefinite.
By the concentration from \Cref{lm:gaussian_Poincare_hpsi} and the convergence $\|\bh\|^2\pto \alpha^2>0, \|\bpsi\|^2/p\pto \beta^2>0$, noting the truncation $x\mapsto \Pi_{[-1,1]}(x)$ is continuous, we find that the random variables $(\hetaH, \hetaG)$ defined above still capture the correlations, that is,
\begin{align*}
\hetaH &= \Pi_{[-1,1]} (\frac{\bpsi^\top\tilde\bpsi}{\|\bpsi\|\|\tilde\bpsi\|}) + \op(1) = \frac{\bpsi^\top\tilde\bpsi}{\|\bpsi\|\|\tilde\bpsi\|} + \op(1).\\
\hetaG &= \Pi_{[-1,1]} (\frac{\bh^\top\tilde\bh}{\|\bh\|\|\tilde\bh\|}) + \op(1) = \frac{\bh^\top\tilde\bh}{\|\bh\|\|\tilde\bh\|} + \op(1). 
\end{align*}
where the second equation follows from the fact that the correlations are less than $1$ in absolute values by the Cauchy--Schwarz inequality.

Now, we use the multivariate normal approximation below to {motivate} \Cref{sys:general_ensemble-M=infty}. 
\begin{lemma}
    [Proposition 5.1 in \cite{bellec2024asymptotics}]
    \label{lm:approx_multi_normal}
    Let $\bz\sim \cN(\bm{0}_q, \bI_q)$ and let $\bF:\R^q\to\R^{q\times M}$ be a locally Lipschitz function with $M\le q$. 
    Then there exists a standard normal vector $\bw\sim \cN(\bm{0}_M, \bI_M)$ such that
    $$
    \E\Bigl[\bigl\|\bF(\bz)^\top \bz - \sum_{l \in [q]}  \frac{\partial \bF(\bz)^\top \be_l}{\partial z_l} - \Bigl\{\bF(\bz)^\top \bF(\bz)\Bigr\}^{1/2} \bw
    \bigr\|^2\Bigr]\le \C \sum_{l \in [q]} \E\Bigl[\bigl\| \frac{\partial \bF(z)}{\partial z_l}\bigr\|_F^2\Bigr],
    $$
    where $\{\cdot\}^{1/2}$ is the square root of the positive semidefinite matrix. 
\end{lemma}
For each $j\in [p]$, applying \Cref{lm:approx_multi_normal} with $\bF = \begin{bmatrix}\frac{\bpsi}{\sqrt{p}\beta} & \frac{\tilde\bpsi}{\sqrt{p}\tilde\beta}\end{bmatrix} \in \R^{n\times 2}$ and $\bz=\bG \be_j \in \R^{n}$, using the derivative formula \eqref{eq:derivative_formula}, 
we find that there exists a random vector $\bw_j\in\R^2 $ such that 
$$\bw_j \mid (\btheta, \bm\epsilon, \bG^{-j}, I, \tilde I)\deq \cN(\bm{0}_2, \bI_2)$$ and 
\begin{align*}
&\bar\E\Bigl[
\Bigl\|
\frac{1}{\sqrt{p}} \begin{pmatrix}(\bpsi^\top \bG\be_j^\top + \tr[\bV]h_j + \bpsi^\top\bD\bG\bA\be_j)/\beta\\
(\tilde\bpsi^\top \bG\be_j^\top + \tr[\tilde\bV]\tilde h_j + \tilde\bpsi^\top\tilde\bD\bG\tilde\bA\be_j)/
\tilde\beta
\end{pmatrix}
-
\begin{pmatrix}
    \|\bpsi\|^2/(p\beta^2) & \bpsi^\top\tilde\bpsi/(p\beta\tilde\beta)\\
    \bpsi^\top\tilde\bpsi/(p\beta\tilde\beta) & \|\tilde\bpsi\|^2/(p \tilde\beta^2)
\end{pmatrix}^{1/2} \bw_j
\Bigr\|^2 
\Bigr] \\
&\le \C \sum_{i \in [n]} \bar\E\Bigl[\frac{1}{p \beta^2}\|\frac{\partial \bpsi}{\partial g_{ij}}\|^2 + \frac{1}{p \tilde\beta^2}\|\frac{\partial \tilde\bpsi}{\partial g_{ij}}\|^2\Bigr].    
\end{align*}
Using %
$\|\bA\|_{\oper}\le (p \mu_n)^{-1}$ from \eqref{bound-A} and 
\Cref{lm:bound_h_psi}, we have 
\begin{align*}
    \bar\E\Bigl[\|\bpsi^\top\bD\bG\bA\|^2\Bigr] &= \Op(\mu^{-2}), \\
    \sum_{ij} \bar\E\Bigl[\|\frac{\partial\bpsi}{\partial g_{ij}}\|^2\Bigr] \le \bar \E\Bigl[2 \|\bD\bG\bA\|_{\frob}^2 \|\bpsi\|^2 + 2\|\bV\|_{\frob}^2\|\bh\|^2\Bigr] &= \Op(n \mu^{-2}). 
\end{align*}
Thus, summing over $j\in [p]$ in the previous display, noting $\mu^{-2} = o(n)$, we are left with 
$$
\sum_{j \in [p]} 
\Bigl\|
\frac{1}{\sqrt{p}} \begin{pmatrix}(\bpsi^\top \bG\be_j^\top + \tr[\bV]h_j)/\beta\\
(\tilde\bpsi^\top \bG\be_j^\top + \tr[\tilde\bV]h_j)/\tilde\beta
\end{pmatrix}
-
\begin{pmatrix}
    \|\bpsi\|^2/(p \beta^2) & \bpsi^\top\tilde\bpsi/(p \beta \tilde\beta)\\
    \bpsi^\top\tilde\bpsi/(p \beta\tilde\beta) & \|\tilde\bpsi\|^2 /(p \tilde\beta^2)
\end{pmatrix}^{1/2} \bw_j
\Bigr\|^2 = \op(n).
$$
By the $(1/2)$-H\"{o}lder continuity for the matrix square root: $\|\bM^{1/2}-\bN^{1/2}\|_{\oper} \le \|\bM-\bN\|_{\oper}^{1/2}$ for positive semidefinite matrices $\bM, \bN$, combined with the convergence $\|\bpsi\|^2/p\to\beta^2$, $\|\tilde \bpsi\|^2/p\to\tilde\beta^2$ and the concentration $\hetaH = \bpsi^\top\tilde\bpsi/(p \beta\tilde\beta) + \op(1)$, we have
\begin{align*}
\Bigl\|\begin{pmatrix}
    |\|\bpsi\|^2/(p\beta^2) & \bpsi^\top\tilde\bpsi/(p\beta \tilde\beta) \\
    \bpsi^\top\tilde\bpsi/(p \beta\tilde\beta) & \|\tilde\bpsi\|^2 /(p \tilde\beta^2)
\end{pmatrix}^{1/2}  - \begin{pmatrix}
   1 & \hetaH\\
   \hetaH & 1
\end{pmatrix}^{1/2}\Bigr\|_{\oper} = \op(1).
\end{align*}
Combined with the previous display, we get 
\begin{align*}
    &\sum_{j \in [p]} \Bigl\|\frac{1}{\sqrt{p}}
\begin{pmatrix}(\bpsi^\top \bG\be_j^\top + \tr[\bV] h_j)/\beta\\
(\tilde\bpsi^\top \bG\be_j^\top + \tr[\tilde\bV] \tilde{h}_j)/\tilde\beta
\end{pmatrix}
- \begin{pmatrix}
    1&\hetaH\\
   \hetaH &1
\end{pmatrix}^{1/2}
\bw_j
\Bigr\|^2 =  \op(n) + \op(1) \sum_{j \in [p]} \|\bw_j\|^2 
= \op(n),
\end{align*}
where the last equation follows from 
$\sum_{j\in [p]} \bar\E[\|\bw_j\|^2] = 2p$, which follows from the fact that the marginal distribution of $\bw_j$ given $(\bz, \btheta, I, \tilde I)$ is $\cN(0, 1)$ (note, however, that we do not establish or take for granted
that $(\bw_j)_{j\in[p]}$ are i.i.d.).
Using $\tr[\bV]/p\pto \nu$ (\Cref{lm:convergence_trace}) and $\|\bh\|^2=\Op(1)$ from \Cref{lm:bound_h_psi}, we can 
 also replace $\tr[\bV]h_j$ by $p \kappa h_j$. As a consequence, we get 
\begin{align}
\sum_{j \in [p]} \Bigl\|
\begin{pmatrix}\frac{1}{\sqrt{p}\beta} (\bpsi^\top \bG\be_j^\top + p\nu h_j) \\
\frac{1}{\sqrt{p}\tilde\beta} (\tilde\bpsi^\top \bG\be_j^\top + p\tilde \nu \tilde h_j)
\end{pmatrix}
-  \begin{pmatrix}
   \hat w_j \\
    \tilde w_j
\end{pmatrix}
\Bigr\|^2
&= \op(n) \text{ where } \begin{pmatrix}
    \hat{w}_j \\
    \tilde{w}_j
\end{pmatrix} \coloneq \begin{pmatrix}
    1 & \hetaH\\
    \hetaH & 1
\end{pmatrix}^{1/2} \bw_j.
\label{prev-display}
\end{align}
Here, we emphasize that the conditional distribution of $(\hat{w}_j, \tilde{w}_j)$ is given by
$$
    \begin{pmatrix}
        \hat{w}_j \\
        \tilde{w}_j
    \end{pmatrix}
    \mid 
    (\bm\eps, \btheta, \bG^{-j}, I, \tilde I) \deq \cN\Bigl(\bm{0}_2, \begin{pmatrix}
        1 & \hetaH\\
        \hetaH & 1
    \end{pmatrix}\Bigr)
$$
since the conditional distribution of $\bw_j$ given $(\bz, \btheta, \bG^{-j}, I, \tilde I)$ is standard normal $\cN(0_2, I_2)$ and $\hetaH$ is $\sigma(\bz, \btheta, I, \tilde I)$-measurable. For each $j\in [p]$, let us define $\Xi_j$ and $\tilde\Xi_j$ as 
\begin{align*}
    \Xi_j = 
    \frac{\bpsi^\top\bG \be_j^\top + p\nu h_j}{\sqrt{p}\beta} - \hat{w}_j , \quad \tilde\Xi_j = 
    \frac{\tilde{\bpsi}^\top\bG \be_j^\top + p\tilde\nu \tilde{h}_j}{\sqrt{p}\tilde\beta} - \tilde{w}_j
\end{align*}
so that the bound \eqref{prev-display} reads $\sum_{j \in [p]} [\Xi_j^2 + \tilde\Xi_j^2] = \op(n)$. By the KKT conditions  $\bG^\top\bpsi \in \sqrt{p}\partial \reg(\sqrt{p}\bh +  \btheta)$ and  $\bG^\top\tilde\bpsi \in \sqrt{p}\partial \tilde \reg(\sqrt{p}\tilde\bh +  \btheta)$ for $\bh$ and $\tilde\bh$, respectively, $\sqrt{p}h_j$ and $\sqrt{p}\tilde{h}_j$ can be written as 
$$
    \begin{pmatrix}
        \sqrt{p} h_j\\
        \sqrt{p} \tilde{h}_j
    \end{pmatrix} = \begin{pmatrix}
        \prox_{\reg}( \theta_j + \frac{\beta}{\nu}(\hat w_j + \Xi_j) ; \nu^{-1}) - \theta_j\\
        \prox_{\tilde \reg}( \theta_j +\frac{\tilde\beta}{\tilde\nu} (\tilde{w}_j +   \tilde \Xi_j) ; \tilde{\nu}^{-1}) - \theta_j\\
    \end{pmatrix}.
$$
Now we claim that the bound $\sum_{j \in [p]} \Xi_j^2 = \op(n)$ also holds in the conditional expectation $\bar\E$, i.e., $\bar\E[\sum_{j \in [p]} \Xi_j^2] = \op(n)$. 
To this end, it suffices to show $\bar\E[(p^{-1}\sum_{j \in [p]} \Xi_j^2)^2] \le C$ with high probability for a constant $C$. Using the upper estimate
\begin{align*}
0 \le \frac{1}{p}\sum_{j \in [p]} \Xi_j^2 &\le 3 \Bigl(\frac{\|\bG^\top\bpsi\|^2}{\beta^2p^2} +  \frac{\nu^2}{\beta^2} \|\bh\|^2 + \frac{1}{p} \sum_{j \in [p]} (\hat{w}_j)^2 \Bigr)
\end{align*}
and the moment bound form \Cref{lm:bound_h_psi}, we get
$$
    \bar\E\Bigl[\Bigl(\frac{1}p \sum_{j \in [p]} \Xi_j^2\Bigr)^2\Bigr] \le \C \Bigl(1 + \bar\E\Bigl[\Bigl(\frac{1}{p}\sum_{j \in [p]} (\hat{w}_j)^2\Bigr)^2\Bigr]\Bigr)
$$
with high probability. 
For the second term, expanding the square of the summation,
\begin{align*}
    \bar\E\Bigl[\Bigl(p^{-1}\sum_{j \in [p]} (\hat{w}_j)^2\Bigr)^2\Bigr] &= \frac{1}{p^2}\sum_{i, j} \bar\E [(\hat{w}_i)^2 (\hat{w}_j)^2] \le \frac{1}{p^2}\sum_{i, j} \sqrt{\bar\E [(\hat{w}_i)^4]} \sqrt{\bar\E[ (\hat{w}_j)^4]}=6.
\end{align*}
where we have used $\bar\E[ (\hat{w}_j)^4] = 6$ for all $j$, which follows from the fact that the marginal law of $\hat{w}_j$ is $\cN(0,1)$. 
This gives $\bar\E[(\frac{1}p \sum_{j \in [p]} \Xi_j^2)^2]\le C$ with high probability for a constant $C$, and hence we get the estimate $\bar\E[\frac{1}p \sum_{j \in [p]} \Xi_j^2] = \op(1)$. The same argument yields $\bar\E[\frac{1}p \sum_{j \in [p]} \tilde \Xi_j^2] = \op(1)$.

Combined with the proximal representation of $(\sqrt{p} h_j, \sqrt{p}\tilde h_j)$ using $(\Xi_j, \tilde\Xi_j)$, since $\prox_{f}(\cdot)$ is $1$-Lipschitz for any convex function $f$, the upper bounds of $\bar\E[\frac{1}p \sum_{j \in [p]} \Xi_j^2]$ and $\bar\E[\frac{1}p \sum_{j \in [p]} \tilde \Xi_j^2]$ yield the following simple proximal approximation of $(\sqrt{p} h_j, \sqrt{p}\tilde h_j)$:
$$
    \frac{1}{p} \sum_{j \in [p]} \bar\E \Bigl[\Bigl\|
    \begin{pmatrix}
    \sqrt{p} h_j\\
    \sqrt{p} \tilde{h}_j
    \end{pmatrix}-
    \begin{pmatrix}
        \prox_{\reg}(\theta_j + (\beta/\nu) \hat{w}_j; \nu^{-1}) - \theta_j\\
        \prox_{\tilde \reg}(\theta_j + (\tilde\beta/\tilde \nu) \tilde{w}_j; \tilde\nu^{-1}) - \theta_j
    \end{pmatrix}
    \Bigr\|^2 \Bigr] = \op(1).
$$
Noting $\bar\E[\|\bh\|^2]=\Op(1)$, this lets us approximate
$\bar\E[\bh^\top\tilde\bh]/\alpha\tilde\alpha$ by the inner product of proximal operators:
\begin{align*}
\frac{\bar\E[\bh^\top\tilde\bh]}{\alpha\tilde\alpha}
+\op(1)
&= \frac{1}{p}\sum_{j \in [p]} \frac{1}{\alpha\tilde\alpha} \bar\E\Bigl[\Bigl(\prox_{\reg}(\theta_j + (\beta/\nu) \hat{w}_j; \nu^{-1}) -\theta_j\Bigr) \Bigl(\prox_{\tilde \reg} (\theta_j + (\tilde\beta/\tilde\nu) \tilde{w}_j; \tilde \nu^{-1}) - \theta_j\Bigr)\Bigr]
\\&=
    \frac{1}{p}\sum_{j \in [p]} {F}_\reg(\hetaH; \theta_j),
\end{align*}
where for the second inequality, we used the fact that the marginal law of $(\hat w_j, \tilde w_j)$ given $(\btheta, \bz, I,\tilde I)$ is jointly Gaussian, with zero mean, unit variance, and correlation $\hetaH$, and 
where $F_\reg (\cdot; \theta_j):[-1,1]\to \R$ is the function defined by
\begin{align*}
   F_\reg(\eta; \theta_j) &= 
   \iint
 \frac{\bigl( \prox_{\reg}(\theta_j + \frac{\beta}{\nu} x; \frac{1}{\nu}) -\theta_j\bigr) \bigl(\prox_{\tilde\reg} (\theta_j + \frac{\tilde\beta}{\tilde\nu} (\eta x + \sqrt{1-\eta^2} y); \frac{1}{\tilde \nu}) - \theta_j\bigr)}{\alpha\tilde\alpha}
 \varphi(x) \varphi(y) \, \mathrm{d}x \, \mathrm{d}y
\end{align*}
with $\varphi$ being the standard normal probability function and $\iint
=\int_{-\infty}^\infty \int_{-\infty}^\infty$.
Notice that for each $\eta\in [-1,1]$, the sequence $(F_\reg (\eta; \theta_j))_{j=1}^p$ are i.i.d.\ random variables with mean $\E[F_\reg(\eta;\theta_j)] = F_\reg(\eta)$. Furthermore, by Jensen's inequality and the Cauchy--Schwarz inequality, the expectation of the absolute value is finite:
\begin{align*}
\E\Bigl[|F_\reg(\eta; \theta_j)|\Bigr] & \le \frac{1}{\alpha\tilde\alpha} \E\Bigl[|\bigl(\prox_\reg(\Theta + \tfrac{\beta}{\nu} H; \tfrac1\nu)-\Theta\bigr)\cdot \bigl(\prox_{\tilde \reg}(\Theta + \tfrac{\tilde \beta}{\tilde \nu} \tilde H; \tfrac1{\tilde{\nu}})-\Theta\bigr) |\Bigr]\\
&\le \frac{1}{\alpha\tilde \alpha} \E\Bigl[\bigl(\prox_\reg(\Theta + \tfrac{\beta}{\nu} H; \tfrac1\nu)-\Theta\bigr)^2\Bigr]^{1/2} \cdot \E\Bigl[\bigl(\prox_{\tilde \reg}(\Theta + \tfrac{\tilde \beta}{\tilde \nu} \tilde H; \tfrac1{\tilde{\nu}})-\Theta\bigr)^2\Bigr]^{1/2} \\
&= 1 \qquad (\text{by \eqref{eq:CGMT-1a} in  \Cref{sys:general_ensemble-M=1}}).
\end{align*}
Thus, the weak law of large numbers gives $p^{-1} \sum_{j\in[p]}  F_\reg (\eta; \theta_j)\pto F_\reg(\eta)$ for each $\eta\in [-1,1]$. Next, we claim that this convergence holds uniformly over $\eta\in [-1,1]$. 
\begin{lemma}\label{lm:uniform_convergence_monotone}
    Let $I$ be a bounded and closed interval of $\RR$ and 
   let $(f_n)_{n\ge 1} : I \to \RR$ be a sequence of non-decreasing functions such that $f_n(x) \pto f(x)$ pointwise for some function $f : I \to \RR$. 
   If $f$ is non-decreasing and continuous, then the uniform convergence $\sup_{x\in I} |f_n(x) - f(x)|\pto 0$ holds. 
\end{lemma}
Note that \Cref{lm:uniform_convergence_monotone} is a probabilistic analogue of a similar statement for deterministic functions: if a sequence of real-valued monotone functions (on $\RR$) converges pointwise to a continuous function on a compact set $I \subset \RR$, then the convergence is uniform on the set $I$.
The probabilistic version is known but in a more general setting, so to keep the treatment self-contained we give a basic proof below which is similar to proof of the Glivenko-Cantelli theorem (cf. \cite[Theorem 19.1]{van2000asymptotic}). 
\begin{proof}
    Let us write $I=[a, b]$.
    Since $f$ is continuous and $I$ is compact, $f$ is uniformly continuous on $I$.
    For any $\epsilon>0$, there exists some $\delta_\epsilon>0$ such that $|f(x)-f(y)|<\epsilon/2$ for all $x, y\in I$ such that $|x-y|\le \delta_\epsilon$. Now for sufficiently large integer $k=k_\epsilon\in\mathbb{N}$ such that $(b-a)/k < \delta_\epsilon$, let us take equally spaced grids $(x_i)_{i=0}^k$ over $[a, b]$ such that $a = x_0 < x_1 <\cdots < x_k = b$ and $(x_i-x_{i-1}) = (b-a)/k$ for all $i\in\{0, \dots k\}$. Let $\Omega_\epsilon = \cap_{i=0}^k \{|f_n(x_i)-f(x_i)|\le \epsilon/2\}$ be the event under which $f_n$ and $f$ are sufficiently close at the finite grids. Note that this event holds with probability converging to $1$, thanks to the pointwise convergence $f_n(x_i)\pto f(x_i)$ at finitely many $x_i$.
    Then, since $f_n$ is non-decreasing while $f$ does not move more than $\epsilon/2$ in $[x_{i-1}, x_{i}]$,
    for all $i\in \{0, 1, \dots, k\}$ and for all $x\in [x_{i-1}, x_i]$ we have
    \begin{align*}
    f_n(x_{i-1})  &\le f_n(x) \le f_n(x_{i}),
    \\
    - \epsilon/2 - f(x_{i-1}) &\le -f(x) \le - f(x_{i}) + \epsilon/2.
    \end{align*}
    In the event $\Omega_\epsilon$, summing the two lines it holds that 
    $-\epsilon   \le f_n(x) - f(x) \le \epsilon$ 
    for all $i\in \{0, 1, \dots, k\}$ and for all $x\in [x_{i-1}, x_i]$.
\end{proof}
Recall that we have shown in the proof of \Cref{th:existence-uniqueness-sys:general_ensemble-M=infty} that $F_\reg$ is differentiable with a non-negative derivative. By the same argument, the map $\eta \mapsto F_\reg(\eta; \theta_j)$ is differentiable with a non-negative derivative. Thus, applying \Cref{lm:uniform_convergence_monotone} with $f_n(\cdot) = p^{-1} \sum_{j \in [p]}  F_{\reg}(\cdot; \theta_j)$, $f(\cdot) = F_\reg(\cdot)$ and $I=[-1,1]$, we get the uniform convergence:
$$
    \sup_{\eta\in [-1,1]} \Big|\frac{1}{p} \sum_{j \in [p]} F_{\reg}(\eta; \theta_j)-F_\reg(\eta)\Big|\pto 0.
$$
Combined with $\frac{\bar\E[\bh^\top\tilde\bh]}{\alpha\tilde\alpha} =  \frac{1}{p}\sum_{j \in [p]} \hat{F}_\reg(\hetaH; \theta_j) + \op(1)$ and $\hetaH\in [-1,1]$, we are left with 
$$
    \frac{\bar\E[\bh^\top\tilde\bh]}{\alpha\tilde\alpha} = \frac{1}{p}\sum_{j \in [p]} {F}_\reg(\hetaH; \theta_j) + \op(1) = F_\reg(\hetaH) + \op(1). 
$$
Recall the definition $\hetaG = \Pi_{[-1,1]}(\frac{\bar\E[\bh^\top\tilde\bh]}{\alpha\tilde\alpha})$ where $\Pi_{[-1,1]}$ is the projection onto $[-1,1]$. By the continuity of the projection map and the bound $\sup_{\eta\in [-1,1]}|F_\reg(\eta)|\le 1$ (see \Cref{th:existence-uniqueness-sys:general_ensemble-M=infty}), the above display yields
$$
    \hetaG = \Pi_{[-1,1]} (F_\reg(\hetaH)) + \op(1) =  F_\reg(\hetaH) + \op(1). 
$$

Next, let us show $\hetaH = F_\loss(\hetaG) + \op(1)$ using the same argument. 
Using \Cref{lm:approx_multi_normal} with $\bF = [\bh/\alpha \mid \tilde\bh/\tilde\alpha] \in \R^{p\times 2}$ and $\bz = \bg_i$, there exists a random vector $\bu_i\in\R^2$ with conditional distribution
$$
    \bu_i \mid \bm\eps, \btheta, \bG_{-i}, I, \tilde I \deq \cN(\bm{0}_2, \bI_2)
$$
and
\begin{multline}
    \label{previous-ineq}
\bar \E\Bigl[
    \Bigl\|
    \begin{pmatrix}
        (\bg_i^\top\bh - \tr[\bA]\psi_i + \bh^\top \bA\bG^\top \bD \be_i)/\alpha\\
        (\bg_i^\top\tilde\bh  - \tr[\tilde\bA]\tilde\psi_i + \tilde\bh^\top \tilde \bA\bG^\top \tilde\bD \be_i)/\tilde\alpha
        \end{pmatrix} - \begin{pmatrix}
        \|\bh\|^2/\alpha^2 & \bh^\top\tilde{\bh}/\alpha\tilde\alpha\\
        \tilde{\bh}^\top\bh/\alpha\tilde\alpha & \|\tilde\bh\|^2/\tilde\alpha^2
    \end{pmatrix}^{1/2}\bu_i
    \Bigr\|^2
\Bigr] \\
\le \C \sum_{j=1}^p \bar\E \frac{1}{\alpha^2} \Bigl\|\frac{\partial \bh }{\partial g_{ij}}\Bigr\|^2 + \frac{1}{\tilde\alpha^2} \Bigl\|\frac{\partial \tilde \bh}{\partial g_{ij}}\Bigr\|^2.
\end{multline}
Using %
$\|\bA\|_{\oper}\le (p\mu)^{-1}$ in \eqref{bound-A}
\Cref{lm:bound_h_psi}, noting $\mu^{-2} = o(n)$, it follows that 
\begin{align*}
    \bar\E\bigl[\|\bh \bA\bG^\top\bD\|^2\bigr]& = \Op(n^{-1}\mu^{-2}) = \op(1)\\
    \sum_{i\in[n]}\sum_{j\in[p]} \bar\E\Bigl[ \|\frac{\partial \bh}{\partial g_{ij}}\|^2\Bigr] 
\le 2 \bar\E\Bigl[\|\bA\|_{\frob}^2 \|\bpsi\|^2 + \|\bA\bG^\top\bD\|_F^2 \|\bh\|^2\Bigr] &=  \Op(\mu^{-2}) = \op(n)
\end{align*}
so that summing over $i\in [n]$ the inequality \eqref{previous-ineq}, we get 
$$
    \sum_{i \in [n]} 
    \Bigl\|
    \begin{pmatrix}
        (\bg_i^\top\bh - \tr[\bA]\psi_i)/\alpha\\
        (\bg_i^\top\tilde\bh  - \tr[\tilde\bA]\tilde\psi_i)/\tilde\alpha
    \end{pmatrix} - \begin{pmatrix}
        \|\bh\|^2/\alpha^2 & \bh^\top\tilde{\bh}/\alpha\tilde\alpha\\
        \tilde{\bh}^\top\bh/\alpha\tilde\alpha & \|\tilde\bh\|^2/\tilde\alpha^2
    \end{pmatrix}^{1/2} \bu_i
    \Bigr\|^2
    = \op(n). 
$$
Appealing to the $(1/2)$-H\"{o}lder continuity for the matrix square root again, now using the two convergences $\|\bh\|^2\to \alpha^2, \|\tilde\bh\|^2\pto \tilde\alpha^2$ and the concentration $\hetaG = \bh^\top\tilde\bh/(\alpha\tilde\alpha) +  \op(1)$, we get 
\begin{align*}
    \Bigl\|\begin{pmatrix}
    \|\bh\|^2/\alpha^2 & \bh^\top\tilde{\bh}/\alpha\tilde\alpha\\
    \tilde{\bh}^\top\bh/\alpha\tilde\alpha & \|\tilde\bh\|^2/\tilde\alpha^2
\end{pmatrix}^{1/2} 
-\begin{pmatrix}
    1 & \hetaG\\
    \hetaG & 1
\end{pmatrix}^{1/2}
\Bigr\|_{\oper} = \op(1). 
\end{align*}
Combined with the convergence $\tr[\bA]\pto\kappa, \tr[\tilde\bA] \pto \tilde\kappa$, noting that $\sum_{i \in [n]} \|\bu_i\|^2$ and $\|\bpsi\|^2 + \|\tilde\bpsi\|^2$ are both $\Op(n)$, we are left with
\begin{align*}
\frac{1}{n} \sum_{i \in [n]} \Bigl\|
    \begin{pmatrix}
        \bg_i^\top\bh - \kappa \psi_i\\
        \bg_i^\top\tilde\bh - \tilde\kappa \tilde\psi_i
    \end{pmatrix} -  \begin{pmatrix}
        \alpha \hat u_i\\
        \tilde \alpha \tilde u_i
    \end{pmatrix}
    \Bigr\|^2 = \op(1) \quad \text{where} \quad \begin{pmatrix}
    \hat{u}_i\\
    \tilde{u}_i
\end{pmatrix} = \begin{pmatrix}
    1 & \hetaG\\
    \hetaG & 1
\end{pmatrix}^{1/2} \bu_i.
\end{align*}
By the same argument that we used to bound $\Xi_j$ and $\tilde\Xi_j$, using 
\Cref{lm:bound_h_psi} and the fact that the marginal law of $\hat u_i$ (and $\tilde u_i)$ is $\cN(0,1)$, 
we can show that the conditional expectation $\bar\E$ of the square of LHS is bounded from above by a constant $C$ with high probability. Thus, the above approximation also holds in the conditional expectation $\bar\E$:
$$
    \frac{1}{n} \sum_{i \in [n]} \bar\E \Bigl[\Bigl\|
    \begin{pmatrix}
        \bg_i^\top\bh - \kappa \psi_i\\
        \bg_i^\top\tilde\bh - \tilde\kappa \tilde\psi_i
    \end{pmatrix} -  \begin{pmatrix}
        \alpha \hat u_i\\
        \tilde \alpha \tilde u_i
    \end{pmatrix}
    \Bigr\|^2 \Bigr] = \op(1).
$$
Let us define $\Xi^{i} := \bg_i^\top\bh - \kappa \psi_i-\alpha\hat u_i$ and $\tilde \Xi^{i} := \bg_i^\top\tilde\bh-\tilde\kappa \tilde\psi_i - \tilde\alpha\tilde u_i$ so that the above display reads $\sum_{i \in [n]} \bar\E[(\Xi^{i})^2 + (\tilde\Xi^{i})^2] = o(n)$. Since $\psi_i = \loss'(z_i - \bg_i^\top\bh)$ and $\tilde\psi_i = \tilde \loss'(\tilde z_i -\bg_i^\top\tilde\bh)$ for all $i\in I\cap \tilde I$, the residuals can be written as 
$$
    z_i - \bg_i^\top\bh = \prox_{\loss}(z_i - \alpha \hat u_i -\Xi^{i}; \kappa), \quad z_i - \bg_i^\top\tilde\bh = \prox_{\tilde \loss}(z_i - \tilde \alpha\tilde u_i -\tilde \Xi^{i}; \tilde \kappa)
$$
for all $i\in I\cap \tilde I$. Since the map $x\mapsto \env_{f}'(x;\tau)=f'\circ \prox_{f} (x; \tau )$ is a composition of Lipschitz functions if $f$ is convex and differentiable with Lipschitz derivative, the moment bound $\sum_{i \in [n]} \bar\E[(\Xi^{i})^2 + (\tilde\Xi^{i})^2] = o(n)$ lets us approximate $\psi_i$ and $\tilde\psi_i$ by $\env_{\loss}'$ and $\env_{\tilde \loss}'$ as follows:
$$
    \frac{1}{n} \sum_{i\in I\cap \tilde I} \bar\E \Bigl\|
    \begin{pmatrix}
    \psi_i - \env_{\loss}'(\eps_i - \alpha \hat{u}_i;\kappa) \\
    \tilde\psi_i - \env_{\tilde \loss}'(\eps_i - \tilde \alpha \tilde{u}_i;\tilde \kappa)
    \end{pmatrix}
    \Bigr\|^2 = \op(1) \quad \text{with} \quad \begin{pmatrix}
        \hat u_i\\
        \tilde u_i
    \end{pmatrix} | \bz, \btheta, \bG_{-i}, I, \tilde I \deq \cN(0_2, \begin{pmatrix}
        1 & \hetaG\\
        \hetaG & 1
    \end{pmatrix})
$$
Noting $\bar\E[\|\bpsi\|^2] = \Op(n)$ by \Cref{lm:bound_h_psi}, this lets us approximate $\bar\E[\bpsi^\top\tilde\bpsi]/(p\beta\tilde\beta)$ by the inner product of $(\env_{\loss}', \env_{\tilde\loss}')$:
\begin{align*}
\frac{\bar\E[\bpsi^\top\tilde\bpsi]}{p\beta\tilde\beta} &= \frac{|I\cap\tilde I|}{p\beta\tilde\beta} \frac{1}{|I\cap\tilde I|} \bar\E[\bpsi^\top\tilde\bpsi]\\
&=\frac{c\tilde c\delta}{\beta\tilde\beta} \frac{1}{|I\cap\tilde I|} \sum_{i\in I\cap\tilde I} \bar\E[\env_{\loss}'(\eps_i - \alpha \hat{u}_i;\kappa) \cdot \env_{\tilde \loss}'(\eps_i - \tilde \alpha \tilde{u}_i;\tilde \kappa)] + \op(1).
\end{align*}
Since the marginal distribution of $(u_i, \tilde u_i)$ given $(\btheta, \bz, I, \tilde I)$ is centered normal with unit variance and correlation $\hetaG$, the above display reads
$$
    \frac{\bar\E[\bpsi^\top\tilde\bpsi]}{p\beta\tilde\beta} = \frac{1}{|I\cap \tilde I|}\sum_{i\in I\cap \tilde I}  F_\loss(\hetaG;\eps_i) + \op(1)
$$
where $F_\loss(\eta;\eps_i): [-1,1]\to\R$ is the function defined as:
$$
    {F}_\loss(\eta; \eps_i) = \frac{c\tilde c\delta}{\beta\tilde\beta} \int_{-\infty}^{\infty} \int_{-\infty}^{\infty} \varphi(x)\varphi(y) \env_{\loss}'(\eps_i + \alpha x ;\kappa) \env_{\tilde \loss}'(\eps_i + \tilde \alpha (\eta x + \sqrt{1-\eta^2} y);\tilde \kappa) \, \mathrm{d}x \, \mathrm{d}y.
$$
By the same argument for $F_\reg (\eta;\theta_j)$, 
the sequence $(F_\loss(\eta; z_i))_{i \in [n]}$ are i.i.d.\ random variables with mean $\E[F_\loss(\eta;z_i)]=F_\loss(\eta)$ and the expectation of the absolute value $|F_\loss(\eta; z_i)|$ is finite. 
Thus, by the weak law of large numbers, we have
$
\frac{1}{m} \sum_{i \in [m]} F_\loss(\eta; z_i) \pto \E[F_\loss(\eta; z_1)] = F_\loss(\eta)
$
for any deterministic integer $m=m_n\to+\infty$. 
Then, we have that 
for any $\epsilon>0$, 
denoting by $\sum_K$ the sum over all possible value $K$ taken by $I\cap \tilde I$,
\begin{align*}
    &\PP\Bigl(|\frac{1}{|I\cap \tilde I|} \sum_{i\in I\cap \tilde I} F_\loss(\eta; z_i)-F_\loss(\eta)| >\epsilon\Bigr) \\
    &=
    \sum_K\PP\Bigl(|\frac{1}{|K|} \sum_{i\in K} F_\loss(\eta; z_i)-F_\loss(\eta)| >\epsilon,
    ~
    I\cap \tilde I = K\Bigr)
    && (\text{by additivity of disjoint events}) \\
    &=\sum_K
    \PP\Bigl(|\frac{1}{|K|} \sum_{i\in K} F_\loss(\eta; z_i)-F_\loss(\eta)| >\epsilon\Bigr)
    \PP(I\cap \tilde I = K)
    && (\text{by independence})
    \\
    &=\sum_{m\ge 0}
    \PP\Bigl( | \frac1m \sum_{i \in [M]} F_\loss(\eta; z_i)-F_\loss(\eta)| >\epsilon\Bigr)
    \PP(|I\cap \tilde I| = m)
    && (\text{since $(z_i)_{i\in K} \deq (z_i)_{i\in[m]}$ for $m=|K|$}).
\end{align*}
We now split the sum over $m\ge 0$ into two as follows:
\begin{align*}
    &\PP\Bigl(|\frac{1}{|I\cap \tilde I|} \sum_{i\in I\cap \tilde I} F_\loss(\eta; z_i)-F_\loss(\eta)| >\epsilon\Bigr) 
    \\
    &=\Bigl(\sum_{m \le nc\tilde c/2} + \sum_{m > nc\tilde c/2}\Bigr) 
    \PP\Bigl( | \frac1m \sum_{i \in [M]} F_\loss(\eta; z_i)-F_\loss(\eta)| >\epsilon\Bigr)
    \PP\Bigl(|I\cap \tilde I| = m\Bigr)
    \\
    &\le
    \PP(|I\cap \tilde I| \le n c\tilde c/2)
    + \sup_{m> nc\tilde c/2}
    \PP\Bigl( | \frac1m \sum_{i \in [M]} F_\loss(\eta; z_i)-F_\loss(\eta)| >\epsilon\Bigr)
    \PP(|I\cap \tilde I| > nc\tilde c/2 )
\end{align*}
where $\PP(|I\cap \tilde I| > nc\tilde c/2 ) \le 1$ in the rightmost term.
The second term converges to 0 by the weak law of large numbers,
and the first by Chebyshev's inequality applied to the hypergeometric distribution.

This gives $\frac{1}{|I\cap \tilde I|} \sum_{i\in I\cap \tilde I} F_\loss(\eta; z_i)\pto F_\loss(\eta)$ pointwise for any $\eta\in[-1,1]$. 
By the same argument we used for $F_\reg$ and $F_\reg(\cdot;\theta_j)$, $F_\loss$ and $F_\loss(\cdot;z_i)$ are non-decreasing and continuous. Thus, we can apply \Cref{lm:uniform_convergence_monotone} with $f_n (\cdot) = \frac{1}{|I\cap \tilde I|} \sum_{i\in I\cap \tilde I} F_\loss(\cdot; z_i)$, $f(\cdot) = F_\loss(\cdot)$ and obtain the uniform convergence: 
$$
\sup_{\eta \in [-1,1]} \Big|\frac{1}{|I\cap \tilde I|} \sum_{i\in I\cap \tilde I} F_\loss(\eta; z_i) - F_\loss(\eta)\Big| = \op(1). 
$$
Combined with $\frac{\bar\E[\bpsi^\top\tilde\bpsi]}{p\beta\tilde\beta} = \frac{1}{|I\cap \tilde I|}\sum_{i\in I\cap \tilde I}  F_\loss(\hetaG;\eps_i)$ and $\hetaG\in [-1,1]$, we are left with
$$
{\bar\E[\bpsi^\top\tilde\bpsi]}/(p\beta\tilde\beta) = F_\loss(\hetaG) + \op(1). 
$$
Recalling $\hetaH = \Pi_{[-1,1]}({\bar\E[\bpsi^\top\tilde\bpsi]}/(p\beta\tilde\beta))$ where $\Pi_{[-1,1]}$ is the projection onto $[-1,1]$, by the continuity of the projection map and $|F_\loss (\eta)|\le \sqrt{c\tilde c}\le 1$ for all $\eta\in[-1,1]$ (see \Cref{th:existence-uniqueness-sys:general_ensemble-M=infty}), we finally obtain
$$
\hetaH = \Pi_{[-1,1]}(F_\loss(\hetaG)) + \op(1) = F_\loss(\hetaG) + \op(1). 
$$

In summary, we have shown that $\hetaG = F_\reg(\hetaH) + \op(1)$ and $\hetaH = F_\loss(\hetaG) + \op(1)$. 
By the continuity of $F_\reg$ and $F_\loss$, it holds that 
$$
\hetaG = F_\reg\circ F_\loss(\hetaG) + \op(1)
\quad \text{and} \quad 
\hetaH = F_\loss \circ F_\reg(\hetaH) + \op(1). 
$$
Let $T:[-1,1]\to\R$ be the map $T: \eta \mapsto \eta -F_\reg\circ F_\loss(\eta)$ so that the above result (for $\hetaG$) reads to $T(\hetaG) = \op(1)$.
Since $T$ is continuous in $[-1,1]$ and has a unique root $\etaG$ by \Cref{th:existence-uniqueness-sys:general_ensemble-M=infty},
for any $\epsilon >0$ let $C(\epsilon) = \min_{u\in[-1,1]\setminus[\etaGstar-\epsilon, \etaGstar+\epsilon]}T(u)$ and note $C(\epsilon)>0$ by compactness.
Hence
$$
\PP(|\hetaG - \etaGstar| \ge \epsilon) \le
\PP(T(\hetaG) \ge C(\epsilon)) \to 0
$$
due to $T(\hetaG) = \op(1)$.
By the same argument, combining $\hetaH = F_\loss \circ F_\reg(\hetaH) + \op(1)$ and the fact that the map $\eta\mapsto \eta -  F_\loss \circ F_\reg(\eta)$ is
continuous and has a unique root $\etaH$ by \Cref{th:existence-uniqueness-sys:general_ensemble-M=infty}, we obtain $\hetaH\pto\etaH$. 

Recalling $\hetaG = \bh^\top\tilde\bh/(\alpha\tilde\alpha) + \op(1)$ and $\hetaH = \bpsi^\top\tilde\bpsi/(p\beta\tilde\beta) + \op(1)$, this gives 
$\bh^\top\tilde\bh/(\alpha\tilde\alpha) \pto \etaG$ and $\bpsi^\top\tilde\bpsi/(p\beta\tilde\beta) \pto \etaH$, thereby completing the proof. 

Finally, let us show the proximal approximation of estimators $(\sqrt{p}h_j, \sqrt{p}\tilde h_j)$ and residuals $(z_i-\bg_i^\top\bh, z_i-\bg_i^\top\tilde\bh)$. 
Now that we have established $\bpsi^\top\tilde\bpsi/(p\beta\tilde\beta) \pto \etaH$ and  $\bh^\top\tilde\bh/(\alpha,\tilde\alpha)\to^p \etaG$, using the above argument again with $\hetaH$ replaced by $\etaH$ and $\hetaG$ by $\etaG$, we obtain the following result. 
{
\begin{corollary}\label{cor:conditional_prox_rep}
Denote by $\bar{\E}[\cdot|\bz, \btheta, I, \tilde I]$ be the conditional expectation. 
Then, we have the following (conditional) proximal representation of the error vectors and residuals:
\begin{align*}
\frac{1}{p} \sum_{j \in [p]} \bar\E \Bigl[
\Bigl\|
\begin{pmatrix}
\sqrt{p} h_j + \theta_j\\
\sqrt{p} \tilde{h}_j +\theta_j 
\end{pmatrix}-
\begin{pmatrix}
    \prox_{\reg}(\theta_j + (\beta/\nu) H_j; \nu^{-1}) \\
    \prox_{\tilde \reg}(\theta_j + (\tilde\beta/\tilde \nu) \tilde{H}_j; \tilde\nu^{-1})
\end{pmatrix}
\Bigr\|^2
\Bigr] = \op(1), \\
\frac{1}{|I\cap\tilde I|} \sum_{i\in I\cap \tilde I} \bar{\E}\Bigl[\Bigl\|
\begin{pmatrix}
z_i - \bg_i^\top\bh\\
z_i - \bg_i^\top \tilde\bh 
\end{pmatrix}-
\begin{pmatrix}
\prox_{\loss}(z_i + \alpha G_i;\kappa)\\
\prox_{\tilde\loss}(z_i + \alpha\tilde{G}_i;\tilde \kappa)
\end{pmatrix}\Bigr\|^2
\Bigr] = \op(1), 
\end{align*}
where $(H_j, \tilde{H}_j)_{j\in[p]}$ and $(G_{i}, \tilde{G}_i)_{i\in I\cap \tilde I}$ are jointly normals such that:
\begin{align*}
\forall j\in [p] \quad 
\begin{pmatrix}
       H_j\\
        \tilde H_j
    \end{pmatrix} | \bz, \btheta, I, \tilde I \deq \cN\biggl(\begin{bmatrix}
        0\\
        0
    \end{bmatrix}, \begin{bmatrix}
        1 & \etaH\\
        \etaH & 1
    \end{bmatrix}\biggr), \\
\forall i\in I\cap \tilde I, \quad 
\begin{pmatrix}
       G_i\\
        \tilde G_i
    \end{pmatrix} | \bz, \btheta, I, \tilde I \deq \cN \biggl(\begin{bmatrix}
        0\\
        0
    \end{bmatrix}, \begin{bmatrix}
        1 & \etaG\\
        \etaG & 1
    \end{bmatrix} \biggr).   
\end{align*}
\end{corollary}
}
Since $|X_n|=\op(1)$ is equivalent to $\E[1\wedge |X_n|]=o(1)$ for any random variable $X_n$, the conditional proximal representation of the error vectors $(\bh, \tilde\bh)$ in \Cref{cor:conditional_prox_rep} reads
$$
\E\Bigl[1 \wedge \frac{1}{p} \sum_{j \in [p]} \bar\E [\mathsf{A}_j]\Bigr]  = o(1) \ \text{where} \  \mathsf{A}_j := \Bigl\|
\begin{pmatrix}
\sqrt{p} h_j + \theta_j\\
\sqrt{p} \tilde{h}_j +\theta_j 
\end{pmatrix}-
\begin{pmatrix}
    \prox_{\reg}(\theta_j + (\beta/\nu) H_j; \nu^{-1}) \\
    \prox_{\tilde \reg}(\theta_j + (\tilde\beta/\tilde \nu) \tilde{H}_j; \tilde\nu^{-1})
\end{pmatrix}
\Bigr\|^2
$$
Here the expectation $\E$ is with respect to $(\bz, \btheta, I, \tilde I)$ (recall the definition $\bar\E[\cdot] = \E[\cdot|\bz,\btheta, I, \tilde I])$. 
Note that the integrand is bounded from below as 
\begin{align*}
1 \wedge \frac{1}{p} \sum_{j \in [p]} \bar\E [\mathsf{A}_j] \ge \frac{1}{p}\sum_{j\in[p]} \Bigl(1\wedge \bar\E[\mathsf{A}_j]\Bigr) \ge \frac{1}{p} \sum_{j\in [p]} \bar\E[1\wedge \mathsf{A}_j], 
\end{align*}
where the second inequality follows from 
the Jensen's inequality $\bar\E[f(X)] \le f(\bar\E[X])$ applied with the concave function $f(x)=1\wedge x$ and $X=\mathsf{A}_j$. Taking the expectation of the above display and using the tower property, we are left with
$$
\frac{1}{p} \sum_{j\in [p]} \E[1\wedge \mathsf{A}_j] = o(1). 
$$
Since the marginal distribution of the integrand $(1\wedge \mathsf{A}_j)$ is the same for all $j\in[p]$ by symmetry, the LHS equals to $\E[1\wedge \mathsf{A}_{j'}]$ for any $j' \in [p]$. This gives $\max_{j\in [p]} \E[1 \wedge \mathsf{A}_j]=o(1)$ and completes the proof of the desired joint approximation of estimators $(\sqrt{p}h_j, \sqrt{p}\tilde h_j)$. The joint approximation of residual follows from the same argument, so we omit the proof.

{
\subsection{Proof of \Cref{theorem:average_prox_rep}}\label{proof:theorem:average_prox_rep}
The proof strategy is the same as \Cref{sec:proof-thm:corr-sigerror-reserror}. We first claim the average $p^{-1}\sum_{j\in[p]} \Phi (\cdots)$ concentrates on its conditional expectation:
\begin{align}\label{eq:concentration_Phi}
    \frac{1}{p} \sum_{j\in[p]}  \Phi \begin{pmatrix}
\sqrt{p} h_j + \theta_j\\
\sqrt{p} \tilde{h}_j +\theta_j\\
\theta_j
\end{pmatrix} =\frac{1}{p} \sum_{j\in[p]} \bar{\E} \Phi \begin{pmatrix}
\sqrt{p} h_j + \theta_j\\
\sqrt{p} \tilde{h}_j +\theta_j\\
\theta_j
\end{pmatrix} + \op(1), 
\end{align}
where recall $\bar{\E}[\cdot] = \E[\cdot|\btheta, \bz, I, \tilde I]$. 
Note that the marginal distribution of $\btheta = (\theta_j)_{j\in[p]}$ is bounded in the second moment by the additional assumption we imposed for this theorem. Furthermore, \Cref{lm:bound_h_psi}-(2) implies that there exists a deterministic constant $C$ which depends on 
$$(\delta, \|\loss'\|_{\lip}, \|\tilde\loss'\|_{\lip}, \alpha, \tilde\alpha, \E[\loss'(Z)^2], \E[\tilde\loss'(Z)^2])$$ only such that $\|\bh\|^2 \vee \|\tilde\bh\|^2\le C$ with probability $1$. Then, combined with \Cref{cor:conditional_prox_rep}, using the pseudo-Lipschitz property of the test function $\Phi$, we have 
$$
 \frac{1}{p} \sum_{j\in[p]} \bar{\E} \Phi \begin{pmatrix}
\sqrt{p} h_j + \theta_j\\
\sqrt{p} \tilde{h}_j +\theta_j\\
\theta_j
\end{pmatrix} = \frac{1}{p}\sum_{j\in[p]} \bar{\E}\Phi \begin{pmatrix}
    \prox_{\reg}(\theta_j + (\beta/\nu) H_j; \nu^{-1}) \\
    \prox_{\tilde \reg}(\theta_j + (\tilde\beta/\tilde \nu) \tilde{H}_j; \tilde\nu^{-1})\\
    \theta_j
\end{pmatrix} + \op(1),
$$
By the law of large numbers, 
$$
\frac{1}{p}\sum_{j\in[p]} \bar{\E}\Phi \begin{pmatrix}
    \prox_{\reg}(\theta_j + (\beta/\nu) H_j; \nu^{-1}) \\
    \prox_{\tilde \reg}(\theta_j + (\tilde\beta/\tilde \nu) \tilde{H}_j; \tilde\nu^{-1})\\
    \theta_j
\end{pmatrix} \pto
\E\Phi \begin{pmatrix}
    \prox_{\reg}(\Theta+ (\beta/\nu) H; \nu^{-1}) \\
    \prox_{\tilde \reg}(\Theta+ (\tilde\beta/\tilde \nu) \tilde{H}; \tilde\nu^{-1})\\
    \Theta
\end{pmatrix},
$$
where the integrand on the RHS is bounded in $L_1$ by \Cref{asm:regularity-conditions}-(1), the pseudo-Lipschitz continuity of order $2$ of $\Phi$, and the additional assumption that $\E[\Theta^2]<+\infty$ where $\Theta\sim F_\theta$. 

Thus, it remains to show the claim \eqref{eq:concentration_Phi}. 
Recall that we have proved this result for the specific test function $\Phi(a, b, c) = (a-c)(b-c)$ in \Cref{lm:gaussian_Poincare_hpsi} by the Gaussian Poincar\'e inequality. Here, we use the same proof strategy. One minor thing we need to be careful of is that we do not assume the differentiability of the test function $\Phi$, so instead of calculating the derivative of $p^{-1}\sum_{j\in[p]} \Phi (\cdots)$, we prove its locally Lipschitz property. Specifically, we claim that for all $\bG, \bG'\in \R^{n\times p}$, letting $\bh=\bh(\bG)$ be the estimator fitted on the design matrix $\bG$, 
\begin{align}\label{eq:locally_Lipschitz_Phi}
    &\left| \frac{1}{p} \sum_{j\in[p]} \Phi \begin{pmatrix}
\sqrt{p} h_j(\bG) + \theta_j\\
\sqrt{p} \tilde{h}_j(\bG) +\theta_j\\
\theta_j
\end{pmatrix} -  \frac{1}{p} \sum_{j\in[p]} \Phi \begin{pmatrix}
\sqrt{p} h_j(\bG') + \theta_j\\
\sqrt{p} \tilde{h}_j(\bG') +\theta_j\\
\theta_j
\end{pmatrix}\right|^2 \nonumber\\
&\le C_* (n \mu_n^2)^{-1} \Bigl(1+ \frac{\|\btheta\|^2}{p} + \frac{\|\loss'(\bz)\|^2}{n} + \frac{\|\tilde{\loss}'(\bz)\|^2}{n} + \frac{\|\bG\|_{\oper}^2}{n}\Bigr) \|\bG-\bG'\|_{\oper}^2,
\end{align}
where $C_*$ is a deterministic constant. If this inequality holds, then the squared Frobenius norm of the weak derivative $\|\nabla_{\bG}\frac{1}{p} \sum_{j\in[p]} \Phi (\cdots)\|_F^2$ is bounded from above by 
$$
C_* (n \mu_n^2)^{-1} \Bigl(1+ \frac{\|\btheta\|^2}{p} + \frac{\|\loss'(\bz)\|^2}{n} + \frac{\|\tilde{\loss}'(\bz)\|^2}{n} + \frac{\|\bG\|_{\oper}^2}{n}\Bigr), 
$$
so that combined with the Gaussian Poincaré inequality, we have 
\begin{align*}
&\bar{\E} \left(\frac{1}{p} \sum_{j\in[p]}  \Phi \begin{pmatrix}
\sqrt{p} h_j + \theta_j\\
\sqrt{p} \tilde{h}_j +\theta_j\\
\theta_j
\end{pmatrix} -\frac{1}{p} \sum_{j\in[p]} \bar{\E} \Phi \begin{pmatrix}
\sqrt{p} h_j + \theta_j\\
\sqrt{p} \tilde{h}_j +\theta_j\\
\theta_j
\end{pmatrix} \right)^2\\
&\le C_* (n \mu_n^2)^{-1} \Bigl(1+ \frac{\|\btheta\|^2}{p} + \frac{\|\loss'(\bz)\|^2}{n} + \frac{\|\tilde{\loss}'(\bz)\|^2}{n} + \frac{\bar\E[\|\bG\|_{\oper}^2]}{n}\Bigr)\\
&= \Op(n \mu_n^2)^{-1}. 
\end{align*}
For the last equality, we have used $\|\btheta\|_2^2/p = \Op(1)$ and $\|f'(\bz)\|_2^2/n = \Op(1)$ by the law of large number with the assumptions $\E[\Theta^2]<+\infty$ and $\E[f'(Z)^2]<+\infty$ for $f\in \{\loss, \tilde\loss\}$. Finally, using $\mu_n^{-1} = o(n^{1/2})$, we complete the proof of \eqref{eq:concentration_Phi}. Thus, it suffices to show the locally Lipschitz property \eqref{eq:locally_Lipschitz_Phi}.

By the pseudo-Lipschitz continuity of order $2$ of the test function $\Phi$, there exists a constant $C$ (which depends on $\Phi$ only) such that 
\begin{align}
&\left| \frac{1}{p} \sum_{j\in[p]} \Phi \begin{pmatrix}
\sqrt{p} h_j(\bG) + \theta_j\\
\sqrt{p} \tilde{h}_j(\bG) +\theta_j\\
\theta_j
\end{pmatrix} -  \frac{1}{p} \sum_{j\in[p]} \Phi \begin{pmatrix}
\sqrt{p} h_j(\bG') + \theta_j\\
\sqrt{p} \tilde{h}_j(\bG') +\theta_j\\
\theta_j
\end{pmatrix}\right|\nonumber \\
&\le 
    \frac{C}{p} \sum_{j\in[p]} \Bigl(1+ |\theta_j|+ \sqrt{p}\bigl(|h_j(\bG)| + |h_j(\bG')| + |\tilde{h}_j(\bG)| +  |\tilde{h}_j(\bG')|\bigr)\Bigr)\nonumber \\
    &\quad \quad \quad \times \sqrt{p} \bigl(|h_j(\bG) - h_j(\bG')| + |\tilde h_j(\bG) - \tilde h_j(\bG')|\bigr)\nonumber \\
&= C \sum_{j\in[p]} v_j(\bG, \bG') \bigl(|h_j(\bG) - h_j(\bG')| + |\tilde h_j(\bG) - \tilde h_j(\bG')|\bigr)\nonumber\\
&\le C \|\bv(\bG, \bG')\| \bigl(\|\bh(\bG)-\bh(\bG')\| + \|\tilde \bh(\bG)-\tilde \bh(\bG')\|\bigr)\label{eq:psi_local_lip_bound}
\end{align}
where $\bv(\bG, \bG') = (v_j(\bG, \bG'))_{j\in[p]}\in\R^p$ is defined as 
$$
v_j(\bG, \bG') := p^{-1/2}(1+ |\theta_j|)+ |h_j(\bG)| + |h_j(\bG')| + |\tilde{h}_j(\bG)| +  |\tilde{h}_j(\bG')|.
$$
Let us bound $\|\bv(\bG, \bG')\|$.
Now, from \Cref{lm:bound_h_psi}-(2), we can take a deterministic constant $C_*$, which depends on $(\delta, \|\loss'\|_{\lip}, \|\tilde\loss'\|_{\lip}, \alpha, \tilde\alpha, \E[\loss'(Z)^2], \E[\tilde\loss'(Z)^2])$ only, such that
\begin{align}\label{eq:upper_bound_h_psi_in_G}
\|\bh(\bG)\|^2 \vee \|\tilde\bh(\bG)\|^2 \le C_* \text{ and } \begin{cases}
    \|\bpsi(\bG)\|^2 \le C_* (\|\loss'(\bz)\|^2 + \|\bG\|_{\oper}^2)\\
    \|\tilde \bpsi(\bG)\|^2 \le C_* (\|\tilde \loss'(\bz)\|^2 + \|\bG\|_{\oper}^2)
\end{cases}    
\end{align}
uniformly for all $\bG\in\R^{n\times p}$. Thus, the upper bound $\|\bh(\bG)\|^2 \vee \|\tilde\bh(\bG)\|^2 \le C_*$ yields 
$$
\|\bv(\bG, \bG')\|^2 \le C (1 + \|\btheta\|^2/p). 
$$
for a constant $C$ depending on $C_*$ only. 

Next, let us control $\|\bh(\bG)-\bh(\bG')\|$. By \cite[Proposition 4.4]{bellec2020out}, we know that $\bG\mapsto \bh(\bG)$ is locally Lipschitz as follows:
$$
\|\bh(\bG) - \bh(\bG')\|^2 \le (n\mu_n^2)^{-1} \|\bG-\bG'\|_{\oper}^2 (\frac{\|\bpsi (\bG)\|_2^2}{n} + \|\bh(\bG)\|_2^2). 
$$
Combined with \eqref{eq:upper_bound_h_psi_in_G}, we obtain 
$$
\|\bh(\bG) - \bh(\bG')\|^2 \le C_* (n\mu_n^2)^{-1}  (1+\frac{\|\loss'(\bz)\|^2}{n} + \frac{\|\bG\|_{\oper}^2}{n}) \|\bG-\bG'\|_{\oper}^2.
$$
Similarly, applying \cite[Proposition 4.4]{bellec2020out} for $\bG\mapsto \tilde \bh(\bG)$ and using \eqref{eq:upper_bound_h_psi_in_G}, we get the same upper bound for $\|\tilde\bh(\bG) - \tilde\bh(\bG')\|^2$ with $\loss'$ on the RHS replaced by $\tilde\loss'$. 

Finally, substituting these upper bounds of $\|\bv(\bG, \bG')\|^2$, $\|\bh(\bG) - \bh(\bG')\|^2$, and $\|\tilde\bh(\bG) - \tilde\bh(\bG')\|^2$ into \eqref{eq:psi_local_lip_bound}, we get the desired locally Lipschitz property \eqref{eq:locally_Lipschitz_Phi}. 

The convergence of $|I\cap \tilde I|^{-1} \sum_{i\in I\cap\tilde I}  \Phi(
y_i - \bx_i^\top\hat\btheta_I,
y_i - \bx_i^\top \hat\btheta_{\tilde I},
z_i)$ follows from the same argument,  so we omit the proof. 
}

\subsection{Proof of \Cref{thm:risk-estimator-general-subagging}}

We assume that $\reg$ and $\tilde\reg$ are $\mu$-strongly convex for a fixed $\mu>0$. 
Then \cite[Theorem 5.1]{bellec2022derivatives} implies that  
$$
\tr[\bA] \tr[\bV] - \df = \Op(\sqrt{n}). 
$$
On the other hand, by the same argument in the proof of \Cref{lm:convergence_trace} with the diminishing ridge penalty $\mu_n$ replaced by the strongly convexity parameter $\mu>0$, we have
$\tr[\bV]/p\pto \nu>0$ and $\tr[\bA]\pto \kappa$. Therefore, we get
$$
\tr[\bA] - \df/\tr[\bV] = \Op(n^{-1/2}) \quad \text{and} \quad \df/\tr[\bV] = \Op(1). 
$$
Note in passing that the same things hold for $\tdf$ and $\tr[\tilde\bV]$. By the Cauchy--Schwarz inequality, the error term due to the replacement of $(\df/\tr[\bV], \tdf/\tr[\tilde\bV])$ by $(\tr[\bA], \tr[\tilde\bA])$ is 
\begin{align*}
&\Bigl|\sum_{i\in [n]} \Bigl(r_i + \frac{\df}{\tr[\bV]} \ind_{\{i\in I\}} \psi_i \Bigr)^\top \Bigl(
    \tilde{r}_i + \frac{\tdf}{\tr[\tilde\bV]} \ind_{\{i\in \tilde I\}}  \tilde \psi_i
    \Bigr) - \sum_{i\in [n]} \Bigl(r_i + \tr[\bA] \ind_{\{i\in I\}}  \psi_i \Bigr)^\top \Bigl(
    \tilde{r}_i + \tr[\tilde \bA] \ind_{\{i\in \tilde I\}} \tilde\psi_i
    \Bigr)\Bigr| \\
    &\le \sqrt{\sum_{i\in [n]} \bigl(r_i + \frac{\df}{\tr[\bV]} \ind_{\{i\in I\}}  \psi_i\bigr)^2} \cdot \Bigl|\frac{\tdf}{\tr[\tilde \bV]}-\tr[\tilde\bA]\Bigr| \|\tilde \bpsi\|  + \sqrt{\sum_{i\in [n]} (\tilde{r}_i + \tr[\tilde{\bA}] \ind_{\{i\in \tilde I\}}  \tilde{\psi}_i)^2} \cdot \Bigl|\frac{\df}{\tr[\bV]}-\tr[\bA]\Bigr| \|\bpsi\| \\
    &= (\|\bz\| + \Op(\sqrt{n})) \cdot \Op(1)
\end{align*}
where the last equality follows from the following fact: $\|\br\|^2 \le 2 (\|\bz\|^2 + \|\bG\bh\|^2)$, $\|\bG\bh\|^2 = \Op(n)$, $\|\bpsi\|^2 = \Op(n)$, and $\df/\tr[\bV] = \Op(1)$. Therefore, it suffices to show
$$
n \bh^\top \tilde\bh + \|\bz\|^2 - \sum_{i\in[n]} (z_i -\bg_i^\top\bh + \tr[\bA] \ind_{\{i\in I\}} \psi_i) (z_i - \bg_i^\top \tilde \bh + \tr[\tilde \bA]\ind_{\{i\in \tilde I\}}\tilde\psi_i) = \Op(1) \|\bz\| + \Op(\sqrt{n}). 
$$
By simple algebra, the LHS can be decomposed into three terms $ \xi_1  + \tilde{\xi}_1 + \xi_2$ with:
\begin{align*}
\xi_1 &= \sum_{i\in [n]} z_i (\bg_i^\top\bh - \tr[\bA] \ind_{\{i\in I\}} \psi_i), \qquad \tilde{\xi}_1 = \sum_{i\in [n]} z_i (\bg_i^\top\tilde\bh - \tr[\tilde\bA] \ind_{\{i\in \tilde I\}} \tilde\psi_i)\\
\xi_2 &= n \bh^\top\tilde\bh - \sum_{i\in[n]} \bigl(\bg_i^\top\bh - \tr[\bA] \ind_{\{i\in I\}} \psi_i\bigr)\bigl(\bg_i^\top \tilde\bh - \tr[\tilde\bA] \ind_{\{i\in \tilde I\}} \tilde\psi_i\bigr)
\end{align*}
For $(\xi_1, \tilde \xi_1)$, the derivative formula \eqref{eq:derivative_h_partial_I} and 
the argument in the proof of \cite[Proposition 18]{bellec2022derivatives} yield
\begin{align*}
\bar\E\Bigl[
\frac{|\xi_1|}{
\{\|\bh\|^2 + p^{-1} \|\bpsi\|^2\}^{1/2} \cdot \|\bz\| }
\Bigr] + \bar\E\Bigl[
\frac{|\tilde {\xi}_1|}{
\{\|\tilde \bh\|^2 + p^{-1} \|\tilde \bpsi\|^2\}^{1/2} \cdot \|\bz\| }
\Bigr] \le C(\mu, \delta). 
\end{align*}
Since the denominators are $\Op(1) \|\bz\|$, we have $\xi_1 + \tilde\xi_1 = \Op(1) \|\bz\|$. 

Below we bound $\xi_2$ using the following moment inequality, which is a variant of \cite[Theorem 7.2]{bellec2020out}.
\begin{lemma}\label{lm:chi_square_general}
Let $\brho, \tilde\brho:\R^{K\times Q}\to\R^Q$ be locally Lipschitz functions and $\bzeta, \tilde\bzeta:\R^{K\times Q}\to\R^L$ be locally Lipschitz functions. If $(\bz_k)_{k\in [K]}$ are i.i.d.\ $\cN(\bm{0}_Q, \bm{I}_Q)$, we have 
\begin{align*}
    \E\Bigl|
    \frac{K \brho^\top \tilde\brho - \sum_{k \in [K]} (\bz_k^\top \brho  - \sum_{q \in Q} \frac{\partial \rho_q}{z_{kq}}) (\bz_k^\top \tilde \brho  - \sum_{q \in Q} \frac{\partial \tilde \rho_q}{z_{kq}})}{
    \{\|\brho\|^2 + \|\bm{\zeta}\|^2\}^{1/2} \{\|\tilde \brho\|^2 + \|\tilde{\bm{\zeta}}\|^2\}^{1/2} 
    }
    \Bigr| \le \C (\sqrt{K} (1+\sqrt{\E[\Xi+\tilde\Xi]}) + \E[\Xi+\tilde\Xi])
\end{align*}
    where $\Xi := \frac{1}{\|\brho\|^2 + \|\bm\zeta\|^2}
   \sum_{k\in [K]}\sum_{q\in [Q]} \Bigl(
    \bigl\|
    \frac{\partial \brho}{\partial g_{kq}}\bigr\|^2 + \bigl\|\frac{\partial\bm{\zeta}}{\partial g_{kq}}\bigr\|^2\Bigr)$. 
\end{lemma}
(The proof of \Cref{lm:chi_square_general} is given in \Cref{sec:proof-lm:chi_square_general}.)
By \Cref{lm:chi_square_general} with $(\brho, \tilde \brho)=(\bh, \tilde\bh)$, $(\bzeta, \tilde\bzeta) = (\bpsi/\sqrt{p}, \tilde\bpsi/\sqrt{p})$, and $K=[n]$, we get
\begin{align*}
&\bar\E\Bigl|
\frac{n \bh^\top\tilde\bh - \sum_{i \in n} (\bg_i^\top \bh - \sum_{j\in [p]} \frac{\partial h_j}{\partial g_{ij}} )(\bg_i^\top\tilde\bh - \sum_{j\in [p]} \frac{\partial \tilde h_j}{\partial g_{ij}})}{\{\|\bh\|^2 + p^{-1} \|\bpsi\|^2\}^{1/2} \{\|\tilde\bh\|^2 + p^{-1} \|\tilde \bpsi\|^2\}^{1/2} }
\Bigr| \le \C (\sqrt{n} (1+\sqrt{\bar\E[\Xi+\tilde\Xi]}) + \bar\E[\Xi+\tilde\Xi]). 
\end{align*}
where
\begin{align}\label{eq:def_Xi}
    \Xi = \frac{1}{\|\bh\|^2 + p^{-1} \|\bpsi\|^2}
   \sum_{i \in [n]}\sum_{j \in [p]} \Bigl(
    \bigl\|
    \frac{\partial \bh}{\partial g_{ij}}\bigr\|^2 + \frac{1}{p} \bigl\|\frac{\partial\bpsi}{\partial g_{ij}}\bigr\|^2\Bigr) 
\end{align}
Let us bound $\Xi$. By the derivative formula \eqref{eq:derivative_formula}, we have 
\begin{align*}
    \sum_{ij}\|\partial_{ij}\bh\|^2 
    \le 2 \|\bA\|_{\frob}^2 \|\bpsi\|^2 + 2 \|\bA\bG^\top\bD\|_F^2 \|\bh\|^2, \quad 
    \sum_{ij}\|\partial_{ij} \bpsi\|^2
    \le 2 \|\bD\bG\bA\|_{\frob}^2 \|\bpsi\|^2 + 2\|\bV\|_{\frob}^2\|\bh\|^2
\end{align*}
where $\bV = \bD - \bD\bG\bA \bG^\top \bD$ and $\bD=\diag\{\loss''(z_i-\bg_i^\top\bh)\}$. By $\|\bA\|_{\oper}\le (p\mu)^{-1}$ and $\|\bD\|_{\oper}\le \|\loss'\|_{\lip}$, $\Xi$ is bounded from above by  
\begin{align}\label{eq:Xi_bound}
    \Xi &\le 2p  \|\bA\|_F^2 + 2\|\bA\bG^\top\bD\|_F^2 + 2\|\bD\bG\bA\|_F^2 + 2p^{-1} \|\bV\|_{\frob}^2
    \le C(\delta, \|\loss'\|_{\lip}) \cdot \|\bA\|_{\oper}^2 \|\bG\|_{\oper}^4
\end{align}
By $\|\bA\|_{\oper}\le (p\mu)^{-1}$ in \eqref{bound-A} and \Cref{lm:bound_h_psi}, we obtain $\bar\E[\Xi] \le \C$. 
This gives
\begin{align*}
\bar\E\Bigl|
\frac{n \bh^\top\tilde\bh - \sum_{i \in n} (\bg_i^\top \bh - \sum_{j\in [p]} \frac{\partial h_j}{\partial g_{ij}} )(\bg_i^\top\tilde\bh - \sum_{j\in [p]} \frac{\partial \tilde h_j}{\partial g_{ij}})}{\{\|\bh\|^2 + p^{-1} \|\bpsi\|^2\}^{1/2} \{\|\tilde\bh\|^2 + p^{-1} \|\tilde \bpsi\|^2\}^{1/2} }
\Bigr| \le \C \sqrt{n}
\end{align*}
Since the denominator is $\Op(1)$, we have 
$$
n \bh^\top\tilde\bh - \sum_{i \in [n]} \Bigl(\bg_i^\top \bh - \sum_{j\in [p]} \frac{\partial h_j}{\partial g_{ij}} \Bigr)\Bigl(\bg_i^\top\tilde\bh - \sum_{j\in [p]} \frac{\partial \tilde h_j}{\partial g_{ij}}\Bigr) = \Op(\sqrt{n}). 
$$
By the derivative formula \eqref{eq:derivative_formula}, noting that $\frac{\partial h_j}{\partial g_{ij}} = 0$ for all $i\notin I$ and $j\in [p]$, it holds that 
\begin{align}\label{eq:derivative_h_partial_I}
\sum_{j\in [p]} \frac{\partial h_j}{\partial g_{ij}} 
 = \begin{dcases}
 0 & i \notin I\\
\tr[\bA] \psi_i - \bh^\top\bA\bG\bD \be_i & i \in I
\end{dcases}     
\end{align}
Combined with %
$\|\bh^\top\bA\bG\bD\|^2 = \Op(n^{-1})$, the Cauchy--Schwarz inequality leads to 
$$
\sum_{i \in [n]} \Bigl(\bg_i^\top \bh - \sum_{j\in [p]} \frac{\partial h_j}{\partial g_{ij}} \Bigr)\Bigl(\bg_i^\top\tilde\bh - \sum_{j\in [p]} \frac{\partial \tilde h_j}{\partial g_{ij}} \Bigr) = \sum_{i\in [n]} \Bigl(\bg_i^\top \bh - \tr[\bA] \ind_{\{i \in I \}} \psi_i\Bigr)\Bigl(\bg_i^\top\tilde\bh - \tr[\tilde\bA] \ind_{\{i \in \tilde I \}} \tilde\psi_i\Bigr) + \Op(1).
$$
This gives $\xi_2 = \Op(\sqrt{n}) + \Op(1) = \Op(\sqrt{n})$ and completes the proof. 

\subsubsection{Proof of \Cref{lm:chi_square_general}}
\label{sec:proof-lm:chi_square_general}
Let us denote $\bm{f} = (\brho^\top, \bzeta^\top)^\top \in \R^{Q+L}$ and $\tilde{\bm{f}} = (\tilde \brho^\top, \tilde \bzeta^\top)^\top \R^{Q\times L}$. Let us write the product as 
$$
\frac{\brho^\top\tilde\brho}{\|\bm{f}\| \|\tilde{\bm{f}}\|} = \Bigl\|\frac{1}{2}\Bigl(\frac{\brho}{\|\bm f\|} + \frac{\tilde\brho}{\|\tilde{\bm f}\|}\Bigr) \Bigr\|^2 - \Bigl\|\frac{1}{2} \Bigl(\frac{\brho}{\|\bm f\|} - \frac{\tilde\brho}{\|\tilde{\bm f}\|}\Bigr)\Bigr\|^2 
$$
Note that $\frac{\brho}{\|\bm f\|}$ and $\frac{\tilde\brho}{\|\tilde{\bm f}\|}$ are bounded by $1$ in the standard Euclid norm $\|\cdot\|$. Applying the $\chi$-square type moment inequality \cite[Theorem 7.2]{bellec2020out} to these terms respectively, 
we observe that the following two terms are bounded from above by $\sqrt{K}\{1+\E[\Xi + \tilde\Xi]\}^{1/2} + \E[\Xi+\tilde \Xi]$ up to some universal constant:
\begin{align*}
    & \E\Bigl|
    K \Bigl\|\Bigl(\frac{\brho}{\|\bm f\|} + \frac{\tilde\brho}{\|\tilde{\bm f}\|}\Bigr) \Bigr\|^2 - \sum_{k\in[K]} \Bigl(
    \bz_k^\top \Bigl(\frac{\brho}{\|\bm f\|} + \frac{\tilde\brho}{\|\tilde{\bm f}\|}\Bigr) - \sum_{q\in[Q]} \frac{\partial}{\partial z_{kq}} \Bigl(\frac{\rho_q}{\|\bm f\|} + \frac{\tilde{\rho}_q}{\|\tilde{\bm f}\|}\Bigr)
    \Bigr)^2 
    \Bigr| \\
    &\E\Bigl|
    K \Bigl\|\Bigl(\frac{\brho}{\|\bm f\|} - \frac{\tilde\brho}{\|\tilde{\bm f}\|}\Bigr) \Bigr\|^2 - \sum_{k\in[K]} \Bigl(
    \bz_k^\top \Bigl(\frac{\brho}{\|\bm f\|} - \frac{\tilde\brho}{\|\tilde{\bm f}\|}\Bigr) - \sum_{q\in[Q]} \frac{\partial}{\partial z_{kq}} \Bigl(\frac{\rho_q}{\|\bm f\|} - \frac{\tilde{\rho}_q}{\|\tilde{\bm f}\|}\Bigr)
    \Bigr)^2 
    \Bigr|
\end{align*}
Thus, by the triangle inequality, we are left with the bound of cross terms:
\begin{equation}\label{eq:cross_term_bound}
\E\Bigl|
K \frac{\brho^\top\tilde\brho}{\|\bm{f}\| \|\tilde{\bm f}\|} - \sum_{k\in [K]} b_k \tilde b_k 
\Bigr| \lesssim \sqrt{K} \{1+\E[\Xi + \tilde\Xi]\}^{1/2} + \E[\Xi+\tilde \Xi]    
\end{equation}
where $b_k = \bz_k^\top \brho/\|\bm{f}\| - \sum_{q\in [Q]} (\partial/\partial z_{kq}) (\rho_q/\|\bm{f}\|)$. 
Expanding the derivative of the second term, $b_k$ can be written as 
\begin{align*}
    b_k  = \underbrace{\frac{\bz_k^\top \brho-\sum_{q\in [Q]}\frac{\partial \rho_q}{\partial z_{kq}}}{\|\bm{f}\|}}_{=:a_k} - \sum_{q\in [Q]} \rho_q \frac{\partial}{\partial z_{kq}} \frac{1}{\|\bm{f}\|} = a_k + \sum_{q\in [Q]} \rho_q \frac{\bm{f}^\top}{\|\bm{f}\|^3} \frac{\partial \bm{f}}{\partial z_{kq}}
\end{align*}
Thus, by multiple applications of Cauchy--Schwarz inequality, using $\|\brho\|^2\le \|\bm f\|^2$, we find that the error $\|\bb-\ba\|^2 = \sum_{k\in [K]} (b_k-a_k)^2$ is bounded from above by $2\Xi$:
$$
\|\ba-\bb\|^2 = \sum_{k\in K} (\sum_{q\in [Q]} \rho_q \frac{\bm{f}^\top}{\|\bm{f}\|^3} \frac{\partial \bm{f}}{\partial z_{kq}})^2 \le \sum_{k} \|\brho\|^2 \sum_{q} \Bigl(\frac{\bm{f}^\top}{\|\bm{f}\|^3} \frac{\partial \bm{f}}{\partial z_{kq}}\Bigr)^2 \le \frac{1}{\|\bm f\|^2} \sum_{k,q} \Bigl\| \frac{\partial \bm{f}}{\partial z_{kq}} \Bigr\|^2 \le 2 \Xi.
$$
The same argument gives $\|\tilde \ba-\tilde \bb\|^2\le 2 \tilde\Xi$. 

Now we claim the following deterministic inequality for all $\bu, \tilde\bu, \ba, \tilde\ba, \bb, \tilde\bb\in\R^K$ with $\|\bu\|\vee \|\tilde\bu\| \le 1$:
\begin{align}
\bigl||K \bu^\top \tilde\bu - \ba^\top \tilde\ba| - |K\bu^\top\tilde\bu - \bb^\top \tilde\bb|\bigr| &\le (\|\ba-\bb\|^2 + \|\tilde\ba - \tilde\bb\|^2) + \sqrt{K} (\|\ba-\bb\| + \|\tilde\ba - \tilde\bb\|) \nonumber \\
&+ 2^{-1} (|K\|\bu\|^2 - \|\bb\|^2| + |K\|\tilde \bu\|^2-\|\tilde\bb\|^2|) \label{eq:deterministic_ineq}
\end{align}
We prove this inequality later. Applying this inequality with $\bu=\bm{\brho}/\|\bm{f}\|$ and $(\ba, \bb)$ defined above, using $\|\ba-\bb\|^2 \le 2\Xi$, we get 
\begin{align*}
|K \bu^\top \tilde\bu-\ba^\top\tilde\ba| &\le  
2(\Xi+\tilde \Xi) + \sqrt{2K} (\Xi^{1/2} + \tilde \Xi^{1/2})
\\
& + 2^{-1} (|K\|\bu\|^2 - \|\bb\|^2| + |K\|\tilde \bu\|^2-\|\tilde\bb\|^2|)\\
&+ |K\bu^\top\tilde\bu - \bb^\top \tilde\bb|
\end{align*}
Taking the expectation, using the moment bound \eqref{eq:cross_term_bound}, we are left with 
\begin{align*}
\E|K \bu^\top\tilde\bu-\ba^\top\tilde\ba| &\lesssim  \E[\Xi+\tilde\Xi] + \sqrt{K} \E[(\Xi^{1/2} 
+ \tilde{\Xi}^{1/2})]\\
&\quad + \{1 + \E[\Xi]\}^{1/2} + \E[\Xi] + \{1 + \E[\tilde \Xi]\}^{1/2} + \E[\tilde \Xi] \\
&\qquad + \{1 + \E[\Xi +\tilde\Xi]\}^{1/2} + \E[\Xi + \tilde\Xi] 
\end{align*}
Using Jensen's inequality $\E[X^{1/2}]\le \sqrt{\E[X]}$ for any non-negative random variable $X$ and 
$\sqrt{a}+\sqrt{b} \le \sqrt{2} \sqrt{a+b}$ for any non-negative scalars $a, b$, the RHS is bounded from above by
$\sqrt{K} \{1+\E[\Xi+\tilde\Xi]\}^{1/2} + \E[\Xi+\tilde\Xi]$ up to some universal constant. This finishes the proof. 

Below we prove the deterministic inequality \eqref{eq:deterministic_ineq}. By multiple applications of the triangle inequality and Cauchy--Schwarz inequality, 
\begin{align*}
    \bigl||K \bu^\top\tilde\bu - \ba^\top\tilde\ba| - |K \bu^\top\tilde\bu - \bb^\top\tilde\bb| \bigr| 
    &\le |\ba^\top\tilde\ba - \bb^\top\tilde\bb|\\
    &\le |(\ba-\bb)^\top(\tilde\ba-\tilde\bb) + (\ba-\bb)^\top\tilde\bb + \bb^\top(\tilde\ba-\tilde\bb)|\\
    &\le \|\ba-\bb\| \|\tilde\ba-\tilde\bb\| + \|\ba-\bb\| \|\tilde \bb\| + \|\tilde \ba-\tilde\bb\| \|\bb\|\\
    &\le \frac{\|\ba-\bb\|^2 +  \|\tilde\ba-\tilde\bb\|^2}{2} + \|\ba-\bb\| \|\tilde \bb\| + \|\tilde \ba-\tilde\bb\| \|\bb\|
\end{align*}
Using $\|\bb\| \le \sqrt{|\|\bb\|^2 - K\|\bu\|^2|} + \sqrt{K\|\bu\|^2}$ and $\|\bu\|\le 1$, $\|\ba-\bb\|\|\tilde\bb\|$ can be bounded from above as 
$$
\|\ba-\bb\|\|\tilde\bb\| \le \|\ba-\bb\| \sqrt{|\|\tilde\bb\|^2 - K\|\tilde \bu\|^2|} + \sqrt{K}\|\ba-\bb\| \le \frac{\|\ba-\bb\|^2}{2} +\frac{ |\|\tilde\bb\|^2 - K\|\tilde \bu\|^2|}{2} + \sqrt{K}\|\ba-\bb\|.
$$
The same argument gives $\|\tilde \ba-\tilde \bb\|\bb\| \le \frac{\|\tilde\ba-\tilde\bb\|^2}{2} + \frac{\|\bb\|^2 - K\|\bu\|^2|}{2} + \sqrt{K}\|\tilde\ba-\tilde\bb\|$. Putting them all together, we obtained the desired upper bound of $\bigl||K\bu^\top\tilde\bu - \ba^\top\tilde\ba| - |K\bu^\top\tilde\bu - \bb^\top\tilde\bb| \bigr|$.  

\subsection{Proof of trace convergence}
\label{subsec:proof_convergence_trace}
We assume without loss of generality that $I=[n]$. 
Throughout this section, we denote $\bh=\hat\bh_{\mu, K}$ and 
$\bpsi = \bpsi_{\mu, K}$ for simplicity. 
\begin{lemma}
    \label{lm:trace_representation}
    For any $\mu \in (0,1]$ with $\mu^{-1} = o(n^{1/4})$, 
    \begin{align*}
    \tr[\bV] &= \|\bh\|^{-2} \bigl(\|\bpsi\|^2\tr[\bA] - \bpsi^\top\bG\bh\bigr) + \Op(n^{1/2} \mu^{-1})\\
        \tr[\bA]^2 &= \frac{(\bpsi^\top \bG\bh)^2 + 
        \|\bh\|^2 (
       p\|\bpsi\|^2-\|\bG^\top\bpsi\|^2
        )}{\|\bpsi\|^4}+  \Op(n^{-1/2}\mu^{-2})
    \end{align*}
    and $\PP(\tr[\bA]^2 \le C) \to 1$ for a constant  $C>0$ that only depend on $(\alpha, \beta, \delta)$. 
\end{lemma}

\begin{proof}
First we show the stochastic representation of $\tr[\bV]$ by $(\bh, \bpsi, \tr[\bA])$. 
\begin{lemma}[\cite{bellec2022derivatives}]\label{lm:sure}
    Let $\brho:\R^{K\times Q}\to\R^Q$ and $\bzeta:\R^{K\times Q} \to \R^K$ be two locally Lipschitz functions with differentiable components. 
    If $\bZ\in \R^{K\times Q}$ has i.i.d.\ $\cN(0,1)$ entries, we have
    \begin{align*}
        \E\Bigl[
        \Bigl(\frac{\bzeta^\top \bZ \brho - \sum_{kq}\frac{\partial (\zeta_k \rho_q)}{\partial g_{kq}}}{\|\bzeta\|^2 + \|\brho\|^2}
        \Bigr)^2
        \Bigr] &\le \C( 1+ \E[\Xi])
    \end{align*}
    where $\Xi := \frac{1}{\|\brho\|^2 + \|\bm\zeta\|^2}
   \sum_{k\in [K]}\sum_{q\in [Q]} \Bigl(
    \bigl\|
    \frac{\partial \brho}{\partial g_{kq}}\bigr\|^2 + \bigl\|\frac{\partial\bm{\zeta}}{\partial g_{kq}}\bigr\|^2\Bigr)$. 
\end{lemma}
By \Cref{lm:sure} with $(\bzeta, \brho) = (p^{-1/2} \bpsi, \bh)$ and $\bZ=\bG$, we have 
$$
\E\Bigl[\Bigl(\frac{\frac{1}{\sqrt{p}} \bpsi^\top\bG \bh - \frac{1}{\sqrt{p}} \sum_{ij}\frac{\partial (\psi_i h_j)}{\partial g_{ij}}}{\|\bh\|^2 + p^{-1} \|\bpsi\|^2}
        \Bigr)^2\Bigr] \le \C(1+\E[\Xi])
$$
where $\Xi$ is the same estimate as in \eqref{eq:def_Xi}. Using the upper estimate $\Xi\le \C \|\bA\|_{\oper}^2 \|\bG\|_{\oper}^4$ in \eqref{eq:Xi_bound}, $\|\bA\|_{\oper}\le (p\mu)^{-1}$ and $\E[\|\bG\|^4] \le C(\delta) n^2$, we find that the RHS in the above display is bounded from away by $C(\delta) (1+\mu^{-2})$. Since the denominator $\|\bh\|^2 + p^{-1} \|\bpsi\|^2$ is $\Op(1)$, we get
$$
\bpsi^\top\bG\bh - \sum_{ij} \frac{\partial \psi_i h_j}{\partial g_{ij}} = \sqrt{p} \cdot \Op(1) \cdot \Op(\mu^{-1}) = \Op(\sqrt{n} \mu^{-1}). 
$$ 
For the sum of derivative $\sum_{ij} \frac{\partial (\psi_i h_j)}{\partial g_{ij}}$, using $\|\bA\|_{\oper}\le (p\mu)^{-1}$ and $\|\bG\|_{\oper}=\Op(\sqrt{n})$, we have 
\begin{align*}
\sum_{ij} \frac{\partial (\psi_i h_j)}{\partial g_{ij}} &= \|\bpsi\|^2\tr[\bA] - \bh^\top \bG^\top\bD \bpsi -\bpsi^\top\bD \bG \bA \bh - \|\bh\|^2 \tr[\bV]\\
&= \|\bpsi\|^2\tr[\bA] + \Op(n^{1/2}) + \Op(\mu^{-1}) - \|\bh\|^2 \tr[\bV] 
\end{align*}
so that we are left with 
$$
\bpsi^\top\bG\bh - \|\bpsi\|^2\tr[\bA] + \|\bh\|^2 \tr[\bV] =\Op(\sqrt{n} \mu^{-1}).
$$
Dividing by $\|\bh\|^2$, noting $\|\bh\|^{-1}=\Op(1)$ from $\|\bh\|^2 \pto \alpha^2>0$, we obtain the representation of $\tr[\bV]$. 

Next, we show the stochastic representation of $\tr[\bA^2]$ by $(\bG, \bh, \bpsi)$. By the stochastic representation of $\tr[\bV]$, 
\begin{align*}
\bigl|\|\bpsi\|^4\tr[\bA]^2 - (\bpsi^\top\bG\bh + \tr[\bV]\|\bh\|^2)^2\bigr|&\le \Op(\sqrt{n}\mu^{-1}) \bigl|\|\bpsi\|^2 \tr[\bA]  + \bpsi^\top\bG\bh + \tr[\bV]\|\bh\|^2\bigr|\\
&\le \Op(\sqrt{n}\mu^{-1}) \Op(n \mu^{-1} + n + n)\\
&= \Op(n^{3/2} \mu^{-2}).
\end{align*}
By \Cref{lm:chi_square_general} with $\brho=\tilde\brho = \bpsi/\sqrt{p}$ and $\bzeta=\tilde\bzeta=\bh$, we have 
\begin{align*}
\E\Bigl[
        \frac{
        \bigl|
        p \|\frac{\bpsi}{\sqrt{p}}\|^2 - \sum_{j \in [p]} (\frac{\bpsi^\top}{\sqrt{p}}\bG\be_j-  \frac{1}{\sqrt{p}} \sum_{i \in [n]}\frac{\partial\bpsi}{\partial g_{ij}})^2\bigr|
        }{\|\bh\|^2 + p^{-1} \|\bpsi\|^2}
        \Bigr] &\le \C (\sqrt{n} (1+\E[\Xi]^{1/2}) + \E[\Xi]).
\end{align*}
where $\Xi$ is the same estimate as in \eqref{eq:def_Xi}. Using $\E[\Xi]\le \C \mu^{-2}$ again, we get 
$$
\|\bpsi\|^2 - \frac{1}{p} \sum_{j \in [p]} (\bpsi^\top\bG\be_j-\sum_{i \in [n]}\frac{\partial\psi_i}{\partial g_{ij}})^2=\Op(1) \cdot \Op(\sqrt{n} (1+\mu^{-1}) + \mu^{-2}) 
= (n^{1/2}\mu^{-1})
$$

Using $\sum_{i \in [n]} \frac{\partial \psi_i}{\partial g_{ij}} = -\bpsi^\top \bD\bG\bA\be_j - \tr[\bV]h_j
$ by the derivative formula \eqref{eq:derivative_formula}, it holds that 
\begin{align*}
    & \Bigl\|\sum_{j \in [p]} (\bpsi^\top\bG\be_j-\sum_{i \in [n]}\frac{\partial\psi_i}{\partial g_{ij}})^2 -  \| \bG^\top \bpsi + \tr[\bV] \bh\|^2\Bigr\|\\
    &=\Bigl|\| \bG^\top \bpsi + \bA^\top\bG^\top \bD \bpsi + \tr[\bV] \bh\|^2 -  \| \bG^\top \bpsi + \tr[\bV] \bh\|^2\Bigr|\\
    &\le \|\bA^\top\bG^\top\bD\bpsi\| \Bigl(\| \bA^\top\bG^\top \bD \bpsi\| +  2 \| \bG^\top \bpsi + \tr[\bV] \bh\|\Bigr)\\
    &= \Op(\mu^{-1}) (\Op(\mu^{-1}) + \Op(n)) = \Op(\mu^{-1} n). 
\end{align*}
Combining the above displays, we are left with 
$$
\|\bpsi\|^2 - p^{-1} \|\bG^\top\bpsi + \tr[\bV]\bh\|^2 = \Op(n^{1/2}\mu^{-1}) + \Op(\mu^{-1} n^{-1}) = \Op(n^{1/2}\mu^{-1}). 
$$
Therefore, 
\begin{align*}
    \|\bpsi\|^4 \tr[\bA]^2 &= (\bpsi^\top\bG\bh + \tr[\bV]\|\bh\|^2)^2 + \Op(n^{3/2}\mu^{-2})\\
    &= (\bpsi^\top \bG\bh)^2 + 
    \|\bh\|^2 (
    2\tr[\bV] \bpsi^\top\bG\bh + \tr[\bV]^2 \|\bh\|^2) + \Op(n^{3/2}\mu^{-2}))\\
    &= (\bpsi^\top \bG\bh)^2 + 
    \|\bh\|^2 (
    \|\bG^\top\bpsi+\tr[\bV]\bh\|^2-\|\bG^\top\bpsi\|^2
    ) + \Op(n^{3/2}\mu^{-2}))\\
     &= (\bpsi^\top \bG\bh)^2 + 
    \|\bh\|^2 (
   p\|\bpsi\|^2-\|\bG^\top\bpsi\|^2
    ) + \Op(n^{3/2} \mu^{-1}) + \Op(n^{3/2}\mu^{-2}). 
 \end{align*}
Multiplying by $\|\bpsi\|^{-4}$, which is $\Op(n^2)$ since $\|\bpsi\|^2/p \pto\beta^2>0$, we obtain this representation of $\tr[\bA]^2$. 
The upper bound $\PP(\tr[\bA]^2\le \C)$ follows from the stochastic representation and the convergences:
$\|\bpsi\|^2/p \pto\beta^2>0$, $\|\bh\|^2 \pto \alpha^2>0$, $\|\bG\|_{\oper}/\sqrt{n} \pto 1+\sqrt{\delta}>0.$
\end{proof}

\begin{lemma}
    \label{lm:concentration_trace_A_RHS}
    Letting $\bar\E[\cdot]=\E[\cdot|\bz, \btheta]$ be the conditional expectation with respect to the design matrix $\bG$, for any $\mu \in (0,1]$ with $\mu^{-2} = o(n)$, we have 
    \small
    $$
    \frac{(\bpsi^\top \bG\bh)^2 + 
    \|\bh\|^2 (
   p\|\bpsi\|^2-\|\bG^\top\bpsi\|^2
    )}{\|\bpsi\|^4} = \frac{(\bar\E[\bpsi^\top \bG\bh])^2 + 
    \bar\E[\|\bh\|]^2 (
   p\bar\E[\|\bpsi\|^2-\|\bG^\top\bpsi\|^2]
    )}{\bar\E[\|\bpsi\|^2]^2}+ \Op(n^{-1/2} \mu^{-1}).
    $$
    \normalsize
\end{lemma}

\begin{proof}
    By the Gaussian Poincar\'e inequality, we claim the following:
    \begin{align*}
           \bar\E[(\|\bpsi\|^2-\bar\E[\|\bpsi\|^2])^2] &\le  \C n \mu^{-2},\\
       \bar \E[(\|\bh\|^2-\bar\E[\|\bh\|^2])^2] &\le  \C n^{-1} \mu^{-2},\\
        \bar \E[( \bpsi^\top\bG\bh -\bar\E[\bpsi^\top\bG \bh])^2] &\le \C n \mu^{-2},\\
        \bar \E[(\|\bG^\top\bpsi\|^2-\bar\E[\|\bG^\top\bpsi\|^2])^2] &\le {\C n^3\mu^{-2}}.
    \end{align*}
    First and second moment inequalities immediately follow from \Cref{lm:gaussian_Poincare_hpsi} with $\tilde I = I = [n]$. 
    For the third moment inequality, the Gaussian Poincar\'e inequality gives the upper bound $ \E[(\bpsi^\top\bG\bh -\bar\E[\bpsi^\top\bG \bh])^2] \le \sum_{ij} \E \Bigl(\frac{\partial \bpsi^\top\bG\bh}{\partial g_{ij}}\Bigr)^2$, where 
    \begin{align*}
        \text{RHS}_{ij} \coloneq
        \frac{\partial \bpsi^\top\bG\bh}{\partial g_{ij}} &=  \Bigl(- \bD\bG\bA \be_j\psi_i - \bV \be_i h_j\Bigr)^\top \bG\bh + \psi_i h_j + \bpsi^\top \bG \bA \Bigl(\be_j\psi_i  - \bG^{\top}\bD\be_i h_j\Bigr)\\
         &= \be_j^\top\bA^\top \Bigl(-\bG^\top\bD\bG \bh + \bG^\top\bpsi\Bigr) \psi_i - \be_i^\top \Bigl(\bV^\top \bG \bh + \bD\bG \bA^\top\bG^\top\bpsi\Bigr)  h_j + \psi_i h_j
    \end{align*}
    By $\bV = \bD - \bD \bG\bA\bG^\top \bD $ and $\|\bA\|_{\oper}\le (p\mu)^{-1}$ in \eqref{bound-A}, we have $\sum_{i,j} \E[\text{RHS}_{ij}^2] \le \C n\mu^{-2}$. 
    Next,
    $$
    \E\Bigl[\Bigl(\|\bG^\top\bpsi\|^2 - \bar\E[\|\bG^\top\bpsi\|^2]\Bigr)^2\Bigr] \le \sum_{i,j}\E\Bigl(\frac{\partial \|\bG^\top\bpsi\|^2}{\partial g_{ij}}\Bigr)^2 = 4 \sum_{i,j}\E
   (\bpsi^\top\bG  
    \frac{\partial \bG^\top\bpsi}{\partial g_{ij}})^2 
    $$
    where 
    \begin{align*}
        \bpsi^\top\bG \frac{\partial\bG^\top\bpsi}{\partial g_{ij}}  &= \bpsi^\top\bG \be_j \psi_i   + \bpsi^\top\bG \bG^\top  (- \bD\bG\bA \be_j\psi_i - \bV \be_i h_j)\\
        &= \bpsi^\top\bG (\bI_p - \bG^\top\bD \bG\bA)\be_j \psi_i - \bpsi^\top\bG\bG^\top\bV \be_i h_j
    \end{align*}
    so $\sum_{ij} \E[\text{RHS}_{ij}^2] \le \C n^3 \mu^{-2}$. 
    Thus, we get 
    \begin{align*}
      \|\bpsi\|^2 &= \bar\E[\|\bpsi\|^2] + \Op(n^{1/2} \mu^{-1}) \\
       \|\bh\|^2&= \bar\E[\|\bh\|^2] + \Op(n^{-1/2} \mu^{-1}) \\
        \bpsi^\top\bG\bh &= \bar\E[\bpsi^\top\bG \bh] + \Op(n^{1/2} \mu^{-1}) \\
        \|\bG^\top\bpsi\|^2 &= \bar\E[\|\bG^\top\bpsi\|^2]+ \Op(n^{3/2} \mu^{-1}) 
    \end{align*}
    Since $\mu^{-2} = o(n)$, the concentration of $\|\bpsi\|^2$ on the conditional expectation $\E[\|\bpsi\|^2]$ and the convergence $\|\bpsi\|^2/p \pto \beta^2$ yield $\E[\|\bpsi\|^2]/p\pto\beta^2>0$. This implies that $1/\|\bpsi\|^2$ and $1/\E[\|\bpsi\|^2]$ are $\Op(n^{-1})$. Then, the error from replacing the denominator $\|\bpsi\|^4$ by $\bar\E[\|\bpsi\|^2]^2$ is estimated as 
    \begin{align*}
    & (\|\bpsi\|^{-4} - \bar\E[\|\bpsi\|^2]^2)\bigl( (\bpsi^\top \bG\bh)^2 + 
    \|\bh\|^2 (
   p\|\bpsi\|^2-\|\bG^\top\bpsi\|^2
    )\bigr) \\
    &= \frac{(\|\bpsi\|^2 - \bar\E[\|\bpsi\|^2]) (\|\bpsi\|^2 + \bar\E[\|\bpsi\|^2])}{\|\bpsi\|^4 \bar\E[\|\bpsi\|^2]^2} \cdot 
    \Op(n^2)\\
    &= \Op(n^{1/2}\mu^{-1}) \Op(n \cdot n^{-4}) \Op(n^2) = \Op(n^{-1/2}\mu^{-1}).  
    \end{align*}
    For the error from replacing the numerator with the conditional one, the error of replacing each term with the conditional one is given by
    \begin{align*}
        (\bpsi^\top \bG\bh)^2 &=  (\bpsi^\top \bG\bh)^2 + \Op(n^{3/2}\mu^{-1})\\
         \|\bh\|^2 \|\bpsi\|^2 &= \bar\E[\|\bh\|^2]\bar\E[\|\bpsi\|^2] + \Op(n^{1/2}\mu^{-1})\\
         \|\bh\|^2 \|\bG^\top\bpsi\|^2 &= \bar\E[\|\bh\|^2] \bar\E[\|\bG^\top\bpsi\|^2] + \Op(n^{3/2}\mu^{-1}) 
    \end{align*}
    Thus, noting that $\bar\E[\|\bpsi\|^2]^{-1} = \Op(n^{-1})$, we get 
    \begin{align*}
    & \frac{(\bpsi^\top \bG\bh)^2 + 
    \|\bh\|^2 (
   p\|\bpsi\|^2-\|\bG^\top\bpsi\|^2
    ) - (\bar\E[\bpsi^\top \bG\bh])^2 -
    \bar\E [\|\bh\|^2] (
   p\bar\E[\|\bpsi\|^2]-\bar\E[\|\bG^\top\bpsi\|^2]
    )}{\bar\E[\|\bpsi\|^2]^2} \\
    &= \Op(n^{-1/2}\mu^{-1}). 
    \end{align*}
    This finishes the proof. 
\end{proof}

\begin{lemma}
    \label{lm:concentrate_trace}
    Suppose $\mu\in (0,1]$ and $\mu^{-1} = O(n^{1/8})$. 
    Then, there exist a non-negative random variables $\hat{\kappa} \ge 0$, which is independent of $\bG$, such that 
    \begin{align*}
       \PP\Bigl(|\tr[\bA] -\hat\kappa|\le 2 n^{-1/16} \quad \text{and} \quad  |\hat\kappa|\le C \Bigr) &\to 1 
    \end{align*}
    where $C$ is a constant depending on $(\alpha, \beta, \delta)$ only. 
\end{lemma}
\begin{proof}
    By \Cref{lm:trace_representation} and \Cref{lm:concentration_trace_A_RHS}, there exists a random variable $A_n$, which is independent of $\bG$, such that
    $$
    \tr[\bA]^2 = A_n + \Op(n^{-1/2}\mu^{-2}) + \Op(n^{-1/2} \mu^{-1}) = A_n + \Op(n^{-1/4}), \quad \text{and} \quad \PP(|A_n|\le C)\to 1
    $$
    Noting $\tr[\bA]^2 \ge 0$ and $\Op(n^{-1/4}) = \op(n^{-1/8})$, this implies that the event 
    $$
    \Omega := \{|\tr[\bA]^2- A_n| \le n^{-1/8}\} \cap \{-n^{-1/8} \le A_n\le C\}
    $$
    holds with high probability. Let us take $\hat\kappa\coloneq\sqrt{(A_n + 2 n^{-1/8})_+}$. Note in passing that under the event $\Omega$, we have $\hat\kappa = \sqrt{A_n + 2 n^{-1/8}}$ and $\hat\kappa \le \sqrt{C+1}$. 
    By the non-negativeness $\tr[\bA]\ge 0$ and the $(1/2)$-H\"{o}lder continuity of the square root $\R_{\ge 0} \ni x\mapsto \sqrt{x}$, under the event $\Omega$, it holds that 
    \begin{align*}
        |\tr[\bA]-\hat\kappa| = |\sqrt{\tr[\bA]^2} - \sqrt{A_n + 2n^{-1/8}}|
        \le |\tr[\bA]^2 - A_n - 2n^{-1/8}|^{1/2}
    \end{align*}
    and the RHS is less than $|n^{-1/8} + 2n^{-1/8}|^{1/2}\le 2n^{-1/16}$ by the triangle inequality. 
\end{proof}

\subsubsection{Proof of \Cref{lm:convergence_trace}}
Applying \Cref{lm:approx_multi_normal} with $M=1$ and $(\bz, \bF(\bz)) = (\bg_i, \bh)$ for each $i\in [n]$, using $\sum_{ij} \bar\E[\|\partial_{ij} \bh\|^2] = \Op(\mu^{-2})=\op(n)$ we get
$$
\sum_{i \in [n]} (\bg_i^\top\bh-\tr[\bA]\psi_i - \|\bh\| \hat{u}_i)^2 =\op(n), \quad \hat{u}_i|\btheta, \bz, \bG_{-i} \deq \cN(0,1)
$$ 
By \Cref{lm:concentrate_trace}, there exists a non-negative random variable $\hat\kappa$ independent of $\bG$ such that 
$$\tr[\bA]= \hat\kappa + \op(1), \quad \PP(\hat\kappa \le C)\to 1$$ 
for a positive constant $C$. Then, combined with $\|\bh\|=\alpha + \op(1)$, we get
$$
\frac{1}{2n}\sum_{i \in [n]} (\bg_i^\top\bh -\hat\kappa \psi_i - \alpha \hat{u}_i)^2 = \op(1) + (\hat\kappa -\tr[\bA])^2 \frac{\|\bpsi\|^2}{n} + (\|\bh\|-\alpha)^2  \frac{\sum_{i \in [n]} \hat\bu_i^2}{n}  = \op(1).
$$
Furthermore, using the independence of $(\hat\kappa, \bG)$ and the upper bound $\hat\kappa \le C$, by the same argument in the proof of \Cref{thm:corr-sigerror-reserror}, we can easily show that the conditional expectation $\bar\E[\cdot] = \E[\cdot|\bz, \btheta]$ of the square of LHS is bounded by $C$ with high probability for some constant $C$. This implies that the conditional expectation $\bar\E$ of LHS is also $o(1)$, i.e., 
$$
\frac{1}{n}\sum_{i \in [n]} \bar\E[(\bg_i^\top\bh -\hat\kappa \psi_i - \alpha \hat{u}_i)^2] = \op(1). 
$$
Let us define $\Xi_i \coloneq \bg_i^\top\bh -\hat\kappa \psi_i - \alpha \hat{u}_i$ so that the above display reads $n^{-1} \sum_i \bar\E [\Xi_i^2] = \op(1)$. Noting $\psi_i = \loss'(z_i-\bg_i^\top\bh)$ for all $i\in [n]$, the residual can be written as 
$$
\eps_i - \bg_i^\top\bh = \prox_{\loss}(\eps_i - \alpha \hat{u}_i -\Xi_i; \hat\kappa)
$$
for all $i\in [n]$. Since $\prox_{f}(\cdot)$ is $1$-Lipschitz for any convex function, we have
$$
\frac{1}{n} \sum_{i\in [n]} \bar\E\Bigl[
\Bigl(\eps_i - \bg_i^\top\bh - \prox_{\loss}(\eps_i - \alpha \hat{u}_i; \hat\kappa)\Bigr)^2 
\Bigr] \le \frac{1}{n} \sum_{i\in [n]} \bar\E[\Xi_i^2] = \op(1).
$$
Since $\loss'$ is Lipschitz, the above display lets us approximate $\psi_i=\loss'(z_i-\bg_i^\top\bh)$ by $\env_{\loss}'(\eps_i - \alpha \hat{u}_i; \hat\kappa)$:
$$
\frac{1}{n}\sum_{i\in [n]} \bar\E\Bigl[\Bigl(\psi_i- \env_{\loss}'(\eps_i - \alpha \hat{u}_i; \hat\kappa) \Bigr)^2 \Bigr] = \op(1). 
$$
Applying this approximation to the concentration $\beta^2 = \|\bpsi\|^2/p + \op(1) = \bar\E[\|\bpsi\|^2]/p + \op(1)$ by \Cref{lm:gaussian_Poincare_hpsi} with $I=\tilde I=[n]$, we get
\begin{align}
\beta^2 &= \frac{n}{p} \frac{1}{n}\sum_{i\in [n]} \bar\E\Bigl[\psi_i^2\Bigr] + \op(1) = \delta \frac{1}{n}\sum_{i\in [n]} \bar\E\Bigl[\env_{\loss}'(\eps_i-\alpha \hat{u}_i;\hat\kappa)^2\Bigr] + \op(1) \nonumber \\
&= \delta \frac{1}{n} \sum_{i\in [n]} \int_{-\infty}^\infty \varphi(x) \env_{\loss}'(z_i + \alpha x; \hat\kappa)^2 \, \mathrm{d}x 
+ \op(1). \label{eq:hat_kappa_system}
\end{align}
where $\varphi(x)$ is the pdf of $\cN(0,1)$. Let us define the functions $F, \hat{F}:[0, \infty)\to\R$ by 
$$
F(\tau) \coloneq \beta^2 - \delta \E[\env_{\loss}'(Z+\alpha G;\tau)^2], \quad \hat{F}(\tau) \coloneq \beta^2 - \delta \frac{1}{n}\sum_{i\in [n]} \int_{-\infty}^\infty \varphi(x) \env_{\loss}'(z_i + \alpha x; \tau)^2 \, \mathrm{d}x 
$$
so that $F(\kappa) = 0$ by \eqref{eq:CGMT-1b} in \Cref{sys:general_ensemble-M=1} (with $c=1$), while \eqref{eq:hat_kappa_system} reads $\hat{F}(\hat\kappa) = \op(1)$. 
Note in passing that the weak law of large numbers implies $\hat F(\tau) \pto F(\tau)$ pointwise, and $F$ and $\hat{F}$ are strictly increasing functions in $\tau$ since $-2^{-1} \env_{\loss}'(x;\tau)^2$ is the derivative of the convex function $\tau\mapsto \env_{\loss}(x;\tau)$, which is strictly convex under \Cref{asm:regularity-conditions}-(4) (see \cite[Lemma 4.4]{thrampoulidis2018precise}). 
Then, for any $\epsilon>0$, we have
$$
F(\kappa + \epsilon) > 0 = F(\kappa) > F(\kappa - \epsilon).
$$
By the pointwise convergence $\hat F(\tau)\pto F(\tau)$, it holds that 
$$
\hat F(\kappa+\epsilon) > 2^{-1} F(\kappa +\epsilon) > 0 > 2^{-1} F(\kappa-\epsilon) >  \hat F(\kappa-\epsilon).
$$
with high probability. 
Then, combined with  $\hat F(\hat \kappa) = \op(1)$, we have 
$$
\hat F(\kappa+\epsilon) > 2^{-1} F(\kappa +\epsilon) > \hat F(\hat\kappa)  > 2^{-1} F(\kappa-\epsilon) >  \hat F(\kappa-\epsilon).
$$
with high probability. 
Since $\hat F$ is non-decreasing with probability $1$, this gives $\PP(|\hat\kappa-\kappa|\le \epsilon) \to 1$. 
Since we took $\epsilon>0$ arbitrarily, we obtain $\hat\kappa\pto \kappa$. 

Going back to the place where we replaced $\tr[\bA]$ by $\hat\kappa$, now replacing $\tr[\bA]$ by $\kappa$ instead, we obtain
$$
\frac{1}{n}\sum_{i \in [n]} (\bg_i^\top\bh -\kappa \psi_i - \alpha \hat{u}_i)^2 = \op(1), 
$$
and this approximation also holds in $\bar\E$. Thus, we get the approximation of residual and $\psi_i$
$$
 \frac{1}{n} \sum_{i\in [n]} \bar \E\Bigl[\Bigl(\eps_i-\bg_i^\top\bh - \prox_{\loss}(\eps_i - \alpha \hat{u}_i; \kappa)\Bigr)^2\Bigr] = \op(1), \quad  \frac{1}{n} \sum_{i\in [n]} \bar \E\Bigl[\Bigl(\psi_i - \env_{\loss}'(z_i-\alpha \hat u_i; \kappa )\Bigr)^2\Bigr] = \op(1)
$$
Let us show $\tr[\bV]/p\pto \nu$. 
Using the concentration 
$\bpsi^\top\bG\bh - \tr[\bA] \|\bpsi\|^2 + \tr[\bV]\|\bh\|^2 = \op(n)$ from \Cref{lm:trace_representation} and $\|\bh\|^2\pto \alpha^2>0$, $\|\bpsi\|^2/p\pto \beta^2$,  $\tr[\bA]\pto \kappa$, we have   
\begin{align*}
p^{-1} \tr[\bV] = \alpha^{-2} \kappa \beta^2 - \alpha^{-2} \bpsi^\top\bG\bh/p + \op(1).     
\end{align*}
Recall that we have shown the concentration $\bpsi^\top\bG\bh= \bar\E[\bpsi^\top\bG\bh] + \op(n)$ in the proof of \Cref{lm:concentration_trace_A_RHS}. 
Applying the proximal approximation of the residual $\eps_i -\bg_i^\top\bh$ and $\psi_i$ to this, noting $\bar\E[\|\bpsi\|^2] = \Op(n)$ and $\bar\E[\|\bG\bh\|^2] = \Op(n)$, we are left with 
\begin{align*}
\frac{1}{n} \bpsi^\top\bG\bh &= \frac{1}{n} \sum_{i \in [n]} \bar\E[\psi_i \cdot \bg_i^\top\bh] + \op(1) \\
&= \frac1n\sum_{i \in [n]} \bar\E \Bigl[\env_{\loss}'(z_i-\alpha\hat u_i;\kappa) (\eps_i - \prox_{\loss}(\eps_i - \alpha \hat{u}_i; \kappa)) \Bigr] + \op(1) \\
&= \frac1n\sum_{i \in [n]} \int_{-\infty}^\infty \varphi(x) \env_{\loss}'(z_i + \alpha x; \kappa) (\eps_i - \prox_{\loss}(\eps_i + \alpha x; \kappa)) \, \mathrm{d}x  + \op(1) 
\end{align*}
so that the weak law of large numbers yields 
\begin{align*}
\frac{1}{n} \bpsi^\top\bG\bh &= \E[\env_{\loss'}(Z+\alpha G;\kappa) (Z - \prox_{\loss}(Z + \alpha G; \kappa))] + \op(1)\\
&= \E[\env_{\loss'}(Z+\alpha G;\kappa) (\alpha G + Z - \prox_{\loss}(Z+\alpha G; \kappa) - \alpha G) ] + \op(1)\\
&= \kappa \E[\env_{\loss}'(\alpha G + Z;\kappa)^2] - \alpha  \E[G \cdot \env_{\loss}'(\alpha G + Z;\kappa)]\\
&= \kappa \cdot \beta^2/\delta - \alpha \cdot \nu\alpha/\delta.
\end{align*}
where we have used \eqref{eq:CGMT-1b} and \eqref{eq:CGMT-1d} in \Cref{sys:general_ensemble-M=1} (with $c=1$) for the last equation. 
Combined with $p^{-1} \tr[\bV] = \alpha^{-2} \kappa \beta^2 - \alpha^{-2} \bpsi^\top\bG\bh/p + \op(1)$, this gives  $p^{-1}\tr[\bV]=\nu + \op(1)$ and finishes  the proof.

\subsection{Convergence of error vector norm squared under \Cref{asm:regularity-conditions}}
\label{subsec:cgmt_assumption}

In this section we verify that \Cref{assu:loss_penalty} and \Cref{asm:regularity-conditions}-(1)-(3) are sufficient for \cite[Theorem 4.1]{thrampoulidis2018precise} to hold. Comparing our assumptions and the condition assumed in the theorem, it suffices to show that the conditions (10) and (12) in \cite{thrampoulidis2018precise} can be omitted. 

Indeed, the authors used the condition (10) to show 
that any $\hat\bu \in \partial{\loss} (\bz - \bG \hat\bh^{B})/\sqrt{p}$ belongs to a compact set with high probability, where 
$\hat{\bh}^{B}$ is a ``bounded'' estimator. More precisely, $\hat{\bh}^{B}$ is a solution to the constrained optimization problem $\min_{\|\bh\| \le K_\alpha} \obj(\bh)$ for a positive constant $K_\alpha>0$ where 
$\obj(\bh) = \sum_{i\in[n]} \loss(z_i-\bg_i^\top\bh)+\sum_{j\in[p]} \reg(\sqrt{p} h_j + \theta_j)$ is the objective function for the original unconstrained M-estimator.  Since $\loss$ is differentiable with Lipschitz derivative $\loss'$ by \Cref{assu:loss_penalty}, using the same argument in \Cref{lm:bound_h_psi}, the norm of $\hat{\bu}$ is bounded as $\|\hat{\bu}\| \le (\|\loss'(\bz)\| + \|\loss'\|_{\lip} \|\bG\|_{\oper} K_\alpha)/\sqrt{p}$. 
Since
$\bG$ has i.i.d.\ $\cN(0, 1)$ entries while $\loss'(z_i)$ has a finite second moment by \Cref{asm:regularity-conditions}-(1), this gives $\|\hat\bu\|^2
\le C$ with probability approaching to $1$ for a positive constant $C$. 

Next, we claim that the condition (12) in \cite{thrampoulidis2018precise} is not necessary. Noting that the condition (12) is used for  \cite[Assumption 2-(b)]{thrampoulidis2018precise}, it suffices to show that the assumption 2-(b) is satisfied given our assumptions. Here we restate the assumption 2-(b) for convenience:
$$
\forall \tau>0, \quad \lim_{c\to+\infty} c^2/(2\tau) - \E\bigl[\env_{\reg}(cH+\Theta) - \reg(\Theta) \bigr] =+\infty. 
$$
By the condition $\PP(\Theta \ne 0)>0$ in \Cref{asm:regularity-conditions}-(3), either $\PP(\Theta>0)>0$ or $\PP(\Theta <0) >0$ holds. Let us consider the case $\PP(\Theta>0)>0$. Define a measurable function $u(H, \Theta)$ of $(H, \Theta)$ as follows:
$$
u(H, \Theta) := -\min(\Theta, 1) I\{\text{$\Theta >0$ and $H<0$}\}. 
$$
Note that $u(H, \Theta)$ is always bounded as $|u(H, \Theta)|\le 1$. By the definition of Moreau envelope, i.e.,  $\env_{\reg}(cH+\Theta) := \argmin_{p\in\R} (cH+\Theta-p)^2/(2\tau) + \reg(p)$, taking the point $p=\Theta + u(H, \Theta)$, we get the lower estimate as follows:
\begin{align*}
\frac{c^2}{2\tau} - \E\bigl[\env_{\reg}(cH+\Theta) - \reg(\Theta) \bigr] 
&\ge \frac{c^2}{2\tau} - \E\bigl[ \frac{(cH-u(H, \Theta))^2}{2\tau} + \reg(\Theta + u(H, \Theta)) - \reg(\Theta) \bigr]\\
&=
\frac{\E[u H]}{\tau}\cdot c
-\frac{\E[u^2]}{2\tau}- \E\Bigl[\reg(\Theta + u) - \reg(\Theta) \Bigr].
\end{align*}
Thus it suffices to show that $\E[uH]>0$ and $\E[\reg(\Theta + u) -\reg(\Theta)]$ is finite for the RHS to diverge as $c\to+\infty$. By the definition of $u=u(\Theta, H)$, the expectation $\E[uH]$ can be written as 
\begin{align*}
    \E[u H] = \E[ - H \min(\Theta,1) I \{\Theta>0, H< 0\}]
    =\E[ |H| \min(\Theta,1) I \{\Theta>0, H< 0\}].
\end{align*}
Here $|H| \min(\Theta, 1)$ is always strictly positive under the event $\{\Theta>0, H>0\}$, and  this event has positive probability $\PP(\Theta>0, H>0) = \PP(\Theta>0)\PP(H>0) = \PP(\Theta>0) \cdot 2^{-1}$ by the assumption $\PP(\Theta>0)>0$ and the independence of $\Theta$ and $H\sim \cN(0,1)$. This means that $\E[uH]$ is strictly positive. Regarding $\E[\reg(\Theta + u) -\reg(\Theta)]$, we have 
$$
\E[\reg(\Theta+u)-
\reg(\Theta)] 
=  \E\Bigl[\bigl\{\reg(\Theta - \min(\Theta,1)) - \reg(\Theta) \bigr\} I\{\Theta>0, H<0\}\Bigr].
$$
Note that $0 < \Theta-\min(\Theta, 1) < \Theta$ for all $\Theta>0$. Then, by the convexity of $\reg$ and the condition $\reg(0) = \min_x\reg(x)$ in \Cref{assu:loss_penalty}, under the event $I\{\Theta>0, H<0\}$ it holds that 
$$
0 > \reg(\Theta - \min(\Theta,1)) - \reg(\Theta) >  - d_{\Theta} \min(\Theta, 1) > - d_\Theta \quad \text{for all $d_\Theta\in \partial \reg(\Theta)$} 
$$ 
By \Cref{asm:regularity-conditions}-(1), $d_\Theta$ has a finite second moment for any choice of sub-derivative $d_\Theta$. Therefore, we have $0 > \E[\reg(\Theta + u)-\reg(\Theta)] > -\E[d_\Theta I\{\Theta>0, H<0\}] > -\infty$ so that $\E[\reg(\Theta+u)-\reg(\Theta)]$ is finite. 

In the other case $\PP(\Theta <0)>0$, we may take $u(H, \Theta) := \min(-\Theta,1) I\{\Theta<0, H>0\}$. Then the same argument leads to $\E[uH]>0$ and $|\E[\reg(\Theta + u) -\reg(\Theta)]| <+\infty$. 

\subsection{Convergence of loss gradient norm squared under \Cref{asm:regularity-conditions}}
\label{subsec:convergence_psi}

Let $\bpsi = \loss'(\bz-\bG\bh)$ and $\loss^*$ be the conjugate of $\loss$. By the same argument in \cite{thrampoulidis2018precise}, restricting the range of $\bpsi$ to a compact set so that the strong duality holds, we observe that $\bpsi$ is a solution to the following min-max problem with probability approaching to $1$:
\begin{align*}
    &\max_{\bpsi\in\R^n} \min_{\bh\in\R^p} \bpsi^\top(\bm\eps-\bG\bh) - \loss^*(\bpsi) + \reg(\sqrt{p}\bh + \btheta)= \max_{\bpsi}  \bpsi^\top\bm\eps - \loss^*(\bpsi) + \bpsi^\top\bG\btheta/\sqrt{p} - \reg^*(\bG^\top\bpsi/\sqrt{p})
\end{align*}
If we write $\hat\bu=\bpsi/\sqrt{p} \in\R^n $ then $\hat\bu$ is the M-estimator of the form:
\begin{align*}
\hat\bu \in \argmin_{\bu\in\R^n} \mathsf{F}(-\bG^\top\bu) + \mathsf{L}(\sqrt{n}\bu), \quad \text{where} \quad  \begin{split}
    &\mathsf{F}: \R^p \mapsto \R, \quad \bv \mapsto \reg^*(-\bv) + \btheta^\top\bv\\
    &\mathsf{L}:\R^n \mapsto \R, \quad \bu \mapsto \loss^*(\sqrt{\frac{p}{n}} \bw) - \sqrt{\frac{p}{n}} \bm\eps^\top \bw 
\end{split}    
\end{align*}
For any $\bh\sim \cN(\bm 0_p, \bI_p)$, $\bg\sim \cN(\bm 0_n, \bI_n)$, for all $c\in \R$ and $\tau>0$ it holds that
\begin{align*}
\frac{1}{p} \bigl(\env_{\mathsf{F}}(c\bh ; \tau) - \mathcal{F}(\bm 0) \bigr) &\pto \frac{c^2}{2\tau} - \E[\env_{\reg}(\frac{c}{\tau} H + \Theta_0; \frac{1}{\tau}) - \reg(0)] \\
\frac{1}{n} \bigl(\env_{\mathsf{L}}(c\bg; \tau) - \mathcal{L}(\bm 0)\bigr) &\pto \frac{c^2}{2\tau} - \E[\env_{\loss}(\sqrt{\delta}\frac{c}\tau G + Z; \frac{\delta}\tau) - \loss(0)]
\end{align*}
Thus, letting $F(c,\tau) := \E[\env_{\reg}(cH+\Theta; \tau)]$ and $L(c, \tau) := \E[\env_{\loss}(cG+\Theta; \tau)]$, \cite{thrampoulidis2018precise} implies that $\|\hat\bu\|^2\pto \tilde\alpha^2$ where $\tilde\alpha$ is the minimizer of the min-max optimization problem:
\begin{align*}
  \inf_{\tilde{\alpha}, \tilde{\tau}_g} \sup_{\tilde{\beta}, \tilde{\tau}_h} \frac{\tilde{\beta}\tilde{\tau}_g}{2} + \delta^{-1} \Bigl(
\frac{\tilde{\alpha}^2\tilde{\beta}}{2\tilde{\tau}_g} - F(\frac{\tilde{\alpha}\tilde{\beta}}{\tilde{\tau}_g}, \frac{\tilde{\beta}}{\tilde{\tau}_g})
\Bigr) - \frac{\tilde{\alpha}\tilde{\tau}_h}{2} - \frac{\tilde{\alpha}\tilde{\beta}^2}{2\tilde{\tau}_h} + \Bigl(\frac{\tilde\alpha \tilde\beta^2}{2\tilde\tau_h} -  L (\sqrt{\delta}\tilde{\beta}, \delta \frac{\tilde{\tau}_h}{\tilde{\alpha}} )\Bigr)
\end{align*}
Thus, by the change of variables $(\tilde\alpha, \tilde\beta, \tilde\tau_g, \tilde\tau_h) \mapsto (\beta, \alpha/\sqrt{\delta}, \tau_h/\sqrt{\delta}, \tau_g/\delta)$ and multiplying the potential by $-\delta$, we are left with $\|\hat\bu\|^2\pto \beta_{*}^2$ where $\beta_{*}$ is the minimizer of 
$$
\sup_{\beta\tau_h}\inf _{\alpha,\tau_g}-
\frac{\alpha \tau_h}{2} - 
\frac{\beta^2\alpha}{2{\tau}_g} + F(\frac{\beta\alpha}{\tau_h}, \frac{\alpha}{{\tau}_h}) + \frac{\beta\tau_g}{2}  +\delta  L (\alpha,  \frac{\tau_g}{\beta}),
$$
which is the same potential in \cite{thrampoulidis2018precise}.

\subsection{Explicit expressions for degrees of freedom}

See \Cref{tab:df-V-examples}.

\begin{table}[!ht]
    \centering
    \caption{
    \label{tab:df-V-examples}
        Explicit formulae for the degrees of freedom and residual degrees of freedom of the estimator $\hat{\btheta}_{I}$ defined using \eqref{eq:def-hbeta}.
        Here the loss functions $\loss(r)$ are: 1) squared loss: $r^2/2$, 2) Huber loss: $r^2/2$ for $|r| \le 1$ and $|r| - 1/2$ for $|r| > 1$;
        and the regularization functions $\reg(b)$ are: 1) ridge penalty: $\frac{\lambda_1}{2} b^2$, 2) lasso penalty: $\lambda_2 | b |$, and 3) elastic net penalty: $\frac{\lambda_1}{2} b^2 + \lambda_2 | b |$.
        And other quantities are: 1) $\hat{S}_{I} = \{j\in [p] \colon \hbeta_{I}(j) \ne 0\}$ is the set of active variables of $\hat{\btheta}_{I}$ and $\bX_{\hat{S}_{I}}$ is the submatrix of $\bX_{I}$ made of columns indexed in $\hat{S}_{I}$, 2) $\hat{T}_{I} = \{ i \in I \colon \loss''(y_i - \bx_i^\top \hat{\btheta}_{I}) > 0 \}$ is the set of detected inliers (active observations), and 3) the matrix $\bD_{I} = \diag(\loss''(\by_{I} - \bX_{I} \hat{\btheta}_{I}))$.
   }
    \small
    \begin{tabular}
    {c c c c}
         \toprule
         \textbf{Loss} & \textbf{Regularizer} & \textbf{Degrees of freedom} & \textbf{Residual degrees of freedom} \\ 
         ($\loss$) & ($\reg$) & ($\df_{I}$) & ($\tr[\bV_{I}]$) \\
         \midrule
         Square &
         Ridge & $\trace\big[\big(\bX_{I}^\top\bX_{I} + \lambda_1 \bI\big)^{-1}\bX_{I}^\top\bX_{I}\big]$ & $|I| - \df_{I}$ \\
         \addlinespace[0.25ex]
         Square &
         Lasso & $|\hat{S}_{I}|$ & $|I| - \df_{I}$ \\
         \addlinespace[0.25ex]
         Square &
         Elastic net &
         $\trace\big[\big(\bX_{\hat{S}_{I}}^\top\bX_{\hat{S}_{I}} + \lambda_1\bI\big)^{-1}\bX^\top_{\hat{S}_{I}}\bX_{\hat{S}_{I}}\big]$ & $|I| - \df_{I}$ \\
         \arrayrulecolor{black!25} \midrule \arrayrulecolor{black}
         Huber & Ridge & $\trace\big[\big(\bX_{I}^\top \bD_{I} \bX_{I} + \lambda_1 \bI\big)^{-1}\bX_{I}^\top \bD_{I} \bX_{I} \big]$ & $|\hat{T}_{I}| - \df_{I}$ \\
         \addlinespace[0.25ex]
         Huber & Lasso & $|\hat{S}_{I}|$ & $|\hat{T}_{I}| - \df_{I}$ \\
         \addlinespace[0.25ex]
         Huber & Elastic net & $\trace\big[\big(\bX_{\hat{S}_{I}}^\top \bD_{I} \bX_{\hat{S}_{I}} +\lambda_1\bI\big)^{-1}\bX^\top_{\hat{S}_{I}} \bD_{I}\bX_{\hat{S}_{I}}\big]$ & $|\hat{T}_{I}| - \df_{I}$ \\
         \arrayrulecolor{black!25} \midrule \arrayrulecolor{black}
         Convex & Convex & $\tr[(\partial/\partial \by_{I}) \bX_{I} \hat{\bbeta}_{I}]$ & $\tr[(\partial/\partial \by_I) \loss'(\by_I - \bX_I \hat{\btheta}_{I})]$ \\
        \bottomrule
    \end{tabular}    
\end{table}

\section{Proofs for results in Section~\ref{sec:specific-examples}}
\label{sec:proofs-sec:specific_examples}

\subsection{Proof of \Cref{prop:monotonicity-ensemble-size}}
Since the risk limit $\cR_M$ is given by $\cR_M = M^{-1} \alpha^2 + (1-M^{-1}) \alpha^2 \etaG$ and $\etaG$ is non-negative in the homogeneous case (see the discussion in \Cref{sec:correlation_signs}), 
it suffices to show $\etaG < 1$. 
We proceed by contradiction. Let $\etaH$ be the associated scalar such that $\etaG = F_\reg(\etaH)$ and $\etaH = F_\loss(\etaG)$. Note that $|\etaH |= |F_\loss(\etaG)|\le c < 1$ by $c < 1$. 
If $|\etaG|=1$ then the equality case $F_\reg(\etaH) = 1$ for the Cauchy--Schwarz inequality implies that with probability $1$, 
$$
\frac{1}{\nu} \cdot \env_{
\reg}' \Bigl(\frac{\beta}{\nu} H + \Theta ; \frac{1}{\nu}\Bigr) - \frac{\beta}{\nu} H
        = \frac{1}{\nu} \cdot \env_{ 
\reg}' \Bigl(\frac{\beta}{\nu}\tilde H + \Theta ; \frac{1}{\nu}\Bigr) - \frac{\beta}{\nu}\tilde H
$$
with $\tilde H := \etaH H + \sqrt{1-\etaH^2} {H}_0$ for any independent standard normals $(H, H_0)$. 
Multiplying the above display by $H_0$ and taking the expectation, the LHS becomes $0$ by the independence of $(H, H_0)$ and $\E[H_0] =0$. 
As for the RHS, using \eqref{eq:CGMT-1c} in \Cref{sys:general_ensemble-M=1} and Stein's lemma, we have
\begin{align*}
\E\Bigl[\Bigl(\frac{1}{\nu} \cdot \env_{
\reg}' \Bigl(\frac{\beta}{\nu}\tilde H + \Theta ; \frac{1}{ \nu}\Bigr) - \frac{\beta}{\nu}\tilde H\Bigr) H_0\Bigr]= \sqrt{1-\etaH^2} (- \kappa  \beta). 
\end{align*}
Thus, we get $0=\sqrt{1-\etaH^2} (- \kappa \beta)$, 
which is a contradiction since $\kappa\beta>0$ and $|\etaH|<1$.

\subsection{Proof of \Cref{prop:monotonicity-risk}}\label{proof:monotonicity-risk}
Let us fix $\psi = (c\delta)^{-1}$ and take the derivative of $\cR_M = M^{-1}\alpha^2 + (1-M^{-1}) \alpha^2\etaG$ with respect to $\phi=\delta^{-1}$. Note $c=\psi/\phi < 1$. With this parameterization, $(\alpha, \beta, \kappa, \nu)$ are all fixed since they depend on $\psi$ only. Below, we derive the partial derivative of $\etaG$ with respect to $\phi$. Using $\psi$ and $\phi$, we observe that $\etaG$ is the unique solution to the fixed point equation
$$
\etaG = F_\reg\circ F_\loss(\etaG; \phi).  
$$
Here, $F_\reg$ does not depend on $\phi$ but $F_\loss$ depends on $\phi$ as:
$$
F_\loss (\etaG; \phi) =  \frac{\phi}{\psi^2 \beta^2}\cdot \E [
          \env_{\loss}'(\alpha G + Z; \kappa) \cdot \env_{\tilde \loss}'(\alpha \tilde G + Z; \kappa)],
$$
Since the map $\etaG\mapsto F_\reg\circ F_\loss(\etaG; \phi)$ is differentiable and $c$-Lipschitz with $c = \psi/\phi <1$ (see \Cref{th:existence-uniqueness-sys:general_ensemble-M=infty}), 
the implicit function theorem implies that $\etaG$ is differentiable with respect to $\phi$ and the derivative satisfies:
$$
\frac{\partial\etaG }{\partial \phi} = (F_\reg\circ F_\loss)' (\etaG) \frac{\partial \etaG}{\partial \phi} + F_\reg'(\etaH) \frac{F_\loss(\etaG)}{\phi}.
$$
Rearranging the above display, we get
$$
\frac{\partial \etaG}{\partial \phi} =\frac{F_\reg'(\etaH)}{\phi} \frac{F_\loss(\etaG)}{1-(F_\reg\circ F_\loss)'(\etaG)} = \frac{F_\reg'(\etaH)}{\phi} \frac{\etaH}{1-(F_\reg\circ F_\loss)'(\etaG)} \ge 0.
$$
where the last inequality follows from the fact that  $F_\reg$ is non-decreasing (\Cref{th:existence-uniqueness-sys:general_ensemble-M=infty}) and $\etaH$ is non-negative (see the discussion on homogeneous cases in \Cref{sec:correlation_signs}). Combined with 
$\partial_{{\phi}} \cR_M = (1-M^{-1}) \alpha^2 \frac{\partial \etaG}{\partial \phi}$, we observe that $\cR_M$ is non-decreasing in $\phi$. Therefore,  
for any two $\phi_1 \le  \phi_2$, we have that 
$$
\cR_{M}(\phi_2, \psi) \ge \cR_{M} (\phi_1, \psi) \ge \inf_{\psi >  \phi_2} \cR_{M}(\phi_1, \psi) \ge \inf_{\psi > \phi_1} \cR_{M}(\phi_1, \psi) \quad \text{for all $\psi\ge \phi_2$}. 
$$
This gives $\inf_{\psi > \phi_2} \cR_{M}(\phi_2, \psi) \ge \inf_{\psi >  \phi_1} \cR_{M}(\phi_1, \psi)$ for any $\phi_1 \le  \phi_2$, which means the map $\phi \mapsto \inf_{\psi >  \phi} \cR_{M}(\phi, \psi)$ is non-decreasing. {Reverting to the original parametrization $(\phi, \psi)\mapsto (\delta, c) = (\phi^{-1}, \psi/\phi)$, we conclude that $\delta \mapsto \inf_{c\in(0,1)}\mathcal{R}_M(\delta, c)$ is non-increasing.
}

\subsection{Derivation of \Cref{sys:ensembles-penalized-least-squares}}

We will first reformulate \Cref{sys:general_ensemble-M=1,sys:general_ensemble-M=infty} in a slightly different set of parameters.
Since the purpose of this reparameterization is to match with existing work, we will also consider regularizer $\reg$ with an explicit regularization level $\lambda$.
The mapping with respect to the parameters in \Cref{sys:general_ensemble-M=1,sys:general_ensemble-M=infty} is as follows: 
$a = \frac{\lambda}{\beta}$, $\tau = \sqrt{c\delta}\frac{\beta}{\nu}$.
Under this parameterization, note that $\frac{\beta}{\nu} = \frac{\tau}{\sqrt{c \delta}}$ and $\tfrac{\lambda}{\nu} = \frac{a \tau}{\sqrt{c \delta}}$. 

\begin{remark}
    [Scaling differences in design]
    It is worth remarking that in the literature on risk characterization of regularized M-estimator under proportional asymptotics, the scaling of $\lambda$ can be slightly different, up to a factor of $\sqrt{c \delta}$.
    One of the reasons for the differences is how the design matrix $\bX$ is scaled.
    We assume that the entries of $\bX \in \RR^{n \times p}$ each have variance $1/p$ and thus each row has a unit average norm squared.
    It is also common to assume that the entries of $\bX$ each have variance $1/n$ and thus each column has a unit average norm squared.
    This brings in a factor of $\sqrt{\delta}$.
    For the subsampled design $\bX_{I} \in \RR^{k \times p}$, we get an additional factor of $\sqrt{c}$.
    Consequently, some expressions may appear different up to this scaling.
\end{remark}

\begin{proof}[Derivation of \Cref{sys:ensembles-penalized-least-squares}]

The derivation is straightforward. 
We will use some simple relationships between proximal operators and Moreau envelopes.
Recall from \eqref{eq:prox_subdiff_relation} that
$
    \env_{f}'( x ; \tau) = \frac{1}{\tau}(x - \prox_{f} ( x ; \tau))
$
so that the derivative identity $
    \env_{f}''( x ; \tau) = \frac{1}{\tau}(1 - \prox_{f}' ( x ; \tau))
$ holds for almost every $x$ by the non-expansiveness of the proximal operator. 
\paragraph*{(1) Case of $m=\ell$ }
The original system of equations is given by \Cref{sys:general_ensemble-M=1}.
For regularizers of the form $\lambda \reg$, from \eqref{eq:CGMT-1a}, we have
\begin{align}
    \alpha^2 &= \EE \big[
      \big(
        \tfrac{\lambda}{\nu} \cdot \env_{\reg}' \big(\tfrac{\beta}{\nu} H + \Theta ; \tfrac{\lambda}{\nu}\big) - \tfrac{\beta}{\nu} H
      \big)^2
    \big] \notag \\
    &= 
    \EE \big[
      \big(
        \prox_{\reg}\big(\tfrac{\beta}{\nu} H + \Theta ; \tfrac{\lambda}{\nu}\big) - \Theta
      \big)^2
    \big] \notag \\
    &=\mathbb{E}\big[\big(\prox_{\reg}\big(\tfrac{\tau}{\sqrt{c\delta}} H + \Theta ; \tfrac{a\tau}{\sqrt{c\delta}}\big) - \Theta\big)^2\big] \label{eq:alpha-squared-3}
\end{align}
where the variables $\tau$ and $a$ are defined as:
\begin{align}
    \tau &= \sqrt{c\delta}\tfrac{\beta}{\nu} \xlongequal{\eqref{eq:CGMT-1b},\eqref{eq:CGMT-1d}} \frac{\alpha\sqrt{\E[\env_{\loss}'(\alpha G + Z; \kappa)^2] }}{\E[\env'_{\loss}(\alpha G + Z; \kappa)\cdot G]  }, \label{eq:tau}\\
    a &= \tfrac{\lambda}{\beta} \xlongequal{\eqref{eq:CGMT-1b}} \frac{\lambda}{\sqrt{c\delta} \sqrt{\E[\env_{\loss}'(\alpha G + Z; \kappa)^2]} }. \label{eq:a}
\end{align}
For squared loss $\loss(x)=x^2/2$, since $\env'_{\loss}(x;\tau) = x / (1 + \tau)$ from \Cref{tab:prox_and_derivatives_ridge_lasso}, squaring the final expression in \eqref{eq:tau} yields
\begin{align}
    \tau^2 &= \alpha^2+\sigma^2.\label{eq:l2-a}
\end{align}
Thus, $\tau^2$ and $\alpha^2$ are the limiting total and excess risks, respectively.
Combining \eqref{eq:alpha-squared-3} and \eqref{eq:l2-a} gives the first desired equation \eqref{eq:amp-bridge-tau}.

Similarly, \eqref{eq:CGMT-1c} after diving by $\tau$ yields
\begin{align*}
      \tfrac{\kappa\beta}{\tau} &= \tfrac{\beta}{\nu \tau} -  \tfrac{\lambda}{\nu \tau} \E\big[
      \env_{\reg}' \big(\tfrac{\tau}{\sqrt{c\delta}}H + \Theta ; \tfrac{a\tau}{\sqrt{c\delta}}\big) \cdot H
    \big].
\end{align*}
Multiplying both sides by $\sqrt{c \delta} a \tau$ and noting that $a \beta = \lambda$, $\tfrac{\beta}{\nu} = \tfrac{\tau}{\sqrt{c \delta}}$, $\tfrac{\lambda}{\nu} = \tfrac{a \tau}{\sqrt{c \delta}}$ then gives
\begin{align}
    \sqrt{c\delta}\kappa\lambda &=  a \tau - \tfrac{(a \tau)^2}{\tau}\E\big[
      \env_{\reg}' \big(\tfrac{\tau}{\sqrt{c\delta}} H + \Theta ; \tfrac{a\tau}{\sqrt{c\delta}}\big) \cdot H
    \big]. \label{eq:kappa-lam}
\end{align}
From Stein's lemma, we have
\begin{align}
    \EE[\env_{\reg}'(\tau H +\Theta;\kappa) \cdot H] = \tau\EE[\env_{\reg}''(\tau H +\Theta;\kappa)], \label{eq:env-prime}
\end{align}
and so \eqref{eq:kappa-lam} reduces to
\begin{align}
    \sqrt{c\delta}\kappa\lambda  &\xlongequal{\eqref{eq:env-prime}} a\tau -   \tfrac{(a\tau)^2}{\sqrt{c\delta}} \E\big[\env_{\reg}'' \big(\tfrac{\tau}{\sqrt{c\delta}} H + \Theta ; \tfrac{a\tau}{\sqrt{c\delta}}\big) 
    \big] \notag\\
    &= 
    a\tau -   a\tau \E\big[
      1 - \prox_{\reg}' \big(\tfrac{\tau}{\sqrt{c\delta}} H + \Theta ; \tfrac{a\tau}{\sqrt{c\delta}}\big)
    \big]  \notag\\
    &=  a\tau \E\big[
      \prox_{\reg}' \big(\tfrac{\tau}{\sqrt{c\delta}} H + \Theta ; \tfrac{a\tau}{\sqrt{c\delta}}\big) \big]. \label{eq:kappa-lam-2}
\end{align}
Using similar manipulations as above, \eqref{eq:tau} reduces to
\begin{align}
    \tau & = \frac{\kappa\sqrt{\E[\env_{\loss}'(\alpha G + Z; \kappa)^2] }}{\E[1 - \prox'_{\loss}(\alpha G + Z; \kappa)]}.\label{eq:tau-2}
\end{align}
Combining \eqref{eq:kappa-lam-2} and \eqref{eq:tau-2}, we get
\begin{align}
    0 &= \kappa\lambda \Big( 1 - \tfrac{a\tau}{\sqrt{c\delta}\kappa \lambda} \E\big[
      \prox_{\reg}' \big(\tfrac{\tau}{\sqrt{c\delta}} H + \Theta ; \tfrac{a\tau}{\sqrt{c\delta}}\big) \big]\Big). \label{eq:kappa-lambda-zero}
\end{align}
From \eqref{eq:a}, under squared $\loss$, we also have
\begin{align}
    1+\kappa &= \tfrac{\sqrt{c\delta}a\tau}{\lambda}. \label{eq:l2-d}
\end{align}
Combining \eqref{eq:kappa-lambda-zero} with \eqref{eq:l2-d} then leads to the second desired equation \eqref{eq:amp-bridge-a}.

\paragraph*{(2) Case of $m\neq\ell$ }

Step 2. $(\etaG, \etaH)$ is the solution to the following 2-scalar fix-point equations:
  \begin{align*}
    \etaG &= \frac{\E \big[
      \big(
        \tfrac{\lambda}{\nu} \cdot \env_{\reg}' \big(\tfrac{\beta}{\nu} H + \Theta ; \tfrac{\lambda}{\nu}\big) - \tfrac{\beta}{\nu} H
      \big)
      \cdot 
      \big(
        \tfrac{\lambda}{\nu} \cdot \env_{\reg}' \big(\tfrac{\beta}{\nu} \tilde H + \Theta ; \tfrac{\lambda}{\nu}\big) - \tfrac{\beta}{\nu} \tilde H
      \big)
    \big]}{\E \big[
      \big(
        \tfrac{\lambda}{\nu} \cdot \env_{\reg}' \big(\tfrac{\beta}{\nu} H + \Theta ; \tfrac{\lambda}{\nu}\big) - \tfrac{\beta}{\nu} H
      \big)^2
    \big]}\\
    \etaH &= c\cdot \frac{\E [
      \env_{\loss}'(\alpha G + Z; \kappa) \cdot \env_{\loss}'(\alpha \tilde G + Z; \kappa)
    ]}{\E[
      \env_{\loss}'(\alpha G + Z; \kappa)^2
    ]},
  \end{align*}
  where 
  \small
  $
  \begin{pmatrix}
    G\\
    \tilde{G}
  \end{pmatrix} \sim N\Bigl(\begin{pmatrix}
    0 \\
    0
  \end{pmatrix}, \begin{pmatrix}
    1 & \etaG\\
    \etaG & 1
  \end{pmatrix}\Bigr)$ and $\begin{pmatrix}
    H\\
    \tilde{H}
  \end{pmatrix} \sim N\Bigl(\begin{pmatrix}
    0 \\
    0
  \end{pmatrix}, \begin{pmatrix}
    1 & \etaH\\
    \etaH & 1
  \end{pmatrix}\Bigr).
  $
  \normalsize
Similarly, by variable substitution, we can rewrite the equations as:
\begin{subequations}    
\label{eq:reform-eq-2}
\begin{align}
      \etaG\alpha^2 &= \mathbb{E}\big[\big(\prox_{\reg}\big(\tfrac{\tau}{\sqrt{c\delta}} H + \Theta ; \tfrac{a\tau}{\sqrt{c\delta}}\big) - \Theta\big)
      \big(\prox_{\reg}\big(\tfrac{\tau}{\sqrt{c\delta}} \tilde{H} + \Theta ; \tfrac{a\tau}{\sqrt{c\delta}}\big) - \Theta\big)\big]\\
    \etaH &= c\cdot \frac{\E [
      \env_{\loss}'(\alpha G + Z; \kappa) \cdot \env_{\loss}'(\alpha \tilde G + Z; \kappa)
    ]}{\E[
      \env_{\loss}'(\alpha G + Z; \kappa)^2
    ]}.  
\end{align}
\end{subequations}
Since $\env_{\loss}'(x;\tau) = x / (1 + \tau)$ for squared loss, we obtain the third desired equation \eqref{eq:amp-bridge-M=infty-a}.
\end{proof}

\subsection{Proximal operators and Moreau envelopes}

See \Cref{tab:prox_and_derivatives_ridge_lasso}.

\begin{table}[!ht]
    \caption{Proximal operators, Moreau envelopes, and their derivatives for ridge (first row) and lasso (second row) regularizers and Huber (third row) loss considered in \Cref{sec:ensembles-bridge-estimators,sec:general-train-loss}.}
    \centering
    \scriptsize
    \begin{tabular}{ccccc}
        \toprule
         $f(x)$ & $\prox_f(x;\tau)$ & $\prox_f'(x;\tau)$ & $\env_f(x;\tau)$ & $\env_f'(x;\tau)$ 
         \\
         \midrule
         $\frac{1}{2}x^2 $ & $\frac{x}{1+\tau}$ & $\frac{1}{1+\tau}$ & $\frac{1}{2} \frac{x^2}{1+\tau}$ & $\frac{x}{1+\tau}$ 
         \\
        \addlinespace[1ex]
        \arrayrulecolor{black!25} \midrule \arrayrulecolor{black}
         $|x|$ & $(|x|-\tau)_+ \sign(x)$ & $\ind\{|x|\geq \tau\}$ &
         $\begin{cases}
             \frac{1}{2 \tau} x^2 &  |x| < \tau \\
             |x| - \frac{1}{2} \tau &  |x| \ge \tau
         \end{cases}$
         & $\min\Big\{\frac{|x|}{\tau},1\Big\}\sign(x)$ 
         \\
         \addlinespace[1ex]
         \arrayrulecolor{black!25} \midrule \arrayrulecolor{black}
         $
         \begin{cases}
             \frac{x^2}{2} &  |x| \le 1 \\
             |x| - \frac{1}{2} &  |x| > 1
         \end{cases}
         $
         & 
         $
         \begin{cases}
         \frac{x}{1+\tau} &  |x| \le 1 + \tau \\
         x - \tau \sign(x) &  |x| > 1 + \tau
         \end{cases}
         $ 
         &
         $
         \begin{cases}
          \frac{1}{1 + \tau} &  |x| \le 1 + \tau \\
          1 &  |x| > 1 + \tau
         \end{cases}
         $
         &
         $
         \begin{cases}
         \frac{1}{2} \frac{x^2}{1+\tau} &  |x| \le 1 + \tau \\
         |x| - \tau - \frac{1}{2} &  |x| > 1 + \tau
         \end{cases}
         $
         & $\left(\frac{|x|}{1+\tau}-1\right)_+ \sign(x)$ \\
         \arrayrulecolor{black}
         \bottomrule
    \end{tabular}
    \label{tab:prox_and_derivatives_ridge_lasso}
\end{table}

\section{Proofs for results in \Cref{sec:ensembles-interpolators}}
\label{sec:proofs-sec:general-properties}

\subsection[]{Proof of \eqref{eq:tau_lower_bound}}\label{sec:proof_tau_lowerbound}
Applying Stein's lemma to \eqref{eq:interpolators-a} and Cauchy--Schwarz inequality, we have
\begin{align*}
1 &= \frac{1}{c\delta} \E\Bigl[\prox_{\reg}'\bigl(\Theta + \tfrac{\tau}{\sqrt{c\delta}} H; \tfrac{a\tau}{\sqrt{c\delta}}\bigr)\Bigr]\\
&= \frac{1}{\sqrt{c\delta}}\frac{1}{\tau} \E\Bigl[H \cdot \Bigl(\prox_{\reg}\bigl(\Theta + \tfrac{\tau}{\sqrt{c\delta}} H; \tfrac{a\tau}{\sqrt{c\delta}}\bigr) - \Theta\Bigr) \Bigr]\\
&\le \frac{1}{\sqrt{c\delta}\tau} \E\Bigl[\Bigl(\prox_{\reg}\bigl(\Theta + \tfrac{\tau}{\sqrt{c\delta}} H; \tfrac{a\tau}{\sqrt{c\delta}}\bigr) - \Theta\Bigr)^2\Bigr]^{1/2}\\
&= \frac{1}{\sqrt{c\delta}\tau} \sqrt{\tau^2-\sigma^2} && (\text{by \eqref{eq:interpolators-tau}).} 
\end{align*}
Taking the square of both sides and rearranging the resulting $\tau^2$ term, we get the lower estimate $\tau^2 \ge \sigma^2/(1-c\delta)$ as desired. 

\subsection[Derivation of \Cref{sys:interpolators} as a limit of \Cref{sys:ensembles-penalized-least-squares} with vanishing regularization]{Derivation of \Cref{sys:interpolators} as a limit of \Cref{sys:ensembles-penalized-least-squares} as $\lambda \to 0^{+}$}\label{subsec:derivation_system_interpolator}

Fix $\reg(x) = |x|^q$ for $q\in \{1, 2\}$. 
Let $(a_\lambda, \tau_\lambda, \xi_\lambda)$ be the solution to \Cref{sys:ensembles-penalized-least-squares} for any $\lambda>0$ and let $(a_{*}, \tau_{*}, \xi_*)$ be the solution to \Cref{sys:interpolators}. 
Note that previous papers showed that $(a_\lambda, \tau_\lambda)$ converges to the solution $(a_*, \tau_*)$ as $\lambda \to (0)^+$; $q=1$ case is given by the proof of Lemma A.1 in \cite{li2021minimum} while the $q=2$ case immediately follows from the explicit formulae of $(a_\lambda, \tau_\lambda)$ and $(a_*, \tau_*)$. 
Below we claim the continuity of $\xi$, i.e., $\xi_\lambda \to \xi_*$.
Denoting $\eta_\lambda = c \xi_\lambda^2/\tau_\lambda^2$ and $\eta_* = c \xi_*^2/\tau_*^2$, what we want to show is $\eta_\lambda \to \eta_*$. 

Observe that the systems for $\xi_\lambda$ and $\xi_*$ read to $\eta_\lambda = F(\eta_\lambda; \tau_\lambda, a_\lambda)$ and $\eta_* = F(\eta_*; \tau_*, a_*)$ where 
\begin{align*}
F(\eta; \tau, a) := \frac{c}{\tau^2} \Bigl\{\E\Bigl[
\Bigl(\prox_\reg\bigl(\Theta + \tfrac{\tau}{\sqrt{c\delta}} H; \tfrac{a\tau}{\sqrt{c\delta}}\bigr) - \Theta\Bigr) \cdot \Bigl(\prox_\reg\bigl(\Theta + \tfrac{\tau}{\sqrt{c\delta}} \tilde H; \tfrac{a\tau}{\sqrt{c\delta}}\bigr) - \Theta\Bigr) 
\Bigr] + \sigma^2\Bigr\}
\end{align*}
with $(H, \tilde H)$ being the mean zero jointly normals such that $\E[H^2]=\E[\tilde H^2] =1$ and $\E[H\tilde H] = \eta$. 
Note that the map $\eta\mapsto F(\eta; \tau_\lambda , a_\lambda)$ is $c$-Lipschitz over $[-1,1]$ by the same argument in \Cref{subsec:proof-exisntence-uniqueness}, while the map $(\tau, a) \mapsto F(\eta; \tau, a)$ is continuous over $(0, \infty)^2$ for any $\eta\in [-1,1]$ by the moment assumption in \Cref{asm:regularity-conditions}-1 and the dominated convergence theorem. 
Then, $|\eta_\lambda - \eta_*|$ is bounded from above as 
\begin{align*}
|\eta_\lambda - \eta_*| &= |F(\eta_\lambda; \tau_\lambda, a_\lambda) - F(\eta_*, \tau_*, a_*)|\\
&\le |F(\eta_\lambda; \tau_\lambda, a_\lambda) - F(\eta_*, \tau_\lambda, a_\lambda)| + |F(\eta_*, \tau_\lambda, a_\lambda) -F(\eta_*, \tau_*, a_*)|\\
&\le c |\eta_\lambda-\eta_*| + |F(\eta_*, \tau_\lambda, a_\lambda) -F(\eta_*, \tau_*, a_*)|
\end{align*}
so that $|\eta_\lambda - \eta_*|\le (1-c)^{-1} |F(\eta_*, \tau_\lambda, a_\lambda) -F(\eta_*, \tau_*, a_*)|$ holds. 
This upper bound converges to $0$ as $\lambda\to 0$ by the continuity of $F(\eta_*, \tau, a)$ in $(\tau, a)$ and the convergence $(\tau_\lambda, a_\lambda) \to (\tau_*, a_*)$.

\subsection[Proof of continuity of full-ensemble risk limit at interpolation]
{Additional details for \Cref{rem:lassoless-fullensemble-continuity-psi=1}}
\label{sec:proof-continuity-Rinfty}

Expanding the product term in the fixed point equation \eqref{eq:full-ensemble-risk-ellq-interpolators} by 
\[
\prox\big(\Theta+ \tfrac{\tau}{\sqrt{c\delta}} H; \tfrac{a\tau}{\sqrt{c\delta}}\big) - \Theta = \tfrac{\tau}{\sqrt{c\delta}} H - \tfrac{a\tau}{\sqrt{c\delta}} \env_{\reg}'\big(\Theta + \tfrac{\tau}{\sqrt{c\delta}} H; \tfrac{a\tau}{\sqrt{c\delta}}\big),
\] we have 
\begin{align*}
    \xi^2 - \sigma^2
    &= \frac{\tau^2}{c\delta} \E[H\tilde H] - 2 \frac{a\tau^2}{c\delta} \E\Big[\tilde{H} \env_{\reg}'\big(\Theta + \tfrac{\tau}{\sqrt{c\delta}}H; \tfrac{a\tau}{\sqrt{c\delta}}\big)\Big] \\
    &+ \frac{(a \tau)^2}{c\delta} \E\Bigl[\env_{\reg}'\big(\Theta + \tfrac{\tau}{\sqrt{c\delta}} H ; \tfrac{a\tau}{\sqrt{c\delta}}\big) \cdot \env_{\reg}'\big(\Theta + \tfrac{\tau}{\sqrt{c\delta}} \tilde{H} ; \tfrac{a\tau}{\sqrt{c\delta}}\big)\Bigr],
\end{align*}
where $\frac{\tau^2}{c\delta} \E[H\tilde H] = \frac{\tau^2}{c\delta} \etaH = \frac{\xi^2}{\delta}$ by the definition $\etaH = c\xi^2/\tau^2$. 
For the second term, realizing $\tilde H = \etaH H + \sqrt{1-\etaH^2} \bar H$ for a standard normal $\bar{H}\sim \cN(0,1)$ independent of $(H, \tilde H, \Theta)$, noting that $\env_{\reg}'(\Theta + \tfrac{\tau}{\sqrt{c\delta}} H)$ is bounded in second moment by \Cref{asm:regularity-conditions}-(1), we have 
\begin{align*}
    \E\Big[\tilde{H} \cdot \env_{\reg}'\big(\Theta + \tfrac{\tau}{\sqrt{c\delta}}H; \tfrac{a\tau}{\sqrt{c\delta}}\big)\Big] 
    &= \etaH\E\Big[H \cdot \env_{\reg}'\big(\Theta + \tfrac{\tau}{\sqrt{c\delta}} H; \tfrac{a\tau}{\sqrt{c\delta}}\big)\Big]\\
    &=\etaH \tfrac{\tau}{\sqrt{c\delta}} \E\Big[\env_{\reg}''\big(\Theta + \tfrac{\tau}{\sqrt{c\delta}} H; \tfrac{a\tau}{\sqrt{c\delta}}\big)\Big] && (\text{by Stein's lemma}) \\
    &= \etaH \tfrac{\tau}{\sqrt{c\delta}} \frac{\sqrt{c\delta}}{a\tau} \E\Big[1-\prox_{\reg}'\big(\Theta + \tfrac{\tau}{\sqrt{c\delta}} H; \tfrac{a\tau}{\sqrt{c\delta}}\big)\Big]\\
    &= \frac{\etaH}{a} (1- c\delta) && (\text{by \eqref{eq:interpolators-a}}).
\end{align*}
Combined with $\etaH = c\xi^2/\tau^2$, we have 
\[
    \frac{a\tau^2}{c\delta} \E\Big[\tilde{H} \env_{\reg}'\big(\Theta + \tfrac{\tau}{\sqrt{c\delta}}H\big)\Big]= \frac{\xi^2}{\delta} (1-c\delta).
\]
Therefore, we get 
$$
\xi^2 \Bigl(1-\frac{1}{\delta} + 2 \frac{(1-c\delta)}{\delta} \Bigr) 
= \sigma ^2 + \frac{(a \tau)^2}{c\delta} \E\Bigl[\env_{\reg}'\big(\Theta + \tfrac{\tau}{\sqrt{c\delta}} H ; \tfrac{a\tau}{\sqrt{c\delta}}\big) \cdot \env_{\reg}'\big(\Theta + \tfrac{\tau}{\sqrt{c\delta}} \tilde{H} ; \tfrac{a\tau}{\sqrt{c\delta}}\big)\Bigr]. 
$$
This means that $\xi^2\to \frac{\delta}{\delta-1}\sigma^2$ holds if and only if the rightmost term converges to $0$ as $c \to (\delta^{-1})^{-}$. This is true if $\reg$ is Lipschitz and $\lim_{c\to (\delta^{-1})^{-}} a \tau = 0$, as $|\env_{\reg}'(x;\tau)|\le \|\reg\|_{\lip}$ for all $x\in\R$ and $\tau>0$.
However, we are not able to provably establish that $a\tau \to 0$ as $c\to\delta^{-1}$.
For the lasso regularizer, we observe in \Cref{fig:atau-limit-to-zero} that $a\tau\to0$ appears to hold as $c\to\delta^{-1}$.

\begin{figure}[!t]
    \centering
  \includegraphics[width=0.9\textwidth]{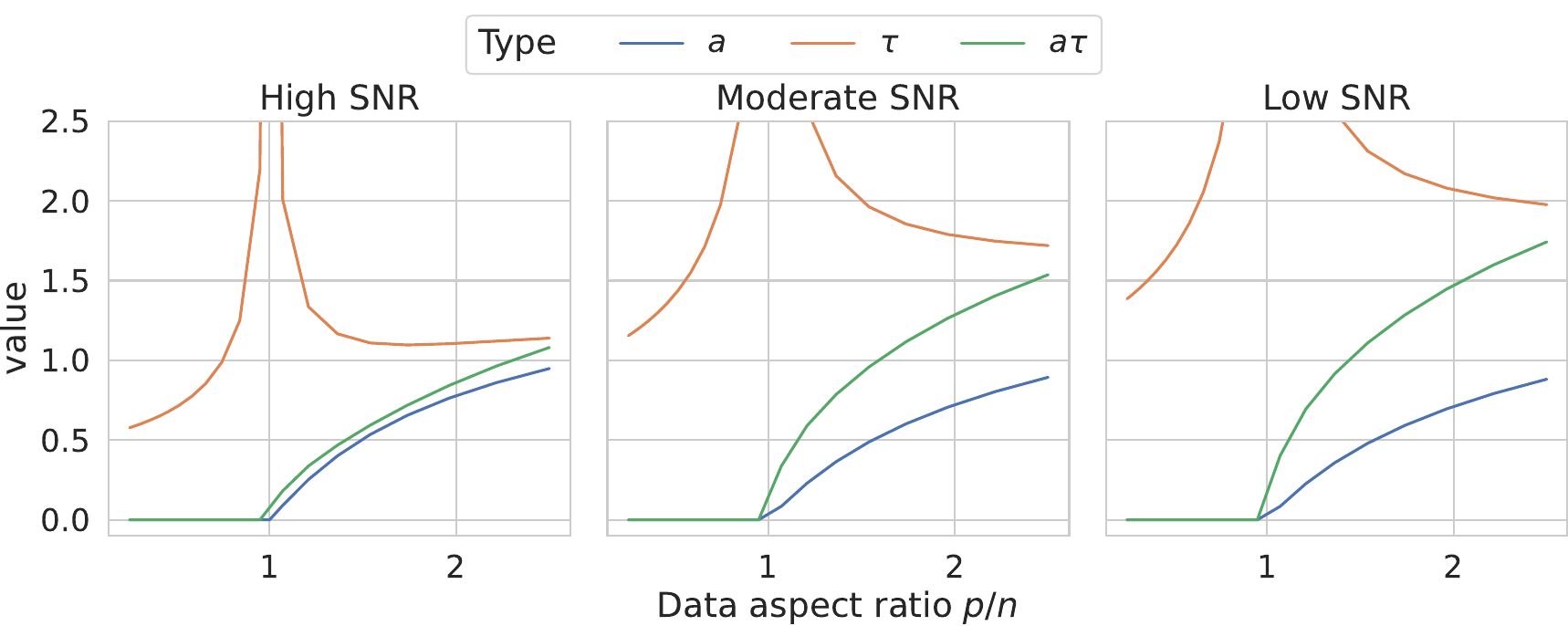}
  \caption{
    Fixed-point quantities for lassoless at different data aspect ratios $p/n$.
    The data model is given by \eqref{model:sparse} with signal strength $\rho=0.5$ and sparsity levels $s=0.2$ at different noise levels $\sigma$.
    \emph{Left}: High SNR $\sigma=0.5$.
    \emph{Middle}: Moderate SNR $\sigma=1$.
    \emph{Right}: Low SNR $\sigma=1.2$.
  }
  \label{fig:atau-limit-to-zero}
\end{figure}

\section{Proof for results in  \Cref{sec:discussion}}
\label{sec:appendix:anisotropic_deterministic}

\subsection{Proof of \Cref{th:contraction_Sigma}}\label{proof:contraction_Sigma}

Here the proof is similar to \Cref{subsec:proof-exisntence-uniqueness}. Our goal is to show
\begin{enumerate}[leftmargin=7mm]
    \item
    $|F_{\loss}(\etaG)| \le \sqrt{c\tilde c}$ and $|F_{\vreg}(\etaH)| \le 1$ for all $\etaG\in [-1,1]$ and $\etaH\in [-1,1]$. 
    \item
    $F_\loss$ and $F_{\vreg}$ are non-decreasing, differentiable, and the compositions $F_\loss \circ F_{\vreg}$ and $F_{\vreg}\circ F_\loss$ are $\min\{c, \tilde c\}$-Lipschitz.
    \item \Cref{sys:general_ensemble-M=infty_sigma} admits a unique solution $(\etaGstar,\etaHstar)
    \in [-1, 1] \times [-\sqrt{c\tilde c},\sqrt{c\tilde c}]$.
\end{enumerate}
The first claim immediately follows from the definition of $(F_\loss, F_{\vreg})$,
the Cauchy--Schwarz inequality, and \eqref{eq:anisotropic_1}-\eqref{eq:anisotropic_2}. The third claim follows from the second claim by the argument in \Cref{subsec:proof-exisntence-uniqueness} using Brouwer's fixed-point theorem. Thus, it suffices to show the second claim. 

For $F_\loss$, we have already shown in \Cref{subsec:proof-exisntence-uniqueness} that $F_\loss$ is differentiable with 
$$
0 \le F_\loss'(\etaG) \le \frac{\alpha\tilde \alpha}{\beta\tilde\beta} \Bigl(\frac{c \tilde\nu}{\kappa} \wedge  \frac{\tilde c \nu}{\tilde \kappa}\Bigr). 
$$
For the derivative of $F_\vreg$, we will use the following lemma that generalizes \Cref{lm:varphi_derivative}.
\begin{lemma}\label{lemma:phi_t}
Given two Lipschitz functions $\bm{f}, \tilde{\bm f}:\R^p\to \R^p$, define $\phi:[-1,1]\to\R$ as 
    $$
    \phi(t) = \E[\bm{f}(\bh)^\top \tilde{\bm{f}}(t \bh + \sqrt{1-t^2}  \bh')]
    $$
    where the expectation is taken with respect to independent standard normal vectors $\bh, \ \bh' \sim \cN(\bm{0}_p, \bI_p)$. Then the map $\phi$ is differentiable with its derivative given by 
    $$
    \phi'(t) = \E\Bigl[\tr\Bigl(\frac{\partial \bm f}{\partial \bx}(\bh) ^\top \frac{\partial \tilde{\bm f}}{\partial \bx}(t\bh + \sqrt{1-t^2}\bh')\Bigr)\Bigr]
    $$
    where $\partial\bm{f}/(\partial \bx)$ and $\partial\tilde{\bm{f}}/(\partial \bx)$ are the weak derivatives of $\bm{f}$ and $\tilde{\bm{f}}$. 
\end{lemma}
We omit the proof of this lemma, as it follows directly from the proof of \Cref{lm:varphi_derivative}.
Now we apply this lemma with \Cref{lemma:generalized_prox} below. Here we recall the generalized Moreau envelope:
$$
\forall \bv\in \R^p, \ \forall\bm\Lambda\succ \bm{0}_{p\times p} \quad 
\env_{\vreg}(\bv; \bm\Lambda) := \min_{\bx\in \R^p} \frac{1}{2} (\bv-\bx)^\top\bm\Lambda^{-1} (\bv-\bx) + \vreg(\bx) 
$$
and denote the unique minimizer by $\bm\prox_{\vreg}(\bv; \bm\Lambda)$. Recall $\nabla \env_{\vreg}(\bv; \bm\Lambda) = \bm\Lambda^{-1} (\bv-\bm\prox_\vreg(\bv; \bm\Lambda))$, where $\nabla$ is the gradient with respect to $\bv$. 
\begin{lemma}\label{lemma:generalized_prox}
Fix $\bm\Lambda\succ \bm{0}_{p\times p}$. Then, for any $\bx, \by\in \R^p$, we have 
\begin{align}\label{eq:lipschitz_continuity_moreau_envelope}
            \|\bm\Lambda^{1/2}(\nabla\env_{\vreg}(\bx;\bm\Lambda) - \nabla \env_{\vreg}(\by;\bm\Lambda))\|_2 \le \|\bm\Lambda^{-1/2}(\bx-\by)\|_2. 
\end{align}
    Therefore, $\bx\mapsto \nabla \env_{\vreg}(\bx;\bm\Lambda)$ is $\lambda_{\min}(\bm{\Lambda})^{-1}$-Lipschitz. Furthermore, there exists a weak derivative $\frac{\partial}{\partial \bm{x}} \nabla \env_{\vreg}(\bx; \bm\Lambda)$, denoted by $\nabla^2 \env_{\vreg}(\bx; \bm\Lambda)$, which is positive semi-definite and satisfies
    $$
    \bm{0}_{p\times p} \preceq \bm{\Lambda}^{1/2} \nabla^2 \env_{\vreg}(\bx; \bm\Lambda) \bm{\Lambda}^{1/2} \preceq \bm{I}_{p}.
    $$
\end{lemma}
\begin{proof}
Note that \eqref{eq:lipschitz_continuity_moreau_envelope} follows easily from the firm nonexpansiveness of the proximal operator $\bm\prox_{\vreg}(\cdot; \bm\Lambda)$ with respect to the inner product $\langle\bv,\bu\rangle_{\bm\Lambda^{-1/2}} := \bv^\top \bm\Lambda^{-1}\bu$. For completeness, however, we prove \eqref{eq:lipschitz_continuity_moreau_envelope} directly.

By the KKT conditions, $\bm\prox_{\vreg}(\bx; \bm\Lambda)$ is also a minimizer of the following convex function $f_{\bx}:\R^p\mapsto \R$:
    $$
    f_{\bx}(\bm{p}) := \frac{1}{2}(\bx-\bp)^\top\bm\Lambda^{-1}(\bx-\bp) - \frac{1}{2}(\bm\prox_{\vreg}(\bx; \bm\Lambda)-\bp)^\top\bm\Lambda^{-1}(\bm\prox_{\vreg}(\bx; \bm\Lambda)-\bp) + \vreg(\bp),
    $$
    so that $f_{\bx}(\bm\prox_\vreg(\bx; \bm\Lambda)) \le  f_{\bx}(\bm\prox_{\vreg}(\by; \bm\Lambda))$. By symmetry, we also have $f_{\by}(\bm\prox_\vreg(\by; \bm\Lambda)) \le  f_{\by}(\bm\prox_{\vreg}(\bx; \bm\Lambda))$. Putting them together, we get
    $$
    f_{\bx}(\bm\prox_\vreg(\bx; \bm\Lambda)) + f_{\by}(\bm\prox_\vreg(\by; \bm\Lambda)) \le f_{\bx}(\bm\prox_{\vreg}(\by; \bm\Lambda)) + f_{\by}(\bm\prox_{\vreg}(\bx; \bm\Lambda)). 
    $$
    Substituting the definitions of $f_{\bx}$ and $f_{\by}$ into the above display and rearranging it with $\nabla \env_{\vreg}(\bv; \bm\Lambda) = \bm\Lambda^{-1} (\bv-\bm\prox_\vreg(\bv; \bm\Lambda))$, we obtain
    \begin{align*}
       \|\bm\Lambda^{1/2}(\nabla\env_{\vreg}(\bx;\bm\Lambda) - \nabla \env_{\vreg}(\by;\bm\Lambda))\|_2^2 
        &\le    (\nabla \env_{\vreg}(\bx; \bm\Lambda)-\nabla \env_{\vreg}(\by; \bm\Lambda))^\top (\bx-\by).
    \end{align*}
    Applying the Cauchy--Schwarz inequality to the RHS, we obtain \eqref{eq:lipschitz_continuity_moreau_envelope}.  

    By Rademacher's theorem, $\bx\mapsto \nabla \env_{\vreg}(\bx; \bm\Lambda)$ is differentiable almost everywhere, and for any differentiable points $\bx$, the Jacobian is $\partial \nabla \env_{\vreg}(\bx; \bm\Lambda) = \nabla^2 \env_{\vreg}(\bx; \bm\Lambda)$, which is positive semi-definite since the mapping $\bx\mapsto\env_{\vreg}(\bx;\bm\Lambda)$ is convex. Let us fix $\bx$ at which $\nabla\env$ is differentiable. For any $\bu\in \R^p$ and any $t>0$, using the established Lipschitz continuity \eqref{eq:lipschitz_continuity_moreau_envelope}, we have  
    $$
         \|\bm\Lambda^{1/2}(\nabla\env_{\vreg}(\bx+t\bm\Lambda^{1/2}\bu;\bm\Lambda) - \nabla \env_{\vreg}(\bx;\bm\Lambda))\|_2 \le \|\bm\Lambda^{-1/2} \cdot t \bm \Lambda^{1/2}\bu\|_2 = t \|\bu\|_2. 
    $$
    Dividing both sides by $t$,  taking the limit $t\to 0$, and using the chain rule, we have 
    $$
    \|\bm\Lambda^{1/2}\nabla^2 \env_{\vreg}(\bx; \bm\Lambda) \bm\Lambda^{1/2} \bu\|_2 \le  \|\bu\|_2
    $$
    for all $\bu\in\R^p$. 
    This means $\bm\Lambda^{1/2} \nabla^2 \env_{\vreg}(\bx; \bm\Lambda)\bm\Lambda^{1/2}\preceq \bm{I}_p$, completing the proof. 
\end{proof}

Let us apply \Cref{lemma:phi_t} and \Cref{lemma:generalized_prox} to $F_{\vreg}$. Note that $F_{\vreg}$ can be written as 
$$
F_\vreg(\etaH) = \frac{1}{\alpha\tilde\alpha} \frac{1}{p}\E[\bm{f}(\bm h)^\top \tilde{\bm{f}}(\tilde\bh)], \quad \tilde\bh = \etaH \bh + \sqrt{1-\etaH^2}\bh', 
$$
where $\bm{f}, \tilde{\bm{f}}:\R^p\to\R^p$ are defined as 
\begin{align*}
    \bm{f}(\bx) &:= \tfrac{\beta}{\nu} \bx -  \nu^{-1} \bm\Sigma^{-1/2} \nabla \env_{\vreg} (\btheta+ \tfrac{\beta}{\nu} \bm{\Sigma}^{-1/2} \bx; (\nu \bm\Sigma)^{-1} ) \\
\tilde{\bm{f}}(\bx) &:= \tfrac{\tilde \beta}{\tilde \nu} \bx- \tilde\nu^{-1} \bm\Sigma^{-1/2} \nabla \env_{\tilde \vreg} (\btheta+ \tfrac{\tilde \beta}{\tilde \nu} \bm{\Sigma}^{-1/2} \bx; (\tilde \nu \bm\Sigma)^{-1}). 
\end{align*}
By  \Cref{lemma:generalized_prox}, $\bm{f}$ and $\tilde{\bm{f}}$ are Lipschitz and their weak derivatives satisfy 
\begin{align*}
    \frac{\partial \bm{f}}{\partial \bx} &= \frac{\beta}{\nu}\Bigl(\bI_p-(\nu\bm\Sigma)^{-1/2} \nabla^2 \env_{\vreg} (\btheta+ \tfrac{\beta}{\nu} \bm{\Sigma}^{-1/2} \bx; (\nu \bm\Sigma)^{-1}) (\nu \bm{\Sigma})^{-1/2} \Bigr), \quad  \bm{0}_{p\times p} \preceq \frac{\partial \bm{f}}{\partial \bx} \preceq \frac{\beta}{\nu} \bm{I}_p\\
     \frac{\partial \tilde{\bm{f}}}{\partial \bx} &= \frac{\tilde \beta}{\tilde \nu}\Bigl(\bI_p-(\tilde \nu\bm\Sigma)^{-1/2} \nabla^2 \env_{\tilde \vreg} (\btheta+ \tfrac{\tilde \beta}{\tilde \nu} \bm{\Sigma}^{-1/2} \bx; (\tilde \nu \bm\Sigma)^{-1}) (\tilde \nu \bm{\Sigma})^{-1/2} \Bigr), \quad  \bm{0}_{p\times p} \preceq \frac{\partial \tilde{\bm{f}}}{\partial \bx} \preceq \frac{\tilde \beta}{\tilde \nu} \bm{I}_p.
\end{align*}
Therefore, using \Cref{lemma:phi_t}, noting that $\frac{\partial {\bm{f}}}{\partial \bx}$ and $\frac{\partial \tilde{\bm{f}}}{\partial \bx}$ are now symmetric, 
we know that $F_{\vreg}$ is differentiable with its derivative given by 
$$
F_{\vreg}'(\etaH) = \frac{1}{\alpha\tilde \alpha p} \E\Bigl[\tr\bigl\{
\frac{\partial \bm{f}}{\partial \bx}(\bh)  \frac{\partial \tilde{\bm{f}}}{\partial \bx}(\tilde \bh)
\bigr\}\Bigr].
$$
Now we use the following fact: $0 \le \tr[\bA\bB]\le \tr[\bA]$ for any positive semi-definite matrix $\bA, \bB$ with $\bm{0}_{p\times p} \preceq \bA, \bB\preceq \bI_{p}$. This inequality immediately follows from the identity $\tr[\bA\bB]=\tr[\bA \bB^{1/2} \bB^{1/2}] = \tr[\bB^{1/2} \bA \bB^{1/2}]$ and $\bm{0}_{p\times p}\preceq \bB^{1/2}\bA\bB^{1/2} \preceq \bA$. Combine this with $ \bm{0}_{p\times p} \preceq \frac{\partial \tilde{\bm{f}}}{\partial \bx} \preceq \frac{\tilde \beta}{\tilde \nu} \bm{I}_p$, $ \bm{0}_{p\times p} \preceq \frac{\partial {\bm{f}}}{\partial \bx} \preceq \frac{ \beta}{ \nu} \bm{I}_p$, we have 
\begin{align*}
     0 \le F_{\vreg}'(\etaH) \le \frac{1}{p\alpha\tilde\alpha} \biggl(\frac{\tilde\beta}{\tilde\nu} \E\Bigl[
     \tr\bigl\{
    \frac{\partial \bm{f}}{\partial \bx}(\bh)
     \bigr\}
     \Bigr] \wedge 
     \frac{\beta}{\nu} \E\Bigl[\tr\bigl\{\frac{\partial \tilde{\bm{f}}}{\partial \bx} (\tilde\bh)\bigr\}\Bigr]
     \biggr).
\end{align*}
Using Stein's lemma and the fact that $(\alpha,\beta, \nu, \kappa)$ satisfies \eqref{eq:anisotropic_3}, recalling the definition of $\bm{f}(\bx)$, we have 
\begin{align*}
    \E\Bigl[
     \tr\bigl\{
    \frac{\partial \bm{f}}{\partial \bx}(\bh)
     \bigr\}
     \Bigr]  = \E[\bh^\top \bm{f}(\bh)] = \E\biggl[\bh^\top \Bigl\{
         \frac{\beta}{\nu}\bh - \frac{\bm{\Sigma}^{-1/2}}{\nu} \nabla \env_{\vreg}\Bigl(\btheta + \frac{\beta}{\nu}\bm{\Sigma}^{-1/2}\bh; \frac{\bm{\Sigma}^{-1}}{\nu}
    \Bigr)
     \Bigr\}\biggr] = p \kappa \beta. 
\end{align*}
By the same argument, we have $\E\Bigl[\tr\bigl\{\frac{\partial \tilde{\bm{f}}}{\partial \bx} (\tilde\bh)\bigr\}\Bigr] = p \tilde\kappa\tilde\beta$. Substituting these equations to the previous display, we get $F_\vreg'(\etaH) \le \frac{\beta\tilde\beta}{\alpha\tilde\alpha} (\frac{\kappa}{ \tilde\nu} \wedge  \frac{\tilde \kappa}{\nu})$. 
 
Putting the above display together, 
$$
0 \le F_\loss'(\etaG) \le \frac{\alpha\tilde \alpha}{\beta\tilde\beta} \Bigl(\frac{c \tilde\nu}{\kappa} \wedge  \frac{\tilde c \nu}{\tilde \kappa}\Bigr), \quad 
0\le  F_\vreg'(\etaH) \le \frac{\beta\tilde\beta}{\alpha\tilde\alpha} \Bigl(\frac{\kappa}{ \tilde\nu} \wedge  \frac{\tilde \kappa}{\nu}\Bigr).
$$
Therefore, by the chain rule, we know that $\etaH \mapsto F_\loss \circ F_{\vreg}(\etaH)$ and $\etaG \mapsto F_{\vreg} \circ F_\loss(\etaG)$ are both $(c\wedge \tilde c)$-Lipschitz. This completes the proof.

\subsection{Heuristic proof of \Cref{conjecture}}\label{proof:conjecture} 
We outline a heuristic proof strategy for \Cref{conjecture}, leaving a complete rigorous proof to future work.

Consider the change of variables $\bb \mapsto \bh = p^{-1/2}\bm{\Sigma}^{1/2}(\bb-\btheta)$, and define $\bG=\sqrt{p}\,\bX\bm{\Sigma}^{-1/2}\in\R^{n\times p}$, whose entries are i.i.d.\ $\cN(0,1)$ (note that this $\bh$ is different from the Gaussian vectors $\bh, \tilde\bh$ appearing in \Cref{sys:general_ensemble-M=infty_sigma}). Then the regularized M-estimator $\hat\btheta$ and its residual vector can be expressed as 
\[
\hat\btheta=\sqrt{p}\,\bm{\Sigma}^{-1/2}\hat\bh+\btheta, 
\quad 
\by-\bX\hat\btheta = \bm\eps-\bG\hat\bh,
\]
where $\hat\bh$ solves the M-estimation problem with a non-separable regularizer $\mathsf{R}(\cdot)$:
\[
\hat\bh \in \argmin_{\bh\in\R^p} 
\sum_{i\in I} \loss(\eps_i-\bg_i^\top\bh) 
+ \mathsf{R}(\bh), 
\quad 
\mathsf{R}(\bh) :=  \vreg(\sqrt{p}\, \bm{\Sigma}^{-1/2}\bh+\btheta).
\]
Assume for simplicity that $\vreg$ is $\tau$-strongly convex (otherwise apply the smoothing argument in \Cref{sec:ridge_smoothing}), so that $\mathsf{R}(\cdot)$ is $(\tau p)/\|\bm\Sigma\|_{\oper}$-strongly convex. 
By \Cref{lm:derivative}, the mappings $\bG\mapsto \hat\bh$ and $\bG\mapsto \bpsi:=\sum_{i\in I}\loss'(\eps_i-\bg_i^\top\hat\bh)\be_i$ are differentiable, with derivatives
\[
\forall i\in[n],\, j\in[p], \quad 
\frac{\partial \hat\bh}{\partial g_{ij}} 
= \bA\Bigl(\be_j\be_i^\top\bpsi - \bG^\top\bD\be_i\be_j^\top\hat\bh\Bigr), 
\quad 
\frac{\partial \bpsi}{\partial g_{ij}} 
= -\bD\bG\bA\be_j\be_i^\top\bpsi - \bV\be_i\be_j^\top\hat\bh.
\]

Following the argument of \Cref{sec:proof-thm:corr-sigerror-reserror}, the Gaussian Poincar\'e inequality yields
\[
\frac{\hat\bh^\top\tilde\bh}{\|\hat\bh\|\|\tilde\bh\|} = \hetaG+\op(1),
\quad
\frac{\hat\bpsi^\top\tilde\bpsi}{\|\bpsi\|\|\tilde\bpsi\|} = \hetaH+\op(1),
\]
for random variables $\hetaG,\hetaH\in[-1,1]$ independent of $\bG$. 

Using the contraction property in \Cref{th:contraction_Sigma}, it remains to show
\[
\hetaH = F_\loss(\hetaG)+\op(1),
\quad
\hetaG = F_\vreg(\hetaH)+\op(1).
\]

The first relation $\hetaH = F_\loss(\hetaG)+\op(1)$ follows from the same reasoning as in \Cref{sec:proof-thm:corr-sigerror-reserror}, as $F_\loss$ is not altered from the original isotropic and random signal setting.   

The main challenge lies in establishing the second one, $\hetaG = F_{\vreg}(\hetaH)+\op(1)$.  
By carefully tracing the argument in \Cref{sec:proof-thm:corr-sigerror-reserror}, applying \Cref{lm:approx_multi_normal} with  $\bF = \begin{bmatrix}\frac{\bpsi}{\sqrt{p}\beta} & \frac{\tilde\bpsi}{\sqrt{p}\tilde\beta}\end{bmatrix} \in \R^{n\times 2}$ and $\bz=\bG \be_j \in \R^{n}$, combined with the concentration results in \eqref{eq:solution_concentrate}, one can construct a random vector $\bw_j\in\R^2$ for each $j\in[p]$  such that
\begin{align}\label{eq:bw_j_marginal_normal}
\bw_j \mid (\btheta, \bz, \bG^{-j}, I, \tilde I) \deq \cN(\bm{0}_2,\bI_2), 
\text{ here } \bG^{-j} \text{ is } \bG \text{ with } j\text{th column removed}     
\end{align}
and obtain
\[
\hetaG = \mathcal{F}_\vreg(\hetaH)+\op(1),
\]
where the map $\mathcal{F}_\vreg:[-1,1]\to\R$ is given by
\footnotesize
\[
\mathcal{F}_\vreg(\eta) 
:= \frac{1}{p\alpha\tilde\alpha}
\E\!\Biggl[\Bigl(\tfrac{\bm{\Sigma}^{-1/2}}{\nu}\nabla \env_{\vreg}\bigl(\btheta+\tfrac{\beta}{\nu}\bm{\Sigma}^{-1/2}\bw_\eta;\tfrac{\bm{\Sigma}^{-1}}{\nu}\bigr) - \tfrac{\beta}{\nu}\bw_\eta\Bigr)^\top
\Bigl(\tfrac{\bm{\Sigma}^{-1/2}}{\tilde\nu}\nabla \env_{\tilde\vreg}\bigl(\btheta+\tfrac{\tilde\beta}{\tilde\nu}\bm{\Sigma}^{-1/2}\tilde\bw_\eta;\tfrac{\bm{\Sigma}^{-1}}{\tilde\nu}\bigr) - \tfrac{\tilde\beta}{\tilde\nu}\tilde\bw_\eta\Bigr)\Biggr],
\]
\normalsize
with $\bw_\eta=(w_{\eta,j})_{j\in[p]}$ and $\tilde\bw_\eta=(\tilde w_{\eta,j})_{j\in[p]}$ defined by
\[
\forall j\in[p], \quad 
\begin{pmatrix}
w_{\eta,j}\\ \tilde w_{\eta,j}
\end{pmatrix}
= 
\begin{pmatrix}
1 & \eta\\ \eta & 1
\end{pmatrix}^{1/2}
\bw_j.
\]
Therefore, it suffices to prove $\mathcal{F}_\vreg=F_{\vreg}$, but at this point, this cannot be concluded directly. Indeed, by \eqref{eq:bw_j_marginal_normal}, we have $(w_{\eta,j},\tilde w_{\eta,j}) \sim \cN\left(\bm{0}_2, \begin{pmatrix}1 & \eta\\ \eta & 1\end{pmatrix}\right)$ marginally, but the vectors $(w_{\eta,j},\tilde w_{\eta,j})_{j\in[p]}$ are not necessarily independent across $j$ {since \eqref{eq:bw_j_marginal_normal} is not strong enough to claim the independence of $(\bw_j)_{j\in[p]}$}. Now recalling the definition of $F_{\vreg}$ in \Cref{sys:general_ensemble-M=infty_sigma}, the bivariate Gaussian vectors $(\bh, \tilde{\bh})=(h_j, \tilde h_j)_{j\in[p]}$ in $F_{\vreg}$ are explicitly assumed to have the \textbf{i.i.d.} marginals $\cN\left(\bm{0}_2, \begin{pmatrix}1 & \eta\\ \eta & 1\end{pmatrix}\right)$. This distinction prevents us from directly concluding $\mathcal{F}_\vreg=F_{\vreg}$. 

Nonetheless, we can conclude that $\mathcal{F}_\vreg = F_{\vreg}$ when $\bSigma$ is diagonal (not necessarily the identity) and $(\vreg, \tilde\vreg)$ are separable such that  $\vreg(\bx) = \sum_{j\in[p]} \reg_j(x_j)$, $\tilde\vreg(\bx) = \sum_{j\in[p]}\tilde\reg_j(x_j)$. 
Indeed, since the Moreau envelope can be expressed as a separable form in this case, we can write $\mathcal{F}_\vreg$ and $F_{\vreg}$ as 
\footnotesize
\begin{align*}
\mathcal{F}_\vreg(\eta)
&= \frac{1}{p\alpha\tilde\alpha}\sum_{j=1}^p
\E\Bigl[\bigl(\tfrac{\Sigma_{jj}^{-1/2}}{\nu}\env_{\reg_j}'(\theta_j+\tfrac{\beta}{\nu}\Sigma_{jj}^{-1/2}w_{\eta,j};\tfrac{\Sigma_{jj}^{-1}}{\nu}) - \tfrac{\beta}{\nu}w_{\eta,j}\bigr)
\bigl(\tfrac{\Sigma_{jj}^{-1/2}}{\tilde\nu}\env_{\tilde{\reg}_j}'(\theta_j+\tfrac{\tilde\beta}{\tilde\nu}\Sigma_{jj}^{-1/2}\tilde w_{\eta,j};\tfrac{\Sigma_{jj}^{-1}}{\tilde\nu}) - \tfrac{\tilde\beta}{\tilde\nu}\tilde w_{\eta,j}\bigr)\Bigr]\\
{F}_\vreg(\eta)
&= \frac{1}{p\alpha\tilde\alpha}\sum_{j=1}^p
\E\Bigl[\bigl(\tfrac{\Sigma_{jj}^{-1/2}}{\nu}\env_{\reg_j}'(\theta_j+\tfrac{\beta}{\nu}\Sigma_{jj}^{-1/2}h_{j};\tfrac{\Sigma_{jj}^{-1}}{\nu}) - \tfrac{\beta}{\nu}h_{j}\bigr)
\bigl(\tfrac{\Sigma_{jj}^{-1/2}}{\tilde\nu}\env_{\tilde{\reg}_j}'(\theta_j+\tfrac{\tilde\beta}{\tilde\nu}\Sigma_{jj}^{-1/2}\tilde h_{j};\tfrac{\Sigma_{jj}^{-1}}{\tilde\nu}) - \tfrac{\tilde\beta}{\tilde\nu}\tilde h_{j}\bigr)\Bigr],
\end{align*}
\normalsize
for all $\eta\in [-1,1]$, which implies $\mathcal{F}_\vreg = F_\vreg$ thanks to $(w_{\eta, j}, \tilde{w}_{\eta, j})=^d (h_{j}, \tilde{h}_{j})$ marginally for $j\in[p]$.
In this case, the independence of $(w_{\eta, j}, \tilde{w}_{\eta, j})$ across $j$ is not required since the expectation $\E$ is taken inside $\textstyle \sum_j$. 

In summary, the key technical obstacle is to establish independence of $(\bw_j)_{j\in[p]}$ in \eqref{eq:bw_j_marginal_normal}, which lies beyond the reach of \Cref{lm:approx_multi_normal}. Overcoming this requires new tools, and we leave this gap for future research.

\clearpage
\section{Additional numerical simulations}\label{sec:additional_numerical_illustrations}

\subsection[Minimum ell1-norm interpolator]{Minimum $\ell_1$-norm interpolator}

\begin{figure}[!ht]
  \includegraphics[width=0.9\textwidth]{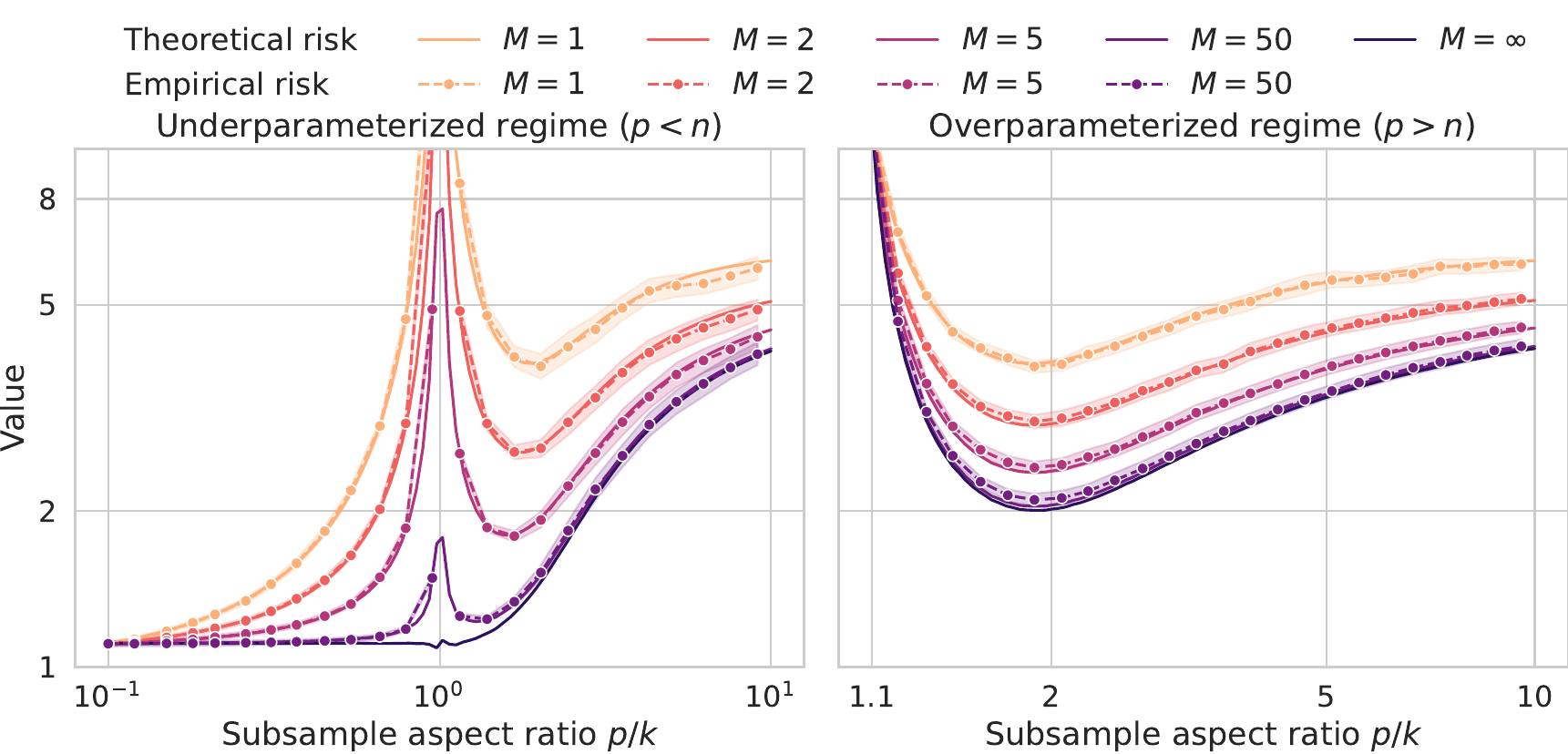}
  \caption{
    The prediction risk for lasso ensemble at different subsample aspect ratios $p/k$ with regularization parameter $\lambda=0.001$ and varying ensemble size $M$.
    The solid lines represent the theoretical risks, the dashed lines represent the empirical risks averaged over $50$ simulations, and the shaded regions represent the empirical standard errors.
    The data model is given by \eqref{model:sparse} with signal strength $\rho=2$, noise level $\sigma=1$, and support proportion $s=0.1$.
    \emph{Left}: underparameterized regime when $p/n=0.1$ and $n = 2000$.
    \emph{Right}: overparameterized regime when $p/n=1.1$ and $n=500$.
  }
  \label{fig:effect-of-M}
\end{figure}

\begin{figure}[!ht]
    \centering
    \includegraphics[width=0.95\textwidth]{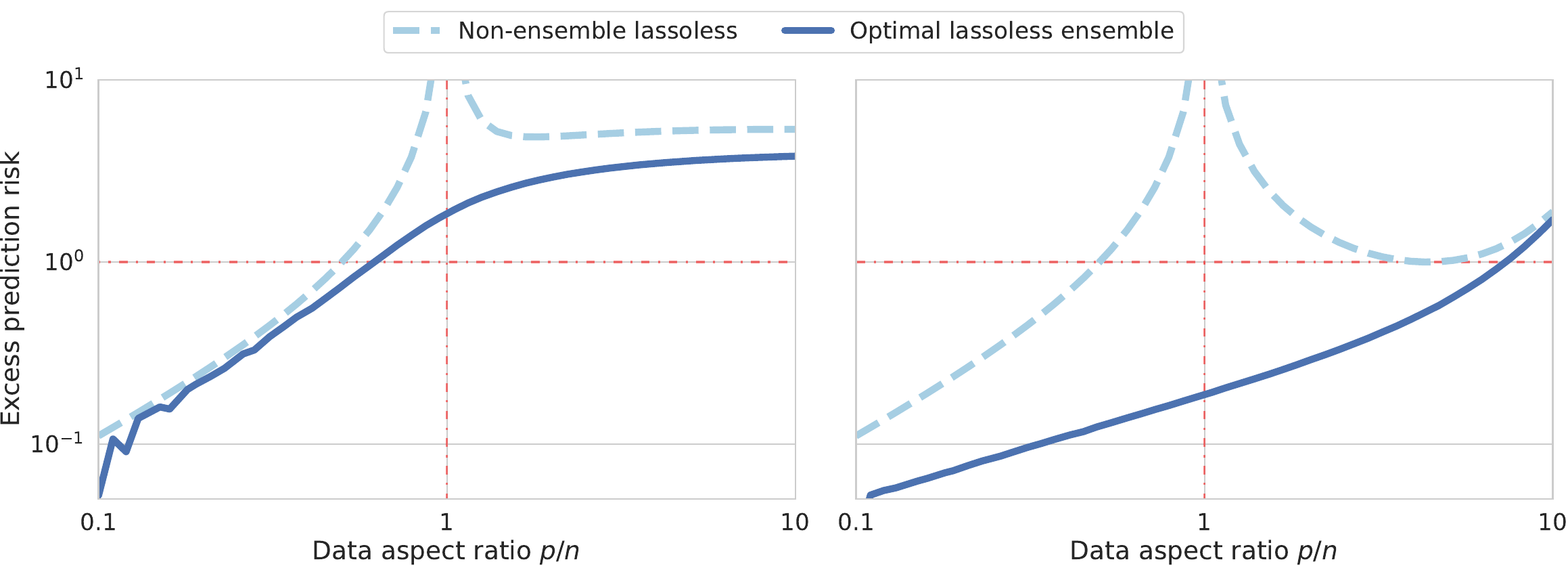} 
    \caption{
        \textbf{Optimal subsample risk of the lassoless ensemble is monotonic in the data aspect ratio.}
        Excess risk of the lasso and optimal lasso ensemble, at different data aspect ratios $p/n$ ranging from $0.1$ to $10$. 
        The data model is given by \eqref{model:sparse} with signal strength $\rho=2$, noise level $\sigma=1$, data aspect ratio $p/n=0.1$, feature size $p=500$, and varying support proportion $s$.
        \emph{Left}: dense regime with $s=0.9$.
        \emph{Right}: sparse regime with $s=0.01$. 
    }
    \label{fig:optimum_phis_overparameterized-2-optrisk}
\end{figure}

\clearpage
\subsection{Lasso under varying covariate distributions}
\begin{figure}[!ht]
    \centering
  \includegraphics[width=0.75\textwidth]{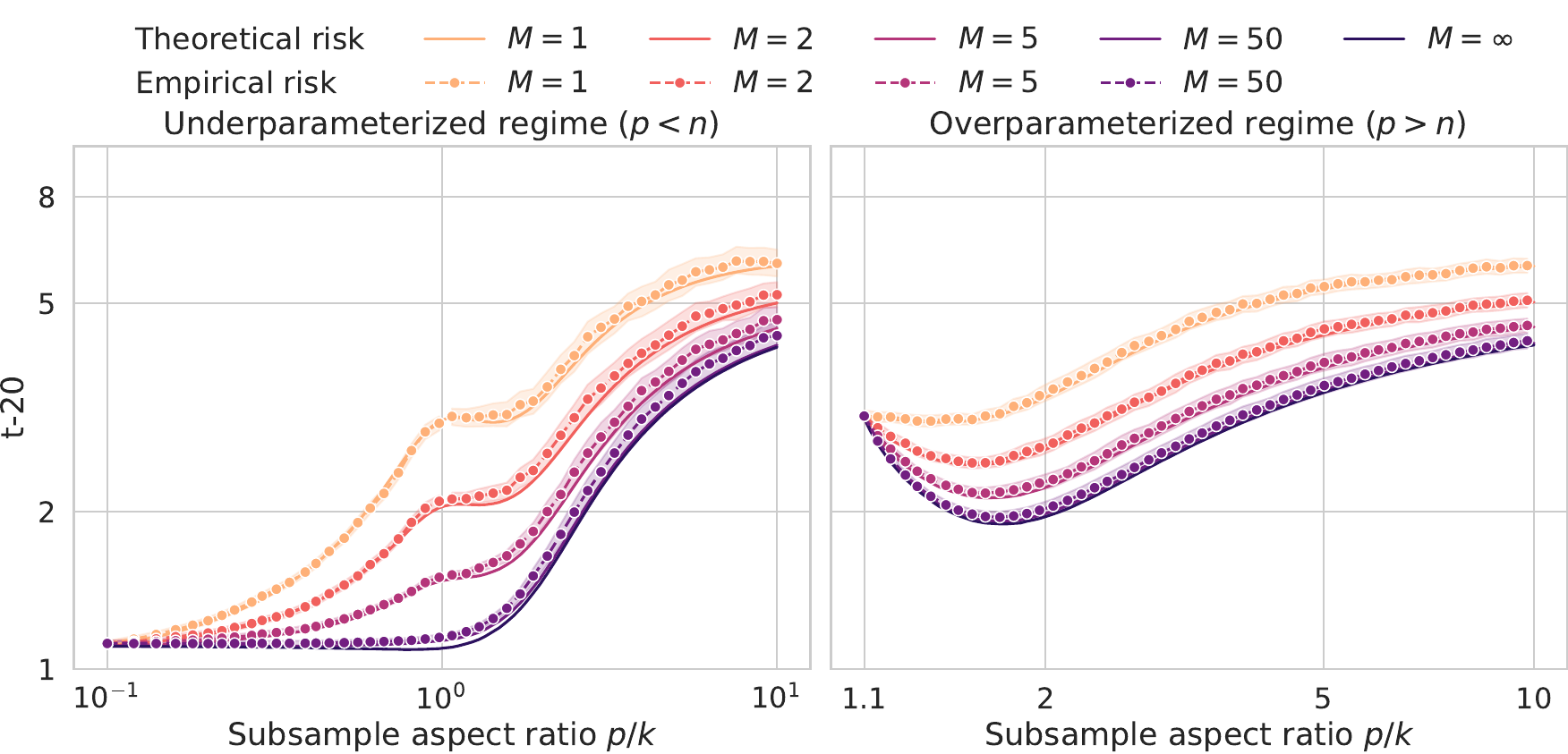}
  \includegraphics[width=0.75\textwidth]{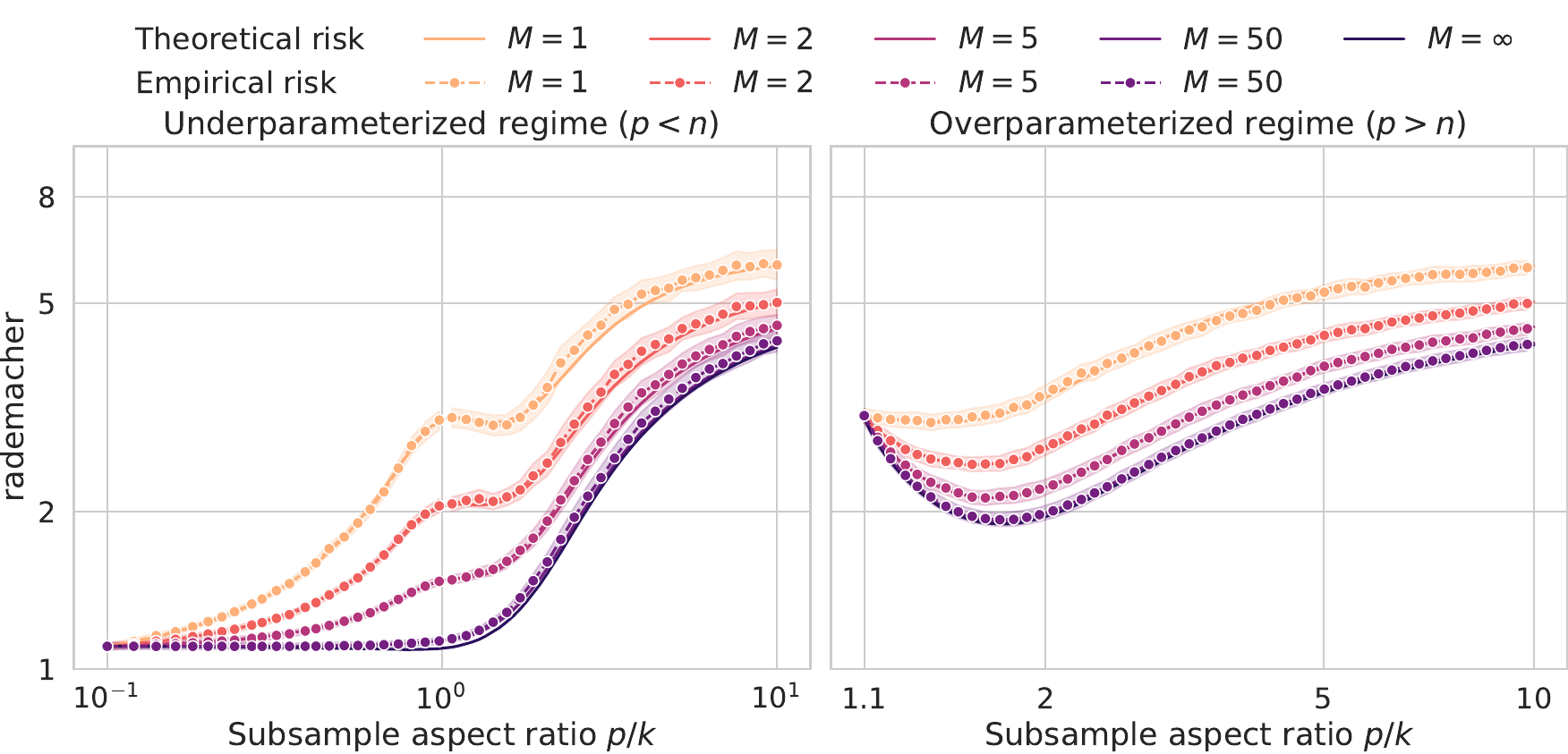}
  \includegraphics[width=0.75\textwidth]{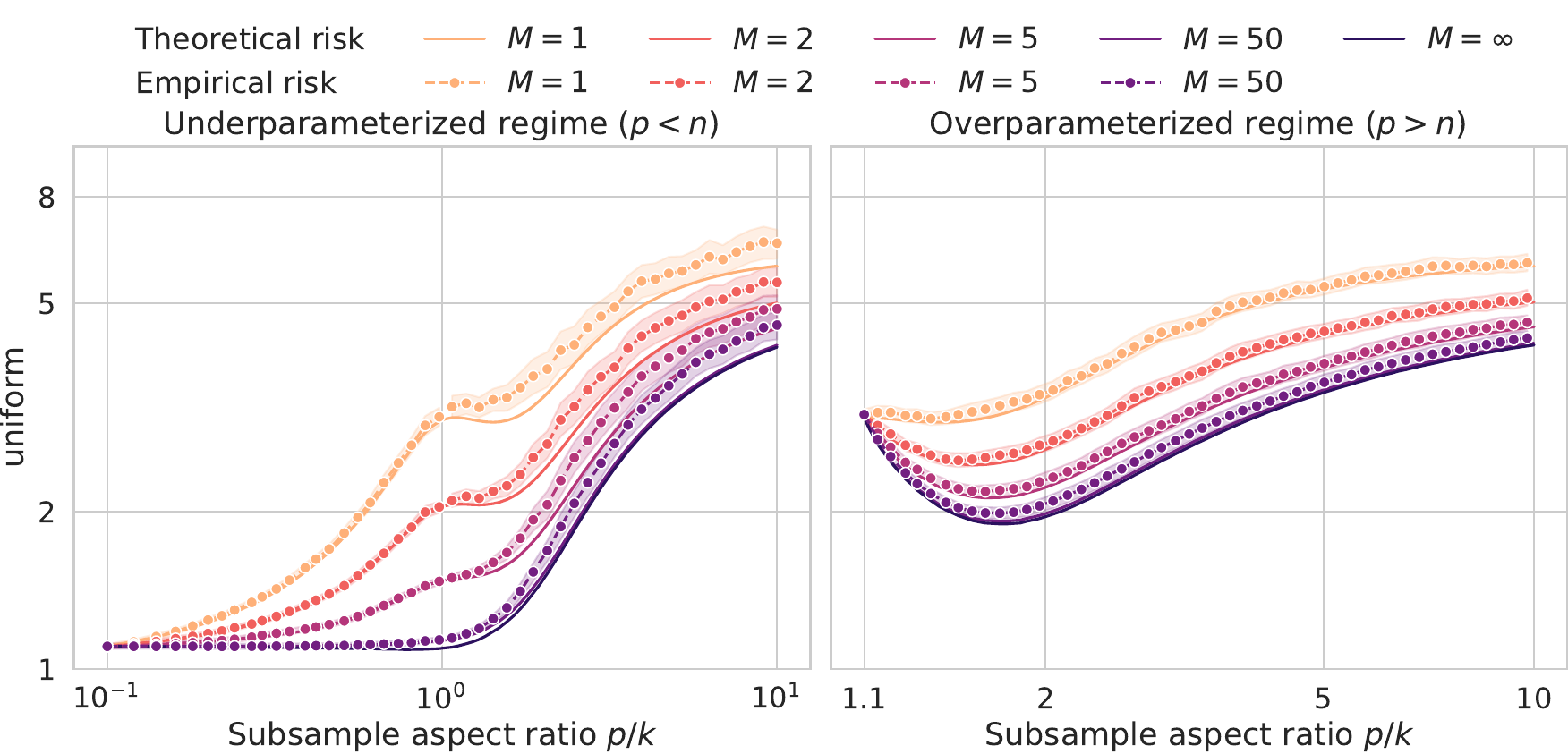}
  \caption{
    The prediction risk for lasso ensemble at different subsample aspect ratios $p/k$ with regularization parameter $\lambda=0.1$ and varying ensemble size $M$.
    The covariates $X_{ij}$'s are i.i.d. and follow (a) $t_{20}$ (b) Rademacher distribution, and (c) uniform distribution in $[-1/1]$, with a proper normalization such that the mean is zero and variance is one.
    The solid lines represent the theoretical risks, the dashed lines represent the empirical risks under Gaussian covariate averaged over $50$ simulations, and the shaded regions represent the empirical standard errors.
    The data model is given by \eqref{model:sparse} with covariate distribution replaced by the corresponding distribution, and with signal strength $\rho=2$, noise level $\sigma=1$, and support proportion $s=0.1$.
    \emph{Left}: underparameterized regime when $p/n=0.1$ and $n = 2000$.
    \emph{Right}: overparameterized regime when $p/n=1.1$ and $n=500$.
    }
  \label{fig:effect-of-M-cov}
\end{figure}

\clearpage

\subsection{Optimal subagging versus optimal regularization}\label{sec:optimal-subagging-versus-optimal-regularization}

\subsubsection{Squared loss, lasso regularizer}

See \Cref{fig:advantage_subsampling_overparameterized}.

\begin{figure}[!ht]
    \centering
    \includegraphics[width=0.95\textwidth]{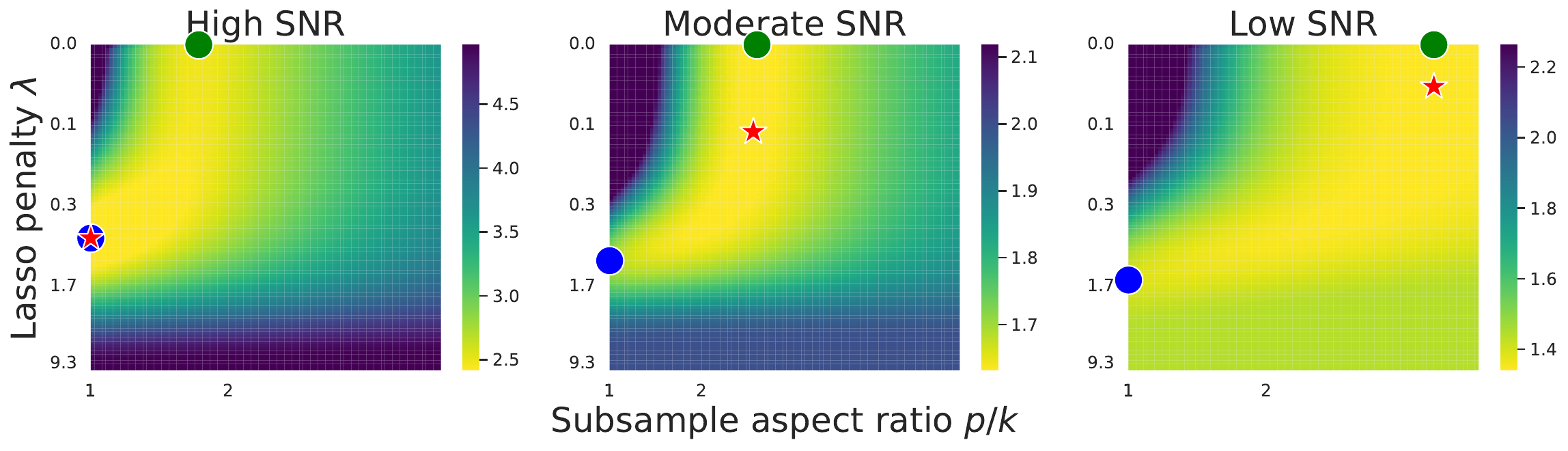}
    \caption{
        Heatmaps of theoretical prediction risk in $\lambda$ and {subsample aspect ratio $p/k = (p/n) \cdot c^{-1}$} of full lasso ensemble in the overparameterized regime ($p/n=1.1$) and sparse regime ($s=0.2$).
        The data model is given by \eqref{model:sparse} with support proportion $s=0.2$ and noise level $\sigma=1$ at different signal levels $\rho$.
        Blue dots ($\color{pyblue}\mysolidcircle$): the optimal lasso. 
        Green dots ($\color{pygreen}\mysolidcircle$): the optimally subsampled lassoless regression. 
        Red stars ($\color{pyred}\star$): the optimal lasso ensemble.
        \emph{Left}: high SNR $\rho=2$. The optimal lasso ($\color{pyblue}\mysolidcircle$) is better than the optimally subsampled lassoless ($\color{pygreen}\mysolidcircle$).
        \emph{Middle}: moderate SNR $\rho=1$. The optimal lasso ensemble ($\color{pyred}\star$) is better than both the optimal lasso ($\color{pyblue}\mysolidcircle$) and the optimally subsampled lassoless($\color{pygreen}\mysolidcircle$).
        \emph{Right}: low SNR $\rho=0.67$. The optimally subsampled lassoless ($\color{pygreen}\mysolidcircle$) is better than the optimal lasso ($\color{pyblue}\mysolidcircle$).
    }
    \label{fig:advantage_subsampling_overparameterized}
\end{figure}

\subsubsection{Huber loss, unregularized}

See \Cref{fig:huber}.

\begin{figure}[!ht]
    \centering
    \includegraphics[width=0.95\textwidth]{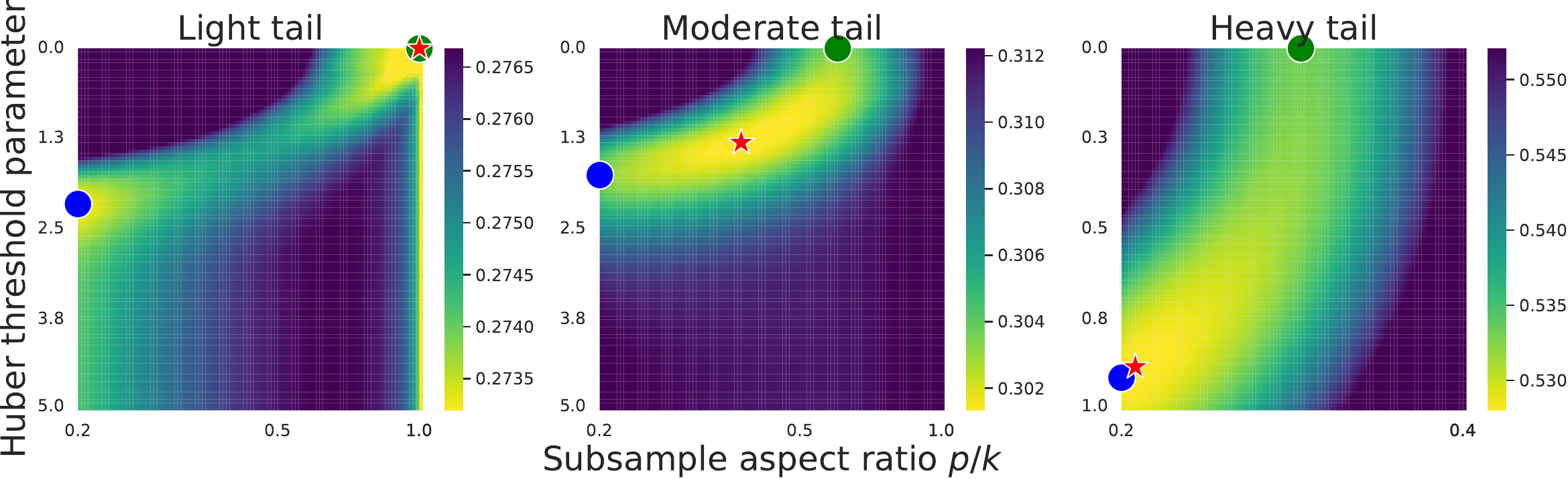}
    \caption{
        Heatmaps of theoretical prediction risk in Huber threshold parameter and subsample aspect ratio $p/k \in [p/n, 1)$ of full Huber ensemble in the underparameterized regime ($p/n=0.2$). Blue dots ($\color{pyblue}\mysolidcircle$): the optimal Huber regression. 
        Green dots ($\color{pygreen}\mysolidcircle$): the optimally subsampled LAD regression (Huber regression with threshold parameter $\to 0^+$). 
        Red stars ($\color{pyred}\star$): the optimal Huber ensemble.
        \emph{Left}: noise follows Student's $t$ distribution $t_{20}$. The optimally subsampled LAD ($\color{pygreen}\mysolidcircle$) is better than the optimal Huber ($\color{pyblue}\mysolidcircle$).
        \emph{Middle}: noise follows Student's $t$ distribution $t_{10}$ The optimal subsample Huber ($\color{pyred}\star$) is better than the  both optimal subsample LAD ($\color{pygreen}\mysolidcircle$) and the optimal Huber ($\color{pyblue}\mysolidcircle$).
        \emph{Right}: noise follows Student's $t$ distribution $t_{2}$. The optimal Huber ($\color{pyblue}\mysolidcircle$) is better than the optimally subsampled LAD ($\color{pygreen}\mysolidcircle$).
    }
    \label{fig:huber}
\end{figure}

\clearpage

\subsubsection[ell1 regularized Huber regression: varying threshold parameter]{$\ell_1$-regularized Huber regression: varying threshold parameter}

See \Cref{fig:huber-l1-rho}.

\begin{figure}[!ht]
    \centering
    \includegraphics[width=0.95\textwidth]{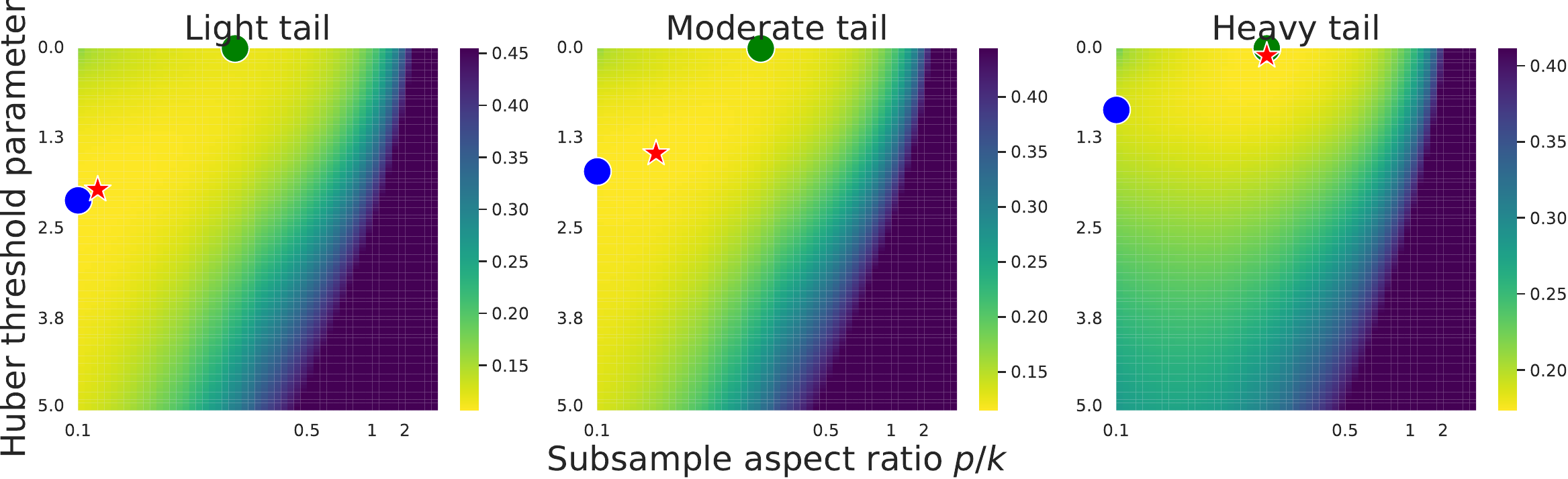}
    \caption{
        Heatmaps of theoretical prediction risk in Huber threshold parameter and {subsample aspect ratio  $p/k = (p/n) \cdot c^{-1}$} of full $\ell_1$-regularized Huber ensemble with lasso regularization level $0.5$ in the underparameterized regime ($p/n=0.1$).  Blue dots ($\color{pyblue}\mysolidcircle$): the optimal $\ell_1$-regularized Huber regression. 
        Green dots ($\color{pygreen}\mysolidcircle$): the optimally subsampled $\ell_1$-regularized LAD regression ($\ell_1$-regularized Huber regression with huber threshold parameter $\to 0^+$). 
        Red stars ($\color{pyred}\star$): the optimal $\ell_1$-regularized Huber ensemble.
        \emph{Left}: noise follows Student's $t$ distribution $t_{20}$. The optimal $\ell_1$-regularized Huber ($\color{pyblue}\mysolidcircle$) is better than the optimally subsampled $\ell_1$-regularized LAD ($\color{pygreen}\mysolidcircle$).
        \emph{Middle}: noise follows Student's $t$ distribution $t_{10}$. The optimal $\ell_1$-regularized Huber ensemble ($\color{pyred}\star$) attains nearly the same risk as both the optimal $\ell_1$-regularized Huber ($\color{pyblue}\mysolidcircle$)  and the optimally subsampled $\ell_1$-regularized LAD ($\color{pygreen}\mysolidcircle$).
        \emph{Right}: noise follows Student's $t$ distribution $t_{2}$. The optimally subsampled $\ell_1$-regularized LAD ($\color{pygreen}\mysolidcircle$) is better than the optimal $\ell_1$-regularized Huber ($\color{pyblue}\mysolidcircle$). 
    }
    \label{fig:huber-l1-rho}
\end{figure}

\subsubsection[ell1 regularized Huber regression: varying regularization level]{$\ell_1$-regularized Huber regression: varying regularization level}

See \Cref{fig:huber-l1-lam}.

\begin{figure}[!ht]
    \centering
    \includegraphics[width=0.95\textwidth]{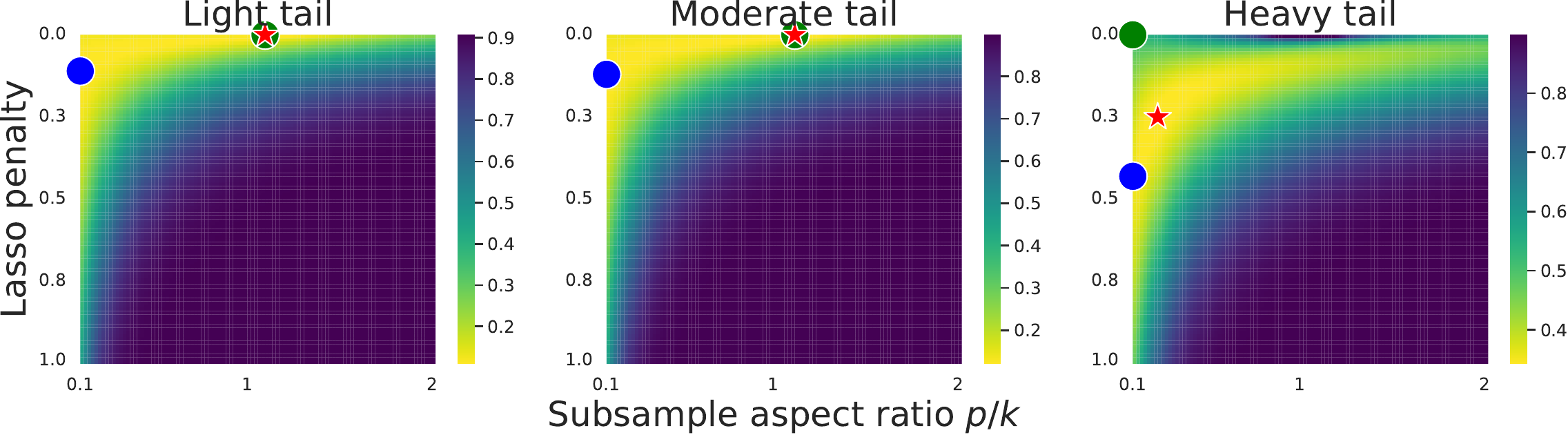}
    \caption{
        Heatmaps of theoretical prediction risk in lasso regularization level and {subsample aspect ratio $p/k = (p/n) \cdot c^{-1}$} of full $\ell_1$-regularized Huber ensemble with Huber threshold parameter $10$ in the underparameterized regime ($p/n=0.1$).  Blue dots ($\color{pyblue}\mysolidcircle$): the optimal $\ell_1$-regularized Huber regression. 
        Green dots ($\color{pygreen}\mysolidcircle$): the optimally subsampled lassoless Huber regression ($\ell_1$-norm interpolator when $p/k \ge 1$ and unregularized Huber regression when $p/k < 1$). 
        Red stars ($\color{pyred}\star$): the optimal $\ell_1$-regularized Huber ensemble.
        \emph{Left}: noise follows Student's $t$ distribution $t_{20}$. The optimally subsampled lassoless Huber ($\color{pygreen}\mysolidcircle$) is better than the optimal $\ell_1$-regularized Huber ($\color{pyblue}\mysolidcircle$).
        \emph{Middle}: noise follows Student's $t$ distribution $t_{10}$. The optimally subsampled lassoless Huber ($\color{pygreen}\mysolidcircle$) is better than the optimal $\ell_1$-regularized Huber ($\color{pyblue}\mysolidcircle$).
        \emph{Right}: noise follows Student's $t$ distribution $t_{2}$
    The optimal $\ell_1$-regularized Huber ensemble ($\color{red}\star$) is better than both the optimally subsampled lassoless Huber ($\color{pygreen}\mysolidcircle$)  and the optimal $\ell_1$-regularized Huber ($\color{pyblue}\mysolidcircle$)
    }   
    \label{fig:huber-l1-lam}
\end{figure}

\clearpage

\subsection{Details of numerical experiments}
\label{sec:experiments-details}

\subsubsection{Data model for lasso experiments}
\label{model:linaer-sparsity}

We consider a linear model whose signal variables are generated from two-point distribution:
\begin{align}
    y = \bx^{\top}\btheta  + \eps,\qquad  \bx\sim\cN(\zero, p^{-1}\bI_p), \quad \eps\sim\cN(0,\sigma^2),\quad  \btheta =  \overset{i.i.d.}{\sim} s \cP_{\rho/\sqrt{s}}+ (1-s)\cP_0, \label{model:sparse}
\end{align}
where $\cP_c$ denotes the Dirac measure at point $c \in \RR$, and $\rho > 0$ is some given quantity that determines the signal magnitude.
Here, $s\in(0,1)$ is the support proportion.
The signal-to-noise-ratio (SNR) under the above model obeys
\[
    \SNR = \frac{\EE[(\bx^{\top}\btheta)^2]}{\sigma^2} = \frac{s\cdot (\rho^2/s)}{\sigma^2} = \frac{\rho^2}{\sigma^2}.
\]

\subsubsection[Data model for ell1 regularized Huber experiments]{Data model for $\ell_1$-regularized Huber experiments}\label{subsec:Huber-ex}

We consider the $\ell_1$-regularized Huber regression
$$
\hat\btheta_I \in \argmin_{\btheta\in\R^p} \sum_{i\in I} 
\mathsf{Huber}(y_i - \bx_i^\top\btheta; \rho) + \lambda \|\btheta\|_1
$$
where $\lambda\ge 0$ is a regularization parameter and $\text{Huber}(\cdot;\mu)$ is the Huber loss with threshold parameter $\mu>0$, which is defined as follows \cite{huber1964robust}:
\begin{equation}
    \label{eq:huber_loss}
    \text{For all $x\in\R$}, \qquad \mathsf{Huber}(x; \mu) 
    := \env_{|\cdot|}(x;\mu) = 
    \begin{cases}
         \frac{1}{2\mu} x^2 & \text{ if } |x| \le \mu \\
         |x| - \frac{1}{2} \mu & \text{ if } |x| > \mu.
    \end{cases}
\end{equation}
Observe that $\mathsf{Huber}(\cdot;\mu)$ behaves like the squared loss for large $\mu$ and like the absolute loss for small $\mu$. 
More precisely, for all $x\in\R$ it holds that 
\begin{align*}
    \lim_{\mu\to 0^+} \mathsf{Huber}(x;\mu) = |x| \quad \text{and} \quad \lim_{\mu\to+\infty} \mu \cdot \mathsf{Huber}(x;\mu) = x^2/2. 
\end{align*}
We consider the linear model $y=\bx^\top\btheta + z$ where the marginal distribution of the signal is set to the mixture $\epsilon \cN(0,1) + (1-\epsilon) \cP_0$ (recall that $\cP_0$ denotes the Dirac measure at $0$) with support proportion $\epsilon=0.1$, while the noise $z$ follows Student's t-distribution $t_{d}$ for some degree of freedom $d\ge 2$, which will be specified for each numerical simulation.
\end{document}